\documentclass[11pt,a4paper]{article}

\usepackage{latexsym}
\usepackage{amsmath}
\usepackage{amssymb}
\usepackage{amsthm}
\usepackage{amscd}
\usepackage{mathrsfs}
\usepackage[all]{xy}
\usepackage{graphicx}
\usepackage{bm}
\usepackage{comment}
\usepackage{url}
\usepackage{enumerate}
\usepackage{calligra}
\usepackage[shortlabels]{enumitem}
\usepackage{cite}

\usepackage[usenames,dvipsnames]{color}
\newcommand{\teruhisa}[1]{{\color{green} \sf  [#1]}}
\newcommand{\ildar}[1]{{\color{red} \sf  [#1]}}

\theoremstyle{plain}

\newtheorem{thm}{Theorem}[section]
\newtheorem{prop}[thm]{Proposition}

\theoremstyle{definition}
\newtheorem{defn}[thm]{Definition}

\theoremstyle{plain}
\newtheorem{lem}[thm]{Lemma}
\newtheorem{cor}[thm]{Corollary}

\newtheorem{rem}[thm]{Remark}

\newtheorem{cons}[thm]{Construction}

\newtheorem*{notation*}{Notation}

\makeatletter

\@addtoreset{equation}{section}
\makeatother

\usepackage{relsize}
\usepackage[bbgreekl]{mathbbol}
\usepackage{amsfonts}
\DeclareSymbolFontAlphabet{\mathbb}{AMSb} 
\DeclareSymbolFontAlphabet{\mathbbl}{bbold}
\newcommand{\Prism}{{\mathlarger{\mathbbl{\Delta}}}}

\DeclareMathOperator{\Spec}{Spec}

\DeclareMathOperator{\id}{id}

\DeclareMathOperator{\Hom}{Hom}

\DeclareMathOperator{\Mod}{Mod}
\DeclareMathOperator{\Spf}{Spf}

\DeclareMathOperator{\Spa}{Spa}

\DeclareMathOperator{\dR}{dR}

\DeclareMathOperator{\Kos}{Kos}

\DeclareMathOperator{\et}{\acute{e}t}

\DeclareMathOperator{\Et}{\acute{E}T}

\DeclareMathOperator{\Inf}{inf}
\DeclareMathOperator{\proet}{pro\acute{e}t}
\DeclareMathOperator{\Zar}{Zar} 
\DeclareMathOperator{\profet}{prof\acute{e}t}
\DeclareMathOperator{\cont}{cont}
\DeclareMathOperator{\fet}{f\acute{e}t}

\DeclareMathOperator{\Shom}{\mathscr{H}\text{\kern -3pt {\calligra\large om}}\,}

\DeclareFontFamily{U}{mathc}{}
\DeclareFontShape{U}{mathc}{m}{it}%
{<->s*[1.03] mathc10}{}
\DeclareMathAlphabet{\mathcal}{U}{mathc}{m}{it}

\newcommand{\bA}{\mathbb{A}}

\newcommand{\bF}{\mathbb{F}}

\newcommand{\bL}{\mathbb{L}}

\newcommand{\bQ}{\mathbb{Q}}

\newcommand{\bZ}{\mathbb{Z}}

\newcommand{\cC}{\mathcal{C}}

\newcommand{\cI}{\mathcal{I}}

\newcommand{\cO}{\mathcal{O}}

\newcommand{\cU}{\mathcal{U}}

\newcommand{\fm}{\mathfrak{m}}

\newcommand{\fp}{\mathfrak{p}}

\newcommand{\fU}{\mathfrak{U}}
\newcommand{\fV}{\mathfrak{V}}

\newcommand{\fX}{\mathfrak{X}}
\newcommand{\fY}{\mathfrak{Y}}
\newcommand{\fZ}{\mathfrak{Z}}

\newcommand{\sA}{\mathscr{A}}
\newcommand{\sB}{\mathscr{B}}
\newcommand{\sC}{\mathscr{C}}

\newcommand{\sE}{\mathscr{E}}
\newcommand{\sF}{\mathscr{F}}
\newcommand{\sG}{\mathscr{G}}

\newcommand{\sI}{\mathscr{I}}

\newcommand{\ol}{\overline}
\newcommand{\ul}{\underline}
\newcommand{\wh}{\widehat}
\newcommand{\wt}{\widetilde}

\newcommand{\ola}{\overleftarrow}
\newcommand{\ora}{\overrightarrow}

\begin{document}

\title{Relative $A_{\Inf}$-cohomology}
\author{Ildar Gaisin and Teruhisa Koshikawa}
\date{}
\maketitle

\begin{abstract}
    We construct a relative version of the $A_{\Inf}$-cohomology theory developed by Bhatt-Morrow-Scholze and relate it to the prismatic theory of Bhatt-Scholze. The construction relies on the fiber product of topoi. As an application we show that there is an étale comparison of the $q$-crystalline pushforward after inverting $\mu \in A_{\Inf}$.
\end{abstract}

\tableofcontents

\newpage

\section{Introduction}

The goal of this paper is to construct and study an $A_{\Inf}$-cohomology theory associated to a \emph{morphism} of $p$-adic formal schemes. The absolute case was initiated by Bhatt-Morrow-Scholze \cite{BhMorSch}. However extending their theory to the relative setting is not a straightforward procedure and requires the consideration of an intermediary topos where the décalage functor operates. We recall the setup and (main) results of \cite{BhMorSch}. Let $C$ be a complete algebraically closed nonarchimedean extension of $\bQ_p$ with ring of integers $\cO$ and residue field $k$. Let $\fX$ be a proper smooth formal scheme over $\cO$. As usual, let $\cO^{\flat}$ be the tilt of $\cO$ and $A_{\Inf} := W(\cO^{\flat})$ be one of the period rings of Fontaine. Then the authors of \cite{BhMorSch} construct an object
$A\Omega_{\fX}$
which lives in $D^{\geq 0}(\fX_{\Zar}, A_{\Inf})$. The main result (cf. Theorem 1.8 in loc.cit.) is that the $A_{\Inf}$-cohomology complex
$R\Gamma_{A_{\Inf}}(\fX) := R\Gamma(\fX_{\Zar}, A\Omega_{\fX})$ \emph{interpolates} 
\begin{enumerate}
    \item the absolute crystalline cohomology of the mod $p$ reduction of $\fX$ and
    \item the integral $p$-adic étale cohomology of the rigid-analytic fiber of $\fX$.
\end{enumerate}
In particular, $R\Gamma_{A_{\Inf}}(\fX)$ also interpolates the de Rham cohomology of $\fX$ and the crystalline cohomology of the special fiber of $\fX$. 

Originally the construction of $A\Omega_{\fX}$ in \cite{BhMorSch} lacked the framework of a site. This was addressed and solved in \cite{Prisms} with the definitions of the $q$-crystalline and prismatic sites. Since the introduction of the $q$-crystalline site it is more natural to directly compare $A_{\Inf}$-cohomology to $q$-crystalline cohomology (or to prismatic cohomology, depending on taste), and to refine (1) as  (cf. Theorem 17.2 in loc.cit.)
\begin{enumerate}[(1)']
    \item $R\Gamma_{A_{\Inf}}(\fX)$ is isomorphic to the $q$-crystalline cohomology of $\fX$.
\end{enumerate}
This is indeed a refinement as the absolute crystalline cohomology is the base change of the $q$-crystalline cohomology (cf. Theorem 1.8(5) and Theorem 16.14 in loc.cit.). 
From this viewpoint, (2) may be regarded as a form of the \'etale comparison for the $q$-crystalline cohomology (or the prismatic cohomology).

Let us mention at this point, generalizing the work of \cite{BhMorSch}, \v{C}esnavi\v{c}ius and one of the authors \cite{SemistabAinfcoh} were able to drop the smoothness assumption in $\fX$ and work with a semistable model\footnote{In this situation the $C$ is less general and the Zariski topology is not fine enough to treat the semistable case, so $\fX_{\Zar}$ is replaced by $\fX_{\et}$. 
}. The generalization of (1)' to the semistable setting was solved in \cite{KoshTerui} (cf. Theorem 8.1 in loc.cit.). 

In this paper we deal with the analogues of (1)' and (2) in the \emph{relative} setting under the smoothness assumption. The setup is as follows. Let $f \colon \fX \to \fY$ be a proper smooth morphism of flat $p$-adic formal schemes over $\cO$. The flatness assumption is natural as $A_{\Inf}$-cohomology theory only remembers information coming from the generic fiber. There should be some assumptions on $\fY$ so that the adic generic fibers $X$ and $Y$ of $\fX$ and $\fY$ respectively, can be defined, but let us ignore this point for now. We construct an object $A\Omega_{\fX/\fY}$ (depending on $f$) and study the induced \emph{relative} $A_{\Inf}$-cohomology complex $R\Gamma_{A_{\Inf}}(\fX/\fY)$. Before we discuss the construction of these objects (including where they \emph{should} live), let us make what information $R\Gamma_{A_{\Inf}}(\fX/\fY)$ should capture more precise. For (2), as integral coefficients (e.g. $\bZ_p$) are not a good coefficient theory for the étale topology, it is more convenient to work in the pro-étale setting. So the relative version of (2) becomes
\begin{enumerate}[$(2)_{\text{rel}}$]
    \item $R\Gamma_{A_{\Inf}}(\fX/\fY)$ \emph{captures} $Rf_{\eta, \proet *}\wh{\bZ}_p$,
\end{enumerate}
where $f_{\eta} \colon X \to Y$ is the induced morphism on the adic generic fibers, and $\wh{\bZ}_p := \varprojlim \bZ/p^n$ as sheaves on the pro-étale site. For (1)' and ignoring the notion of a $q$-PD pair for the moment, \cite{Prisms} construct a (big, global) site $q-\textnormal{CRYS}(\fX/A_{\Inf})$ equipped with a structure sheaf $\cO_{q-\textnormal{CRYS}}$  (idem. for $\fY$). The morphism $f$ induces a cocontinuous functor of sites
\[
f_{q-\textnormal{CRYS}} \colon q-\textnormal{CRYS}(\fX/A_{\Inf}) \to q-\textnormal{CRYS}(\fY/A_{\Inf})
\]
and so the relative version of (1)' becomes
\begin{enumerate}[(1)'$_{\text{rel}}$]
    \item $R\Gamma_{A_{\Inf}}(\fX/\fY)$ \emph{captures} $Rf_{q-\textnormal{CRYS}*}\cO_{q-\textnormal{CRYS}}$.
\end{enumerate}

Let us now briefly review the construction of $A\Omega_{\fX}$ (in the setting of \cite{BhMorSch}), keeping in mind that it should be a complex whose cohomology interpolates between (1) and (2). On the site $X_{\proet}$, there is the sheaf $W(\wh{\cO}^{+, \flat}_X)$ defined as the usual Fontaine period ring on affinoid perfectoid objects in $X_{\proet}$ (cf. \cite[Definition 6.1]{SchpHrig}). For technical reasons one works with the derived $p$-adic completion of $W(\wh{\cO}^{+, \flat}_X)$, denoted by $\bA_{\Inf,X}$. There is a natural morphism of sites
\[
\nu_X \colon X_{\proet} \to \fX_{\Zar}
\]
and an initial approximation for $A\Omega_{\fX}$ is the pushforward $R\nu_{X*}\bA_{\Inf,X}$. For an explanation of how to arrive at this approximation we refer the reader to \cite[\S 1.4]{BhMorSch}. As a consequence of the primitive version of a comparison theorem (cf. \cite[Theorem 1.3]{SchpHrig}), this complex \emph{almost} captures the integral $p$-adic étale cohomology of $X$ in a technical sense. One way to obtain a result on the nose is to invert a single element in $A_{\Inf}$ whose image in $W(k)$ is zero and this is essentially the content of \cite[Theorem 1.8(iv)]{BhMorSch}. Regarding the ($q$-)crystalline side, it will be important to modify $R\nu_{X*}\bA_{\Inf,X}$. To see why, let us review the ``local" part of the argument of (1)/(1)'. The local computations are made in the case when $\fX$ has a description by tori-coordinates and can be summarised by the following schematic:
\begin{equation*} \label{eq:localcompdia}
 \xymatrix{
 R\nu_{X*}\bA_{\Inf,X} \ar@{<~>}[rr]^-{\text{Faltings'}}_-{\text{almost purity}}  && \text{(continuous) group cohomology} \ar@{<~>}[rr]^-{\text{décalage}}_-{\text{functor}}  && \text{$q$-de Rham complex.}
 } 
\end{equation*}
In this case, the upshot is that there is an explicit pro-(finite étale) affinoid perfectoid cover of the generic fiber (by repeatedly taking degree $p$ Kummer extensions) and Faltings' almost purity theorem means that $R\nu_{X*}\bA_{\Inf,X}$ can be \emph{almost} computed by the Koszul complex associated to the Galois group acting on the explicit cover. On the other hand an explicit (Koszul-like) complex calculating $q$-crystalline cohomology (in the context of (1)') is the $q$-de Rham complex\footnote{
In the context of (1), as one is  interested in the (less refined) absolute crystalline cohomology the base change of the $q$-de Rham complex to $A_{\textnormal{crys}}$ was considered \cite[Corollary 12.5]{BhMorSch}.}. In order to compare the Koszul complex almost representing $R\nu_{X*}\bA_{\Inf,X}$ and the similarly-shaped $q$-de Rham complex one must kill the \emph{non-integral} contribution\footnote{Also coined as the ``junk torsion" in \cite{BhMorSch}.} of the Koszul complex. To achieve this, we need a special element of $A_{\Inf}$. Fix a compatible system of primitive $p$-power roots of unity $\xi_{p^r} \in \cO$. The system $(1, \xi_p, \xi_{p^2}, \ldots)$ defines elements $\epsilon \in \cO^{\flat}$, $\mu = [\epsilon] - 1 \in A_{\Inf}$, and let $L\eta_{\mu}$ be the associated décalage functor. Then, in the case $\fX$ is a torus, for instance, taking cohomology of the previous \emph{almost} approximation and applying $L\eta_{\mu}$ gives an honest quasi-isomorphism:
\[
L\eta_{\mu} R\Gamma_{\cont} (\bZ_p, A_{\Inf} \{x^{\pm 1}\}) \to L\eta_{\mu} R\Gamma (X_{\proet}, \bA_{\Inf, X}),
\]
with the source identifying with the $q$-de Rham complex. This is one way of motivating the construction of
\[
A\Omega_{\fX} := L\eta_{\mu}R\nu_{X*}\bA_{\Inf,X}
\]
and the passage from $R\nu_{X*}\bA_{\Inf,X} \rightsquigarrow A\Omega_{\fX}$  is the main novelty of the approach of Bhatt-Morrow-Scholze. Let us also mention that several key inputs (albeit in a different setting) mentioned so far stretch back to works of Faltings \cite{FaltpadHodg, Faltalmoetalext}.

Let us now return to the relative situation. To get an idea of where $A\Omega_{\fX/\fY}$ should live, say ?, the following two observations are key:

\begin{enumerate}[(a)]
    \item Since the (relative) cohomology of $A\Omega_{\fX/\fY}$ should capture $Rf_{\eta, \proet *}\wh{\bZ}_p$, one idea is that there should be a map $? \to Y_{\proet}$.
    \item By taking modulo $q-1$, the $q$-de Rham complex identifies with the usual de Rham complex. So we see that the (relative) cohomology of $A\Omega_{\fX/\fY}$ should capture the usual de Rham complex $\Omega^{\bullet}_{\fX/\fY}$. As $\Omega^{\bullet}_{\fX/\fY}$ lives in $D(\fX_{\Zar})$, there should be a map $? \to \fX_{\Zar}$. 
\end{enumerate}

So far (a)-(b) give us two maps $? \to Y_{\proet}$ and $? \to \fX_{\Zar}$. Moreover as $A\Omega_{\fX/\fY}$ should in principle only depend on $f: \fX \to \fY$, these two maps should be compatible with their projections to $\fY_{\Zar}$. Finally it goes without saying that $A\Omega_{\fX/\fY}$ should be some kind of relative cohomology of $\bA_{\Inf,X}$ coming from $X_{\proet}$. So there should be a map $X_{\proet} \to ?$. In summary there should be a commutative diagram (of topoi):
\begin{equation} \label{eq:diagmotiva}
 \xymatrix{
 X_{\proet} \ar@{->}[rrd]^{\nu_{X}} \ar@{->}[ddr]_{f_{\eta, \proet}} \ar@{->}[rd]|-{\nu_f} & &  \\
 & ? \ar@{->}[d]^{p} \ar@{->}[r] & \fX_{\Zar} \ar@{->}[d]  \\ 
  & Y_{\proet} \ar@{->}[r]  & \fY_{\Zar}.
 } 
\end{equation}
The shape of diagram \eqref{eq:diagmotiva} means that one candidate for ? is some kind of \emph{product} of the topoi $Y_{\proet}$, $\fX_{\Zar}$ and $\fY_{\Zar}$. Indeed such a notion exists and is expounded in \cite[Exposé XI]{TravdeGabb}. Moreover we choose to work with étale sites (although the same arguments work with the Zariski sites), so we take $? := Y_{\proet} \times_{\fY_{\et}} \fX_{\et}$. This leads to the following definition
\[
A\Omega_{\fX/\fY} := L\eta_{\mu}R\nu_{f*}\bA_{\Inf,X}.
\]
We caution the reader that in general, we can only prove that $A\Omega_{\fX/\fY}$ is \emph{almost} derived $p$-adically complete.

Having established where the décalage functor should operate, we define relative $A_{\Inf}$-cohomology as
\[
R\Gamma_{A_{\Inf}}(\fX/\fY) := Rp_{*}A\Omega_{\fX/\fY}. 
\] 
We can now state our first comparison result:
\begin{thm}[Comparison with the pro-étale pushforward] \label{thm:comparisononee}
Let $\fX$ and $\fY$ be flat $p$-adic formal schemes locally of finite type over $\cO$, and let $f \colon \fX \to \fY$ be a proper smooth morphism. Then 
    \[R\Gamma_{A_{\Inf}}(\fX/\fY) \otimes_{A_{\Inf}}^{\bL} A_{\Inf}[\tfrac{1}{\mu}] \cong R\Gamma_{A_{\Inf}}(\fX/\fY)^{\wedge} \otimes_{A_{\Inf}}^{\bL} A_{\Inf}[\tfrac{1}{\mu}] \cong Rf_{\eta, \proet*}\wh{\bZ}_p \otimes_{\wh{\bZ}_p}^{\bL} \bA_{\Inf,Y}[\tfrac{1}{\mu}],
    \]
    where $R\Gamma_{A_{\Inf}}(\fX/\fY)^{\wedge}$ means the derived $p$-adic completion of $R\Gamma_{A_{\Inf}}(\fX/\fY)$.
\end{thm}

The proof of Theorem \ref{thm:comparisononee} follows the same lines as its absolute analogue and the relevant relative comparison result was already proved by Scholze (cf. the morphism \eqref{eq:takdsjdf} which was proved to be an \emph{almost} isomorphism in the the proof of {\cite[Theorem 8.8]{SchpHrig}}). The additional technical difficulty present in the relative case is proving that $Rf_{\eta, \proet*}\wh{\bZ}_p \otimes_{\wh{\bZ}_p}^{\bL} \bA_{\Inf,Y}$ is derived $p$-adically complete. In the absolute case (cf. the proof of \cite[Theorem 5.7]{BhMorSch}) the underlying objects are modules and integral étale cohomology is a perfect complex of $\bZ_p$-modules. It turns out that $Rf_{\eta, \proet*}\wh{\bZ}_p$ has enough finiteness properties so that the tensor product can still be controlled. It is based on the local constancy of $R^{i}f_{\eta, \et*}\bF_p$ proved in \cite[Theorem 10.5.1]{SchWeinsberkelec}. This is also the source of the ``locally of finite type" assumption on $\fX$ and $\fY$ in Theorem \ref{thm:comparisononee}. 

Without further mention, for the rest of the introduction we work in the setting of Theorem \ref{thm:comparisononee}. Although in the body of the paper we work with more general $\fX$ and $\fY$, the main obstacles are already present in the current setting. 

We now push on with the comparison with $Rf_{q-\textnormal{CRYS}*}\cO_{q-\textnormal{CRYS}}$. First note that $R\Gamma_{A_{\Inf}}(\fX/\fY)$ is a complex of $W(\wh{\cO}^{+, \flat}_Y)$-modules living in $Y_{\proet}$ and this immediately raises the question of \emph{where} this comparison should take place. Recall that there is the  site $q-\textnormal{CRYS}(\fY/A_{\Inf})$ and the object $Rf_{q-\textnormal{CRYS}*}\cO_{q-\textnormal{CRYS}}$ lives in the associated topos $(\fY/A_{\Inf})_{q-\textnormal{CRYS}}$. We construct a morphism of topoi
\begin{equation} \label{eq:introprovsqcry}
\mu_{\proet}: Y_{\proet} \to (\fY/A_{\Inf})_{q-\textnormal{CRYS}},
\end{equation}
which sends an affinoid perfectoid object to its pro-infinitesimal thickening given by Fontaine's period ring (a $q$-PD pair). A similar morphism is used in the work of \cite{MorrTsuj} when they relate relative Breuil-Kisin-Fargues modules with prismatic crystals (cf. proof of Theorem 5.15 in loc.cit.). It is along $\mu_{\proet}$ that we make the comparison of $R\Gamma_{A_{\Inf}}(\fX/\fY)$ and $Rf_{q-\textnormal{CRYS}*}\cO_{q-\textnormal{CRYS}}$. The next theorem represents the bulk of our work.

\begin{thm}[Comparison with the $q$-crystalline pushforward] \label{thm:introcomparqcrystal}
There is a canonical functorial morphism 
\begin{equation} \label{eq:compmorAOMeqOMe12} 
\mu_{\proet}^{-1}Rf_{q-\textnormal{CRYS}*}\cO_{q-\textnormal{CRYS}} \to R\Gamma_{A_{\Inf}}(\fX/\fY)
\end{equation}
in $D(Y_{\proet}, W(\wh{\cO}^{+, \flat}_Y))$ compatible with the Frobenius such that the cohomology sheaves of the cone of \eqref{eq:compmorAOMeqOMe12} 
 are killed by $[\fm^{\flat}] \subset A_{\Inf}$ (where $\fm^{\flat} \subset \cO^{\flat}$ is the maximal ideal).
\end{thm}

We now discuss the proof of Theorem \ref{thm:introcomparqcrystal}. The starting point is constructing the $q$-crystalline analogue of the complex $A\Omega_{\fX/\fY}$. As a first attempt it is natural to consider the product of topoi $(\fY/A_{\Inf})_{q-\textnormal{CRYS}} \times_{\fY_{\et}} \fX_{\et}$ and for the ``$q$-analogue" of $\nu_f$ (coming from the universal property of the product of topoi):
\[
\nu_{f}^{q} \colon (\fX/A_{\Inf})_{q-\textnormal{CRYS}} \to (\fY/A_{\Inf})_{q-\textnormal{CRYS}} \times_{\fY_{\et}} \fX_{\et}, 
\]
one defines
\[
q\Omega_{\fX/\fY} := R\nu_{f*}^{q}\cO_{q-\textnormal{CRYS}}.
\]
The issue is that the underlying site of $(\fY/A_{\Inf})_{q-\textnormal{CRYS}} \times_{\fY_{\et}} \fX_{\et}$ is difficult to describe due to the first projection $(\fY/A_{\Inf})_{q-\textnormal{CRYS}} \to \fY_{\et}$ essentially coming from a \emph{cocontinuous} functor of sites. In particular it is difficult to describe the (local) sections of $q\Omega_{\fX/\fY}$. The fix is to work with the \emph{oriented} product $(\fY/A_{\Inf})_{q-\textnormal{CRYS}} \ola{\times}_{\fY_{\et}} \fX_{\et}$, which, at least heuristically, takes into account the asymmetric nature of the two projections to $\fY_{\et}$. We make this precise, describing an explicit site for $(\fY/A_{\Inf})_{q-\textnormal{CRYS}} \ola{\times}_{\fY_{\et}} \fX_{\et}$ allowing us to study $q\Omega_{\fX/\fY}$ (now defined with the oriented product in place of the usual product in the target of $\nu_{f}^{q}$). 

With the correct setup in place, the next step in the proof of Theorem \ref{thm:introcomparqcrystal} is to compare both $q\Omega_{\fX/\fY}$ and $A\Omega_{\fX/\fY}$. The morphism of topoi \eqref{eq:introprovsqcry} extends to a morphism between the relevant products
\[
\mu \colon Y_{\proet} \times_{\fY_{\et}} \fX_{\et} \to (\fY/A_{\Inf})_{q-\textnormal{CRYS}} \ola{\times}_{\fY_{\et}} \fX_{\et}
\]
and the comparison between $q\Omega_{\fX/\fY}$ and $A\Omega_{\fX/\fY}$ takes the form:
\begin{thm} \label{thm:introqOmegvsAinome}
There is a canonical functorial isomorphism 
\begin{equation} \label{eq:compmorAOMeqOMed}
\mu^{-1}q\Omega_{\fX/\fY} \xrightarrow{\sim} A\Omega_{\fX/\fY}
\end{equation}
in the relevant derived category compatible with the Frobenius.
\end{thm}

The proof of Theorem \ref{thm:introqOmegvsAinome} follows the same lines as its absolute analogue \cite[Theorem 17.2]{Prisms}. One key ingredient is the comparison of $A\Omega_{\fX/\fY}$ with the de Rham complex. However, in the relative situation, one has to work harder in order to compare the cohomological Koszul complex coming from continuous group cohomology with the (local) sections of $A\Omega_{\fX/\fY}$. In particular we show that affinoid perfectoid objects over $Y$ are ``rigid", guaranteeing that the relevant continuous group cohomology groups have no non-zero \emph{almost}-zero elements. 

Returning to the proof of Theorem \ref{thm:introcomparqcrystal}, the morphism \eqref{eq:compmorAOMeqOMe12} is then obtained from pushing forward \eqref{eq:compmorAOMeqOMed} to the pro-étale site $Y_{\proet}$, composed with a base change morphism. This base change morphism involves both products of topoi and is difficult to study on its own. However, after reducing coefficients from $A_{\Inf}$ to $\cO$, the situation becomes more manageable. In particular, one has Hodge-Tate comparisons for both $q\Omega_{\fX/\fY}$ and $A\Omega_{\fX/\fY}$ and it suffices to study two different base change morphisms, one for each of the product topoi. These two base change morphisms both involve the sheaf of K\"{a}hler differentials on $\fX_{\et}$. Both base change morphisms turn out to be \emph{almost} isomorphisms and this, in short, ends the proof of Theorem \ref{thm:introcomparqcrystal}. Along the way we prove that the source of \eqref{eq:compmorAOMeqOMe12} is a perfect complex of $W(\wh{\cO}^{+, \flat}_Y)$-modules. As we are unable to do the same for $R\Gamma_{A_{\Inf}}(\fX/\fY)$, we cannot speculate whether \eqref{eq:compmorAOMeqOMe12} can be improved to an actual isomorphism.

We now discuss some applications of our comparison results: Theorems \ref{thm:comparisononee}-\ref{thm:introcomparqcrystal}. One immediate consequence is that relative $q$-crystalline/prismatic cohomology recovers the étale pushforward after inverting $\mu$:

\begin{thm} \label{thm:introqcrysvset}
There is a quasi-isomorphism
 \begin{equation} \label{eq:introqcrysvsetal}
(\mu_{\proet}^{-1}Rf_{q-\textnormal{CRYS}*}\cO_{q-\textnormal{CRYS}})^{\wedge} \otimes_{A_{\Inf}}^{\bL} A_{\Inf}[\tfrac{1}{\mu}] \cong Rf_{\eta, \proet*}\wh{\bZ}_p \otimes_{\wh{\bZ}_p}^{\bL} \bA_{\Inf,Y}[\tfrac{1}{\mu}],
\end{equation}
where $(\mu_{\proet}^{-1}Rf_{q-\textnormal{CRYS}*}\cO_{q-\textnormal{CRYS}})^{\wedge}$ denotes the derived $p$-adic completion of $\mu_{\proet}^{-1}Rf_{q-\textnormal{CRYS}*}\cO_{q-\textnormal{CRYS}}$.
\end{thm}

To the best of our knowledge, this seems to be a version of an étale comparison which does not follow directly from \cite[Theorem 1.8(4) and Theorem 1.16(4)]{Prisms} and is possibly one advantage of the theory constructed in this paper over prismatic cohomology. However the statement obtained by base changing \eqref{eq:introqcrysvsetal} to a statement over $W(C^{\flat})$ is closer to loc.cit. due to \cite[Proposition 3.6]{BhattSchopriFcr}.

Another application is by imitating the strategy used in \cite{AbbGors}, we are able to prove a \emph{relative} Hodge-Tate spectral sequence:

\begin{thm} \label{thm:introspectseque}
Suppose that for any $i,j \geq 0$, $R^{i}f_{\eta*}\Omega^{j}_{X/Y}$ is locally free of finite type over $\cO_{Y}$.
Then there is an $E_2$-spectral sequence
\[
E_2^{i,j} = R^{i}f_{\eta*}\Omega^{j}_{X/Y} \otimes_{\cO_{Y}} \wh{\cO}_{Y}(-j) \implies R^{i+j}f_{\eta, \proet *}\wh{\bZ}_p \otimes_{\wh{\bZ}_p} \wh{\cO}_{Y}
\]
and it degenerates.
\end{thm}

In their work, Abbes-Gros consider the ``\emph{topos de Faltings relatif}" (cf. \S 3.4 in loc.cit.) which plays a similar role to that of $Y_{\proet} \times_{\fY_{\et}} \fX_{\et}$. The spectral sequences established in \cite[Corollary 2.2.4]{CaScGenSV} and \cite[Théorème 1.4.5]{AbbGors}, together with Theorem \ref{thm:introspectseque} give now a priori, three different Hodge-Tate filtrations on the rational $p$-adic étale higher direct images, albeit in different mutually non-exclusive settings. The generic analogue $Y_{\proet} \times_{Y_{\et}} X_{\et}$ of our product of topoi can potentially be used to check the compatibility of the filtrations by Caraiani-Scholze and that of Theorem \ref{thm:introspectseque}. This would also involve checking the compatibility of the relative de Rham comparison in \cite[Theorem 8.8]{SchpHrig} and the one hypothetically implied by relative $A_{\Inf}$-cohomology theory, which we decided not to work out. One approach is to generalize the argument of \cite[Theorem 13.1]{BhMorSch} using $Y_{\proet} \times_{Y_{\et}} X_{\et}$. 

Finally regarding torsion in the relative étale pushforward given that the de Rham pushforwards are locally free, we refer the reader to Theorem \ref{thm:torsionbigerindic}. We give a proof which is based on the relative $A_{\Inf}$-cohomology theory (in particular Theorem \ref{thm:introqcrysvset}), however as noted by Takeshi Tsuji, the result essentially follows from the absolute case (cf. Remark \ref{rem:indivaTsjma}). We are not sure how to obtain a degree-wise result like in \cite[Theorem 14.5(ii)]{BhMorSch}.

\section*{Outline}

We briefly explain the content of the different sections. In Section \ref{sec:secontion2premlim} we describe a basis for the fiber product $Y_{\proet} \times_{\fY_{\et}} \fX_{\et}$. In Section \ref{sec:section3relapadcomp} we establish some properties of lisse $\wh{\bZ}_p$-sheaves on the pro-étale site in the context of complexes, and use this to prove Theorem \ref{thm:comparisononee}. In Section \ref{sec:localanal} we perform the local computations relating continuous group cohomology and \emph{framed} sections of $A\Omega_{\fX/\fY}$. In Section \ref{sec:TheHodhTaspe} we use these local computations to relate $A\Omega_{\fX/\fY}$ to the relative de Rham complex $\Omega_{\fX/\fY}^{\bullet}$. In Section \ref{$q$-crystalline sitewed} we introduce the $q$-crystalline analogue $(\fY/A_{\Inf})_{q-\textnormal{CRYS}} \ola{\times}_{\fY_{\et}} \fX_{\et}$ of $Y_{\proet} \times_{\fY_{\et}} \fX_{\et}$, and together with the results in Section \ref{sec:TheHodhTaspe}, use this to prove Theorem \ref{thm:introcomparqcrystal}. Finally, in Section \ref{sec:sec7applica}, we prove some consequences including Theorems \ref{thm:introqcrysvset}-\ref{thm:introspectseque}.

\section*{Acknowledgements}
The second author found the approach to relative $A_{\Inf}$-cohomology via the fiber product of topoi after conversations with Hiroki Kato during the ``$p$-adic cohomology and arithmetic geometry 2018'' conference at Tohoku University, before we learned the work of Abbes-Gros. We thank him for the inspiration. We thank Ahmed Abbes, Bhargav Bhatt, Lars Hesselholt, Naoki Imai, Jean-Stefan Koskivirta, and Takeshi Tsuji for helpful discussions and communications. The first author would like to thank the University of Tokyo for its support during this project.
This work was supported by JSPS KAKENHI Grant Number 20K14284. 

\section*{Additional notation}
Given $C$ and $\cO$ as in the introduction, let $\fm \subset \cO$ be its maximal ideal. We let $W(-)$ denote the $p$-typical Witt vectors of a (not necessarily unital) ring so that $A_{\Inf} := W(\cO^{\flat})$ is Fontaine's ring. For instance we often consider the non-unital ring $W(\fm^{\flat})$. The ring $A_{\Inf}$ comes equipped with the Frobenius $\varphi \colon A_{\Inf} \xrightarrow{\sim} A_{\Inf}$ and de Rham specialization map $\theta \colon A_{\Inf} \twoheadrightarrow \cO$. Its kernel $\ker(\theta) \subset A_{\Inf}$ is generated by a single element $\xi := \tfrac{[\epsilon] - 1}{[\epsilon^{1/p}]-1} \in A_{\Inf}$ and we denote $\tilde{\xi} := \varphi(\xi)$. For $X$ a locally noetherian adic space over $\Spa(\bQ_p, \bZ_p)$ we use the pro-étale site $X_{\proet}$ constructed in \cite{SchpHrig}. We adopt the notation for the various structure sheaves\footnote{Strictly speaking $\bA_{\Inf,X}$ is a complex.} on $X_{\proet}$ given in \cite[Definition 5.4]{BhMorSch}.

\section{Preliminaries} \label{sec:secontion2premlim}

The main goal of this section is to review the construction of the fiber product of topoi given in \cite[Exposé XI]{TravdeGabb} (cf. \cite[Chapter VI]{AbGrosTsuj}), relevant to our context. We also describe a suitable basis for the underlying site, which will facilitate extending the local computations performed in \S \ref{sec:localanal} to sheaf-theoretic results such as the Hodge-Tate specialization in \S \ref{sec:TheHodhTaspe}. The main result is Proposition \ref{prop:goodobjformbas}. We only urge the reader to understand the definition of relative $A_{\Inf}$-cohomology (cf. Definition \ref{def:relAOmegacomplex}) and skip the remaining contents of this section upon first viewing.

\subsection{Product of topoi: $Y_{ \proet}\times_{\fY_{\et}}\fX_{\et}$} \label{sec:prodoftopo}
In this section, unless otherwise stated, $\fX$ and $\fY$ will be any $p$-adic formal schemes of type (S)(b) over $\cO$, as in \cite[\S 1.9]{HubEtadic}. This means that for all $y \in \fY$  there exists an open affine neighborhood $\fU \subseteq \fY$ of $y$ such that the topology on $\cO_{\fY}(\fU)$ is $p$-adic and $\cO_{\fY}(\fU)[p^{-1}]$ is strongly noetherian (idem. for $\fX$). Let $f \colon \fX \to \fY$ be any $\Spf(\cO)$-morphism. Note that in this situation the adic generic fibers $X$ and $Y$ of $\fX$ and $\fY$, respectively, are locally noetherian adic spaces. In particular the framework established in \cite[\S 3]{SchpHrig} applies. We begin by recalling the construction of the fiber product of topoi $Y_{ \proet}\times_{\fY_{\et}}\fX_{\et}$ following the treatment in \cite[Exposé XI]{TravdeGabb}. As explained in the introduction, this will be a crucial ingredient for our definition of relative $A_{\Inf}$-cohomology (cf. Definition \ref{def:relAOmegacomplex}). We  begin with a simple reality check on the objects involved.

\begin{lem} \label{lem:finprolimandmorpsic}
\begin{enumerate}
    \item The sites $Y_{ \proet}$, $\fY_{\et}$ and $\fX_{\et}$ admit finite projective limits.
    \item The natural functors $\fY_{\et} \to Y_{ \proet}$ (obtained by taking the adic generic fiber) and $\fY_{\et} \to \fX_{\et}$ commute with finite projective limits and induce morphisms of sites 
   $\nu_Y \colon Y_{ \proet} \to \fY_{\et}$ and $f_{\et} \colon \fX_{\et} \to \fY_{\et}$.
\end{enumerate} 
\end{lem}

\begin{proof}
\begin{enumerate}
\item Firstly $Y_{ \proet}$ admits finite projective limits due to \cite[Lemma 3.10(vii)]{SchpHrig}. By \cite[\S3.5]{HubEtadic}, there is an equivalence of sites $\fY_{\et} \cong \fY_{\text{red}, \et}$ (idem. for $\fX_{\et}$), where $\fY_{\text{red}} = (\fY, \cO_{\fY}/\sI)$ is the reduced scheme underlying $\fY$ (here $\sI$ is the vanishing sheaf of ideals of $\cO_{\fY}$). Now $\fY_{\text{red}, \et}$ has a final object (namely $\fY_{\text{red}}$) and fiber products exist. Thus by \cite[{Tag 04AS}, Lemma 4.18.4(1) $\iff$ (3)]{stacks-project} $\fY_{\et}$ (idem. for $\fX_{\et}$) admits finite projective limits.
\item Both functors preserve final objects and commute with fiber products. Therefore they also commute with finite projective limits. With the identification $\fY_{\et} \cong \fY_{\text{red}, \et}$ (idem. for $\fX_{\et}$) from (1), we get a morphism of sites $\fX_{\et} \to \fY_{\et}$ by \cite[{Tag 04I0}]{stacks-project}. There is a natural projection $\nu \colon Y_{ \proet} \to Y_{ \et}$, so it suffices to show that the natural functor $\fY_{\et} \to Y_{ \et}$ induces a morphism of sites $Y_{ \et} \to \fY_{\et}$. This is \cite[Lemma 3.5.1(ii)]{HubEtadic}.
\end{enumerate}
\end{proof}

Lemma \ref{lem:finprolimandmorpsic} guarantees that we can use the formalism in \cite[Exposé XI]{TravdeGabb} to describe the underlying site for our product (cf. \S 1.1 in loc.cit.). This will not be the case for the $q$-crystalline (oriented) product which we deal with in Section \ref{$q$-crystalline sitewed}. Let $D$ be the following site:
\begin{enumerate}[(i)]
    \item As a category, the objects of $D$ are     \emph{commutative diagrams}
\[
 \xymatrix{
 W \ar@{->}[r] \ar@{->}[d] & \fU \ar@{->}[d] & \fV \ar@{->}[l] \ar@{->}[d]  \\ 
 Y \ar@{->}[r] & \fY  & \ar@{->}[l] \fX
 } 
\]
where $\fU \to \fY$ and $\fV \to \fX$ are étale, $W \to Y$ is pro-étale, and $W \to \fU$ (resp. $Y \to \fY$) denotes the morphism $W \to U$ (resp. $\id \colon Y \to Y$) of objects in $Y_{ \proet}$ where $U$ is the adic generic fiber of $\fU$. We explain the \emph{commutativity} condition. We demand the right-most square to be a commutative diagram of formal schemes. For the left-most square we demand that the diagram in $Y_{ \proet}$:
\[
 \xymatrix{
 W \ar@{->}[r] \ar@{->}[d] & U \ar@{->}[d]   \\ 
 Y \ar@{=}[r] & Y 
 } 
\]
is commutative. For brevity we denote such an object by $(W \to \fU \leftarrow \fV)$.

Given two such objects $(W_1 \to \fU_1 \leftarrow \fV_1)$ and $(W_2 \to \fU_2 \leftarrow \fV_2)$, a morphism from $(W_1 \to \fU_1 \leftarrow \fV_1)$ to $(W_2 \to \fU_2 \leftarrow \fV_2)$ is the datum of three morphisms $W_1 \to W_2$, $\fU_1 \to \fU_2$ and $\fV_1 \to \fV_2$ such the diagrams
\[
 \xymatrix{
 W_1 \ar@{->}[r] \ar@{->}[d] & U_1 \ar@{->}[d]   \\ 
 W_2 \ar@{->}[r] & U_2,
 }
 \hspace{10mm}
 \xymatrix{
 \fU_1 \ar@{->}[d] & \fV_1 \ar@{->}[d] \ar@{->}[l]   \\ 
 \fU_2  & \fV_2 \ar@{->}[l]
 }
\]
are commutative (here $U_1 \to U_2$ is the induced morphism on the adic generic fibers of $\fU_1 \to \fU_2$). For brevity, we will denote such a morphism by the following diagram:
\[
 \xymatrix{
 W_1 \ar@{->}[r] \ar@{->}[d] & \fU_1 \ar@{->}[d] & \fV_2 \ar@{->}[l] \ar@{->}[d]  \\ 
 W_2 \ar@{->}[r] & \fU_2  & \ar@{->}[l] \fV_2.
 } 
\]
\item We equip $D$ with the topology generated by coverings
\[
\left\{ (W_i \to \fU_i \leftarrow \fV_i) \to (W \to \fU \leftarrow \fV) \right\}_{i \in I}
\]
of the following types:
\begin{enumerate}
    \item $\fU_i=\fU, \fV_i = \fV \text{ }$ for all $i \in I$ and $(W_i \to W)_{i \in I}$ is a covering in $Y_{\proet}$, 
    \item $W_i=W, \fU_i = \fU \text{ }$ for all $i \in I$ and $(\fV_i \to \fV)_{i \in I}$ is a covering in $\fX_{\et}$, 
    \item $I$ is a singleton and a commutative diagram
        \[
 \xymatrix{
 W' \ar@{->}[r] \ar@{=}[d] & \fU' \ar@{->}[d] & \fV' = \fU' \times_\fU \fV \ar@{->}[l] \ar@{->}[d]  \\ 
 W \ar@{->}[r] & \fU  & \ar@{->}[l] \fV, 
 } 
\]
\item $I$ is a singleton and a commutative diagram
        \[
 \xymatrix{
 W' = W \times_U U' \ar@{->}[r] \ar@{->}[d] & \fU' \ar@{->}[d] & \fV' \ar@{->}[l] \ar@{=}[d]  \\ 
 W \ar@{->}[r] & \fU  & \ar@{->}[l] \fV.
 } 
\]
\end{enumerate}
\end{enumerate}

\begin{notation*}
We will refer to coverings of the above by ``type (a)", ``type (b)", ``type (c)" and ``type (d)", respectively.
\end{notation*}

Let $\wt{D}$ be the topos of sheaves of sets in $D$. The following is an instance of \cite[Exposé XI, Définition 3.3]{TravdeGabb}.
\begin{defn}
The topos $\wt{D}$ is called the fiber product of $\wt{Y_{ \proet}}$ and $\wt{\fX_{\et}}$ above $\wt{\fY_{\et}}$, which we denote (via an abuse of language\footnote{So as to avoid confusion, we will continue to denote the associated site by $D$. When the context is clear we will drop the tilde from $\wt{Y_{ \proet}}$, $\wt{\fX_{\et}}$ and $\wt{\fY_{\et}}$ to ease notation.}) by
$Y_{ \proet}\times_{\fY_{\et}}\fX_{\et}$.
\end{defn}

One can verify the sheaf condition on $D$ as follows.

\begin{lem}[{\cite[Exposé XI, \S 3.1]{TravdeGabb}}] \label{lem:covrshecond}
Let $\sF$ be a presheaf on $D$. For $\sF$ to be a sheaf, it suffices to verify the usual conditions of exactness for coverings of type (a) and (b), and for all covering families $Z' \to Z$ of type (c) and (d), $\sF(Z) \to \sF(Z')$ is an isomorphism.
\end{lem}

By construction there are natural projections of topoi
\[
p \colon Y_{ \proet}\times_{\fY_{\et}}\fX_{\et} \rightarrow Y_{\proet} \text{ and } q \colon Y_{ \proet}\times_{\fY_{\et}}\fX_{\et} \rightarrow \fX_{\et}, 
\]
which are given by
\[
p^{-1}(W) = (W \to \fY \leftarrow \fX) \text{ and } q^{-1}(\fV) = (Y \to \fY \leftarrow \fV),
\]
for $W$ an object in $Y_{\proet}$ and $\fV$ an object in $\fX_{\et}$. Moreover the diagram of topoi 
\[
 \xymatrix{
 & Y_{ \proet}\times_{\fY_{\et}}\fX_{\et} \ar@{->}[ld]_{p} \ar@{->}[rd]^{q} &  \\
 Y_{ \proet} \ar@{->}[rd]^{\nu_Y} && \fX_{\et} \ar@{->}[ld]_{f} \\
 & \fY_{\et} & 
 } 
\]
admits a canonical 2-isomorphism
\begin{equation} \label{eq:cantwoisom}
\varepsilon \colon  fq \to \nu_{Y}p
\end{equation}
given by an isomorphism of functors $(fq)_{*} \to (\nu_{Y}p)_{*}$ defined as follows: for $\sF$ a sheaf on $D$ and $\fV$ an object in $\fY_{\et}$ (with generic fiber $V$),
\[
\varepsilon \colon ((fq)_{*}\sF)(\fV) \to ((\nu_{Y}p)_{*}\sF)(\fV)
\]
is the composition
\[
\sF(Y \to \fY \leftarrow \fV \times_{\fY} \fX) \to \sF(V \to \fV \leftarrow \fV \times_{\fY} \fX) \to \sF(V \to \fY \leftarrow \fX)
\]
where the first isomorphism is induced by a covering of type (d) and the second isomorphism is induced by a covering of type (c). In the above diagram, we hope the reader does not find confusing in continuing to use $f$ and $\nu_Y$ to denote the induced morphisms of topoi. We can now state the universal property that $Y_{ \proet}\times_{\fY_{\et}}\fX_{\et}$ satisfies:

\begin{thm}[{\cite[Exposé XI, Théorème 3.2]{TravdeGabb}}] \label{thm:universprofibprodto}
Let $T$ be any topos equipped with morphisms $a \colon T \to Y_{ \proet}$ and $b \colon T \to \fX_{\et}$ and an isomorphism $t \colon fb \xrightarrow{\sim} \nu_{Y}a$. Then there exists a triple $(h \colon T \to Y_{ \proet}\times_{\fY_{\et}}\fX_{\et}, \alpha \colon ph \xrightarrow{\sim } a, \beta \colon qh \xrightarrow{\sim} b)$, unique up to unique isomorphism, such the composition
\[
fb \xrightarrow{\beta^{-1}} fqh \xrightarrow{\varepsilon} \nu_{Y}ph \xrightarrow{\alpha} \nu_{Y}a
\]
is equal to $t$. Moreover for $Z = (W \to \fU \leftarrow \fV)$ an object in $D$,
\[
h^{-1}Z = a^{-1}W \times_{b^{-1}(\fU \times_{\fY} \fX)} b^{-1}\fV,
\]
where $a^{-1}W \to b^{-1}(\fU \times_{\fY} \fX)$ is the composition $a^{-1}W \to a^{-1}U \xrightarrow{t} b^{-1}(\fU \times_{\fY} \fX)$, where $U$ is the generic fiber of $\fU$.
\end{thm}

Applying Theoem \ref{thm:universprofibprodto} to the topos $X_{ \proet}$ equipped with the natural morphisms $f_{\eta} \colon X_{ \proet} \to Y_{ \proet}$ and $\nu_X \colon X_{ \proet} \to \fX_{\et}$ gives a morphism of topoi
\begin{equation} \label{eq:definovf}
(\nu_f^{-1},\nu_{f*})  \colon X_{\proet} \rightarrow Y_{ \proet}\times_{\fY_{\et}}\fX_{\et},
\end{equation}
where 
\begin{equation} \label{eq:pullbaobobjec}
\nu_f^{-1}(W \to \fU \leftarrow \fV) = (W \times_Y X) \times_{(U \times_Y X)} \times V = W \times_U V.
\end{equation}

\begin{defn}\label{def:comtensostrushea}
We define the completed integral structure sheaf on $D$ by the $p$-adically completed tensor product 
$$q^{-1}\cO_{\fX,\et} \wh{\otimes}_{p^{-1}\nu_{Y}^{-1}\cO_{\fY, \et}} p^{-1}\wh{\cO}^{+}_Y$$
and denote it by $\wh{\cO}^+_D$. 
\end{defn}

\begin{rem}
A word of explanation for the definition of $\wh{\cO}_{D}^{+}$. The morphism $p^{-1}\nu_{Y}^{-1}\cO_{\fY, \et} \to p^{-1}\wh{\cO}^{+}_Y$ comes via the natural map of sheaves of rings $\nu_{Y}^{-1}\cO_{\fY, \et} \to \wh{\cO}^{+}_{Y}$. By \eqref{eq:cantwoisom}, there is a canonical isomorphism $p^{-1}\nu_{Y}^{-1}\cO_{\fY, \et} \xrightarrow{\sim} q^{-1}f^{-1}\cO_{\fY, \et}$. One then composes the latter isomorphism with the morphism $q^{-1}f^{-1}\cO_{\fY, \et} \to q^{-1}\cO_{\fX, \et}$ (coming from the natural map of sheaves of rings $f^{-1}\cO_{\fY, \et} \to \cO_{\fX, \et}$). This procedure produces the morphism $p^{-1}\nu_{Y}^{-1}\cO_{\fY, \et} \to q^{-1}\cO_{\fX, \et}$.
\end{rem}

By construction the projections $p$ and $q$ extend to morphisms of ringed topoi:
\begin{equation} \label{eq:morringtop}
(p, p^{\sharp}) \colon (Y_{ \proet}\times_{\fY_{\et}}\fX_{\et}, \wh{\cO}^+_D) \rightarrow (Y_{\proet}, \wh{\cO}^{+}_{Y}) \text{ and }(q, q^{\sharp}) \colon (Y_{ \proet}\times_{\fY_{\et}}\fX_{\et}, \wh{\cO}^+_D) \rightarrow (\fX_{\et}, \cO_{\fX, \et}).
\end{equation}
We delay extending $\nu_f$ to a morphism of ringed topoi to \S\ref{subsec:complfsdas} (cf. Proposition \ref{lem:morringtopnuf}).

\begin{prop} \label{prop:ousiDalge}
\begin{enumerate}
    \item The family of objects of the form $Z:= (\varprojlim_{i \in I} \Spa(S_{i}, S_{i}^{+}) \to \Spf(S) \leftarrow \Spf(R))$ is generating for $D$. Moreover the family is stable under fiber products.
    \item All objects as in (1) are quasi-compact.
    \item All objects as in (1) are coherent and the topos $Y_{ \proet}\times_{\fY_{\et}}\fX_{\et}$ is algebraic (cf. 
    \cite[Exposé VI, Définition 2.3]{SGA4}).
    \item Assume that $f \colon \fX \to \fY$ is qcqs. Then the projection morphism $p \colon Y_{ \proet}\times_{\fY_{\et}}\fX_{\et} \rightarrow Y_{\proet}$ is a coherent morphism (cf. \cite[Exposé VI, Définition 3.1]{SGA4}) of algebraic topoi.
\end{enumerate}
\end{prop}

\begin{proof}
\begin{enumerate}
\item Choose any object $T = (W \to \fU \leftarrow \fV)$ in $D$ and fix an affine covering $\{\fU_i \}$ of $\fU$ in $\fU_{\et}$. Then $\{\fU_i \times_{\fU} \fV \}$ is a covering of $\fV$ in $\fV_{\et}$. Therefore 
\[
\{(W \to \fU \leftarrow \fU_i \times_{\fU} \fV) \}
\]
is a covering of $T$ of type (b). Then by using coverings of type (d):
        \[
 \xymatrix{
 W' = W \times_U U_i \ar@{->}[r] \ar@{->}[d] & \fU_i \ar@{->}[d] & \fU_i \times_{\fU} \fV \ar@{->}[l] \ar@{=}[d]  \\ 
 W \ar@{->}[r] & \fU  & \ar@{->}[l] \fU_i \times_{\fU} \fV,
 } 
\]
where $U$ and $U_i$ are the adic generic fibers of $\fU$ and $\fU_i$, respectively, we can assume $\fU$ is affine. Using coverings of type (b) again, we can assume $\fV$ is affine. Finally by using coverings of type (a) and  \cite[Lemma 3.12(ii)]{SchpHrig}, proves that the family is generating. For the last part, note that taking the adic generic fiber commutes with fiber products of formal schemes. Therefore fiber products in $D$ are calculated component-wise and now it is clear that the family is stable under fiber products. 

    \item By part (1), the objects appearing there are generating, so we can consider the site whose objects are the objects appearing in (1) together with the coverings of type (a)-(d). This is a site of definition for $Y_{ \proet}\times_{\fY_{\et}}\fX_{\et}$.  By \cite[Lemma 3.12(i)]{SchpHrig}, $\varprojlim_{i \in I} \Spa(S_{i}, S_{i}^{+})$ is a quasi-compact object in $Y_{\proet}$, and $\Spf(R)$ is quasi-compact in $\fX_{\et}$. It follows that all covering families of type (a) or (b) of $Z$ have a finite refinement. 
    \item By parts (1) and (2), and the criterion of \cite[Exposé VI, Proposition 2.1]{SGA4}, all objects as in (1) are coherent. We check that (i ter) of Proposition 2.2 in loc.cit. is satisfied. For this we further restrict ourselves to objects of the form 
    \[
    Z':= (\varprojlim_{i \in I} \Spa(S_{i}, S_{i}^{+}) \to \Spf(S) \leftarrow \Spf(R))
    \]
    such that $\Spf(S) \to \fY$ (resp. $\Spf(R) \to \fX$) factors over an affine open subset $\fY' \subset \fY$ (resp. $\fX' \subset \fX$). This is still a generating family for $D$ and 
     \[
    Z' \times_{(Y \to \fY \leftarrow \fX)} Z' = Z'\times_{(Y ' \to \fY' \leftarrow \fX')} Z',
    \]
    where $Y'$ is the (affinoid) adic generic fiber of $\fY'$.
    
    \item By part (3) the topos $Y_{ \proet}\times_{\fY_{\et}}\fX_{\et}$ is algebraic. Moreover by \cite[Proposition 3.12(iii)]{SchpHrig}, $Y_{\proet}$ is an  algebraic topos. We use the criterion in \cite[Exposé VI, Proposition 3.2]{SGA4} to check that $p$ is a coherent morphism. By \cite[Proposition 3.12(ii)-(iii)]{SchpHrig}, objects of the form $\varprojlim_{i \in I} \Spa(S_{i}, S_{i}^{+})$ are coherent and are a generating family for $Y_{\proet}$. We can further restrict to the class $\varprojlim_{i \in I} \Spa(S_{i}, S_{i}^{+})$ that have the additional property that $\varprojlim_{i \in I} \Spa(S_{i}, S_{i}^{+}) \to Y$ factors over an open affinoid $Y' \subset Y$ where $Y'$ is the adic generic fiber of some open affine $\fY' \subset \fY$. This is still a generating family for $Y_{\proet}$. It suffices to check that
    \begin{equation}\label{eq:inversimpd}
    p^{-1}(\varprojlim_{i \in I} \Spa(S_{i}, S_{i}^{+})) = (\varprojlim_{i \in I} \Spa(S_{i}, S_{i}^{+}) \to \fY \leftarrow \fX)
    \end{equation}
    is a coherent object of $Y_{ \proet}\times_{\fY_{\et}}\fX_{\et}$. For this we will use the criterion in \cite[Exposé VI, Corollaire 1.17]{SGA4}. We have a covering of type (c):
            \[
 \xymatrix{
 \varprojlim_{i \in I} \Spa(S_{i}, S_{i}^{+}) \ar@{->}[r] \ar@{=}[d] & \fY' \ar@{^{(}->}[d] & \fX' = \fY' \times_\fY \fX \ar@{->}[l] \ar@{^{(}->}[d]  \\ 
 \varprojlim_{i \in I} \Spa(S_{i}, S_{i}^{+}) \ar@{->}[r] & \fY  & \ar@{->}[l] \fX.
 } 
\]
Since the morphism $\fX \to \fY$ is quasi-compact and quasi-separated, it follows that $\fX' \subset \fX$ is quasi-compact and quasi-separated. Thus by using coverings of type (b), we obtain a finite covering of \eqref{eq:inversimpd}
\[
\{ Z_j:= (\varprojlim_{i \in I} \Spa(S_{i}, S_{i}^{+}) \to \fY' \leftarrow \fX_j) \to (\varprojlim_{i \in I} \Spa(S_{i}, S_{i}^{+}) \to \fY \leftarrow \fX) \}_{j \in J}
\]
where each $\fX_j \subset \fX'$ is affine open (in particular by part (3) each $Z_j$ is coherent). Moreover for $j, k \in J$
\begin{equation} \label{eq:fintunofap}
Z_j \times_{(\varprojlim_{i \in I} \Spa(S_{i}, S_{i}^{+}) \to \fY \leftarrow \fX)} Z_k = (\varprojlim_{i \in I} \Spa(S_{i}, S_{i}^{+}) \to \fY' \leftarrow \fX_j \cap \fX_k).
\end{equation}
Since $\fX'$ is quasi-separated, $\fX_j \cap \fX_k$ is a finite union of affine opens. By using coverings of type (b) and \cite[Exposé VI, Proposition 1.3]{SGA4}, it follows that the object in \eqref{eq:fintunofap} is quasi-compact.
\end{enumerate}
\end{proof}

Next we describe a basis for the topology on $D$.

\begin{prop} \label{prop:basisDnosmor}
The set of object  of the form $(\varprojlim_{i \in I} \Spa(S_{i}, S_{i}^{+}) \to \Spf(S) \leftarrow \Spf(R))$ where $W =
\varprojlim_{i \in I} \Spa(S_{i}, S_{i}^{+}) \in Y_{\proet}$
is some affinoid perfectoid form a basis for $D$.
\end{prop}

\begin{proof}
This is proved in the same way as the first part of Proposition \ref{prop:ousiDalge}(1), except one uses that affinoid perfectoids form a basis for the pro-étale topology (cf. \cite[Proposition 4.8]{SchpHrig}).
\end{proof}

When $f \colon \fX \to \fY$ is smooth, we will need a refinement of Proposition  \ref{prop:basisDnosmor}, in order to relate our upcoming definition of $A\Omega_{\fX/\fY}$ and the relative de Rham complex, via the local computations of Section \ref{sec:localanal}. For this reason we make the following definition and we assume $f$ is smooth for the rest of this section.

\begin{defn}  \label{def:goodtriples}
We call an object $Z = (W \to \fU \leftarrow \fV) \in D$  \emph{good} if it is of the form
\[
(\varprojlim_{i \in I} \Spa(S_{i}, S_{i}^{+}) \to \Spf(S) \leftarrow \Spf(R))
\]
where 
\begin{enumerate}
\item The object $W =
\varprojlim_{i \in I} \Spa(S_{i}, S_{i}^{+}) \in Y_{\proet}$
is some affinoid perfectoid.

\item For some $d \geq 0$, there is a $\Spf(S)$-étale morphism 
\[
\fV = \Spf (R) \rightarrow \Spf (R^{\square}) =: \fV^{\square},
\]
where $R^{\square} := S \{t_1^{\pm 1}, \ldots, t_d^{\pm 1} \}$.
\end{enumerate}

If in addition  $(S_{\infty}[\tfrac{1}{p}])^{\circ} = S_{\infty}$, then we call $Z$ \emph{very good}. Here $S_{\infty}$ is the $p$-adic completion of $\varinjlim_{i \in I} S_i^{+}$.
\end{defn}

\begin{rem}
The last condition ``$(S_{\infty}[\tfrac{1}{p}])^{\circ} = S_{\infty}$" of Definition \ref{def:goodtriples} may appear somewhat odd at this stage, but is precisely what is required in order to match up (on the nose) the local computations of Section \ref{sec:localanal} regarding continuous group cohomology and the pushforward of $\bA_{\Inf,X}$. Moreover it will turn out that under some finiteness assumptions on $Y$, this condition is always satisfied (cf. Lemma \ref{lem:rigspacirplus}).
\end{rem}

We will prove that under some mild assumptions, $D$ has a basis consisting of very good objects. First we need some preparation. 

\begin{lem} \label{lem:nonoto}
Suppose $T$ is a perfectoid $C$-algebra. Then for each $x \in \cO$, the $\cO$-module $T^{\circ}/x$ has no nonzero $\fm$-torsion.
\end{lem}

\begin{proof}
This is essentially proven in \cite[\S 5]{SchPerf}. Lemma 5.3(ii) in loc.cit. gives $T^{\circ} \subset (T^{\circ a})_{*} \subset T$, which implies 
\begin{equation} \label{eq:alintel}
(T^{\circ a})_{*}[\tfrac{1}{p}] = T.
\end{equation}
By Proposition 5.5 in loc.cit., $T^{\circ a}$ is a perfectoid $\cO^{a}$-algebra and so by Lemma 5.6 in loc.cit. 
\begin{equation} \label{eq:genfipush}
(T^{\circ a})_{*}[\tfrac{1}{p}]^{\circ} = (T^{\circ a})_{*}.
\end{equation}
Therefore \eqref{eq:alintel}-\eqref{eq:genfipush} give $T^{\circ} = (T^{\circ a})_{*}$. Therefore by Lemma 5.3(iii) in loc.cit.  
\[
T^{\circ}/x = (T^{\circ a})_{*}/x \subset (T^{\circ a}/x)_{*}.
\]
Finally by Proposition 4.4 in loc.cit. $(T^{\circ a}/x)_{*} = \Hom_{\cO^{a}}(\cO^{a}, T^{\circ a}/x)$ has no non-zero $\fm$-torsion.
\end{proof}

\begin{rem} \label{rem:quotvan}
In the situation of Lemma \ref{lem:nonoto}, note that if $(T, T^{+})$ is an affinoid pair over $(C,\cO)$ with $T^{+} \subsetneq T^{\circ}$ and $0 \not= x \in \fm$ then $T^{+}/x$ does have nonzero $\fm$-torsion. Indeed take $y \in T^{\circ} \backslash T^{+}$. Then $xy \in T^{+}$ and its class in the quotient $T^{+}/x$ is nonzero $\fm$-torsion. 
\end{rem}

We are now ready to control powerbounded elements in an affinoid perfectoid. For a related result, cf. \cite[Example 1.6(iii)]{MorrTsuj}.

\begin{lem} \label{lem:rigspacirplus}
Suppose that $Y$ is locally of finite type over $\Spa(C, \cO)$. Then an affinoid perfectoid $\varprojlim_{i \in I} \Spa(S_{i}, S_{i}^{+}) \in Y_{\proet}$ satisfies $(S_{\infty}[\tfrac{1}{p}])^{\circ} = S_{\infty}$. 
\end{lem}

\begin{proof}
Since $Y$ is locally of finite type over $\Spa (C, \cO)$, it follows that $\Spa(S_i, S_i^{+})$ is topologically of finite type over $\Spa (C, \cO)$ and so by \cite[Lemma 4.4]{Hugfs} $S_i^{+} = S_i^{\circ}$.  The idea of the proof is to modify the terms in the colimit $\varinjlim_{i \in I} S_i^{\circ}$, so that the projection morphisms $S_i^{\circ} \to S_{\infty}$ become well behaved. A clue is that $S_{\infty}$ is reduced (as it is a perfectoid ring).

For each $i \in I$, let $S_{i, \text{red}}$ be the quotient of $S_i$ by its nilradical. The quotient morphism $S_i \to S_{i, \text{red}}$ induces a morphism of adic spaces $\Spa(S_{i, \text{red}}, S_{i, \text{red}}^\circ) \to \Spa(S_i, S_i^{\circ})$ which is an isomorphism on the underlying topological spaces (note that $S_{i, \text{red}}^{\circ}$ is the integral closure of $S_i^{\circ}$ in $S_{i, \text{red}}$). Furthermore for $j \geq i$ we claim that the diagram
\[
 \xymatrix{
 \Spa(S_{j, \text{red}}, S_{j, \text{red}}^\circ) \ar@{->}[r] \ar@{->}[d] & \Spa(S_j, S_j^{\circ}) \ar@{->}[d]   \\ 
 \Spa(S_{i, \text{red}}, S_{i, \text{red}}^\circ) \ar@{->}[r] & \Spa(S_i, S_i^{\circ})
 }
\]
is cartesian: indeed the ring morphism $S_i \to S_j$ is finite étale and so $S_{j, \text{red}} = S_{i, \text{red}} \otimes_{S_i} S_j$. Moreover, again by Lemma 4.4 in loc.cit., $S_{j, \text{red}}^{\circ}$ is the integral closure of $S_{i, \text{red}}^{\circ} \otimes_{S_i^{\circ}} S_j^{\circ}$ in $S_{j, \text{red}}$. In particular the morphism $\Spa(S_{j, \text{red}}, S_{j, \text{red}}^\circ) \to \Spa(S_{i, \text{red}}, S_{i, \text{red}}^\circ)$ is surjective. This implies that for each $n \geq 1$, $S_{i, \text{red}}^{\circ}/p^n \to S_{j, \text{red}}^{\circ}/p^n$ is injective and therefore after taking limits, the same is true for $S_{i, \text{red}}^{\circ} \to S_{j, \text{red}}^{\circ}$ (for each $i \in I$, $S_{i, \text{red}}^{\circ}$ is reduced and so by \cite[\S 6.2.4, Theorem 1]{BGR}, $S_{i, \text{red}}^{\circ}$ is bounded and in particular a ring of definition for $S_{i, \text{red}}$ equipped with the $p$-adic topology).

The morphisms $S_i \to S_{\infty}[\tfrac{1}{p}]$ factor through $S_{i, \text{red}}$ (as $S_{\infty}[\tfrac{1}{p}]$ is a perfectoid $C$-algebra and in particular reduced). Therefore $S_i^{\circ} \to S_{\infty}$ factors through $S_{i, \text{red}}^{\circ}$. We will need to compare $\varinjlim_{i \in I} S_i^{\circ}$ and $\varinjlim_{i \in I} S_{i, \text{red}}^{\circ}$. For this we show:

\begin{lem} \label{lem: surjintredri}
The ring morphism $S_i^{\circ} \to S_{i, \text{red}}^{\circ}$ is surjective.
\end{lem}

\begin{proof}
Recall that the morphism $\Spa(S_{i, \text{red}}, S_{i, \text{red}}^\circ) \to \Spa(S_i, S_i^{\circ})$ is an isomorphism on topological spaces.
The morphism $S_i \to S_{i, \text{red}}$ is surjective. For $s \in S_i \backslash S_i^{\circ}$, there exists $x \in \Spa(S_i, S_i^{\circ})$ such that $\lvert x(s) \lvert > 1$. The preimage $x' \in \Spa(S_{i, \text{red}}, S_{i, \text{red}}^\circ)$ of $x$ has the property that $\lvert x'(s') \lvert = \lvert x(s) \lvert >1$, where $s' \in S_{i, \text{red}}$ the image of $s$. Therefore $s' \in S_{i, \text{red}} \backslash S_{i, \text{red}}^{\circ}$ and the induced morphism $S_i^{\circ} \to S_{i, \text{red}}^{\circ}$ is surjective.
\end{proof}

By Lemma \ref{lem: surjintredri}, it follows that $S_{\infty}$ coincides with the $p$-adic completion of $\varinjlim_{i \in I} S_{i, \text{red}}^{\circ}$: indeed we obtain a surjection $\varinjlim_{i \in I} S_i^{\circ} \to \varinjlim_{i \in I} S_{i, \text{red}}^{\circ}$ which induces a surjection on the $p$-adic completions (cf. \cite[{Tag 0315}]{stacks-project}). Finally the composition $(\varinjlim_{i \in I} S_i^{\circ})^{\wedge} \to (\varinjlim_{i \in I} S_{i, \text{red}}^{\circ})^{\wedge} \to S_{\infty}$ is an isomorphism.

Since $S_{i, \text{red}}^{\circ}$ is a subring of $S_{i, \text{red}}$ (a $C$-algebra), it follows that $S_{i, \text{red}}^{\circ}$ is torsionfree as an $\cO$-module and hence flat. Moreover by the finiteness theorem of Grauert-Remmert (cf. \cite[\S 6.4.1, Corollary 5]{BGR}), it follows that $S_{i, \text{red}}^{\circ}$ is topologically of finite type over $\cO$ (cf. \cite[Theorem 9.4]{SchpHrig}). By a flattening theorem of Raynaud-Gruson, flatness together with topologically of finite type implies topologically of finite presentation (cf. \cite[\S 7.3, Corollary 5]{Bosch}) and so $S_{i, \text{red}}^{\circ}$ is topologically of finite presentation over $\cO$.

Therefore by \cite[Lemma 8.10]{BhMorSch}\footnote{Technically the result in loc.cit. assumes that $S_{i, \text{red}}^{\circ}$ is $p$-completely smooth $\cO$-algebra, however the same argument works under the slightly weaker assumptions: 
\begin{enumerate}
    \item $S_{i, \text{red}}^{\circ}/p$ is finitely presented over $\cO/p$ and 
    \item $S_{i, \text{red}}^{\circ}$ is flat over $\cO$.
\end{enumerate}\label{fnote}}, $S_{i, \text{red}}^{\circ}$ is the $p$-adic completion of a free $\cO$-module. In particular for each $0 \not= x \in \fm$, the $\cO$-modules $S_{i, \text{red}}^{\circ}/x$ have no nonzero $\fm$-torsion. The surjectivity of $\Spa(S_{j, \text{red}}, S_{j, \text{red}}^\circ) \to \Spa(S_{i, \text{red}}, S_{i, \text{red}}^\circ)$ implies $S_{i, \text{red}}^{\circ}/x \to S_{j, \text{red}}^{\circ}/x$ is injective. It follows that the direct limit $\varinjlim_{i \in I} S_{i, \text{red}}^{\circ}/x$ has no nonzero $\fm$-torsion. Therefore $S_{\infty}/x$ has no nonzero $\fm$-torsion. The result now follows from Remark \ref{rem:quotvan}.
\end{proof}

\begin{lem} \label{lem:facetalmor}
Suppose $\fV \to \Spf(S)$ is a smooth morphism of $p$-adic formal schemes over $\cO$. Then locally for the Zariski topology on $\fV$, the morphism factors as \[
\Spf(R) \to \Spf(S \{t_1^{\pm 1}, \ldots, t_d^{\pm 1} \}) \to \Spf(S),
\]
where the first morphism is étale (for some $d \geq 0$). 
\end{lem}

\begin{proof}
This is the essentially the same argument as \cite[Lemma 4.9]{Bhatspecvar} (in the case $S = \cO$ the argument in loc.cit. shows that it can be arranged so that the first morphism $\Spf(R) \to \Spf(S \{t_1^{\pm 1}, \ldots, t_d^{\pm 1} \})$ is finite étale).
\end{proof}

\begin{prop} \label{prop:goodobjformbas}

\begin{enumerate}
\item The set of $Z \in D$ which are good form a basis for $D$.
\item If in addition to (1), $Y$ is locally of finite type over $\Spa (C, \cO)$, then a good object is very good. 
\end{enumerate}
\end{prop}

\begin{proof}
 The first statement follows from Proposition \ref{prop:basisDnosmor} together with Lemma \ref{lem:facetalmor} (by using coverings of type (b)). Finally the last statement follows from Lemma \ref{lem:rigspacirplus}.
\end{proof}

\subsection{The object $A\Omega_{\fX/\fY}$}

In this section, let $\fX$ and $\fY$ be any flat $p$-adic formal schemes of type (S)(b) over $\cO$, as in \cite[\S 1.9]{HubEtadic}. Let $f \colon \fX \to \fY$ be any $\Spf(\cO)$-morphism. Considering $\nu_{f}$ and $p$ as the following morphism of ringed topoi
\[
(\nu_f, \nu_f^{\sharp}) \colon (X_{\proet},A_{\Inf}) \rightarrow (Y_{ \proet}\times_{\fY_{\et}}\fX_{\et},A_{\Inf})
\text{ and }
(p, p^{\sharp}) \colon (Y_{ \proet}\times_{\fY_{\et}}\fX_{\et}, A_{\Inf}) \rightarrow (Y_{\proet}, A_{\Inf}),
\]
respectively, we make the following definition. 

\begin{defn} \label{def:relAOmegacomplex}
The relative complex $A\Omega_{\fX / \fY} \in D(Y_{ \proet}\times_{\fY_{\et}}\fX_{\et}, A_{\Inf})$ is defined by
\[
A\Omega_{\fX / \fY} := L\eta_{\mu}(R\nu_{f*}\bA_{\Inf,X}). 
\]
In addition we set $R\Gamma_{A_{\Inf}}(\fX/\fY) := Rp_{*}A\Omega_{\fX / \fY}$, an object living in $D(Y_{ \proet}, A_{\Inf})$. We call $R\Gamma_{A_{\Inf}}(\fX/\fY)$ the (relative) $A_{\Inf}$-cohomology.
\end{defn}

We will sometimes need to \emph{upgrade} Definition \ref{def:relAOmegacomplex} and consider $A\Omega_{\fX/\fY}$ as an object over the sheaf of Witt vectors. To make this precise we now consider $\nu_f$ and $p$ as the following morphism of ringed topoi 
\[
(\nu_f^{W}, \nu_f^{W\sharp}) \colon (X_{ \proet}, W(\wh{\cO}^{+, \flat}_X)) \to (Y_{ \proet}\times_{\fY_{\et}}\fX_{\et}, p^{-1}W(\wh{\cO}^{+, \flat}_Y))
\]
and 
\[
(p^W, p^{W\sharp}) \colon (Y_{ \proet}\times_{\fY_{\et}}\fX_{\et}, p^{-1}W(\wh{\cO}^{+, \flat}_Y)) \rightarrow (Y_{\proet}, W(\wh{\cO}^{+, \flat}_Y)),
\]
and we make the following definition (note that the ideal sheaf $p^{-1}((\mu)) \subset p^{-1}W(\wh{\cO}^{+, \flat}_Y)$ is invertible).

\begin{defn} \label{defn:WversofAInf}
By considering $p$-adic derived completion taking place in the ringed site $(X_{\proet}, W(\wh{\cO}^{+, \flat}_X))$ we can
consider $\bA_{\Inf,X}$ as an object of $D(X_{\proet}, W(\wh{\cO}^{+, \flat}_X))$, in which case we denote it by $\bA_{\Inf,X}^{W}$. The relative complex $A\Omega_{\fX / \fY}^{W} \in D(D, p^{-1}W(\wh{\cO}^{+, \flat}_Y))$ is defined by
\[
A\Omega_{\fX / \fY}^{W} := L\eta_{p^{-1}((\mu))}(R\nu^{W}_{f*}\bA_{\Inf,X}^{W}). 
\]
In addition we set $R\Gamma^{W}_{A_{\Inf}}(\fX/\fY) := Rp^{W}_{*}A\Omega_{\fX / \fY}^{W}$, an object living in $D(Y_{ \proet}, W(\wh{\cO}^{+, \flat}_Y))$.
\end{defn}

We will now check that $A\Omega_{\fX/\fY}$ and $R\Gamma_{A_{\Inf}}(\fX/\fY)$ are compatible with their $W$-variants (after which we will drop the superscript $W$ from the notation when context is clear). 
Let 
\[
(r,r^{\sharp}) \colon (Y_{ \proet}\times_{\fY_{\et}}\fX_{\et}, p^{-1}W(\wh{\cO}^{+, \flat}_Y)) \to (Y_{ \proet}\times_{\fY_{\et}}\fX_{\et}, A_{\Inf})
\]
and 
\[
s \colon (X_{ \proet}, W(\wh{\cO}^{+, \flat}_X)) \to (X_{\proet},A_{\Inf})
\]
be the natural morphisms of ringed topoi. 

\begin{lem} \label{lem:quasiisormAOmegW}
There is a natural quasi-isomorphism $A\Omega_{\fX/\fY}  \cong Rr_{*}A\Omega_{\fX/\fY}^{W}$.
\end{lem}

\begin{proof}
Note that $Rs_{*}\bA_{\Inf,X}^{W} = \bA_{\Inf,X}$ (as $p$-adic derived completion commutes with forgetting $W$) and since $\nu_f \circ s = r \circ v_{f}^{W}$ we obtain $R\nu_{f*}\bA_{\Inf,X} = Rr_{*}R\nu_{f*}^{W}\bA_{\Inf,X}^{W}$. Therefore $A\Omega_{\fX/\fY} = L\eta_{\mu}Rr_{*}R\nu_{f*}^{W}\bA_{\Inf,X}^{W}$. Pick a  $p^{-1}((\mu))$-torsion-free resolution $C^{\bullet}$ of $R\nu_{f*}^{W}\bA_{\Inf,X}^{W}$ in the sense of \cite[pg. 288]{BhMorSch}. Since $r_*$ is an exact functor (because it is identity at the level of topoi), it follows that $r_{*}C^{\bullet}$ is a $(\mu)$-torsion free complex representing $Rr_{*}C^{\bullet}$. Thus $L\eta_{\mu}Rr_{*}C^{\bullet}$ is represented by the complex $\eta_{\mu}r_{*}C^{\bullet}$. Finally note that $\eta_{\mu}r_{*} = r_{*}\eta_{p^{-1}((\mu))}$.
\end{proof}

In the case $f$ is smooth and $\fY = \Spf(\cO)$, we can compare our objects to those of \cite{BhMorSch}. We delay this to Section \ref{sec:TheHodhTaspe} after the local computations have been made, cf. Proposition \ref{prop:compaBMSvsours}.

\subsection{Almost $\cO$-sheaves}\label{subsec:almostOsheav}
In this section, we will recall the notion of $\alpha$-sheaves (or almost sheaves) as introduced in \cite[\S 2.6]{AbbGors}. There is also recent work of \cite{ZavBog}. Throughout we will fix some site $\sC$ and employ the notation in \cite{GabRam}. The following is an instance of \cite[Définition 2.6.16]{AbbGors}. 

\begin{defn}
We call the category of \emph{presheaves of} $\cO^{a}$-\emph{modules} (\emph{or almost $\cO$-presheaves}) on $\sC$ to be the category of presheaves on $\sC$ whose values are in the category $\cO^{a}$. We denote this category by $\widehat{\sC}^{a}$. We call an isomorphism in $\widehat{\sC}^{a}$ an almost isomorphism.  
\end{defn}

Armed with the notion of an almost $\cO$-presheaf, we now define almost $\cO$-sheaves, cf. Définition 2.6.24 in loc.cit.

\begin{defn}
An object $\sF$ in $\widehat{\sC}^{a}$ is called an \emph{almost $\cO$-sheaf} if for each object $X$ of $\sC$ and for each covering family $\{U_i\}_i$ of $X$ the canonical morphism
\[
\sF(X) \to \varprojlim_{V \to U_i} \sF(V)
\]
is an isomorphism.  We denote $\wt{\sC}^{a}$ the full subcategory of $\widehat{\sC}^{a}$ formed by the almost $\cO$-sheaves and $\iota \colon \wt{\sC}^{a} \to \widehat{\sC}^{a}$, the natural inclusion.
\end{defn}
There is a natural sheafification functor, cf. Proposition 2.6.31 and Proposition 2.6.33 in loc.cit.

\begin{prop}
The inclusion $\iota \colon \wt{\sC}^{a} \to \widehat{\sC}^{a}$ admits a left-adjoint
\[
\overline{a} \colon  \widehat{\sC}^{a} \to \wt{\sC}^{a}
\]
and is exact.
\end{prop}

We will need to compare the categories $\wt{\sC}^{a}$ and $\widehat{\sC}^{a}$ to the usual categories of presheaves and sheaves of $\cO$-modules.

\begin{defn}
We denote by $\Mod_{\cO}(\widehat{\sC})$ (resp. $\Mod_{\cO}(\wt{\sC})$) the category of preheaves (resp. sheaves) in $\cO$-modules on $\sC$ and $i \colon \Mod_{\cO}(\wt{\sC}) \to \Mod_{\cO}(\widehat{\sC})$ the natural inclusion.
\end{defn}

The categories $\Mod_{\cO}(\widehat{\sC})$ (resp. $\Mod_{\cO}(\wt{\sC})$) together with their almost counterparts $\widehat{\sC}^{a}$ (resp. $\wt{\sC}^{a}$) are abelian cf. \cite[{Tag 01AG}]{stacks-project} and \cite[Corollaire 2.6.18, Proposition 2.6.34]{AbbGors}. Thus one can define their associated derived categories. We will denote the derived category of $\wt{\sC}^{a}$ by $D(\sC, \cO^{a})$.

\begin{defn}\label{def:quaisisoder}
We call a morphism $\sF \to \sG$ in $D(\sC, \cO^{a})$ or $D(\widehat{\sC}^{a})$  an \emph{almost isomorphism} if it is a quasi-isomorphism. For $\sF \in D(\sC, \cO)$, let $\sF^{a}$ be its image in $D(\sC, \cO^{a})$.  
\end{defn}

The localization functor
\begin{align*}
    \cO-\text{mod} &\to \cO^{a}-\text{mod}\\
    M &\mapsto M^{a}
\end{align*}
induces the exact functors
\[
\widehat{\alpha} \colon \Mod_{\cO}(\widehat{\sC}) \to \widehat{\sC}^{a} \text{ and } \wt{\alpha} \colon \Mod_{\cO}(\wt{\sC}) \to \wt{\sC}^{a}.
\]

The next proposition says, in particular, that to verify whether a morphism of almost $\cO$-sheaves is an almost isomorphism, it suffices to verify it on the level of presheaves, cf. Proposition 2.6.36 in loc.cit.

\begin{prop} \label{prop:commsqaalmssnosh}
The diagram 
\[
 \xymatrix{
 \Mod_{\cO}(\widehat{\sC}) \ar@{->}[r]^{a} \ar@{->}[d]_{\widehat{\alpha}} & \Mod_{\cO}(\wt{\sC}) \ar@{->}[d]^{\wt{\alpha}}   \\ 
 \widehat{\sC}^{a} \ar@{->}[r]^{\overline{a}} & \wt{\sC}^{a}
 }
\]
is commutative, up to unique isomorphism (where $a \colon \Mod_{\cO}(\widehat{\sC}) \to \Mod_{\cO}(\wt{\sC})$ is the usual sheafification). 
\end{prop}

\subsection{Miscellany}

In this short section we collect some properties of the ring $A_{\Inf}$ (and its relative version) for which we could not find a reference for.

\begin{lem} \label{lem:promaxidtensp}
The tensor product $W(\fm^{\flat}) \otimes_{A_{\Inf}, \theta \circ \varphi^{-1}} \cO$ identifies with $\fm$.
\end{lem}

\begin{proof}
From the exact sequence
\[
0 \to W(\fm^\flat) \to A_{\Inf} \to W(k) \to 0
\]
it suffices to show $\textnormal{Tor}_{A_{\Inf}}^1(W(k), A_{\Inf}/\tilde{\xi})$ vanishes. The image of $\tilde{\xi}$ in $W(k)$ is $p$. Since $W(k)$ is $p$-torsion free, the result follows.
\end{proof}

In addition to the site $X_{\proet}$, we consider the site $X_{\proet}^{\text{psh}}$, whose objects are affinoid perfectoids in $X_{\proet}$, and coverings are isomorphisms. In this way, every presheaf is already a sheaf. We denote the associated topos by $X_{ \proet}^{\text{psh}}$. There is a morphism of topoi
\[
(\phi^{-1}, \phi_{*}) \colon X_{ \proet} \to X_{ \proet}^{\text{psh}},
\]
for which $\phi_{*}$ is given by restricting sheaves on $X_{\proet}$ to $X_{ \proet}^{\text{psh}}$ and $\phi^{-1}$ sheafification.

\begin{lem} \label{lem:derived completionaind}
The $R\phi_{*}\bA_{\Inf, X}$ is derived $(p, \mu)$-adically complete in the presheaf topos.
\end{lem}

\begin{proof}
By \cite[{Tag 0BLX}]{stacks-project}, it suffices to show that for each affinoid perfectoid $U \in X_{\proet}$, $R\Gamma(U, \bA_{\Inf, X})$ is derived $(p, \mu)$-adically complete. It is certainly derived $p$-adically complete by definition, so it suffices to show it is derived $\mu$-adically complete. Since we are in the presheaf topos (it is replete), by \cite[Lemma 6.15]{BhMorSch}, it suffices to check each $H^{i}(U, \bA_{\Inf,X})$ is derived $\mu$-adically complete. Now $\mu \in W(\fm^{\flat})$ and $H^{i}(U, \bA_{\Inf,X})$ is killed by $W(\fm^{\flat})$ for $i >0$ (cf. the statement just after \cite[Lemma 5.6]{BhMorSch}). Therefore $H^{i}(U, \bA_{\Inf,X})$ is derived $\mu$-adically complete for $i >0$. For $i = 0$, one has $H^{0}(U, \bA_{\Inf,X}) = W(R^{\flat +})$ where $\Spa(R,R^{+})$ is the perfectoid space determined by $U$ and this is even classically $\mu$-adically complete. 
\end{proof}

\begin{cor} \label{cor:xiadcomiofAinf}
The object $\bA_{\Inf,X}$ is derived $\tilde{\xi}$-adically complete.
\end{cor}

\begin{proof}
Since  $\phi^{-1} \circ R\phi_{*} \cong \id$, this is an immediate consequence of Lemma \ref{lem:derived completionaind} (note that $\tilde{\xi} \in (p,\mu)$).
\end{proof}

\section{Relative $p$-adic étale comparison}\label{sec:section3relapadcomp}

In this section we record the relative $p$-adic \'{e}tale comparison (cf. \cite[Theorem 14.3(iv)]{BhMorSch} and \cite[Theorem 2.3]{SemistabAinfcoh}) in Theorem \ref{thm:padcohse}. The proof follows the same strategy as those in loc.cit., although in the current (relative) situation, the relevant statements are formulated in terms of complexes of \emph{sheaves} rather than modules. For this reason we need to check certain complexes of $\bZ_p$-sheaves are well behaved and we begin with some preliminary work guaranteeing this. 

\subsection{Lisse $\wh{\bZ}_p$-sheaves} \label{sec:lisfdjkl}

In this section, unless otherwise stated, $X$ will temporarily denote a locally noetherian adic space over $\Spa (\bQ_p, \bZ_p)$. In \cite[Definition 8.1]{SchpHrig}, Scholze defines the notion of a lisse $\wh{\bZ}_p$-sheaf on $X_{\proet}$. Roughly speaking a lisse $\wh{\bZ}_p$-sheaf on $X_{\proet}$ is a literal projective limit of a classical lisse $\bZ_p$-sheaf appearing on the étale site $X_{\et}$ (cf. Proposition 8.2 in loc.cit.). This is possible because (derived) limits on the pro-étale site are well behaved, in that they (usually) become limits. Much like in \cite{BS} where the authors define and study the notion of a \emph{constructible complex} on the pro-étale site of a qcqs scheme, we will do the same for lisse $\wh{\bZ}_p$-sheaves. We follow closely the framework developed in \cite[{Tag 09C0}]{stacks-project}.

\begin{notation*}
For $K \in D(X_{\proet}, \bZ_p)$ we set $K_n := K \otimes^{\bL}_{\bZ_p} \bZ/p^n$. We write $\ul{M}$ for the image of $M \in D(\bZ_p)$ under the pullback $D(\bZ_p) \to D(X_{\proet}, \bZ_p)$.
\end{notation*}

\begin{defn} \label{defn:lissed}
An object $K \in D(X_{\proet}, \bZ_p)$ is called \emph{lisse} if:
\begin{enumerate}
    \item $K$ is $p$-adically derived complete and
    \item $K_1$ has finite locally constant cohomology sheaves and locally has finite tor dimension.
\end{enumerate}
\end{defn}

\begin{rem} \label{rem:troskfsd}
By \cite[Proposition 8.2]{SchpHrig}, local systems of $\bZ/p^n$-modules on $X_{\proet}$ are equivalent (via pullback) to local systems of $\bZ/p^n$-modules on $X_{\et}$. By Corollary 3.17(i) in loc.cit. the pushforward is a quasi-inverse. A detailed proof can be found in \cite[Theorem C.4]{ZavBog}\footnote{In the beginning of \cite[Appendix C]{ZavBog}, Zavyalov makes the assumption that $X$ is a locally \emph{strongly} noetherian adic space over $\Spa(\bQ_p, \bZ_p)$, but the ``\emph{strongly}" assumption is not essential.}.
\end{rem}

\begin{lem} \label{lem:finitosroes}
Suppose $K$ is lisse and $X$ is quasi-compact. There exists $a$ and $b$ such that
\begin{enumerate}
    \item each $K_n$ has tor-amplitude in $[a,b]$,
    \item each $K_n$ has finite locally constant cohomology sheaves,
    \item $H^{i}(K) = \varprojlim H^{i}(K_n)$,
    \item $H^{i}(K) = 0$ for $i \notin [a,b]$,
    \item each $K_n$ is obtained via pullback of a bounded $\bZ/p^n$-complex with finite locally constant cohomology sheaves from $X_{\et}$.
\end{enumerate}
\end{lem}

\begin{proof}
\begin{enumerate}
\item Since $X$ is quasi-compact and pro-étale morphisms are open, it follows that we can find $a \leq b$ such that $K_1$ has tor amplitude in $[a,b]$. For the remaining $K_n$, the result follows from \cite[{Tag 0942}]{stacks-project}.

\item The proof is the same as \cite[{Tag 094I}]{stacks-project}.

\item Since $K$ is $p$-adically derived complete $K=R\varprojlim K_n$. A priori, by part (2) each $K_n$ has finite locally constant cohomology sheaves on $X_{\proet}$, but by Remark \ref{rem:troskfsd}, 
they can be considered as finite locally constant sheaves on $X_{\et}$. The result now follows from \cite[Theorem 4.9]{SchpHrig} and \cite[{Tag 0A09}]{stacks-project}. 
\item This follows from parts (1) and (3).
\item This is an application of \cite[{Tag 0D7U}]{stacks-project} to Remark \ref{rem:troskfsd}. 
\end{enumerate}
\end{proof}

It will useful to strengthen part (5) of Lemma \ref{lem:finitosroes} to $X_{\fet}$:

\begin{lem} \label{lem:gososdma}
Let $X = \Spa(A, A^{+})$ be an affinoid 
connected noetherian adic space over $\Spa(\bQ_p, \bZ_p)$. Suppose $K \in D(X_{\proet}, \bZ_p)$ is lisse. Then each $K_n$ is obtained via pullback of a bounded $\bZ/p^n$-complex with finite locally constant cohomology sheaves from $X_{\fet}$.
\end{lem}

\begin{proof}
Indeed by part (5) of Lemma \ref{lem:finitosroes}, it suffices to verify the conditions of \cite[{Tag 0D7U}]{stacks-project} (comparing finite locally constant sheaves on $X_{\et}$ and $X_{\text{f}\et}$). This is achieved in \cite[Theorem 4.9]{SchpHrig}.
\end{proof}

\begin{defn} \label{defn:adiclisse}
An object $K \in D(X_{\proet}, \bZ_p)$ is called \emph{adic lisse} if $K$ is locally isomorphic to 
\begin{equation} \label{eq:lissadilocfo}
\wh{\bZ}_p \otimes_{\bZ_p}^{\bL} \ul{M}
\end{equation}
where $M \in D(\bZ_p)$ a finite complex of finite free $\bZ_p$-modules.  
\end{defn}

For our next task, we want to compare the notions of lisse and adic lisse. It is not difficult to trivialize each $K_n$ one by one. The key idea however is to control these trivializations as one jumps from $n$ to $n+1$. For instance, in the setting of \cite{BS}, this is achieved via the notion of \emph{w-strictly local affine scheme} (cf. Lemma 6.6.8 in loc.cit.). In the current setup, we don't quite have access to such a nice basis (it's unlikely that $X_{\proet}$ gives rise to a replete topos), but nevertheless an analogue does exist via \cite[Theorem 4.9]{SchpHrig}.

\begin{lem} \label{lem:trivaknallaton}
Let $X = \Spa(A, A^{+})$ be an affinoid 
connected noetherian adic space over $\Spa(\bQ_p, \bZ_p)$. Suppose $K \in D(X_{\proet}, \bZ_p)$ is lisse. Then there exists a pro-(finite étale) cover $U \to X$ trivializing all $K_n$ (i.e. there are perfect complexes $M_n \in D(\bZ/p^n)$ such that $K_n \lvert_{U} \cong \ul{M_n} \lvert_{U}$).
\end{lem}

\begin{proof}
We will use the results of Lemma \ref{lem:finitosroes} without mention. The idea is to first trivialize $K_1$ for a finite étale covering $U_1 \to X$ and then construct a sequence of finite étale covers $U_{n+1} \to U_n$ trivializing $K_{n+1}$. We proceed by induction on $n$.

Let us deal with $K_1$ first.  We produce a finite étale cover $ U_{1} \to X$ and a perfect complex of $\bZ/p$-modules $M_1$ such that $K_1 \lvert_{U_{1}} \cong \ul{M_1}\lvert_{U_{1}}$. By Lemma \ref{lem:gososdma}, first note that $K_1$ comes (via pullback) of a complex in $D(X_{\fet}, \bZ/p)$ which has finite locally constant cohomology sheaves (in $X_{\fet}$). By \cite[{Tag 094G}]{stacks-project} we can choose a finite étale cover $U_{1} \to X$ (with $U_{1}$ affinoid connected)  such that $K_1 \lvert_{U_{1}} \cong \ul{M_{1}}\lvert_{U_{1}}$ for some finite complex of finite $\bZ/p$-modules $M_{1}$. 

Suppose we have trivialized $K_n$ via surjective finite étale maps $U_{n} \to U_{1}$ with each $U_{n}$ affinoid connected. Then repeating the argument for $K_1$ with $X$ replaced by $U_n$, we produce a surjective finite étale map $U_{(n+1)} \to U_{n}$ trivializing $K_{n+1}$ with $U_{(n+1)}$ affinoid connected.

In summary we have produced a sequence of finite étale connected (affinoid) covers $U_{n+1} \to U_n$, trivializing $K_{n+1}$. The pro-object $U := \varprojlim U_n$ then gives a pro-(finite étale) cover trivializing all $K_n$.
\end{proof}

Recall that for $X$ an affinoid connected noetherian adic space over $\Spa (\bQ_p, \bZ_p)$, there is a pro-object $X_{\infty} := \varprojlim_i X_i \to X$ constructed in \cite[Theorem 4.9]{SchpHrig}\footnote{Technically in loc.cit. $X_{\infty}$ is used to denote the \emph{associated} perfectoid affinoid space (to the pro-object). In our exposition we will denote the pro-object by $X_{\infty}$.}. This is called the \emph{universal cover} of $X$ (the $X_i$ are the Galois objects in $X_{\fet}$) and it is indeed a cover in $X_{\proet}$ as explained in \cite{SchpHrigcorr}. The point being that for $G$ a profinite group, the map $G \to *$ is a cover in $G$-pfsets. One then uses \cite[Proposition 3.5]{SchpHrig} to give a covering $X_{\infty} \to X$ in $X_{\proet}$ (even in $X_{\profet}$).

\begin{cor} \label{cor:unicofstr}
In the situation of Lemma \ref{lem:trivaknallaton}, the universal cover $X_{\infty} \to X$ of $X$ trivializes all $K_n$. Moreover in this case we have canonical trivializations: $R\Gamma(X_{\infty}, K_n) \in D(\bZ/p^n)$ is a perfect complex and the canonical morphism
\[
\ul{R\Gamma(X_{\infty}, K_n)} \to K_n \lvert_{X_{\infty}}
\]
is a quasi-isomorphism.
\end{cor}

\begin{proof}
For the first part, note that the universal cover $X_{\infty}$ covers $U$ (appearing in Lemma \ref{lem:trivaknallaton}). Thus by Lemma \ref{lem:trivaknallaton}, $K_n \lvert_{X_{\infty}} \cong \ul{M_n}$ for some perfect complex $M_n \in D(\bZ/p^n)$. Fix such a trivialization. Then by naturality it suffices to show that the canonical morphism
\[
\ul{R\Gamma(X_{\infty}, \ul{M_n})} \to \ul{M_n} \lvert_{X_{\infty}}
\]
is a quasi-isomorphism. This follows because taking sections of constant sheaves of $\bZ/p^n$-modules over $X_{\infty}$ is exact (cf. proof of \cite[Theorem 4.9]{SchPerf}) and $X_{\infty}$ is connected (i.e. $H^{0}(X_{\infty}, \bF_p) = \bF_p$). Finally perfectness of $M_n$, implies the perfectness of $R\Gamma(X_{\infty}, K_n)$.
\end{proof}

The trivializations constructed in Lemma \ref{lem:trivaknallaton} are not neccessarily compatible with the transition morphisms $K_{n+1} \to K_n$. On the other hand, Corollary \ref{cor:unicofstr} provides a fix for this. The reason is that $X_{\infty}$ is at the same time connected and a kind of \emph{weakly contractible} object for local systems of $\bZ/p^n$-modules. This simplicity is in contrast to the situation in \cite{BS}, where weakly contractible objects are rarely connected. The authors in loc.cit. must therefore work harder in order to glue the trivializations (for each $K_n$) together and this is solved by Lemma 6.6.2 in loc.cit.

We now come to the main result of this section.

\begin{prop} \label{prop:lisseimplisadicl}
Suppose $K \in D(X_{\proet}, \bZ_p)$ is lisse. Then $K$ is adic lisse. 
\end{prop}

\begin{proof}
Working locally, we can assume that $X = \Spa(A, A^{+})$ is affinoid connected. Then by Corollary \ref{cor:unicofstr}, 
\begin{equation} \label{eq:seisimco}
K \lvert_{X_{\infty}} = R\varprojlim K_n\vert_{X_{\infty}} \cong R\varprojlim \ul{R\Gamma(X_{\infty}, K_n)}.
\end{equation}
By \cite[{Tag 09AW}]{stacks-project} we can find a finite complex of finite free $\bZ_p$-modules $M$ and isomorphisms $M \otimes^{\bL}_{\bZ_p} \bZ/p^n \cong R\Gamma(X_{\infty}, K_n)$ compatible with the transition map $R\Gamma(X_{\infty}, K_{n+1}) \to R\Gamma(X_{\infty}, K_n)$. Therefore \eqref{eq:seisimco} implies
\[
K \lvert_{X_{\infty}} \cong R\varprojlim (\ul{M}\lvert_{X_{\infty}} \otimes^{\bL}_{\bZ_p} \bZ/p^n) \cong  \wh{\bZ}_p \otimes_{\bZ_p}^{\bL} \ul{M}\lvert_{X_{\infty}}.
\]
\end{proof}

\subsection{The comparison isomorphism}

In this section we assume that $f \colon \fX \to \fY$ is a proper morphism of $p$-adic formal schemes locally of finite type and flat over $\cO$ such that the induced morphism on the adic generic fibers $f_{\eta} \colon X \to Y$ is smooth. Note that in this case $X$ and $Y$ are locally of finite type over $\Spa (C, \cO)$ and are rigid spaces (considered as adic spaces). Moreover $f_{\eta}$ is also proper. 

\begin{rem}
In fact if $\fX'$ and $\fY'$ are formal schemes locally of finite type and flat over $\cO$, and $f' \colon \fX' \to \fY'$ an $\cO$-morphism, then by \cite[Remark 1.3.18(ii)]{HubEtadic}, the properness of $f'$ is equivalent to the properness of $f'_{\eta}$.
\end{rem}

With the framework established in \S\ref{sec:lisfdjkl}, we are ready to prove the relative $p$-adic \'{e}tale comparison (cf. Theorem \ref{thm:padcohse}). 

\begin{lem} \label{lem:pretalecompas}
We have an identification
\begin{equation} \label{eq:AOmegtowithmu}
A\Omega_{\fX/\fY} \otimes_{A_{\Inf}}^{\bL} A_{\Inf}[\tfrac{1}{\mu}] \cong R\nu_{f*}\bA_{\Inf, X} \otimes_{A_{\Inf}}^{\bL} A_{\Inf}[\tfrac{1}{\mu}]
\end{equation}
\end{lem}

\begin{proof}
The localization morphism $A_{\Inf} \to A_{\Inf}[\tfrac{1}{\mu}]$ is flat, so the claim is a consequence of \cite[Lemma 6.14]{BhMorSch}.
\end{proof}
For completeness we also record the following completed version, which will be useful when comparing $q$-crystalline cohomology with $p$-adic étale cohomology (cf. Theorem \ref{thm:qrysvsetale}). By \cite[Tag 0A0G]{stacks-project} $R\nu_{f*}\bA_{\Inf, X}$ is $p$-adically derived complete and so the natural morphism $A\Omega_{\fX/\fY} \to R\nu_{f*}\bA_{\Inf, X}$ induces a morphism
\begin{equation} \label{eq:fromuniprodeciom}
A\Omega_{\fX/\fY}^{\wedge} \to R\nu_{f*}\bA_{\Inf, X},
\end{equation}
where $A\Omega_{\fX/\fY}^{\wedge}$ is the $p$-adic derived completion of $A\Omega_{\fX/\fY}$.

\begin{lem} \label{lem:imprlem3.10}
The morphism 
\begin{equation} \label{eq:derivceajmo}
A\Omega_{\fX/\fY}^{\wedge} \otimes_{A_{\Inf}}^{\bL} A_{\Inf}[\tfrac{1}{\mu}] \to R\nu_{f*}\bA_{\Inf, X} \otimes_{A_{\Inf}}^{\bL} A_{\Inf}[\tfrac{1}{\mu}]
\end{equation}
induced by \eqref{eq:fromuniprodeciom} is a quasi-isomorphism. 
\end{lem}

\begin{proof}
This requires a slightly more delicate analysis compared to its uncompleted version in Lemma \ref{lem:pretalecompas}. Let $C^{\bullet}$ be the cone of canonical morphism $A\Omega_{\fX/\fY} \to R\nu_{f*}\bA_{\Inf, X}$ given by \cite[Lemma 6.10]{BhMorSch}. For each $m \geq 0$, Lemma 6.9 in loc.cit. provides a map $(\mu)^{\otimes m} \otimes_{A_{\Inf}} \tau^{[0,m]}R\nu_{f*}\bA_{\Inf, X} \to \tau^{[0,m]}A\Omega_{\fX/\fY}$ which is part of the compositions
\begin{equation} \label{eq:firstmods}
(\mu)^{\otimes m} \otimes_{A_{\Inf}} \tau^{[0,m]}A\Omega_{\fX/\fY} \to (\mu)^{\otimes m} \otimes_{A_{\Inf}} \tau^{[0,m]}R\nu_{f*}\bA_{\Inf, X} \to \tau^{[0,m]}A\Omega_{\fX/\fY}
\end{equation}
and 
\begin{equation} \label{eq:seconlkfdso}
(\mu)^{\otimes m} \otimes_{A_{\Inf}} \tau^{[0,m]}R\nu_{f*}\bA_{\Inf, X} \to \tau^{[0,m]}A\Omega_{\fX/\fY} \to \tau^{[0,m]}R\nu_{f*}\bA_{\Inf, X}
\end{equation}
Moreover both \eqref{eq:firstmods}-\eqref{eq:seconlkfdso} are just the identity tensored with the inclusion $(\mu)^{\otimes m} \hookrightarrow A_{\Inf}$. Let $C_m$ be the cone of the truncated map $\tau^{[0,m]}A\Omega_{\fX/\fY} \to \tau^{[0,m]}R\nu_{f*}\bA_{\Inf, X}$. Then by \eqref{eq:firstmods}-\eqref{eq:seconlkfdso}, $H^{i}(C_m)$ is killed by $\mu^{2m}$ and so $H^{m}(C^{\bullet})$ is killed by $\mu^{2m+2}$. The Tor-spectral sequence then shows that $H^{m}(C^{\bullet} \otimes_{\bZ}^{\bL} \bZ/p^n)$ is killed by $\mu^{4m+6}$ for each $n \geq 1$. By functoriality one obtains $R^{i}\varprojlim_{n}(H^{m}(C^{\bullet} \otimes_{\bZ}^{\bL} \bZ/p^n))$ is also killed by $\mu^{4m+6}$ for all $i \geq 0, n \geq 1$. The spectral sequence
\[
R^{i}\varprojlim_{n}(H^{m}(C^{\bullet} \otimes_{\bZ}^{\bL} \bZ/p^n)) \implies H^{i+m} (R\varprojlim_{n}(C^{\bullet} \otimes^{\bL}_{\bZ} \bZ/p^n))
\]
then shows that $H^{m}(C^{\bullet \wedge})$ is killed by some power of $\mu$, where $C^{\bullet \wedge}$ is the derived $p$-adic completion of $C^{\bullet}$. Finally note that the cone of \eqref{eq:derivceajmo} is $C^{\bullet \wedge} \otimes^{\bL}_{A_{\Inf}} A_{\Inf}[\tfrac{1}{\mu}]$ which is thus quasi-isomorphic to the zero complex.
\end{proof}

For integral $p$-adic étale cohomology, we will use the framework established in \cite[\S 8]{SchpHrig}. Recall the morphism $f \colon \fX \to \fY$ gives rise to a morphism of pro-étale topoi on the generic fiber:
\[
f_{\eta, \proet} \colon X_{\proet} \to Y_{\proet}.
\]
This gives rise to \emph{relative} integral $p$-adic étale cohomology (cf. Proposition 8.2 in loc.cit.) 
\begin{equation} \label{eq:etapadco}
Rf_{\eta, \proet*}\wh{\bZ}_p.
\end{equation}
The following lemma says that the complex $Rf_{\eta, \proet*}\wh{\bZ}_p$ is well behaved.

\begin{lem} \label{lem:procomsisasli}
The complex $Rf_{\eta, \proet*}\wh{\bZ}_p \in D(Y_{\proet}, \bZ_p)$ is adic lisse (cf. Definition \ref{defn:adiclisse}).
\end{lem}

\begin{proof}
By Proposition \ref{prop:lisseimplisadicl}, it suffices to show $Rf_{\eta, \proet*}\wh{\bZ}_p \in D(Y_{\proet}, \bZ_p)$ is lisse. We need to check that the conditions of Definition \ref{defn:lissed} are satisifed. By \cite[{Tag 0A0G}]{stacks-project} it is $p$-adically derived complete. Moreover by the projection formula (cf. \cite[{Tag 0944}]{stacks-project}) and \cite[Corollary 3.17(ii)]{SchpHrig}
\[
Rf_{\eta, \proet*}\wh{\bZ}_p \otimes^{\bL}_{\bZ_p} \bZ/p \cong \nu^{*}Rf_{\eta, \et*}(\bZ/p),
\]
where $\nu \colon Y_{\proet} \to Y_{\et}$ is the natural morphism of sites. By \cite[Theorem 10.5.1]{SchWeinsberkelec}, $Rf_{\eta, \et*}(\bZ/p)$ has finite locally constant cohomology sheaves. Furthermore by base change (cf. \cite[Proposition 2.6.1]{HubEtadic}) and vanishing of cohomology in large enough degrees (cf. \cite[Theorem 5.1]{SchpHrig}), $Rf_{\eta, \et*}(\bZ/p)$ has finite tor-dimension. This proves the lemma.
\end{proof}

We need an analogue of \cite[Theorem 5.7]{BhMorSch}. First we need some preparation. For a complex $\sC^{\bullet}$ of $\bZ_p$-modules (in a ringed site), denote by $\sC^{\bullet}_n := \sC^{\bullet} \otimes^{\bL}_{\bZ_p} \bZ/p^n$.

\begin{lem} \label{lem:critkillms}
Suppose $\sA^{\bullet} \to \sB^{\bullet}$ is a morphism in $D(Y_{\proet}, A_{\Inf})$ of $p$-adically derived complete $A_{\Inf}$-complexes such that the induced morphism
\[
H^{i}(\sA^{\bullet}_n)(U) \to H^{i}(\sB^{\bullet}_n)(U)
\]
is an almost isomorphism (w.r.t. $[\fm^{\flat}]$) for all $i$, $n \geq 1$ and $U \in Y_{\proet}$. Then the cohomology sheaves of the cone of $\sA^{\bullet} \to \sB^{\bullet}$ are killed by $W(\fm^{\flat})$.
\end{lem}

\begin{proof}
Let $\sC$ be the cone. We will prove the statement in four steps.
\begin{enumerate}[\textbf{Step} 1:]
    \item  ``$[\fm^{\flat}]$ kills $H^{i}(\sC^{\bullet}_n)(U)$". Note that $H^{i}(\sC^{\bullet}_n)$ sits in an exact sequence
    \begin{equation} \label{eq:longsnoska}
        H^{i}(\sA^{\bullet}_n) \to H^{i}(\sB^{\bullet}_n) \to H^{i}(\sC^{\bullet}_n) \to  H^{i+1}(\sA^{\bullet}_n) \to H^{i+1}(\sB^{\bullet}_n).
    \end{equation}
The kernel and cokernel presheaves of the first and last arrows of \eqref{eq:longsnoska} are killed by $[\fm^{\flat}]$. Since sheafification preserves the property 
\begin{equation} \label{eq:propofkill}
``\text{local sections killed by } [\fm^{\flat}]"
\end{equation} it follows that the cokernel of the first and last arrows is killed by $[\fm^{\flat}]$. It is easy to check that \eqref{eq:propofkill} is also stable under extensions and this proves that $[\fm^{\flat}]$ kills $H^{i}(\sC^{\bullet}_n)(U)$.
\item ``$[\fm^{\flat}]$ kills $H^{i}(U, \sC^{\bullet}_n)$". By the Leray spectral sequence (cf. \cite[{Tag 0732}]{stacks-project})
\[
H^{p}(U, H^q(\sC^{\bullet}_n)) \implies H^{p+q}(U, \sC^{\bullet}_n),
\]
it suffices to show that each $H^{p}(U, H^q(\sC^{\bullet}_n))$ is killed by $[\fm^{\flat}]$. \textbf{Step} 1 says that this is true for $p = 0$. In other words for $m \in [\fm^{\flat}]$, the multiplcation map
\[
m \colon H^q(\sC^{\bullet}_n) \to H^q(\sC^{\bullet}_n)
\]
is zero. By functorality, this implies that the induced map 
\[
m \colon H^{p}(U, H^q(\sC^{\bullet}_n))  \to H^{p}(U, H^q(\sC^{\bullet}_n)) 
\]
is also zero.
\item ``$W(\fm^{\flat})$ kills $H^{i}(U, \sC^{\bullet})$". By \cite[{Tag 0A0G}]{stacks-project} and \cite[Lemma 6.15]{BhMorSch}, $H^{i}(U, \sC^{\bullet})$ is $p$-adically derived complete. Thus by \cite[Lemma 3.17]{SemistabAinfcoh}, it suffices to check that it is killed by $[\fm^{\flat}]$. This follows from \textbf{Step} 2 and \cite[{Tag 0D6K}]{stacks-project} (note that the $R^{1}\varprojlim$ of almost zero modules is almost zero).
\item ``$W(\fm^{\flat})$ kills $H^{i}(\sC^{\bullet})(U)$". This follows from \textbf{Step} 3 and the fact that $H^{i}(\sC^{\bullet})$ is the sheafification of $ U \mapsto H^{i}(U, \sC^{\bullet})$.
\end{enumerate}
\end{proof}

\begin{prop} \label{prop:exrn517shco}
The morphism $f_{\eta, \proet}^{-1}\bA_{\Inf,Y} \to \bA_{\Inf, X}$ induces a morphism
\begin{equation}\label{eq:relfinres}
Rf_{\eta, \proet*}\wh{\bZ}_p \otimes_{\wh{\bZ}_p}^{\bL} \bA_{\Inf,Y} \to Rf_{\eta, \proet*}\bA_{\Inf,X}
\end{equation}
such that the cohomology sheaves of its cone are killed by $W(\fm^{\flat})$.
\end{prop}

\begin{proof}
We check that the conditions of Lemma \ref{lem:critkillms} are satisifed. By Lemma \ref{lem:procomsisasli}, $Rf_{\eta, \proet*}\wh{\bZ}_p$ is an adic lisse complex on $Y_{\proet}$. It follows that the LHS of \eqref{eq:relfinres} is $p$-adically derived complete. By \cite[{Tag 0A0G}]{stacks-project} the RHS of \eqref{eq:relfinres}, $Rf_{\eta, \proet*}\bA_{\Inf,X}$ is $p$-adically derived complete. 

Finally after tensoring with $\bZ/p^n$ and taking cohomology, \eqref{eq:relfinres} takes the shape
\begin{equation} \label{eq:takdsjdf}
\nu^{*}R^{i}f_{\eta, \et*}\bZ/p^n \otimes_{\wh{\bZ}_p} W(\wh{\cO}^{+, \flat}_{Y}) \to R^{i}f_{\eta, \proet*}(W(\wh{\cO}^{+, \flat}_{X})/p^n).
\end{equation}
The proof of \cite[Theorem 8.8]{SchpHrig} shows that (after taking sections) \eqref{eq:takdsjdf} is an almost isomorphism (w.r.t to $[\fm^{\flat}]$). 
\end{proof}

After inverting $\mu$, the relative $A_{\Inf}$-cohomology $R\Gamma_{A_{\Inf}}(\fX/\fY)$ (cf. Definition \ref{def:relAOmegacomplex}) and its derived $p$-adic completion capture \eqref{eq:etapadco}. More precisely:

\begin{thm} \label{thm:padcohse}
We have the identifications 
\[
R\Gamma_{A_{\Inf}}(\fX/\fY) \otimes_{A_{\Inf}}^{\bL} A_{\Inf}[\tfrac{1}{\mu}] \cong R\Gamma_{A_{\Inf}}(\fX/\fY)^{\wedge} \otimes_{A_{\Inf}}^{\bL} A_{\Inf}[\tfrac{1}{\mu}] \cong Rf_{\eta, \proet*}\wh{\bZ}_p \otimes_{\wh{\bZ}_p}^{\bL} \bA_{\Inf,Y}[\tfrac{1}{\mu}].
\]
\end{thm}

\begin{proof}
First we show the left-most quasi-isomorphism. Indeed Lemmas \ref{lem:pretalecompas} and \ref{lem:imprlem3.10} imply that the canonical morphism 
\[
A\Omega_{\fX/\fY} \otimes_{A_{\Inf}}^{\bL} A_{\Inf}[\tfrac{1}{\mu}] \to A\Omega_{\fX/\fY}^{\wedge} \otimes_{A_{\Inf}}^{\bL} A_{\Inf}[\tfrac{1}{\mu}]
\]
is a quasi-isomorphism. The result now follows from \cite[Exposé VI, Théor\`{e}me 5.1]{SGA4} (applied to Proposition \ref{prop:ousiDalge}(4)) and \cite[{Tag 0A0G}]{stacks-project}.

It remains to compare the two outer terms.
Since $\mu \in W(\fm^{\flat})$, this is a consequence of Lemma  \ref{lem:pretalecompas}, Proposition \ref{prop:exrn517shco} and \cite[Exposé VI, Théor\`{e}me 5.1]{SGA4} (again applied to Proposition \ref{prop:ousiDalge}(4)).
\end{proof}

\section{The local analysis of $A\Omega_{\fX / \fY}$} \label{sec:localanal}

Let $R$ and $S$ be flat $\cO$-algebras equipped with the $p$-adic topology such that $R[p^{-1}]$ and $S[p^{-1}]$ are strongly noetherian. Moreover we assume $R$ and $S$ are complete. We assume throughout \S \ref{sec:localanal} that $\fY = \Spf (S)$, $\fX = \Spf (R)$ and $f \colon \fX \to \fY$ is any smooth $\Spf(\cO)$-morphism such that for some $d \geq 0$, there is an étale $\Spf(S)$-morphism
\[
\fX = \Spf (R) \rightarrow \Spf (R^{\square}) =: \fX^{\square}
\]
where $R^{\square} := S \{t_1^{\pm 1}, \ldots, t_d^{\pm 1} \}$. Let $f_{\eta} \colon X \to Y$ denote the induced morphism on the adic generic fibers. Armed with Definition \ref{def:relAOmegacomplex}, the first goal of this section is to relate the cohomology of the complex $Rf_{\eta *}\bA_{\Inf,X}$, in terms of continuous group cohomology (cf. \cite[Lemma 1.16]{BhMorSch} in the smooth case and \cite[Theorem 3.20]{SemistabAinfcoh} in the semistable case), particularly after killing $\mu$-torsion. The origin of such a computation goes back to ideas of Faltings. The main steps are analogous to those appearing in the absolute case (cf. \cite[Lemma 1.16]{BhMorSch} in the smooth case and \cite[Theorem 3.20]{SemistabAinfcoh} in the semistable case). The main difference in the current situation is that the pro-étale site of the target $Y$ adds some extra technicalities.  

\subsection{The perfectoid cover $X_{\infty}$}\label{sec:perfcov}

We use the setup in \cite[\S 5.1]{BhMorSch} when dealing with the pro-étale site associated to a locally noetherian adic space. We fix some pro-étale morphism  
\[
Y_{\infty} := \varprojlim_{i \in I} \Spa(S_{i}, S_{i}^{+}) \rightarrow  Y,
\]
where $Y_{\infty}$ is some affinoid perfectoid object in $Y_{\proet}$. Setting $S_{\infty} := (\varinjlim_i S_i^{+})^{\wedge}$ where the completion is $p$-adic one has the following fact which is implicit in loc.cit.:
\begin{lem} \label{baslem:torfree}
$S_{\infty}$ is a torsion free $\cO$-module. 
\end{lem}

\begin{proof} 
Note first that each $S_i$ is a $C$-vector space and $S_i^{+} \subset S_i$ is a subring and therefore $S_i^{+}$ is a torsion free $\cO$-module. The direct limit of torsion free $\cO$-modules is again torsion free and so $\varinjlim_i S_i^{+}$ is a torsion free $\cO$-module. Since $\cO$ is a valuation ring by \cite[{Tag 0539}]{stacks-project}, this is equivalent to $\varinjlim_i S_i^{+}$ being a flat $\cO$-module. One then applies \cite[Lemma 7.1.6(i)]{GabRam} to conclude that the $p$-adic completion $S_{\infty}$ is a flat and therefore a torsion free $\cO$-module. 
\end{proof}

The pro-object $Y_{\infty}$ therefore determines the perfectoid space $\Spa (S_{\infty}[\tfrac{1}{p}], S_{\infty})$. Base changing along $Y_{\infty}$, we will construct a perfectoid cover over $X \times_{Y} Y_{\infty}$. More precisely, for each $m \geq 0$, we consider the $R^{\square}$-algebra
\begin{equation} \label{eq:defnofRmsqau}
R^{\square}_{m} := S_{\infty}\{t_1^{\pm 1/p^{m}}, \ldots, t_d^{\pm 1/p^{m}} \} \text{ and } R_{\infty}^{\square} := (\varinjlim R^{\square}_m)^{\wedge},
\end{equation}
where the completion is $p$-adic. Explicitly, we have the $p$-adically completed direct sum decomposition
\begin{equation} \label{eq:smallsum}
R_{\infty}^{\square} \cong \widehat{\bigoplus}_{(a_1,\ldots, a_d) \in \bZ [\tfrac{1}{p}]^{\oplus d}} S_{\infty} \cdot t_{1}^{a_1} \cdots t_d^{a_d}.
\end{equation}
The corresponding $R$-algebras are 
\begin{equation}\label{eq:definRinfty}
R_m := (R \wh{\otimes}_{R^{\square}}R^{\square}_0) \otimes_{R^{\square}_0} R^{\square}_m \text{ and } R_{\infty} := (\varinjlim R_m)^{\wedge} \cong R \wh{\otimes}_{R^{\square}} R_{\infty}^{\square}.
\end{equation}
As a consequence of Lemma \ref{baslem:torfree}, we obtain:

\begin{cor} \label{cor:comtentorfree}
$R_{\infty}$ is a torsion free $\cO$-module.
\end{cor}

\begin{proof}
By Lemma \ref{baslem:torfree} (together with \cite[Lemma 7.1.6(i)]{GabRam}), it follows that $R_{\infty}^{\square}$ is a torsion free $\cO$-module. Since $R$ is a $p$-completely étale $R^{\square}$-algebra, it follows that $R$ is the $p$-adic completion of an étale $R^{\square}$-algebra $R'$. Therefore $R' \otimes_{R^{\square}} R_{\infty}^{\square}$ is a flat $R_{\infty}^{\square}$-algebra. Hence, $R' \otimes_{R^{\square}} R_{\infty}^{\square}$ is a flat $\cO$-algebra. One then applies \cite[Lemma 7.1.6(i)]{GabRam} to conclude.
\end{proof}

We collect some properties about the algebras $R_m$.

\begin{lem}
For all $m \geq 0$, $R_m$ is $p$-adically complete and $R_m$ is integrally closed in $R_m[\tfrac{1}{p}]$.
\end{lem}

\begin{proof}
The first part follows the same idea as the proof of Corollary \ref{cor:comtentorfree} and in fact we use the algebra $R'$ appearing there. Indeed $R' \otimes_{R^{\square}} R_{0}^{\square}$ is a flat $R_{0}^{\square}$-algebra. Hence by \cite[Lemma 7.1.6(i)]{GabRam}, so is $R \wh{\otimes}_{R^{\square}}R^{\square}_0$. As $R^{\square}_m$ is a finitely presented $R^{\square}_0$-module, the result follows \cite[Lemma 7.1.6(ii)]{GabRam}. The second part follows from \cite[{Tag 03GG}]{stacks-project} (by working with uncompleted algebras).
\end{proof}

The summands in \eqref{eq:smallsum} with $a_j \notin \bZ$ for some $1 \leq j \leq d$ comprise an $R^{\square}$-submodule $M^{\square}_{\infty}$ of $R^{\square}_{\infty}$ and we set $M_{\infty} := R \wh{\otimes}_{R^{\square}} M_{\infty}^{\square}$. Thus, we have the $R^{\square}$-module (resp. $R$-module) decomposition

\begin{equation} \label{eq:smallsumdec}
R_{\infty}^{\square} \cong (R^{\square} \wh{\otimes}_S S_{\infty}) \oplus M_{\infty}^{\square} \text{  (resp. }  R_{\infty} \cong (R \wh{\otimes}_S S_{\infty}) \oplus M_{\infty}\text{)}.
\end{equation}

The profinite group
\[
\Delta := \left\{ (\epsilon_1, \ldots, \epsilon_d) \in \varprojlim_{m \geq 0} \left( \mu_{p^m}(\cO) \right)^{\oplus d} \right\} \simeq \bZ_p^{\oplus d}
\]
acts $(R^{\square} \wh{\otimes}_S S_{\infty})$-linearly on $R_{m}^{\square}$ by scaling each $t_{j}^{1/p^m}$ by the $\mu_{p^m}$-component of $\epsilon_j$. The induced actions of $\Delta$ on $R_{\infty}^{\square}$ and $R_{\infty}$ are continuous, compatible, and preserve the decompositions \eqref{eq:smallsum} and \eqref{eq:smallsumdec}.

After inverting $p$, for each $m \geq 0$, we have
\[
R^{\square}_{m}[\tfrac{1}{p}] \cong \bigoplus_{a_1, \ldots, a_d \in \{ 0, \tfrac{1}{p^m}, \ldots, \tfrac{p^m -1}{p^m} \}} (R^{\square} \wh{\otimes}_S S_{\infty})[\tfrac{1}{p}] \cdot t_{1}^{a_1} \cdots t_d^{a_d}
\]
so $R^{\square}_m [\frac{1}{p}]$ is the $(R^{\square} \wh{\otimes}_S S_{\infty})[\tfrac{1}{p}]$-algebra obtained by adjoining the $(p^{m})^{\text{th}}$ roots of $t_1, \ldots, t_d \in \left( (R^{\square} \wh{\otimes}_S S_{\infty})[\tfrac{1}{p}] \right)^{\times}$, and hence is finite étale over $(R^{\square} \wh{\otimes}_S S_{\infty})[\tfrac{1}{p}]$. Therefore, $\varinjlim_{m} \left( R_m^{\square}[\tfrac{1}{p}] \right)$ is a pro-(finite étale) $\Delta$-cover of $(R^{\square} \wh{\otimes}_S S_{\infty})[\tfrac{1}{p}]$. We have $R_m^{\square} = \left( R_m^{\square}[\tfrac{1}{p}] \right)^{\circ}$, so the pro-object\ildar{\footnote{Strictly speaking, to get an honest pro-étale presentation, one should replace each $\Spa (R_m^{\square}[\tfrac{1}{p}], R_m^{\square})$ by the pro-object $X_{m}^{\square} \times_{Y} Y_{\infty}$ where $X_{m}^{\square} \to X^{\square}$ is a finite étale Kummer cover obtained by adjoining the $(p^{m})^{\text{th}}$ roots of $t_1, \ldots, t_d \in R^{\square}$.   }} 
\[
X_{\infty}^{\square} := \varprojlim_{m} \Spa (R_m^{\square}[\tfrac{1}{p}], R_m^{\square})
\]
is a pro-(finite étale) $\Delta$-cover of the pro-object $X^{\square} \times_{Y} Y_{\infty}$. Of course, $X_{\infty}^{\square}$ is an affinoid perfectoid object in the site $X^{\square}_{\proet}$. Base changing the situation along $X \to X^{\square}$ we obtain the tower\footnote{Analogous to the previous footnote, one should consider the finite étale Kummer cover $X_m \to X$ to obtain an honest pro-étale presentation.}
\[
X_{\infty} := \varprojlim_{m} \Spa (R_m[\tfrac{1}{p}], R_m)
\]
which is a pro-(finite étale) $\Delta$-cover of the pro-object $X \times_{Y} Y_{\infty}$. As before $X_{\infty}$ becomes an affinoid perfectoid object in the site $X_{\proet}$.

We end this section with a result that allows us to control powerbounded elements in the ring of functions of $X_{\infty}$ (or rather the adic space associated to $X_{\infty}$).

\begin{lem} \label{lem:topfinflasir}
Suppose that $S$ is topologically of finite type 
over $\cO$. Then $(R_{\infty}[\tfrac{1}{p}])^{\circ} = R_{\infty}$.
\end{lem}

\begin{proof}
We have $R_{\infty} \cong R \wh{\otimes}_{R^{\square}} R_{\infty}^{\square}$ and the completed direct sum decomposition \eqref{eq:smallsum} shows that $R_{\infty}^{\square}$ has a $p$-adically completed direct sum decomposition into free $R^{\square}\widehat{\otimes}_{\cO} S_{\infty}$-modules. Therefore $R_{\infty}$ has a $p$-adically completed direct sum decomposition into free $R \widehat{\otimes}_{\cO} S_{\infty}$-modules. Since $S$ is topologically of finite type over $\cO$, so is $R$. Moreover $R$ is topologically of finite presentation by \cite[\S 7.3, Corollary 5]{Bosch} and therefore $R$ is the $p$-adic completion of a free $\cO$-module by \cite[Lemma 8.10]{BhMorSch} (cf. footnote \ref{fnote}). The result now follows from Lemma \ref{lem:rigspacirplus}.
\end{proof}

\subsection{The cohomology of $\wh{\cO}^{+}$ and continuous group cohomology} \label{subsec:cohcongrp}

We prove a relative version of \cite[Corollary 8.13(i)]{BhMorSch}. By \cite[Proposition 3.5, Proposition 3.7(iii), Corollary 6.6]{SchpHrig} the \v{C}ech complex of the sheaf $\wh{\cO}^{+}_X$ with respect to the pro-(finite étale) affinoid perfectoid cover
\[
X_{\infty} \to X \times_{Y} Y_{\infty}
\]
is identified with the continuous cochain complex $R\Gamma_{\cont}(\Delta, R_{\infty})$.  By \cite[{Tag 01GY}]{stacks-project} this gives an edge map 
\begin{equation} \label{eq:relver3.3.1}
 e \colon R\Gamma_{\cont}(\Delta, R_{\infty}) \rightarrow R\Gamma_{\text{proét}} (X \times_{Y} Y_{\infty},  \wh{\cO}^{+}_X).
\end{equation}
By the almost purity theorem \cite[Lemma 4.10(v)]{SchpHrig}, the maximal ideal $\fm \subset \cO$ kills the cohomology groups of $\text{Cone}(e)$.
Given an injective resolution $\wh{\cO}^{+}_X \to \cI^{\bullet}$, we have
\begin{align*}
R\Gamma_{\text{proét}} (X \times_{Y} Y_{\infty},  \wh{\cO}^{+}_X) &\text{ is represented by } \Gamma (X \times_{Y} Y_{\infty}, \cI^{\bullet}) \\
Rf_{\eta*}\wh{\cO}^{+}_X &\text{ is represented by } f_{\eta *}\cI^{\bullet} \\
R\Gamma_{\text{proét}} (Y_{\infty},  Rf_{\eta*}\wh{\cO}^{+}_X) &\text{ is represented by } \Gamma (Y_{\infty}, f_{\eta*}\cI^{\bullet}).
\end{align*}
The second and third points follow from \cite[{Tag 0731}]{stacks-project} and \cite[{Tag 072Z}]{stacks-project} when viewing $\wh{\cO}^{+}_X$ as an object of $D(X_{\text{proét}}, \cO)$ or $D(X_{\text{proét}}, \wh{\cO}^{+}_X)$. By definition of the pushforward we get that $\Gamma (X \times_{Y} Y_{\infty}, \cI^{\bullet}) = \Gamma (Y_{\infty}, f_{\eta*}\cI^{\bullet})$. Therefore the edge map can be rewritten as
\begin{equation} \label{eq:relver3.3.2}
 e \colon R\Gamma_{\cont}(\Delta, R_{\infty}) \rightarrow R\Gamma_{\text{proét}} (Y_{\infty},  Rf_{\eta*}\wh{\cO}^{+}_X),
\end{equation}
viewed as a morphism in $D(\cO)$.

We will now show that $L\eta_{(\zeta_p - 1)}(e)$ is an isomorphism, so that $L\eta_{(\zeta_p - 1)}R\Gamma (Y_{\infty},  Rf_{\eta*}\wh{\cO}^{+}_X)$ is computed in terms of continuous group cohomology. For this, we will use \cite[Lemma 8.11(i)]{BhMorSch} (cf. \cite[Lemma 3.4]{SemistabAinfcoh}). In preparation, first we need to control the $\fm$-torsion in the group cohomology of $R_{\infty}/x$ for some $x \in \cO$. 

\begin{prop} \label{prop:mtorsioncoho}
The element $\zeta_{p}-1$ kills the $\cO$-modules $H^{i}_{\cont}(\Delta, M_{\infty})$. If moreover $(R_{\infty}[\tfrac{1}{p}])^{\circ} = R_{\infty}$, then for each $b \in \cO$, the $\cO$-modules $R_{\infty}/b$ and $H^{i}_{\cont}(\Delta, R_{\infty}/b)$ have no nonzero $\fm$-torsion. 
\end{prop}

\begin{proof}
Firstly by Lemma \ref{lem:nonoto}, it follows that $R_{\infty}/b$ has no nonzero $\fm$-torsion.  We will now show that $H^{i}_{\cont}(\Delta, R_{\infty}/b)$ has no non-zero $\fm$-torsion. As we are in the smooth case, we will have access to a decomposition of $M_{\infty}$ of which all of the terms will be $R \wh{\otimes}_{R^{\square}} R_{0}^{\square}$-modules of rank 1. This will allow us to treat each term one by one, analogous to the situation in \cite[Proposition 8.9]{BhMorSch}. The semistable case handled in \cite[Proposition 3.8]{SemistabAinfcoh} is slightly more delicate. With this in mind consider $S_{\infty} \cdot t_{1}^{a_1} \cdots t_d^{a_d}$ a summand of \eqref{eq:smallsum} and set $T := R_{0}^{\square} \cdot t_{1}^{a_1} \cdots t_d^{a_d}$. 

By \cite[Lemma 3.7]{SemistabAinfcoh} (cf. \cite[Lemma 7.3(ii)]{BhMorSch}) the $\cO$-module $H^{i}_{\cont}(\Delta, T/b)$ is the $i$-th cohomology of the $R_{0}^{\square}/b$-tensor product of $d$-complexes of the form $R_{0}^{\square}/b \xrightarrow{\zeta -1} R_{0}^{\square}/b$ for suitable $p$-power roots of unity $\zeta$. By Lemma \ref{baslem:torfree}, one obtains
\begin{equation} \label{eq:cohofonesumaman}
H^{i}_{\cont}(\Delta, T/b) = M \otimes_{D} R_{0}^{\square},
\end{equation}
where $M$ is a $D$-module for some discrete valuation ring $D \subset \cO$ (the $b$ is absorbed into the $D$-module $M$). Consider the decomposition
\begin{equation} \label{eq:decompmtor}
R_{\infty}/b \cong  (R \wh{\otimes}_S S_{\infty})/b \oplus \underbrace{\widehat{\bigoplus}_{T} ((R \wh{\otimes}_{R^{\square}}R^{\square}_0) \otimes_{R^{\square}_0} T)/b}_{M_{\infty}/b}. 
\end{equation}
where the direct sum runs over those $T$ which contribute to $M_{\infty}$. In order to show that $H^{i}_{\cont}(\Delta, R_{\infty}/b)$ has no non-zero $\fm$-torsion, we need to show that both $H^{i}_{\cont}(\Delta, (R \wh{\otimes}_S S_{\infty})/b)$ and $H^{i}_{\cont}(\Delta, M_{\infty}/b)$ have no non-zero $\fm$-torsion.

Since $R$ is $R^{\square}$-flat,  \cite[Lemma 7.1.6(i)]{GabRam} implies $R \wh{\otimes}_{R^{\square}}R^{\square}_0$ is $R^{\square}_0$-flat. Moreover $(R \wh{\otimes}_{R^{\square}}R^{\square}_0) \otimes_{R^{\square}_0} T$ is $p$-adically complete and so 
\cite[Lemma 3.7]{SemistabAinfcoh} (cf. \cite[Lemma 7.3(ii)]{BhMorSch}) gives
\begin{equation} \label{eq:cohdecolikmf}
H^{i}_{\cont}(\Delta, ((R \wh{\otimes}_{R^{\square}}R^{\square}_0) \otimes_{R^{\square}_0} T)/b) \cong (R \wh{\otimes}_{R^{\square}}R^{\square}_0) \otimes_{R^{\square}_0} H^{i}_{\cont}(\Delta, T/b).
\end{equation}
Therefore by \eqref{eq:cohofonesumaman} and \cite[Lemma 3.6]{SemistabAinfcoh}, to show that $H^{i}_{\cont}(\Delta, M_{\infty}/b)$ has no nonzero $\fm$-torsion, it suffices to show that $(R \wh{\otimes}_{R^{\square}}R^{\square}_0) \otimes_{R^{\square}_0} (M \otimes_{D} R_{0}^{\square})$ has no nonzero $\fm$-torsion. This reduces to showing that $(R \wh{\otimes}_{R^{\square}}R^{\square}_0) \otimes_{R^{\square}_0} R^{\square}_0/d \cong (R \wh{\otimes}_{R^{\square}}R^{\square}_0)/d$ has no nonzero $\fm$-torsion for some $d \in D$. However
\[
(R \wh{\otimes}_{R^{\square}}R^{\square}_0)/d \cong (R \wh{\otimes}_S S_{\infty})\{t_1^{\pm 1}, \ldots, t_d^{\pm 1} \}/d
\]
and since $(R \wh{\otimes}_S S_{\infty})/d$ sits inside $R_{\infty}/d$, it follows that $(R \wh{\otimes}_{R^{\square}}R^{\square}_0)/d$ has no nonzero $\fm$-torsion. Finally $\Delta$ acts trivially on $(R \wh{\otimes}_S S_{\infty})/b$ and so $H^{i}_{\cont}(\Delta,(R \wh{\otimes}_S S_{\infty})/b)$ has no nonzero $\fm$-torsion. 

It remains to show $\zeta_{p}-1$ kills the $\cO$-modules $H^{i}_{\cont}(\Delta, M_{\infty})$. If $T$ contributes to $M_{\infty}$, then as in the proof of \cite[Proposition 3.8]{SemistabAinfcoh}, we conclude
\begin{equation} \label{eq:zetakillscohom}
\zeta_{p}-1 \text{ kills }H^{i}_{\cont}(\Delta, T).
\end{equation}
By \eqref{eq:decompmtor}-\eqref{eq:cohdecolikmf} and \cite[Lemma 3.6]{SemistabAinfcoh}, we get the desired result.
\end{proof}

\begin{thm} \label{thm:edgemapisoOsheaf}
The edge map $e$ defined in \eqref{eq:relver3.3.2} induces an almost isomorphism
\[
 L\eta_{(\zeta_p - 1)}(e) \colon L\eta_{(\zeta_p - 1)}R\Gamma_{\cont}(\Delta, R_{\infty}) \rightarrow L\eta_{(\zeta_p - 1)}R\Gamma_{\proet} (Y_{\infty},  Rf_{\eta*}\wh{\cO}^{+}_X).
\]
Similarly the induced map (coming from the inclusion $i \colon R\widehat{\otimes}_S S_{\infty} \hookrightarrow R_{\infty}$)
\begin{equation}\label{eq:fsmsfdssd}
  L\eta_{(\zeta_p - 1)}(e \circ R\Gamma_{\cont}(i)) \colon L\eta_{(\zeta_p - 1)}R\Gamma_{\cont}(\Delta, R \widehat{\otimes}_S S_{\infty}) \rightarrow L\eta_{(\zeta_p - 1)}R\Gamma_{\proet} (Y_{\infty},  Rf_{\eta*}\wh{\cO}^{+}_X)
\end{equation}
is an almost isomorphism. Moreover if $(R_{\infty}[\tfrac{1}{p}])^{\circ} = R_{\infty}$, then both almost isomorphisms are  isomorphisms.
\end{thm}

\begin{proof}
Recall by \S\ref{subsec:cohcongrp}, the edge map $e$ is an almost isomorphism. Therefore $L\eta_{(\zeta_p - 1)}(e)$ is an almost isomorphism. By Proposition \ref{prop:mtorsioncoho}, $i$ induces an isomorphism
\[
\frac{H^{i}_{\cont}(\Delta, R_{\infty})}{H^{i}_{\cont}(\Delta, R_{\infty})[\zeta_p -1]} \cong \frac{H^{i}_{\cont}(\Delta, R\widehat{\otimes}_S S_{\infty})}{H^{i}_{\cont}(\Delta, R\widehat{\otimes}_S S_{\infty})[\zeta_p -1]}.
\]
Therefore $L\eta_{(\zeta_p - 1)}(R\Gamma_{\cont}(i))$ is an isomorphism and hence $L\eta_{(\zeta_p - 1)}(e \circ R\Gamma_{\cont}(i))$ is an almost isomorphism.

Suppose now $(R_{\infty}[\tfrac{1}{p}])^{\circ} = R_{\infty}$. To show that $L\eta_{(\zeta_p - 1)}(e)$ is an isomorphism, it is enough to show that $L\eta_{(\zeta_p - 1)}(e \circ R\Gamma_{\cont}(i))$ is an ismorphism. We will apply \cite[Lemma 8.11(i)]{BhMorSch} (cf. \cite[Lemma 3.4]{SemistabAinfcoh}). Proposition \ref{prop:mtorsioncoho} ensures that the $\cO$-modules $H^{i}_{\cont}(\Delta, R\widehat{\otimes}_S S_{\infty})$ have no nonzero $\fm$-torsion and that the quotient $\tfrac{H^{i}_{\cont}(\Delta, R\widehat{\otimes}_S S_{\infty})}{(\zeta_p -1)H^{i}_{\cont}(\Delta, R\widehat{\otimes}_S S_{\infty})}$ is a direct sum of $\cO$-modules of the form $(R\widehat{\otimes}_S S_{\infty})/(\zeta_p -1)$. These also have no nonzero $\fm$-torsion. Therefore $L\eta_{(\zeta_p - 1)}(e \circ R\Gamma_{\cont}(i))$ is an isomorphism. 
\end{proof}

The main goal of this section is an analogue of Theorem \ref{thm:edgemapisoOsheaf} for the complex $\bA_{\Inf,X}$ (see Theorem \ref{edgemapisoAinf}). To prepare for it, in \S\ref{tiltofRinfty} and \S\ref{ringAinfonRinfty} we describe the values of the sheaves $\wh{\cO}^{+, \flat}_X$ and $W(\wh{\cO}^{+, \flat}_X)$. 

\subsection{The tilt $R_{\infty}^{\flat}$}\label{tiltofRinfty}

We fix a system of compatible $p^n$-power roots $p^{1/p^{\infty}} := (p^{1/p^n})_{n \geq 0}$ of $p$ in $\cO$. Thanks to the explicit description \eqref{eq:smallsum} of the perfectoid ring $R_{\infty}^{\square}$, its tilt $(R_{\infty}^{\square})^{\flat} := \varprojlim_{y \mapsto y^p} (R_{\infty}^{\square}/p)$ is described explictly by the identification
\begin{align*} 
(R_{\infty}^{\square})^{\flat} &\cong (\varinjlim_{m} S_{\infty}^{\flat}[x_1^{\pm 1/p^{m}}, \ldots, x_d^{\pm 1/p^{m}}])^{\wedge} \\
&\cong \widehat{\bigoplus}_{(a_1,\ldots, a_d) \in \bZ [\tfrac{1}{p}]^{\oplus d}} S_{\infty}^{\flat} \cdot x_{1}^{a_1} \cdots x_d^{a_d},
\end{align*}
where $x_{i}^{1/p^m}$ corresponds to the $p$-power compatible sequence $(\ldots, t_i^{1/p^{m+1}}, t_i^{1/p^m})$ of elements of $R_{\infty}^{\square}$, where the completions are $p^{1/p^{\infty}}$-adic, and the decomposition is as $S_{\infty}^{\flat}$-modules. Thus,
\[
\text{the tilt }R_{\infty}^{\flat} := \varprojlim_{y \mapsto y^p} (R_{\infty}/p) \text{ of the perfectoid ring }R_{\infty}
\]
is identified with the $p^{1/p^{\infty}}$-adic completion of any lift of the étale $R_{\infty}^{\square}/p$-algebra $R_{\infty}/p$ to an étale $(R_{\infty}^{\square})^{\flat}$-algebra (such a lift exists by \cite[{Tag 04D1}]{stacks-project}). By \cite[Lemma 5.11(i)]{SchpHrig} the value on $X_{\infty}$ of the sheaf $\wh{\cO}^{+, \flat}_X$ is the ring $R_{\infty}^{\flat}$. 

By functoriality, the group $\Delta$ acts continuously and $S_{\infty}^{\flat}$-linearly on $(R_{\infty}^{\square})^{\flat}$ and $R_{\infty}^{\flat}$. Explicitly, $\Delta$ respects the completed direct sum decomposition and an $(\epsilon_1, \ldots, \epsilon_d) \in \Delta$ scales $x_{j}^{a_j}$ by $\epsilon_j^{a_j} \in \cO^{\flat}$. 

We will need the following analogue of \cite[Lemma 3.12]{SemistabAinfcoh} for our analysis in \S\ref{ringAinfonRinfty} of the value on $X_{\infty}$ of the sheaf $\bA_{\Inf, X}$.

\begin{lem} \label{lem:nonozeromflatro}
Assume that $(R_{\infty}[\tfrac{1}{p}])^{\circ} = R_{\infty}$. Then both $R_{\infty}^{\flat}/b$ and $H^{i}_{\cont}(\Delta, R_{\infty}^{\flat}/b)$ for each $b \in \cO^{\flat} \backslash \{ 0\}$ have no nonzero $\fm^{\flat}$-torsion.
\end{lem}

\begin{proof}
By using Frobenius we can assume that $b \text{ }\lvert \text{ }p^{1/p^{\infty}}$ in $\cO^{\flat}$. Then Proposition \ref{prop:mtorsioncoho} and the $\Delta$-isomorphism $R_{\infty}^{\flat}/b \cong R_{\infty}/b^{\sharp}$ for some $b^{\sharp} \in \cO$ gives the claim.
\end{proof}

\subsection{The ring $\bA_{\Inf}(R_{\infty})$}\label{ringAinfonRinfty}

By \cite[Theorem 6.5(i)]{SchpHrig}, the value of $H^{0}(R\Gamma(X_{\infty}, \bA_{\Inf,X}))$ is the ring
\[
\bA_{\Inf}(R_{\infty}):=W(R_{\infty}^{\flat}).
\]
In general we will denote by $\bA_{\Inf}(P)$ for a perfectoid ring $P$ over $\cO$ to be $W(P^{\flat})$.

\begin{lem} \label{lem:nononzerowfalttor}
Assume that $(R_{\infty}[\tfrac{1}{p}])^{\circ} = R_{\infty}$. Then each quotient 
\[
\bA_{\Inf}(R_{\infty})/(p^{n},\mu^{n'})  \text{ so also } \bA_{\Inf}(R_{\infty})/\mu \text{ has no nonzero } W(\fm^{\flat})\text{-torsion}.
\]
\end{lem}

\begin{proof}
This follows from Lemma \ref{lem:nonozeromflatro} (cf. \cite[(3.14.1)]{SemistabAinfcoh}).
\end{proof}

Due to the explicit construction of Witt vectors we obtain
\begin{align*} 
\bA_{\Inf}(R_{\infty}^{\square}) &\cong (\varinjlim_{m} \bA_{\Inf}(S_{\infty})[X_1^{\pm 1/p^{m}}, \ldots, X_d^{\pm 1/p^{m}}])^{\wedge} \\
&\cong \widehat{\bigoplus}_{(a_1,\ldots, a_d) \in \bZ [\tfrac{1}{p}]^{\oplus d}} \bA_{\Inf}(S_{\infty}) \cdot X_{1}^{a_1} \cdots X_d^{a_d},
\end{align*}
where the completions are $(p, \mu)$-adic, the decomposition is as $\bA_{\Inf}(S_{\infty})$-modules and in terms of the notation in \S\ref{tiltofRinfty}, $X_i^{1/p^{m}} = [x_i^{1/p^{m}}]$. The summands for which $a_i \in \bZ$ for all $i$ comprise a subring 
\[
A(R^{\square}) = \bA_{\Inf}(S_{\infty}) \{X_1^{\pm 1}, \ldots, X_d^{\pm 1} \} \text{ inside }\bA_{\Inf}(R_{\infty}^{\square})
\]
where the convergence is $(p, \mu)$-adic. The remaining summands, that is, those for which $a_i \notin \bZ$ comprise a $A(R^{\square})$-submodule $N^{\square}_{\infty} \subset \bA_{\Inf}(R_{\infty}^{\square})$.

There is a canonical surjective ring homomorphism
\[
\theta \colon \bA_{\Inf}(R_{\infty}^{\square}) \twoheadrightarrow R_{\infty}^{\square}
\]
whose kernel is a principal ideal, cf. \cite[Lemma 3.10]{BhMorSch}, such that
\begin{equation} \label{eq:thetares}
\theta \lvert_{A(R^{\square})} \colon A(R^{\square}) \twoheadrightarrow R^{\square} \text{ is described by } X_i \mapsto t_i.
\end{equation}
The $p$-completely étale map $R^{\square} \to R$ deforms uniquely (along \eqref{eq:thetares}) to a $(p,\mu)$-completely étale map $A(R^{\square}) \to A(R)$, where $A(R)$ is $(p, \mu)$-adically complete. By construction, we have the identification
\begin{equation} \label{eq:idfrobmor}
\bA_{\Inf}(R_{\infty}) \cong \bA_{\Inf}(R_{\infty}^{\square}) \widehat{\otimes}_{A(R^{\square})} A(R),
\end{equation}
where the completion is $(p, \mu)$-adic. Therefore, by setting $N_{\infty} := N_{\infty}^{\square} \widehat{\otimes}_{A(R^{\square})} A(R)$, we arrive at the decompositions of $\bA_{\Inf}(R_{\infty}^{\square})$ and $\bA_{\Inf}(R_{\infty})$ into   ``integral" and ``nonintegral" parts:
\begin{equation} \label{eq:compdesde}
\bA_{\Inf}(R_{\infty}^{\square}) \cong A(R^{\square}) \oplus N_{\infty}^{\square} \text{ and } \bA_{\Inf}(R_{\infty}) \cong A(R) \oplus N_{\infty}.
\end{equation}
Modulo $\xi$, these decompositions reduce to the decompositions in \eqref{eq:smallsumdec}.

The Witt vector Frobenius of $\bA_{\Inf}(R_{\infty}^{\square})$ preserves $A(R^{\square})$; explictly: it is semilinear with respect to the Frobenius of $\bA_{\Inf}(S_{\infty})$ and raises each $X_i^{1/p^{m}}$ to the $p$-th power. By construction, $A(R)$ inherits a Frobenius ring endomorphism from $A(R^{\square})$, and the identification \eqref{eq:idfrobmor} is Frobenius-equivariant.

The natural $\Delta$-action on $\bA_{\Inf}(R_{\infty})$ is continuous and commutes with the Frobenius. Explictly, $\Delta$ respects the completed direct sum decomposition \eqref{eq:compdesde} and an $(\epsilon_1, \ldots, \epsilon_d) \in \Delta$ scales $X_j^{a_j}$ by $[\epsilon_j^{a_j}] \in A_{\Inf}$. The $\Delta$-action on $A(R^{\square})$ lifts uniquely to a necessarily Frobenius-equivariant $\Delta$-action on $A(R)$, which is trivial on the quotient $A(R) \to R$. In particular, $\Delta$ acts trivially on $A(R)/\mu$. The identifications \eqref{eq:idfrobmor}-\eqref{eq:compdesde} are $\Delta$-equivariant.

\subsection{The cohomology of $\bA_{\Inf,X}$ and continuous group cohomology} \label{sec:cohcongorAnin}

Similarly to \S\ref{subsec:cohcongrp}, the \v{C}ech complex of the sheaf $W(\wh{\cO}^{+, \flat}_X)$ with respect to the pro-(finite étale) affinoid perfectoid cover
\[
X_{\infty} \to X \times_{Y} Y_{\infty}
\]
is identified with the continuous cochain complex $R\Gamma_{\cont}(\Delta, \bA_{\Inf}(R_{\infty}))$. Thus, by \cite[{Tag 01GY}]{stacks-project}, we obtain the map to the pro-étale cohomology of $\bA_{\Inf,X}$ (obtained by composing the edge map with $R\Gamma_{\text{proét}} (Y_{\infty},  Rf_{\eta*}W(\wh{\cO}^{+, \flat}_X)) \to R\Gamma_{\text{proét}} (Y_{\infty},  Rf_{\eta*}\bA_{\Inf,X})$):
\begin{equation} \label{eq:edgemapforainf}
 e \colon R\Gamma_{\cont}(\Delta, \bA_{\Inf}(R_{\infty})) \rightarrow R\Gamma_{\text{proét}} (Y_{\infty},  Rf_{\eta*}\bA_{\Inf,X}).
\end{equation}
By the almost purity theorem, more precisely, by \cite[Theorem 6.5(ii)]{SchpHrig}, the subset $[\fm^{\flat}] \subset A_{\Inf}$ that consists of the Teichmüller lifts of the elements in the maximal ideal $\fm^{\flat} \subset \cO_{C^{\flat}}$ kills all the cohomology groups of $\text{Cone}(e)$. In fact using \cite[Lemma 3.17]{SemistabAinfcoh} and the fact that pushforwards preserve derived $p$-completions (cf. \cite[{Tag 0A0G}]{stacks-project}), it follows that each $H^{i}(\text{Cone}(e))$ is killed by $W(\fm^{\flat})$.

\begin{prop} \label{prop:relAinfsitf}
For each $i \in \bZ$, the $A_{\Inf}$-module $H^i_{\cont}(\Delta, \bA_{\Inf}(R_{\infty})/\mu)$ is $p$-torsion free and $p$-adically complete; moreover, the following natural maps are isomorphisms:
\begin{equation}\label{eq:tensoris}
H^i_{\cont}(\Delta, \bA_{\Inf}(R_{\infty})/\mu) \otimes_{A_{\Inf}} A_{\Inf}/p^n \xrightarrow{\sim} H^i_{\cont}(\Delta, \bA_{\Inf}(R_{\infty})/(\mu, p^n)) \text{ for } n>0
\end{equation}
and
\begin{equation}\label{eq:varprolisim}
H^i_{\cont}(\Delta, \bA_{\Inf}(R_{\infty})/\mu) \xrightarrow{\sim} \varprojlim_{n} H^i_{\cont}(\Delta, \bA_{\Inf}(R_{\infty})/(\mu, p^n)).
\end{equation}
In addition, if  $(R_{\infty}[\tfrac{1}{p}])^{\circ} = R_{\infty}$, then $H^i_{\cont}(\Delta, \bA_{\Inf}(R_{\infty})/(\mu, p^n))$ and $H^i_{\cont}(\Delta, \bA_{\Inf}(R_{\infty})/\mu)$ have no nonzero $W(\fm^{\flat})$-torsion.
\end{prop}

\begin{proof}
The first step is to use the decomposition \eqref{eq:compdesde} to reduce proving the proposition for the  ``nonintegral" part of $\bA_{\Inf}(R_{\infty})$. Since $A(R)/\mu$ is $p$-adically complete and has a trivial $\Delta$-action (see \cite[Lemma 3.13]{SemistabAinfcoh} and \S\ref{ringAinfonRinfty}), \cite[Lemma 3.7]{SemistabAinfcoh} (cf. \cite[Lemma 7.3(ii)]{BhMorSch}) implies that $H^i_{\cont}(\Delta, A(R)/\mu)$ is a direct sum of copies of $A(R)/\mu$ and likewise for $H^i_{\cont}(\Delta, A(R)/(\mu,p^n))$. Furthermore, if  $(R_{\infty}[\tfrac{1}{p}])^{\circ} = R_{\infty}$ then Lemma \ref{lem:nononzerowfalttor} implies that the rings $A(R)/\mu$ and $A(R)/(\mu,p^n)$ have no nonzero $W(\fm^{\flat})$-torsion. Thus we only need to establish all claims with $N_{\infty}$ in place of $\bA_{\Inf}(R_{\infty})$.

We now reduce the situation to that of \cite[Proposition 3.19]{SemistabAinfcoh}. For this we will need to keep track of the $S_{\infty}$ appearing in the definition of $N_{\infty}^{\square}$. Therefore it will be convenient to temporarily denote $N_{\infty}^{\square}$ by $N_{\infty}^{\square}(S_{\infty})$.  The identifications
\begin{align*}
N_{\infty}^{\square}(S_{\infty})/(\mu, p^n) &\cong N_{\infty}^{\square}(\cO)/(\mu, p^n) \otimes_{A_{\Inf}} \bA_{\Inf}(S_{\infty}) \text{ for } n>0 \\
N_{\infty}/(\mu, p^n) &\cong N_{\infty}^{\square}(S_{\infty})/(\mu, p^n) \otimes_{A(R^{\square})} A(R)
\end{align*}
are $\Delta$-equivariant. Furthermore $\bA_{\Inf}(S_{\infty})/(\mu, p^n)$ is $A_{\Inf}/(\mu, p^n)$-flat (cf. Lemma 3.13 in loc.cit.) and $A(R)/(\mu, p^n)$ is $A(R^{\square})/(\mu, p^n)$-flat (because $A(R)$ is a $(\mu, p)$-completely étale $A(R^{\square})$-algebra). Thus the short exact sequence \cite[(3.19.9)]{SemistabAinfcoh} established, in particular, for $N_{\infty}^{\square}(\cO)$ holds also for $N_{\infty}$:
\begin{equation} \label{eq:prolimexse}
0 \to H^i_{\cont}(\Delta, N_{\infty}/(\mu,p^n))[p] \to H^i_{\cont}(\Delta, N_{\infty}/(\mu,p^n)) \to H^i_{\cont}(\Delta, N_{\infty}/(\mu,p^{n-1})) \to 0,
\end{equation}
for $n>1$. The transition morphisms $H^i_{\cont}(\Delta, N_{\infty}/(\mu,p^n)) \to H^i_{\cont}(\Delta, N_{\infty}/(\mu,p^{n-1}))$ are in particular surjective and so $R^{1}\varprojlim_{n}H^i_{\cont}(\Delta, N_{\infty}/(\mu,p^n))$ vanishes. Therefore by \cite[{Tag 0D6K}]{stacks-project}, we obtain
\begin{equation} \label{eq:projliverN}
H^i_{\cont}(\Delta, N_{\infty}/\mu) \xrightarrow{\sim} \varprojlim_{n} H^i_{\cont}(\Delta, N_{\infty}/(\mu, p^n)),
\end{equation}
which is the sought analogue of \eqref{eq:varprolisim}. The $p$-torsion freeness of $H^i_{\cont}(\Delta, N_{\infty}/\mu)$ follows from \eqref{eq:prolimexse}-\eqref{eq:projliverN}. Therefore by \cite[Lemma 3.7]{SemistabAinfcoh} (cf. \cite[Lemma 7.3(ii)]{BhMorSch}) and \cite[{Tag 061Z}]{stacks-project}, we obtain
\[
H^i_{\cont}(\Delta, N_{\infty}/\mu) \otimes_{A_{\Inf}} A_{\Inf}/p^n \xrightarrow{\sim} H^i_{\cont}(\Delta, N_{\infty}/(\mu, p^n)) \text{ for } n>0,
\]
which is the sought analogue of \eqref{eq:tensoris}. 

It remains to show that each $H^i_{\cont}(\Delta, N_{\infty}/(\mu, p^n))$ has no nonzero $W(\fm^{\flat})$-torsion, under the assumption $(R_{\infty}[\tfrac{1}{p}])^{\circ} = R_{\infty}$. The short exact sequence \eqref{eq:prolimexse} shows that $H^i_{\cont}(\Delta, N_{\infty}/(\mu, p^n))$ is a successive extension of copies of $H^i_{\cont}(\Delta, N_{\infty}/(\mu, p))$. Now $N_{\infty}/(\mu, p)$ is a direct summand of $\bA_{\Inf}(R_{\infty})/(\mu, p) \cong R_{\infty}^{\flat}/\mu$. The result now follows from Lemma \ref{lem:nonozeromflatro}.
\end{proof}

\begin{rem}
A remark about the proof of Proposition \ref{prop:relAinfsitf}. We have
\[
A(R^{\square}) = \bA_{\Inf}(S_{\infty}) \{X_1^{\pm 1}, \ldots, X_d^{\pm 1} \} \cong A_{\Inf}\{X_1^{\pm 1}, \ldots, X_d^{\pm 1} \} \widehat{\otimes}_{A_{\Inf}} \bA_{\Inf}(S_{\infty}),
\]
where the completion is $(p,\mu)$-adic. Thus it would be natural to attempt to give a more direct reduction of  Proposition \ref{prop:relAinfsitf} to \cite[Proposition 3.19]{SemistabAinfcoh} where the statement is proved, in particular, for the ring $A_{\Inf}\{X_1^{\pm 1}, \ldots, X_d^{\pm 1} \}$. However we do not know whether $\bA_{\Inf}(S_{\infty})$ is a flat $A_{\Inf}$-algebra and it is unclear how to bring the completed tensor product outside the cohomology groups.
\end{rem}

\begin{thm} \label{edgemapisoAinf}
The edge map defined in \eqref{eq:edgemapforainf} induces a morphism 
\begin{equation*} 
 L\eta_{\mu}(e) \colon L\eta_{\mu}R\Gamma_{\cont}(\Delta, \bA_{\Inf}(R_{\infty})) \to L\eta_{\mu}R\Gamma_{\proet} (Y_{\infty},  Rf_{\eta*}\bA_{\Inf,X})
 \end{equation*}
 such that the cohomology groups of $\text{Cone}(L\eta_{\mu}(e))$ are killed by $W(\fm^{\flat}) \subset A_{\Inf}$.
 Moreover if $(R_{\infty}[\tfrac{1}{p}])^{\circ} = R_{\infty}$, then this morphism is an  isomorphism.
\end{thm}

\begin{proof}
Recall by \S\ref{sec:cohcongorAnin}, the cohomology groups of $\text{Cone}(e)$ are killed by $W(\fm^{\flat})$. Therefore $L\eta_{\mu}(e)$ satisfies the analogous claim. 

Suppose now $(R_{\infty}[\tfrac{1}{p}])^{\circ} = R_{\infty}$. By the projection formula \cite[{Tag 0944}]{stacks-project},
\begin{equation} \label{eq:projfusdba}
R\Gamma_{\cont}(\Delta, \bA_{\Inf}(R_{\infty})) \otimes_{A_{\Inf}}^{\bL} A_{\Inf}/\mu \cong R\Gamma_{\cont}(\Delta, \bA_{\Inf}(R_{\infty})/\mu).
\end{equation}
Therefore by Proposition \ref{prop:relAinfsitf}, the cohomology modules of $R\Gamma_{\cont}(\Delta, \bA_{\Inf}(R_{\infty})) \otimes_{A_{\Inf}}^{\bL} A_{\Inf}/\mu$ have no nonzero $W(\fm^{\flat})$-torsion. Thus the claim follows from \cite[Lemma 3.18]{SemistabAinfcoh}.
\end{proof}
Later when we will compare $A\Omega_{\fX/\fY}$ with $q$-crystalline cohomology (cf. Theorem \ref{thm:AOmegvsqCrys}) we will need to vary $\Delta$ in Theorem \ref{edgemapisoAinf}  to accommodate for the ``all possible coordinates" technique. The following remark will make this possible (cf. \cite[Remark 3.21]{SemistabAinfcoh}). 

\begin{rem} \label{rem:extenthem4.12}
For any profinite group $\Delta'$ equipped with a continuous surjection $\Delta' \twoheadrightarrow \Delta$ and any pro-(finite étale) affinoid perfectoid $\Delta'$-cover
\[
\Spa(R_{\infty}'[\tfrac{1}{p}], R'_{\infty}) \to X \times_{Y} Y_{\infty} \hspace{2mm} \text{ that refines the $\Delta$-cover } \hspace{2mm} X_{\infty} \to X \times_{Y} Y_{\infty}
\]
compatibly with the surjection $\Delta' \twoheadrightarrow \Delta$, the edge map $e'$ defined analogously to \eqref{eq:edgemapforainf} induces a morphism 
\begin{equation*} 
 L\eta_{\mu}(e') \colon L\eta_{\mu}R\Gamma_{\cont}(\Delta', \bA_{\Inf}(R'_{\infty})) \to L\eta_{\mu}R\Gamma_{\proet} (Y_{\infty},  Rf_{\eta*}\bA_{\Inf,X})
 \end{equation*}
 which satisfies the same conditions as in Theorem \ref{edgemapisoAinf}. The same line of reasoning as in loc.cit. applies. In general by the octahedral axiom the morphism 
\begin{equation} \label{eq:letversoshi}
L\eta_{\mu}R\Gamma_{\cont}(\Delta, \bA_{\Inf}(R'_{\infty})) \to L\eta_{\mu}R\Gamma_{\cont}(\Delta', \bA_{\Inf}(R'_{\infty}))
\end{equation}
has cone whose cohomology groups are killed by $[\fm^{\flat}]$ and therefore by \cite[Lemma 3.17]{SemistabAinfcoh}, it is also killed by $W(\fm^{\flat})$. If in addition $(R_{\infty}[\tfrac{1}{p}])^{\circ} = R_{\infty}$, then \eqref{eq:letversoshi} is an isomorphism.
\end{rem}

\section{The Hodge-Tate specialization of $A\Omega_{\fX / \fY}$}\label{sec:TheHodhTaspe}

In this section we assume that $\fX$ and $\fY$ are flat $p$-adic formal schemes of type (S)(b) over $\cO$, as in \cite[\S 1.9]{Hugfs}, and $f \colon \fX \to \fY$ is a smooth morphism. To state the results in this section, we will often have to resort to using the language of almost $\cO$-sheaves. This is particularly the case, outside of some finite type assumptions. For this we refer the reader to \S\ref{subsec:almostOsheav}, and in particular the functor (cf. Definition \ref{def:quaisisoder})
\begin{align*}
D(\sC, \cO) &\to D(\sC, \cO^{a})
\\
\sF &\to \sF^{a}.
\end{align*}

With the local analysis of \S\ref{sec:localanal} at our disposal, we turn to the analysis of the Hodge-Tate specialization of $A\Omega_{\fX / \fY}$. Analogous to \cite[\S 9.2]{BhMorSch}, we consider a presheaf version of $A\Omega_{\fX / \fY}$.

\subsection{The presheaf version $A\Omega_{\fX / \fY}^{\text{psh}}$} \label{subsec:presver}

We begin this section with a result that good objects in $D$ (cf. Definition \ref{def:goodtriples}) are, under some mild conditions, sufficiently well behaved. Following the notation of \S \ref{sec:prodoftopo} and \ref{sec:perfcov}, for a good object
\begin{equation}\label{eq:goodpair}
(\varprojlim_{i \in I} \Spa(S_{i}, S_{i}^{+}) \to \Spf(S) \leftarrow \Spf(R))
\end{equation}
in $D$, one has the associated rings $S_{\infty}$ and $R_{\infty}$ (cf. \eqref{eq:definRinfty}). 

\begin{lem} \label{lem:Rinftucirc}
If $\fY$ is locally of finite type over $\cO$, then for every good object (cf. \eqref{eq:goodpair}), one has $(R_{\infty}[\tfrac{1}{p}])^{\circ} = R_{\infty}$.
\end{lem}

\begin{proof}
This is an immediate consequence of Lemma \ref{lem:topfinflasir}.
\end{proof}

In addition to the site $D$ constructed in \S\ref{sec:prodoftopo}, we consider the site $D^{\text{psh}}$, whose objects are good objects in $D$, and coverings are isomorphisms. In this way, every presheaf is already a sheaf. We denote the associated topos by $(Y_{ \proet}\times_{\fY_{\et}}\fX_{\et})^{\text{psh}}$. There is a morphism of topoi
\[
(\phi^{-1}, \phi_{*}) \colon Y_{ \proet}\times_{\fY_{\et}}\fX_{\et} \to (Y_{ \proet}\times_{\fY_{\et}}\fX_{\et})^{\text{psh}}
\]
for which $\phi_{*}$ is given by restricting sheaves on $D$ to $D^{\text{psh}}$ and $\phi^{-1}$ is given by sheafification. In particular, since any sheaf is the sheafification of its associated presheaf, $\phi^{-1}\circ \phi_{*} \cong \id$. We let
\[
\nu_f^{\text{psh}} := \phi \circ \nu_f \colon X_{ \proet} \to (Y_{ \proet}\times_{\fY_{\et}}\fX_{\et})^{\text{psh}}
\]
be the indicated composition of morphisms of topoi (with $\nu_f$ defined in \eqref{eq:definovf}) and set
\[
A\Omega_{\fX / \fY}^{\text{psh}} := L\eta_{\mu}(R\nu_{f*}^{\text{psh}}\bA_{\Inf,X}) \in D^{\geq 0}(D^{\text{psh}}, A_{\Inf}).
\]
Since $L\eta$ commutes with pullback under flat morphisms of ringed topoi (cf. \cite[Lemma 6.14]{BhMorSch})
\begin{equation} \label{eq:pullbackofpregivshe}
\phi^{-1}(A\Omega_{\fX / \fY}^{\text{psh}}) \cong A\Omega_{\fX / \fY}.
\end{equation}
Moreover, again by \cite[Lemma 6.14]{BhMorSch}, $A\Omega_{\fX / \fY}^{\text{psh}}$ may be described explictly: for every object $Z = (W \to \fU \leftarrow \fV)$ of $D^{\text{psh}}$, we have (cf. \eqref{eq:pullbaobobjec})
\begin{equation} \label{eq:evapresatsec}
R\Gamma(Z, A\Omega_{\fX / \fY}^{\text{psh}}) = L\eta_{\mu}(R\Gamma(W \times_U V, \bA_{\Inf,X})).
\end{equation}
In particular, since, by \cite[Lemma 6.19]{BhMorSch}, the functor $L\eta_{\mu}$ preserves derived completeness when used in the context of a \emph{replete} topos (such as the topos of sets), we see from \eqref{eq:evapresatsec} that $A\Omega_{\fX / \fY}^{\text{psh}}$ is derived $\xi$-adically (and also $\tilde{\xi}$-adically) complete (cf. Lemma \ref{lem:derived completionaind}). Armed with the formalism of \S\ref{subsec:presver}, we now identify the Hodge-Tate specialization of $A\Omega_{\fX / \fY}$.

\begin{thm} \label{thm:almosishodsa}
There exists a morphism in $D(D, \cO)$
\begin{equation} \label{eq:hodgtatespecmap}
A\Omega_{\fX / \fY} \otimes^{\bL}_{A_{\Inf}, \theta \circ \varphi^{-1}} \cO \to L\eta_{(\zeta_p - 1)}(  R\nu_{f*}(\wh{\cO}^{+}_X)),
\end{equation}
whose image in $D(D, \cO^a)$ is an almost isomorphism. If in addition $\fY$ is locally of finite type over $\cO$ then \eqref{eq:hodgtatespecmap} is already an isomorphism in $D(D, \cO)$.
\end{thm}

\begin{proof}
We will employ the language of \S\ref{subsec:almostOsheav}. Let $\theta_X \colon W(\wh{\cO}^{+, \flat}_X) \twoheadrightarrow \wh{\cO}^{+}_X$ be the natural map so that the kernel of $\theta_{X} \circ \varphi^{-1} \colon W(\wh{\cO}^{+, \flat}_X) \twoheadrightarrow \wh{\cO}^{+}_X$ is generated by the nonzero divisor $\tilde{\xi}$. The projection formula (cf. \cite[{Tag 0944}]{stacks-project}) provides the identification
\begin{equation} \label{eq:newadjisom}
R\nu_{f*}(W(\wh{\cO}^{+, \flat}_X)) \otimes^{\bL}_{A_{\Inf}, \theta \circ \varphi^{-1}} \cO \cong R\nu_{f*}(\wh{\cO}^{+}_X).
\end{equation}
By \cite[Lemma C.11]{ZavBog} $\wh{\cO}^{+}_X$ is derived $p$-adically complete and so by \cite[{Tag 0A0G}]{stacks-project}, the RHS of \eqref{eq:newadjisom} is derived $p$-adically complete. Thus there is no harm in replacing $W(\wh{\cO}^{+, \flat}_X)$ appearing on the LHS of \eqref{eq:newadjisom} with it's derived $p$-adic completion to obtain:
\[
R\nu_{f*}(\bA_{\Inf,X}) \otimes^{\bL}_{A_{\Inf}, \theta \circ \varphi^{-1}} \cO \cong R\nu_{f*}(\wh{\cO}^{+}_X).
\]
Since $(\theta \circ \varphi^{-1})(\mu) = \zeta_p -1$, this induces the map \eqref{eq:hodgtatespecmap} and, likewise also its presheaf version
\begin{equation} \label{eq:preshvers}
A\Omega_{\fX / \fY}^{\text{psh}} \otimes^{\bL}_{A_{\Inf}, \theta \circ \varphi^{-1}} \cO \to L\eta_{(\zeta_p - 1)}(R\phi_{*}  (R\nu_{f*}(\wh{\cO}^{+}_X))).
\end{equation}
Due to \eqref{eq:pullbackofpregivshe}, $\phi^{-1}$ brings \eqref{eq:preshvers} to \eqref{eq:hodgtatespecmap}, so by Proposition \ref{prop:commsqaalmssnosh}, it suffices to show that \eqref{eq:preshvers} is an almost isomorphism (if $\fY$ is in addition locally of finite type over $\cO$, it suffices to show \eqref{eq:preshvers} is an isomorphism).

Consider an object $Z = (W \to \fU \leftarrow \fV) =  (\varprojlim_{i \in I} \Spa(S_{i}, S_{i}^{+}) \to \Spf(S) \leftarrow \Spf(R))$ of $D^{\text{psh}}$. The discussion and notation of \S\ref{sec:localanal} apply. In particular, Proposition \ref{prop:relAinfsitf} and \eqref{eq:projfusdba} ensure that the cohomology of $R\Gamma_{\cont}(\Delta, \bA_{\Inf}(R_{\infty})) \otimes_{A_{\Inf}}^{\bL} A_{\Inf}/\mu$ is $p$-torsion free. Thus, since $\tilde{\xi} \equiv p \pmod {\mu}$, \cite[Lemma 5.16]{Bhatspecvar} implies that
\begin{equation}\label{eq:truisnoalis}
L\eta_{\mu}R\Gamma_{\cont}(\Delta, \bA_{\Inf}(R_{\infty})) \otimes_{A_{\Inf},\theta \circ \varphi^{-1}}^{\bL} \cO \xrightarrow{\sim} L\eta_{(\zeta_p -1)}(R\Gamma_{\cont}(\Delta,R_{\infty})).
\end{equation}
Now Theorem \ref{edgemapisoAinf} gives a morphism
\begin{equation} \label{eq:almosWmbmor}
 L\eta_{\mu}R\Gamma_{\cont}(\Delta, \bA_{\Inf}(R_{\infty})) \otimes^{\bL}_{A_{\Inf}, \theta \circ \varphi^{-1}} \cO \to L\eta_{\mu}(R\Gamma(W \times_U V, \bA_{\Inf,X})) \otimes^{\bL}_{A_{\Inf}, \theta \circ \varphi^{-1}} \cO 
\end{equation}
whose cone has cohomology groups  killed by $W(\fm^{\flat}) \otimes_{A_{\Inf}, \theta \circ \varphi^{-1}} \cO \cong \fm$ (cf. Lemma \ref{lem:promaxidtensp}). 
Similarly Theorem \ref{thm:edgemapisoOsheaf} gives a morphism
\begin{equation} \label{eq:almostmomor}
L\eta_{(\zeta_p -1)}(R\Gamma_{\cont}(\Delta,R_{\infty})) \to L\eta_{(\zeta_p -1)}(R\Gamma(W \times_U V, \wh{\cO}^{+}_X))
\end{equation}
whose cone has cohomology groups  killed by $\fm$. Since the edge maps \eqref{eq:relver3.3.1} and \eqref{eq:edgemapforainf} are compatible \eqref{eq:truisnoalis}-\eqref{eq:almostmomor} give the desired conclusion (if $\fY$ is in addition locally of finite type over $\cO$, then by Lemma \ref{lem:Rinftucirc}, \eqref{eq:almosWmbmor}-\eqref{eq:almostmomor} are isomorphisms).
\end{proof}

\subsection{The completed integral structure sheaf $\wh{\cO}_{D}^{+}$}\label{subsec:complfsdas}

In this section, we compute the (local) sections of $\wh{\cO}_{D}^{+}$ for good objects, cf. Proposition \ref{lem:morringtopnuf}. The strategy is as follows
\begin{enumerate}
    \item Compute the sections of the pushforward $\nu_{f*}\wh{\cO}^{+}_{X}$ for good objects (cf. Lemma \ref{lem:valuesofpushfosab}).
    \item Show $\nu_{f*}\wh{\cO}^{+}_{X}$ is $p$-adically complete (cf. Lemma \ref{lem:varithin}(iv)).
    \item Compare $\nu_{f*}\wh{\cO}^{+}_{X}$ with $\wh{\cO}_{D}^{+}$.
\end{enumerate}

\begin{lem} \label{lem:valuesofpushfosab}
For a good object 
\[
Z = (\varprojlim_{i \in I} \Spa(S_{i}, S_{i}^{+}) \to \Spf(S) \leftarrow \Spf(R)),
\] 
we have $\nu_{f*}\wh{\cO}^{+}_{X}(Z) = R \wh{\otimes}_S S_{\infty}$.
\end{lem}

\begin{proof}
We employ the notation of \S\ref{sec:perfcov}. By \eqref{eq:smallsumdec}, recall that we have a decomposition 
\[
R_{\infty} \cong (R \wh{\otimes}_S S_{\infty}) \oplus M_{\infty}
\]
and since $M_{\infty}^{\Delta}$ is killed by $(\zeta_p -1)$ (cf. Proposition \ref{prop:mtorsioncoho}), we immediately obtain by Corollary \ref{cor:comtentorfree},
\[
R \wh{\otimes}_S S_{\infty}  \cong H^{0}_{\cont}(\Delta, R_{\infty}).
\]
The edge map \eqref{eq:relver3.3.1} is an isomorphism in degree 0. Thus we a get a natural isomorphism (in $R$, $S$ and $S_{\infty}$)
\begin{equation} \label{eq:isoedhmosf}
R \wh{\otimes}_S S_{\infty}  \xrightarrow{\sim} H^{0}(X \times_{Y} Y_{\infty},  \wh{\cO}^{+}_X).
\end{equation}
By \eqref{eq:pullbaobobjec}, the RHS is precisely $\nu_{f*}\wh{\cO}^{+}_{X}(Z)$.
\end{proof}
Lemma \ref{lem:valuesofpushfosab} computes the sections of $\nu_{f*}\wh{\cO}^{+}_{X}$ for good objects. We now begin our comparison of $\nu_{f*}\wh{\cO}^{+}_{X}$ with $\wh{\cO}_{D}^{+}$. This will involve showing that $\nu_{f*}\wh{\cO}^{+}_X$ is $p$-adically complete (as a sheaf). The following (slight) strengthening of the proof of \cite[Proposition 7.13]{SchPerf} will be useful.

\begin{lem} \label{lem:vanishcecomspld}
Let $U_{\infty}$ be an affinoid perfectoid space over $\Spa (C, \cO)$. Then for any covering $\{U_i \to U_{\infty} \}$ (with each $U_i$ affinoid) in $U_{\infty, \et}$ (in the sense of \cite[Definition 7.1(iii)]{SchPerf}), the complex
\[
0 \to \cO_{U_{\infty}}(U_{\infty})^{\circ a} \to \prod_{i} \cO_{U_{i}}(U_{i})^{\circ a} \to \prod_{i,j} \cO_{U_{i} \times_{U_{\infty}} U_{j}}(U_{i} \times_{U_{\infty}} U_{j})^{\circ a} \to \ldots
\]
is exact.
\end{lem}

\begin{proof}
The statement is proved in the proof of \cite[Proposition 7.13]{SchPerf} where in addition the covering is finite and each $U_i \to U_{\infty}$ is a composition of rational embeddings and finite étale maps. To obtain the general statement one takes a finite refinement $\{ V_j \to U_{\infty}\}$ of $\{U_i \to U_{\infty} \}$ such that each $V_j \to U_{\infty}$ is a composition of rational embeddings and finite étale maps. The key point is that the \v{C}ech complex associated to the covering of $\{ V_j \to U_{\infty}\}$ restricted to each $U_{i_1} \times_{U_{\infty}} \ldots \times_{U_{\infty}} U_{i_p}$ is almost exact (again by the proof in loc.cit.). By the same argument as \cite[\S 8.1.4, Corollary 3]{BGR}, this implies that the augmented \v{C}ech complex associated to $\{V_j \to U_{\infty} \}$ is almost exact iff the same is true for $\{ U_{i} \to U_{\infty}\}$\footnote{For the corresponding argument \emph{outside} of the almost setting, cf. \cite[Proposition 8.2.21]{KeLiRpHF}.}. 
\end{proof}

We now show an analogue of \cite[Lemma 4.10]{SchpHrig} for $\nu_{f*}\wh{\cO}^{+}_{X}$.

\begin{lem} \label{lem:varithin}
Let $Z = (W \to \fU \leftarrow \fV) = (\varprojlim_{i \in I} \Spa(S_{i}, S_{i}^{+}) \to \Spf(S) \leftarrow \Spf(R)) \in D$ be a good object. In the following, we use the almost setting w.r.t. $\fm \subset \cO$.
\begin{enumerate}[(i)]
    \item We have $\nu_{f*}\wh{\cO}^{+}_{X}(Z)/p^n = (R \otimes_S S_{\infty})/p^n$, and this is almost equal to $((\nu_{f*}\wh{\cO}^{+}_X)/p^n)(Z)$.
    \item The cohomology groups $H^{i}(Z, (\nu_{f*}\wh{\cO}^{+}_X)/p^n)$ are almost zero for $i >0$ and $n \geq 1$.
\item The image of $((\nu_{f*}\wh{\cO}^{+}_X)/p^{n+1})(Z)$ in $((\nu_{f*}\wh{\cO}^{+}_X)/p^n)(Z)$ is $(R \otimes_S S_{\infty})/p^n$ for all $n \geq 1$.
\item The sheaf $\nu_{f*}\wh{\cO}^{+}_X$ is $p$-adically complete: $\nu_{f*}\wh{\cO}^{+}_X = \varprojlim (\nu_{f*}\wh{\cO}^{+}_X)/p^n$.
\item The cohomology groups $H^{i}(Z, \nu_{f*}\wh{\cO}^{+}_X)$ are almost zero for $i >0$.
\end{enumerate}
\end{lem}

\begin{proof}
The equality $\nu_{f*}\wh{\cO}^{+}_X(Z)/p^n = (R \otimes_S S_{\infty})/p^n$ follows from Lemma \ref{lem:valuesofpushfosab}. For parts (i)-(ii), by Proposition \ref{prop:commsqaalmssnosh} and the exactness of $\wt{\alpha} \colon \Mod_{\cO}(\wt{\sC}) \to \wt{\sC}^{a}$ (cf. \S \ref{subsec:almostOsheav})
it suffices to show
\[
Z \mapsto \sF (Z) := (\nu_{f*}\wh{\cO}^{+}_{X}(Z)/p^n)^{a}
\]
is a sheaf of $\cO^{a}$-modules and $H^{i}(Z, \sF) = 0$ for $i >0$.

Let $Z$ be covered by $Z_{k} \to Z$, where $Z_k$ are objects 
as in Proposition \ref{prop:ousiDalge}(1). By Proposition \ref{prop:ousiDalge}(2) $Z$ is quasi-compact, so we can assume that there are only finitely many $Z_k$ or a single $T$ (by taking the disjoint union componentwise). After taking a refinement we can assume $T$ is good (cf. Proposition \ref{prop:goodobjformbas}(1)), and of the form 
\[
T = (\varprojlim W_j \to \fU_0 \leftarrow \fV_0)
\]
such that
\begin{enumerate}
   \item $W_j \to W_{j'}$ is finite étale surjective for $j \geq j' \geq 0$.
    \item $\varprojlim W_j \to U_0$ factors as  $\varprojlim W_j \to W_0 \to U_0$, where $U_{0}$ is the adic generic fiber of $\fU_0$.
\end{enumerate}
In particular this implies that $W_j \times_{U_{0}} V_{0} \to W \times_{U} V$ is étale surjective for $j \geq 0$, where $V_0$ is the adic generic fiber of $\fV_0$. It suffices to check that the complex
\begin{equation} \label{eq:tensofeasr}
C(T \to Z, \sF)\colon 0 \to \sF(Z) \to \sF(T) \to \sF(T \times_Z T) \to \cdots
\end{equation}
is exact. Note that by Lemma \ref{lem:valuesofpushfosab}
\begin{enumerate}
    \item $\sF(Z) = (R \otimes_S S_{\infty})^{a}/p^n$, and
    \item $\sF(T) = \wh{\cO}^{+}_X(\varprojlim W_j \times_{U_0} V_{0})^{a}/p^n = \varinjlim (\wh{\cO}^{+}_{X}(W_j \times_{U_{0}} V_0) )^{a}/p^n$. 
\end{enumerate}
Moreover the values of $\sF(Z)$ and $\sF(T)$ determine the values of all the terms in \eqref{eq:tensofeasr} as each $\underbrace{T \times_{Z} \ldots \times_Z T}_{
    \text{$r$ times}}$ is good (good objects are stable under fiber products) and 
    \[
    \sF (\underbrace{T \times_{Z} \ldots \times_Z T}_{
    \text{$r$ times}}) = \underbrace{\sF(T) \times_{\sF(Z)} \ldots \times_{\sF(Z)} \sF(T)}_{
    \text{$r$ times}}.
    \]
Now $R \wh{\otimes}_S S_{\infty}$ is not quite perfectoid, so the final step is to base change \eqref{eq:tensofeasr} to a sequence of perfectoid algebras.

Employing the notation in \S \ref{sec:localanal} we have the pro-(finite étale) cover $U_{\infty} \to W \times_U V$ in $X_{\proet}$. Let $U_{j, \infty} \to U_{j}$ be the étale cover determined by $W_j \times_{U_{0}} V_{0} \to W \times_{U} V$.

\begin{lem}
    The morphism $R \wh{\otimes}_S S_{\infty} \to R_{\infty}$ is $p$-completely faithfully flat.
\end{lem}

\begin{proof}
Note that each $R_m$ is nonzero and finite free over $R \wh{\otimes}_S S_{\infty}$ and $R_{\infty}$ is the (derived) $p$-adic completion of $\varinjlim R_m$.
\end{proof}
So it suffices to show that the complex in \eqref{eq:tensofeasr} base changed along $(R \wh{\otimes}_S S_{\infty})^{a} \to R_{\infty}^{a}$:

\begin{equation} \label{eq:modvsdf}
 0 \to (\cO^{+}_{\wh{U}_{\infty}}(\wh{U}_{\infty})/p^n)^{a} \to \varinjlim (\cO^{+}_{\wh{U}_{ \infty}}(\wh{U}_{j,\infty})/p^n)^{a} \to \varinjlim (\cO^{+}_{\wh{U}_{ \infty}}(\wh{U}_{j,\infty} \times_{\wh{U}_{\infty}} \wh{U}_{j,\infty})/p^n)^{a} \to \cdots
\end{equation}
is exact, where $\wh{U}_{\infty}$ (resp. $\wh{U}_{j,\infty}$) is the corresponding affinoid perfectoid  space associated to $U_{\infty}$ (resp. $U_{j, \infty}$). This follows from Lemma \ref{lem:vanishcecomspld}.

We now show part (iii). Let $f \in ((\nu_{f*}\wh{\cO}^{+}_{X})/p^{n+1})(Z)$. By part (1), there exists $g \in R \wh{\otimes}_S S_{\infty}$ such that $pf = g$ in $((\nu_{f*}\wh{\cO}^{+}_{X})/p^{n+1})(Z)$. This means that there exists a covering $\{ Z_k \to Z \}$ and $f_k \in \nu_{f*}\wh{\cO}^{+}_{X}(Z_k)$ such that 
\begin{equation} \label{eq:somequalf}
pf_k = g \text{ in } \nu_{f*}\wh{\cO}^{+}_{X}(Z_k).
\end{equation}
On the other hand we can write $g = ph$ for some $h \in (R \wh{\otimes}_S S_{\infty})[\tfrac{1}{p}]$. By \eqref{eq:somequalf}, we get that $h \in R \wh{\otimes}_S S_{\infty}$. As multiplication by $p \colon (\nu_{f*}\wh{\cO}^{+}_{X})/p^{n} \to (\nu_{f*}\wh{\cO}^{+}_{X})/p^{n+1}$ is an injection, we get that $f = h$ in $((\nu_{f*}\wh{\cO}^{+}_{X})/p^{n})(Z)$. This completes the proof of part (iii).

Part (iv) is a consequence of part (iii). Finally part (v) is a consequence of parts (ii) and (iv) together with the almost version of \cite[Lemma 3.18]{SchpHrig}.
\end{proof}

We can finally compare $\nu_{f*}\wh{\cO}^{+}_{X}$ and $\wh{\cO}^{+}_{D}$ and establish a morphism of ringed topoi.

\begin{prop} \label{lem:morringtopnuf}
There is a morphism of ringed topoi
\[
(\nu_f, \nu_{f}^{\sharp}) \colon (X_{ \proet}, \wh{\cO}^{+}_X) \to (Y_{ \proet}\times_{\fY_{\et}}\fX_{\et}, \wh{\cO}^{+}_D)
\]
with $\nu_{f}^{\sharp} \colon  \wh{\cO}^{+}_D \xrightarrow{\sim} \nu_{f^{*}}(\wh{\cO}^{+}_X)$.  Moreover for a good object 
\[
Z = (\varprojlim_{i \in I} \Spa(S_{i}, S_{i}^{+}) \to \Spf(S) \leftarrow \Spf(R)),
\] 
we have $\wh{\cO}^{+}_D(Z) = R \wh{\otimes}_S S_{\infty}$.
\end{prop}

\begin{proof}
By definition
\[
\wh{\cO}^{+}_D := \varprojlim( q^{-1}\cO_{\fX,\et} \otimes_{p^{-1}\nu_{Y}^{-1}\cO_{\fY, \et}} p^{-1}\wh{\cO}^{+}_Y)/p^n 
\]
and so by Lemma \ref{lem:varithin}(iv), it suffices to construct a morphism
\begin{equation} \label{eq:consmorpfs}
q^{-1}\cO_{\fX,\et} \otimes_{p^{-1}\nu_{Y}^{-1}\cO_{\fY, \et}} p^{-1}\wh{\cO}^{+}_Y \to \nu_{f*}(\wh{\cO}^{+}_{X})
\end{equation}
which gives an isomorphism modulo $p^n$. By definition the LHS of \eqref{eq:consmorpfs} is sheafification of the presheaf $Z \mapsto \sF(Z) := R \otimes_S S_{\infty}$. Since $\nu_{f*}(\wh{\cO}^{+}_{X})(Z) = R \wh{\otimes}_S S_{\infty}$ which is natural in $R$, $S$ and $S_{\infty}$ (cf. Lemma \ref{lem:valuesofpushfosab} and its proof), we get a map $\sF \to \nu_{f*}(\wh{\cO}^{+}_{X})$ of presheaves which is an isomorphism modulo $p^n$. Sheafifying this map gives the sought after \eqref{eq:consmorpfs}.
\end{proof}

Finally we conclude this section by comparing $\wh{\cO}^{+}_{D}$ with its derived $p$-adic completion $R\varprojlim (\wh{\cO}^{+}_{D}/p^n)$: 
\begin{lem} \label{lem:Odsupad}
The completed integral structure sheaf $\wh{\cO}^{+}_{D}$ is derived $p$-adically complete.
\end{lem}

\begin{proof}
By Proposition \ref{lem:morringtopnuf} we identify $\wh{\cO}^{+}_D$ with $\nu_{f^{*}}(\wh{\cO}^{+}_X)$ and use the notation of Lemma \ref{lem:varithin}. By a similar proof to \cite[Lemma C.11]{ZavBog}, it suffices to show $R\Gamma(Z, \wh{\cO}^{+}_D)$ is derived $p$-adically complete. By \cite[Lemma 6.15]{BhMorSch}, it suffices to show for each $i \geq 0$, $H^{i}(Z, \wh{\cO}^{+}_{D})$ is derived $p$-adically complete. For $i = 0$, this follows from second part of Proposition \ref{lem:morringtopnuf}. For $i \geq 1$, this follows from Lemma \ref{lem:varithin}(v).
\end{proof}

\subsection{The object $\wt{\Omega}_{\fX / \fY}$}\label{subsec:objecttidlome}

Armed with \S \ref{subsec:complfsdas}, to proceed further, we need to analyze the right side of \eqref{eq:hodgtatespecmap}, namely
\[
\wt{\Omega}_{\fX / \fY} := L\eta_{(\zeta_p - 1)}(  R\nu_{f*}(\wh{\cO}^{+}_X)) \in D^{\geq 0}(D, \wh{\cO}_{D}^{+}).
\]
First we show that the morphism of ringed topoi $(\nu_f, \nu_{f}^{\sharp})$ constructed in Proposition \ref{lem:morringtopnuf} satisfies the expected base change:

\begin{lem}\label{lem:baschaexpect}
Suppose the morphism $f \colon \fX \to \fY$ sits in a commutative diagram
\begin{equation} \label{eq:comdaireshr}
\xymatrix{
\fX' \ar@{->}[r]^{f'} \ar@{->}[d] & \fY' \ar@{->}[d]   \\ 
\fX \ar@{->}[r]^{f} & \fY,
}
\end{equation}
where $\fX' \to \fX$ and $\fY' \to \fY$ are étale. Then \eqref{eq:comdaireshr} induces a commutative diagram of ringed topoi:
\[
\xymatrix{
 X'_{\proet} \ar@{->}[rr]_{\nu_{f'}} \ar@{->}[d]^{i} && Y'_{\proet} \times_{\fY'_{\et}} \fX'_{\et} \ar@{->}[d]^{j}   \\ 
 X_{\proet} \ar@{->}[rr]_{\nu_f} && Y_{\proet} \times_{\fY_{\et}} \fX_{\et}
 }
\]
and the base change morphism  $j^{*}R^{n}\nu_{f*}\sF \to R^{n}\nu_{f'*}i^{*}\sF$
is an isomorphism (here $\sF$ is an $\wh{\cO}^{+}_X$-module).
\end{lem}

\begin{proof}
The key point is that if 
\[
Z' := (W' \to \fU' \leftarrow \fV')
\]
is a good object in $Y'_{\proet} \times_{\fY'_{\et}} \fX'_{\et}$, then it is also a good object in $Y_{\proet} \times_{\fY_{\et}} \fX_{\et}$. The commutativity follows from the functoriality of the edge map \eqref{eq:relver3.3.1} (cf. proof of Lemma \ref{lem:valuesofpushfosab}). For the isomorphism claim, one checks that both sides (of the base change morphism) is sheafification of the presheaf
\[
Z' \mapsto H_{\text{proét}}^{n} (V' \times_{U'} W',  \sF),
\]
where as usual, $U'$ (resp. $V'$) is the adic generic fiber of $\fU'$ (resp. $\fV'$).
\end{proof}

A basic property is that $\wt{\Omega}_{\fX / \fY}$ behaves well with respect to base change:
\begin{cor} \label{cor:swibasometil}
In the situation of Lemma \ref{lem:baschaexpect}, $j^{*}\wt{\Omega}_{\fX / \fY} \cong \wt{\Omega}_{\fX' / \fY'}$.
\end{cor}

\begin{proof}
The result follows by Lemma \ref{lem:baschaexpect} together with the fact that $L\eta$ commutes with flat base change (cf. \cite[Lemma 6.14]{BhMorSch}).
\end{proof}

\begin{prop} \label{prop:twistOmegfrrmo}
We have $H^{0}(\wt{\Omega}_{\fX / \fY}) \cong \wh{\cO}^{+}_D$. For $i \geq 1$, the $\wh{\cO}^{+a}_D$-module $H^{i}(\wt{\Omega}_{\fX / \fY})^{a}$ is Zariski locally on $\fX$ the pullback of a free $\cO_{\fX,\et}^{a}$-module of rank $\binom{\dim_{x} f_{k}^{-1}(f_{k}(x))}{i}$ at a variable closed point $x$ of $\fX_{k}$. If in addition $\fY$ is locally of finite type over $\cO$, then the almost can be dropped.
\end{prop}

\begin{proof}
The isomorphism $H^{0}(\wt{\Omega}_{\fX / \fY}) \cong \wh{\cO}^{+}_D$ is a consequence of Proposition \ref{lem:morringtopnuf}. So assume $i \geq 1$.

By Lemma \ref{lem:facetalmor} and Corollary \ref{cor:swibasometil}, we  can assume that that $\fY = \Spf (S)$, $\fX = \Spf (R)$ and for some $d \geq 0$, 
there is an étale $\Spf(S)$-morphism
\begin{equation} \label{eq:locsigood}
\fX = \Spf (R) \rightarrow \Spf (R^{\square}) =: \fX^{\square} \text{ with }R^{\square} := S \{t_1^{\pm 1}, \ldots, t_d^{\pm 1} \}.
\end{equation}
To achieve this, we can first shrink $\fX$ (as the statement depends Zariski locally on $\fX$) and then $\fY$ as this does not change the sections of sheaves in the fibered topoi (cf. Lemma \ref{lem:covrshecond} for coverings of type (c)). 
Henceforth the discussion and notation of \S \ref{sec:localanal} apply. In particular, since $R$ is $R^{\square}$-flat and $\Delta$ acts trivially on $R^{\square}$ and $R$, \cite[Lemma 3.7]{SemistabAinfcoh} (cf. \cite[Lemma 7.3(ii)]{BhMorSch}) and Proposition \ref{prop:mtorsioncoho} imply that 
\begin{equation} \label{eq:lonconcesd}
(R \wh{\otimes}_S S_{\infty})^{\oplus \binom{d}{i}} \cong H^i_{\cont}(\Delta, R^{\square}) \otimes_{R^{\square}} R \wh{\otimes}_S S_{\infty} \cong \frac{H^{i}_{\cont}(\Delta, R_{\infty}^{\square})}{H^{i}_{\cont}(\Delta, R_{\infty}^{\square})[\zeta_p -1]} \otimes_{R^{\square}} R \cong \frac{H^{i}_{\cont}(\Delta, R_{\infty})}{H^{i}_{\cont}(\Delta, R_{\infty})[\zeta_p -1]}.
\end{equation}
Theorem \ref{thm:edgemapisoOsheaf} shows that
\begin{equation} \label{eq:rtensssinal}
(R \wh{\otimes}_S S_{\infty})^{\oplus \binom{d}{i}} \cong \frac{H^{i}(X \times_{Y} Y_{\infty},  \wh{\cO}^{+}_X)}{H^{i}(X \times_{Y} Y_{\infty},  \wh{\cO}^{+}_X)[\zeta_p -1]}
\end{equation}
is an almost isomorphism of free $R \wh{\otimes}_S S_{\infty}$-modules of rank $\binom{d}{i}$. Consequently by Proposition \ref{lem:morringtopnuf} we get almost isomorphisms
\begin{equation} \label{eq:Odplusdialmos}
(\wh{\cO}_{D}^{+})^{\oplus \binom{d}{i}} \cong \frac{R^{i}\nu_{f*}(\wh{\cO}^{+}_X)}{(R^{i}\nu_{f*}(\wh{\cO}^{+}_X))[\zeta_p -1]} \cong H^{i}(\wt{\Omega}_{\fX / \fY}),
\end{equation}
to the effect that $H^{i}(\wt{\Omega}_{\fX / \fY})$ is the pullback of an almost free $\cO_{\fX, \et}$-module of rank $\binom{d}{i}$ as desired. Finally if $\fY$ is locally of finite type over $\cO$, then by Lemma \ref{lem:Rinftucirc}, \eqref{eq:rtensssinal}-\eqref{eq:Odplusdialmos} are isomorphisms.
\end{proof}

\begin{rem} \label{rem:preshaissheaal}
The proof of Proposition \ref{prop:twistOmegfrrmo}, specifically \eqref{eq:rtensssinal}-\eqref{eq:Odplusdialmos}, shows that if $\fX$ and $\fY$ are affine, then the presheaf assigning $\frac{H^{i}(X' \times_{Y'} Y_{\infty}',  \wh{\cO}^{+}_{X'})}{H^{i}(X' \times_{Y'} Y_{\infty}',  \wh{\cO}^{+}_{X'})[\zeta_p -1]}$ to a variable good object $(\varprojlim_j Y_{j}' \to \fY' \leftarrow \fX')$ over $(Y \to \fY \leftarrow \fX)$ is an almost $\cO$-sheaf (as usual  $Y_{\infty}'$ is the perfectoid affinoid determined by  $\varprojlim_j Y_j'$). If in addition $\fY$ is locally of finite type over $\cO$,
then it is a sheaf, although we will not use this fact.
\end{rem}
Almost in the following corollary is w.r.t. $[\fm^{\flat}]$.

\begin{cor} \label{cor:objecxiadcom}
The object $A\Omega_{\fX/\fY}$ is almost derived $\xi$-adically complete (i.e. the canonical morphism $A\Omega_{\fX / \fY} \to R\varprojlim_{n}(A\Omega_{\fX / \fY} \otimes^{\bL}_{A_{\Inf}} A_{\Inf}/\xi^n)$ is an almost isomorphism) and 
\begin{equation} \label{eq:adjuisoap}
A\Omega_{\fX / \fY}^{\text{psh}} \to R\phi_{*}A\Omega_{\fX / \fY} \overset{\eqref{eq:pullbackofpregivshe}}{\cong}  R\phi_{*}(\phi^{-1}(A\Omega_{\fX / \fY}^{\text{psh}}))
\end{equation}
is an almost isomorphism. If in addition $\fY = \Spf(\cO)$ and $f$ is smooth, then the almost can be dropped.
\end{cor}

\begin{proof}
Since  $\phi^{-1} \circ R\phi_{*} \cong \id$, to show that $A\Omega_{\fX/\fY}$ is almost derived $\xi$-adically complete, it suffices to show that 
\[
R\phi_{*}A\Omega_{\fX / \fY} \to R\phi_{*}R\varprojlim_{n}(A\Omega_{\fX / \fY} \otimes^{\bL}_{A_{\Inf}} A_{\Inf}/\xi^n)
\]
is an almost isomorphism. Since $A\Omega_{\fX / \fY}^{\text{psh}}$ is derived $\xi$-adically complete (cf. \S \ref{subsec:presver}), we reduce  to proving \eqref{eq:adjuisoap}. We can assume $\fX$ and $\fY$ are affine, and there is a framing as in \eqref{eq:locsigood}. By the almost analogue of \cite[Lemma 9.15]{BhMorSch} for the site $D^{\text{psh}}$ defined in \S \ref{subsec:presver} and the fact that $A\Omega_{\fX / \fY}^{\text{psh}}$ is derived $\tilde{\xi}$-adically complete, it suffices to show
\[
A\Omega_{\fX / \fY}^{\text{psh}} \otimes^{\bL}_{A_{\Inf}} A_{\Inf}/\tilde{\xi}^n \to R\phi_{*}(\phi^{-1}(A\Omega_{\fX / \fY}^{\text{psh}}\otimes^{\bL}_{A_{\Inf}} A_{\Inf}/\tilde{\xi}^n))
\]
is an almost isomorphism.
By induction on $n$, we may assume $n=1$ and in this case by \eqref{eq:preshvers}
\begin{equation} \label{eq:anoalfuth}
A\Omega_{\fX / \fY}^{\text{psh}} \otimes^{\bL}_{A_{\Inf}} A_{\Inf}/\tilde{\xi} \to L\eta_{(\zeta_p - 1)}(R\phi_{*}  (R\nu_{f*}(\wh{\cO}^{+}_X)))
\end{equation}
is an almost isomorphism. Finally by Remark \ref{rem:preshaissheaal}, the cohomology presheaves of the RHS of \eqref{eq:anoalfuth} are in fact almost $\cO$-sheaves (note that by Lemma \ref{lem:varithin}(ii), $R^{i}\phi_{*}$ of these almost $\cO$-sheaves are almost zero for $i >0$). Of course in the case $\fY = \Spf(\cO)$ and $f$ smooth, then the $R^{i}\phi_{*}$ of these almost $\cO$-sheaves vanishes for $i > 0$, as a geometric point has no higher cohomology groups.
\end{proof}

\begin{rem} \label{rem: almostnodropallo}
In general, the ``almost" in Corollary \ref{cor:objecxiadcom} cannot be dropped even in the case $\fY$ is locally of finite type over $\cO$. The issue is that even in this case, $R^{i}\phi_{*}$ of the sheaves appearing in Remark \ref{rem:preshaissheaal} are only almost zero in general. 
\end{rem}

Before we continue, we actually have enough at our disposal to compare our objects with the ones constructed in \cite{BhMorSch}. So let us take a brief detour to do this. Let $r \colon \fX_{\et} \to \fX_{\Zar}$ be the natural morphism of topoi and recall the projection morphism $q \colon Y_{\proet} \times_{\fY_{\et}} \fX_{\et} \to \fX_{\et}$.

\begin{prop} \label{prop:compaBMSvsours}
Suppose $\fY = \Spf(\cO)$ and $f$ is smooth. Then there is a natural quasi-isomorphism
\[
A\Omega_{\fX} \cong Rr_{*}Rq_{*}A\Omega_{\fX/\Spf(\cO)}
\]
where $A\Omega_{\fX} := L\eta_{\mu}(Rr_{*}Rq_{*}R\nu_{f*}\bA_{\Inf,X})$ as defined in \cite[Definition 9.1]{BhMorSch}.
\end{prop}

\begin{proof}
By \cite[Corollary 4.21]{SemistabAinfcoh} it suffices to show that there is a natural quasi-isomorphism
\[
A\Omega_{\fX_{\et}} \cong Rq_{*}A\Omega_{\fX/\Spf(\cO)},
\]
where $A\Omega_{\fX_{\et}} := L\eta_{\mu}(Rq_{*}R\nu_{f*}\bA_{\Inf,X})$ is the main object studied in \cite{SemistabAinfcoh}. Let $q^{\text{psh}} \colon (Y_{ \proet}\times_{\fY_{\et}}\fX_{\et})^{\text{psh}} \to \fX_{\et}^{\text{psh}}$ denote the presheaf version of $q$. Then by (4.1.3) in loc.cit. and \eqref{eq:evapresatsec} we get
\[
A\Omega_{\fX_{\et}}^{\text{psh}} \cong Rq_{*}^{\text{psh}}A\Omega_{\fX/\Spf(\cO)}^{\text{psh}},
\]
where $A\Omega_{\fX_{\et}}^{\text{psh}}$ is the presheaf version of $A\Omega_{\fX_{\et}}$ defined in (4.1.1) in loc.cit. However $A\Omega_{\fX_{\et}}^{\text{psh}}$ is already a sheaf by (4.6.1) in loc.cit. It remains to show $A\Omega_{\fX/\Spf(\cO)}^{\text{psh}}$ is also a sheaf. This follows from the last statement of Corollary \ref{cor:objecxiadcom}.
\end{proof}

In the spirit of \cite[Theorem 8.3]{BhMorSch} and \cite[Theorem 4.11]{SemistabAinfcoh}, our next task is to identify $H^{i}(\wt{\Omega}_{\fX / \fY})$ with the twists of the bundles given by (continuous) K\"ahler differentials. For this, in Proposition \ref{prop:expexpro}, we first express $H^{i}(\wt{\Omega}_{\fX / \fY})$ as $\bigwedge^{i} H^{1}(\wt{\Omega}_{\fX / \fY})$.

\subsubsection{The cup product maps} 

By the same arguments as in \cite[{Tag 068G}]{stacks-project}, there are product maps
\[
R^{j}\nu_{f*}(\wh{\cO}^{+}_X) \otimes_{\wh{\cO}_D^{+}} R^{j'}\nu_{f*}(\wh{\cO}^{+}_X) \xrightarrow{- \cup -} H^{j+ j'}(R\nu_{f*}(\wh{\cO}^{+}_X) \otimes_{\wh{\cO}_D^{+}}^{\bL} R\nu_{f*}(\wh{\cO}^{+}_X))
\]
that satisfy $x \cup y = (-1)^{jj'}y \cup x$ (cf. \cite[{Tag 0BYI}]{stacks-project}). By \cite[{Tag 0B6C}]{stacks-project}, there is a cup product map 
\[
R\nu_{f*}(\wh{\cO}^{+}_X) \otimes_{\wh{\cO}_D^{+}}^{\bL} R\nu_{f*}(\wh{\cO}^{+}_X) \to R\nu_{f*}(\wh{\cO}^{+}_X).
\]
These maps combine to give the \emph{cup product map} (where the tensor product is over $\wh{\cO}_D^{+}$)
\begin{equation} \label{eq:cuptenprod}
\bigotimes_{s= 1}^{i} R^{1}\nu_{f*}(\wh{\cO}^{+}_X) \to R^{i}\nu_{f*}(\wh{\cO}^{+}_X) \text{ for each } i >0
\end{equation}

\begin{prop} \label{prop:expexpro}
For each $i > 0$, the map \eqref{eq:cuptenprod} induces an almost isomorphism
\begin{equation} \label{eq:almowedgetois}
    \bigwedge^{i} \left( \frac{R^{1}\nu_{f*}(\wh{\cO}^{+}_X)}{(R^{1}\nu_{f*}(\wh{\cO}^{+}_X))[\zeta_p -1]} \right)^{a} \cong \bigwedge^{i} H^{1}(\wt{\Omega}_{\fX / \fY})^{a} \xrightarrow{\sim} H^{i}(\wt{\Omega}_{\fX / \fY})^{a} \cong \left( \frac{R^{i}\nu_{f*}(\wh{\cO}^{+}_X)}{(R^{i}\nu_{f*}(\wh{\cO}^{+}_X))[\zeta_p -1]} \right)^{a}.
\end{equation}
If in addition, $\fY$ is locally of finite type over $\cO$, then the almost can be dropped.
\end{prop}

\begin{proof}
By Proposition \ref{prop:twistOmegfrrmo}, each $H^{i}(\wt{\Omega}_{\fX / \fY})^{a}$ has no nonzero 2-torsion, so the antisymmetry of the map \eqref{eq:cuptenprod}, in each pair of the variables indeed induces the $\wh{\cO}_D^{+a}$-module map \eqref{eq:almowedgetois}. For the isomorphism claim, we may work étale locally (on $\fX$), so we put ourselves in the situation of \eqref{eq:locsigood} (cf. Corollary \ref{cor:swibasometil}). Similar to the proof of \cite[Proposition 4.8]{SemistabAinfcoh}, the edge maps \eqref{eq:relver3.3.1} are compatible with cup products. By Theorem \ref{thm:edgemapisoOsheaf} and Remark \ref{rem:preshaissheaal}, it then remains to argue that via the cup product the identification
\[
H^1_{\cont}(\Delta, R) \cong R^d \text{ induces } H^i_{\cont}(\Delta, R) \cong \bigwedge^{i}(R^d).
\]
This is precisely the content of \cite[Lemmas 7.3 and 7.5]{BhMorSch}. 

The final statement follows from Lemma \ref{lem:Rinftucirc} and running the argument on the nose.
\end{proof}

To relate $H^{1}(\wt{\Omega}_{\fX / \fY})$ to K\"ahler differentials, we now review the needed material on cotangent complexes. 

\subsubsection{The completed cotangent complex $\wh{\bL}_{\wh{\cO}_D^{+}/\wh{\cO}_Y^{+}}$}

We will need to analyze a cotangent complex naturally appearing on our (fibered product) site $D$. To phrase results, however, we need two facts coming from \cite[\S 8.2]{BhMorSch} (cf. \cite[\S 4.9]{SemistabAinfcoh}):
\begin{enumerate}
    \item
    $\wh{\bL}_{\wh{\cO}^{+}_X / \cO} \cong \wh{\bL}_{\wh{\cO}^{+}_Y / \cO} \cong 0$ and
    \item $\wh{\bL}_{\wh{\cO}_{X}^{+}/\bZ_p} \cong \wh{\cO}^{+}_{X} \{ 1\}[1]$.
\end{enumerate}
We take some time to review them. Affinoid perfectoids form a basis of $X_{\proet}$ by \cite[Proposition 4.8]{SchpHrig}. Therefore \cite[Lemma 3.14]{BhMorSch} ensures that for the sheaf of rings $\wh{\cO}_{X}^{+}$, the cotangent complex $\bL_{\wh{\cO}^{+}_X / \cO}$, whose terms are $\wh{\cO}_X^{+}$-flat and which give an object of $D^{\leq 0}(\wh{\cO}_X^{+})$, satisfies
\[
\bL_{\wh{\cO}^{+}_X / \cO} \otimes_{\bZ}^{\bL} \bZ/p^n \cong 0 \text{ and therefore (by derived Nakayama) } \wh{\bL}_{\wh{\cO}^{+}_X / \cO} \cong 0.
\]
This shows (1). We now show (2). By \cite[Definition 8.2]{BhMorSch}, the completed cotangent complex $\wh{\bL}_{\cO/\bZ_p}$ is the (shifted) Breuil-Kisin-Fargues twist $\cO\{ 1\}[1]$ and so the derived $p$-adic completion turns the canonical morphism 
\[
\bL_{\cO/\bZ_p} \otimes_{\cO} \wh{\cO}^{+}_{X} \to \bL_{\wh{\cO}^{+}_X/\bZ_p} \text{ into an isomorphism } \wh{\cO}^{+}_{X} \{ 1\}[1] \xrightarrow{\sim} \wh{\bL}_{\wh{\cO}^{+}_X/\bZ_p},
\]
since by (1), the cone (of the said morphism) $\wh{\bL}_{\wh{\cO}^{+}_X / \cO} \cong 0$. 

Armed with (1)-(2) we turn to our analysis of $\wh{\bL}_{\cO_D^{+}/\wh{\cO}_{Y}^{+}}$ and compare it to the sheaf $q^{*}\Omega^{1}_{\fX/\fY}$ (cf. Proposition \ref{prop:complcotcomispus}). We will need some preparation:

\begin{lem}\label{lem:secoullshsdd}
For a good object 
$Z = (\varprojlim_{i \in I} \Spa(S_{i}, S_{i}^{+}) \to \Spf(S) \leftarrow \Spf(R))$, 
\[
q^{*}\Omega^{1}_{\fX/\fY}(Z) = \Omega^{1}_{R/S} \wh{\otimes}_S S_{\infty}.
\]
\end{lem}

\begin{proof}
A priori $q^{*}\Omega^{1}_{\fX/\fY}$ is the sheaf associated to the presheaf $Z \mapsto \wh{\Omega}^{1}_{R/S} \otimes_R (R \wh{\otimes}_S S_{\infty})$. On the other hand by \cite[Chapter I, Proposition 5.3.17]{FujKat}, $\wh{\Omega}^{1}_{R/S}$ is finite free $R$-module. Since the $Z$-sections $\wh{\cO}_D^{+}$ are $R \wh{\otimes}_S S_{\infty}$ (cf. Proposition \ref{lem:morringtopnuf}), it follows that the presheaf is already a sheaf. 
\end{proof}

\begin{prop} \label{prop:complcotcomispus}
The complex $\wh{\bL}_{\cO_D^{+}/\wh{\cO}_{Y}^{+}}$ is quasi-isomorphic to $q^{*}\Omega^{1}_{\fX/\fY}$ placed in degree $0$.
\end{prop}

\begin{proof}
We use \cite[{Tag 0D6K}]{stacks-project} and begin analyzing $\bL_{\wh{\cO}_{D}^{+}/\wh{\cO}_{Y}^{+}} \otimes_{\bZ}^{\bL} \bZ/p^n$. By Lemma \ref{baslem:torfree} (resp. Corollary \ref{cor:comtentorfree}) together with the fact that good objects form a basis for the topology on $D$ (cf. Proposition \ref{prop:goodobjformbas}),  $p^{-1}\wh{\cO}_{Y}^{+}$ (resp. $\wh{\cO}_{D}^{+}$) is $p$-torsion free on local sections. Therefore there is a canonical quasi-isomorphism (cf. \cite[{Tag 08QQ}]{stacks-project})
\begin{equation}\label{eq:redmodpncomcon}
\bL_{\wh{\cO}_{D}^{+}/\wh{\cO}_{Y}^{+}} \otimes_{\bZ}^{\bL} \bZ/p^n \xrightarrow{\sim} \bL_{(\wh{\cO}_{D}^{+}/p^n)/(\wh{\cO}_{Y}^{+}/p^n)}.
\end{equation}
Now for a good object 
$Z = (\varprojlim_{i \in I} \Spa(S_{i}, S_{i}^{+}) \to \Spf(S) \leftarrow \Spf(R))$, $H^{i}(\bL_{(\wh{\cO}_{D}^{+}/p^n)/(\wh{\cO}_{Y}^{+}/p^n)})$ is the sheaf associated to the presheaf (cf. \cite[{Tag 08SW}]{stacks-project}) 
\begin{equation} \label{eq:preshvercolc}
Z \mapsto H^{i}(\bL_{(R \otimes_S S_{\infty}/p^n)/(S_{\infty}/p^n)}).
\end{equation} 
However as $R/p^n$ is smooth over $S/p^n$ we obtain (again by \cite[{Tag 08QQ}]{stacks-project})
\[
\bL_{(R \otimes_S S_{\infty}/p^n)/(S_{\infty}/p^n)} \cong \bL_{(R/p^n)/(S/p^n)} \otimes_{S/p^n} S_{\infty}/p^n
\]
is concentrated in degree 0. Therefore \eqref{eq:redmodpncomcon} implies $\bL_{\wh{\cO}_{D}^{+}/\wh{\cO}_{Y}^{+}} \otimes_{\bZ}^{\bL} \bZ/p^n$ is concentrated in degree 0. Hence by \cite[{Tag 0D6K}]{stacks-project}, $\wh{\bL}_{\cO_D^{+}/\wh{\cO}_{Y}^{+}}$ is also concentrated in degree $0$.

At degree 0, $H^{0}(\bL_{(R \otimes_S S_{\infty}/p^n)/(S_{\infty}/p^n)})$ coincides with $\Omega^{1}_{(R \otimes_S S_{\infty}/p^n)/(S_{\infty}/p^n)}$ (cf. \cite[{Tag 08UR}]{stacks-project}). Thus by the description of, in particular $H^{0}(\bL_{\wh{\cO}_{D}^{+}/\wh{\cO}_{Y}^{+}} \otimes_{\bZ}^{\bL} \bZ/p^n)$ (cf. \eqref{eq:preshvercolc}),
\begin{equation} \label{eq:reducHzeroco}
\varprojlim_n  H^{0}(U, \bL_{\wh{\cO}_{D}^{+}/\wh{\cO}_{Y}^{+}} \otimes_{\bZ}^{\bL} \bZ/p^n) = \varprojlim_n \Omega^{1}_{(\wh{\cO}_{D}^{+}/p^n)/(\wh{\cO}_{Y}^{+}/p^n)}(U).
\end{equation}
We now analyse the RHS of \eqref{eq:reducHzeroco}. The sheaf $\varprojlim_n \Omega^{1}_{(\wh{\cO}_{D}^{+}/p^n)/(\wh{\cO}_{Y}^{+}/p^n)}$ is associated to the presheaf $Z \mapsto \wh{\Omega}^{1}_{R \wh{\otimes}_S S_{\infty}/S_{\infty}}$. By \cite[Chapter I, Proposition 5.1.5]{FujKat}, there is a canonical isomorphism 
\begin{equation} \label{eq:FKlocsitu}
\Omega^{1}_{R/S} \wh{\otimes}_S S_{\infty} \xrightarrow{\sim} \wh{\Omega}^{1}_{R \wh{\otimes}_S S_{\infty}/S_{\infty}}.
\end{equation}
Finally by lemma \ref{lem:secoullshsdd}, $\Omega^{1}_{R/S} \wh{\otimes}_S S_{\infty}$ are precisely the $Z$-sections of $q^{*}\Omega^{1}_{\fY/\fX}$.
\end{proof}

\subsubsection{The relation between $\wt{\Omega}_{\fX / \fY}$ and the sheaf of relative K\"ahler differentials}

The functoriality of the cotangent complex supplies the pullback morphism
\begin{equation} \label{eq:funcpulmorshd}
\wh{\bL}_{\wh{\cO}_D^{+}/\bZ_p} \to R\nu_{f*}\wh{\bL}_{\wh{\cO}_{X}^{+}/\bZ_p} \overset{(2)}{\cong} R\nu_{f*}\wh{\cO}^{+}_{X} \{ 1\}[1].
\end{equation}
We first relate $\wh{\bL}_{\wh{\cO}_D^{+}/\bZ_p}$ to the completed cotangent complex of the previous section: $\wh{\bL}_{\cO_D^{+}/\wh{\cO}_{Y}^{+}}$. Indeed $\wh{\bL}_{\cO/\bZ_p}$ is the (shifted) Breuil-Kisin-Fargues twist $\cO \{ 1\}[1]$ (cf. \cite[Definition 8.2]{BhMorSch}) and so
\[
(\bL_{\cO/\bZ_p} \otimes_{\cO} \wh{\cO}_{D}^{+})^{\wedge} \cong \wh{\cO}_{D}^{+}\{1 \}[1].
\]
This implies $H^{0}(\wh{\bL}_{\wh{\cO}^{+}_{D}/\bZ_p}) \cong H^{0}(\wh{\bL}_{\wh{\cO}^{+}_{D}/\wh{\cO}^{+}_{Y}}) \cong  q^{*}\Omega^{1}_{\fX/\fY}$, where the last isomorphism is Proposition \ref{prop:complcotcomispus}. Therefore by applying $H^{0}(-)$ to \eqref{eq:funcpulmorshd} and twisting by $\cO \{-1 \}$ we obtain:
\begin{equation} \label{eq:composshealscont}
q^{*}\Omega^{1}_{\fX/\fY}\{-1 \} \to R^{1}\nu_{f*}\wh{\cO}^{+}_{X} \twoheadrightarrow \frac{R^{1}\nu_{f*}(\wh{\cO}^{+}_X)}{(R^{1}\nu_{f*}(\wh{\cO}^{+}_X))[\zeta_p -1]} \cong H^{1}(\wt{\Omega}_{\fX / \fY}).
\end{equation}

\begin{thm} \label{thm:HTspcssd}
The composition \eqref{eq:composshealscont} induces an $\wh{\cO}_{D}^{+a}$-isomorphism (after dividing the map by $\zeta_p-1$) 
\begin{equation} \label{eq:selfindsad}
q^{*}\Omega^{1}_{\fX/\fY}\{-1 \}^{a} \xrightarrow{\sim} H^{1}(\wt{\Omega}_{\fX / \fY})^{a},
\end{equation}
which by Proposition \ref{prop:twistOmegfrrmo} (in particular $H^{0}(\wt{\Omega}_{\fX / \fY}) \cong \wh{\cO}^{+}_D$) and Proposition \ref{prop:expexpro}, induces an $\wh{\cO}_{D}^{+a}$-module identification  
\begin{equation} \label{eq:sdfsods}
q^{*}\Omega^{i}_{\fX/\fY}\{-i \}^{a} \xrightarrow{\sim} H^{i}(\wt{\Omega}_{\fX / \fY})^{a}\text{, for every }i \geq 0. 
\end{equation}
If in addition $\fY$ is locally of finite type over $\cO$, then the almost can be dropped.
\end{thm}

\begin{proof}
By Corollary \ref{cor:swibasometil} and Lemma \ref{lem:secoullshsdd} we may assume $\fX = \Spf (R)$ and $\fY = \Spf (S)$ with $R = S \{t_1^{\pm 1}, \ldots, t_d^{\pm 1} \}$. In this case $\Delta = \bZ_p(1)^{d}$ is the natural group acting $R \wh{\otimes}_S S_{\infty}$-linearly on $R_{\infty}$. The $p$-adic derived completion of the pullback morphism
\[
\bL_{R \wh{\otimes}_S S_{\infty}/\bZ_p} \otimes^{\bL}_{R \wh{\otimes}_S S_{\infty}} R_{\infty} \to \bL_{R_{\infty}/\bZ_p}
\]
supplies a morphism $\wh{\bL}_{R \wh{\otimes}_S S_{\infty}/\bZ_p} \to \wh{\bL}_{R_{\infty}/\bZ_p} \cong R_{\infty}\{1 \}[1]$. We remind the reader that the isomorphism $\wh{\bL}_{R_{\infty}/\bZ_p} \cong R_{\infty}\{1 \}[1]$ is coming from the Breuil-Kisin-Fargues twist $\wh{\bL}_{\cO/\bZ_p} = \cO \{ 1\}[1]$. The $\Delta$-equivariance of the map $R \wh{\otimes}_S S_{\infty} \to R_{\infty}$ induces a map
\begin{equation} \label{eq:analmorlocalca}
H^{0}(\wh{\bL}_{R \wh{\otimes}_S S_{\infty}/\bZ_p}) \to H^{1}_{\cont}(\Delta, R_{\infty})\{1 \}
\end{equation}
The LHS of \eqref{eq:analmorlocalca} $H^{0}(\wh{\bL}_{R \wh{\otimes}_S S_{\infty}/\bZ_p}) = \wh{\Omega}^{1}_{R \wh{\otimes}_S S_{\infty}/S_{\infty}} \overset{\eqref{eq:FKlocsitu}}{\cong} \Omega^{1}_{R/S} \wh{\otimes}_S S_{\infty}$ is a free $R \wh{\otimes}_S S_{\infty}$-module of rank $d$ generated by $\tfrac{dt_i}{t_i}$. By the proof of \cite[Proposition 8.15]{BhMorSch} \eqref{eq:analmorlocalca} is an isomorphism onto 
\begin{equation} \label{eq:usinfgolsds}
(\zeta_p -1)H^{1}_{\cont}(\Delta, R \wh{\otimes}_S S_{\infty})\{1 \} = (\zeta_p -1)H^{1}_{\cont}(\Delta, R_{\infty})\{1 \}
\end{equation}
Note the equality of \eqref{eq:usinfgolsds} follows from Proposition \ref{prop:mtorsioncoho}. The isomorphism sends $\tfrac{dt_i}{t_i}$ to 
\begin{equation} \label{eq:shifcasds}
d\log_i \otimes 1 \in H^{1}_{\cont}(\bZ_p(1), \cO\{1 \}) \otimes_{\cO}  (R \wh{\otimes}_S S_{\infty}) = \Hom_{\cont}(\bZ_p(1), T_p(\Omega^{1}_{\cO/\bZ_p})) \otimes_{\cO}  (R \wh{\otimes}_S S_{\infty}),
\end{equation}
is the map on $p$-adic Tate modules induced by the map $d\log_i \colon \mu_{p^{\infty}} \to \Omega^{1}_{\cO/\bZ_p}$. The $\bZ_p(1)$ in \eqref{eq:shifcasds} corresponds to the $i$th component in $\Delta$. The result now follows from Theorem \ref{thm:edgemapisoOsheaf}.
\end{proof}

\subsection{The de Rham specialization of $A\Omega_{\fX/\fY}$}

In this section, we relate the de Rham specialization of $A\Omega_{\fX/\fY}$ to the relative de Rham complex, cf. \cite[Theorem 14.1]{BhMorSch} and \cite[Theorem 4.17]{SemistabAinfcoh}. The additional difficulty present in the relative situation is that in general, we only have almost isomorphisms (cf. \eqref{eq:hodgtatespecmap}). This means we have to take care with the construction of the map \eqref{eq:dehamsps}. For the next theorem we abuse notation and define the complex $q^{*}\Omega^{\bullet}_{\fX/\fY}$ to have $i$th term $q^{*}\Omega^{i}_{\fX/\fY}$ and induced differentials. 

\begin{thm} \label{thm:regdeRhamse}
There is a morphism in $D(D,  \cO)$ 
\begin{equation} \label{eq:dehamsps}
q^{*}\Omega^{\bullet}_{\fX/\fY} \to A\Omega_{\fX/\fY} \otimes^{\bL}_{A_{\Inf}, \theta} \cO.
\end{equation}
whose image in $D(D, \cO^{a})$ is an almost isomorphism. If in addition $\fY$ is locally of finite type over $\cO$, then \eqref{eq:dehamsps} is already an isomorphism in $D(D, \cO)$.
\end{thm}

\begin{proof}
We have
\[
A\Omega_{\fX/\fY} \otimes^{\bL}_{A_{\Inf}, \theta} \cO \cong A\Omega_{\fX/\fY} \otimes^{\bL}_{A_{\Inf}, \varphi} A_{\Inf} \otimes^{\bL}_{A_{\Inf}, \theta \circ \varphi^{-1}} \cO \cong L\eta_{\tilde{\xi}}(A\Omega_{\fX/\fY}) \otimes^{\bL}_{A_{\Inf}, \theta \circ \varphi^{-1}} \cO
\]
where the second isomorphism follows from \cite[Lemma 6.11]{BhMorSch} and the identity $\varphi(\mu) = \tilde{\xi}\mu$. By \cite[Proposition 6.12]{BhMorSch} the object $L\eta_{\tilde{\xi}}(A\Omega_{\fX/\fY}) \otimes^{\bL}_{A_{\Inf}, \theta \circ \varphi^{-1}} \cO$ is identified with the complex $C^{\bullet}$ whose $i$th degree term is
\[
H^{i}(A\Omega_{\fX/\fY} \otimes^{\bL}_{A_{\Inf}, \theta \circ \varphi^{-1}} \cO) \otimes_{\cO} \left( \tfrac{\ker (\theta \circ \varphi^{-1})}{(\ker (\theta \circ \varphi^{-1}))^{2}} \right)^{\otimes i}
\]
and the differentials are given by Bockstein homomorphisms. By the proof of \cite[Theorem 4.17]{SemistabAinfcoh}, $\tfrac{\ker (\theta \circ \varphi^{-1})}{(\ker (\theta \circ \varphi^{-1}))^{2}} \cong \cO \{1 \}$. The complex $C^{\bullet}$ has a natural structure of a commutative differential graded $\cO$-algebra with multiplication induced by the composition:
\[
H^{i}(A\Omega_{\fX/\fY}/\tilde{\xi}) \times H^{j}(A\Omega_{\fX/\fY}/\tilde{\xi}) \to H^{i+j}(A\Omega_{\fX/\fY}/\tilde{\xi} \otimes^{\bL}_{\cO} A\Omega_{\fX/\fY}/\tilde{\xi}) \to H^{i+j}(A\Omega_{\fX/\fY}/\tilde{\xi}),
\]
where the first map is the cup product and the second map is induced from the lax symmetric monoidal structure on $L\eta_{\mu}$ (cf. \cite[Lemma 6.7]{BhMorSch}) and the commutative algebra structure on $R\nu_{f*}\bA_{\Inf,X}$.

The strategy now for producing a morphism $q^{*}\Omega^{\bullet}_{\fX/\fY} \to C^{\bullet}$ in $D(D, \cO)$ is as follows:
\begin{enumerate}
    \item Construct a morphism $\wh{\cO}_{D}^{+} \to C^{0}$.
    \item Construct a (compatible) morphism $q^{*}\Omega^{1}_{\fX/\fY} \to C^{1}$ (universal property of K\"ahler differentials).
    \item Extend the morphism from (2) to higher forms using the differential graded $\cO$-algebra structure of $q^{*}\Omega^{\bullet}_{\fX/\fY}$ and $C^{\bullet}$.
\end{enumerate}
A problem arises in step (2), as we do not know in general whether $H^{1}(A\Omega_{\fX/\fY} \otimes^{\bL}_{A_{\Inf}, \theta \circ \varphi^{-1}} \cO)$ is (derived) $p$-adically complete. In the case $\fY$ is locally of finite type over $\cO$ however, this is indeed true\footnote{This follows from Theorem \ref{thm:almosishodsa}, Proposition \ref{prop:twistOmegfrrmo} and Lemma \ref{lem:Odsupad}.}. Moreover with regards to step (3), we do not know whether $C^{2}$ is 2-torsion free (in the case $p =2$), in order to pass to exterior algebras. A solution for these obstacles is to work with the (intermediary) group cohomology terms appearing as the source of \eqref{eq:almosWmbmor}. More precisely for each good object $Z$ of the form \eqref{eq:goodpair} we have a map of complexes
\begin{equation} \label{eq:mapofcombock}
 \xymatrix{
 C^{0}_{Z} \ar@{->}[r] \ar@{->}[d]^{\wr} & C^{1}_{Z} \ar@{->}[d] \ar@{->}[r] & C^{2}_{Z} \ar@{->}[d] \ar@{->}[r]  & \cdots    \\
 C^{0}(Z) \ar@{->}[r] & C^{1}(Z) \ar@{->}[r] & C^{2}(Z) \ar@{->}[r] & \cdots
 }
\end{equation}
where $C^{i}_{Z} := H^{i}(L\eta_{\mu}R\Gamma_{\cont}(\Delta, \bA_{\Inf}(R_{\infty})) \otimes^{\bL}_{A_{\Inf}, \theta \circ \varphi^{-1}} \cO)\{ i\}$. The top complex is a representative for the complex $L\eta_{\tilde{\xi}}L\eta_{\mu}R\Gamma_{\cont}(\Delta, \bA_{\Inf}(R_{\infty})) \otimes^{\bL}_{A_{\Inf}, \theta \circ \varphi^{-1}} \cO$ via \cite[Proposition 6.12]{BhMorSch}. The vertical arrows are induced via \eqref{eq:almosWmbmor} and the functoriality of the Bockstein homomorphism provides the map \eqref{eq:mapofcombock}. 

We now construct a morphism of complexes $(q^{*}\Omega^{\bullet}_{\fX/\fY})(Z) \to C^{\bullet}_{Z}$ (the complex $(q^{*}\Omega^{\bullet}_{\fX/\fY})(Z)$ has $i$th term $(q^{*}\Omega^{i}_{\fX/\fY})(Z)$) using the strategy outlined above.
\begin{enumerate}
    \item $\wh{\cO}^{+}_{D}(Z) = R \wh{\otimes}_S S_{\infty} \overset{\eqref{eq:lonconcesd}}{=} \frac{H^{0}_{\cont}(\Delta, R_{\infty})}{H^{0}_{\cont}(\Delta, R_{\infty})[\zeta_p -1]} \xrightarrow{\eqref{eq:truisnoalis}^{-1}} C^{0}_{Z}$.
    \item Note that the complex $C^{i}_{Z}$ is a commutative differential graded $\cO$-algebra (by a similar argument used to show $C^{\bullet}$ is a commutative differential graded $\cO$-algebra). By \eqref{eq:truisnoalis} and \eqref{eq:lonconcesd}, $C^{1}_{Z}$ is $p$-adically complete. Therefore to construct a compatible morphism $(q^{*}\Omega^{1}_{\fX/\fY})(Z) \to C^{1}_Z$, we need to check that the composition $\wh{\cO}^{+}_{D}(Z) \to C^{0}_{Z} \to C^{1}_{Z}$ is $S_{\infty}$-linear (we endow $C^{0}_Z$ and $C^{1}_Z$ with their natural structure of $S_{\infty}$-modules via the identifications \eqref{eq:truisnoalis} and \eqref{eq:lonconcesd}). By definition the first map $\wh{\cO}^{+}_{D}(Z) \to C^{0}_{Z}$ is $S_{\infty}$-linear. To see that the Bockstein homomorphism $C^{0}_{Z} \to C^{1}_{Z}$ is $S_{\infty}$-linear, note that $L\eta_{\mu}R\Gamma_{\cont}(\Delta, \bA_{\Inf}(R_{\infty}))$ has an explicit representation by a complex of $\bA_{\Inf}(S_{\infty})$-modules involving Koszul complexes (cf. \cite[Lemma 4.6]{Bhatspecvar}\footnote{Technically the proof in loc.cit. is for the ring $A_{\Inf}$, however the same proof works for $\bA_{\Inf}(S_{\infty})$ (and its étale extensions).} and \cite[Lemma 7.9]{BhMorSch}). This means that the Bockstein homomorphism is also $S_{\infty}$-linear. Note that the $S_{\infty}$-module structures put on $C^0_Z$ and $C^{1}_Z$ are compatible with the ones coming from the Koszul complexes.
    \item By \eqref{eq:truisnoalis}, \eqref{eq:lonconcesd} and Corollary \ref{cor:comtentorfree}, $C^{2}_{Z}$ is 2-torsion free. Thus the map constructed in (2) extends (compatibly) to higher forms. 
\end{enumerate}

We now show that the morphisms $(q^{*}\Omega^{\bullet}_{\fX/\fY})(Z) \to C^{\bullet}_{Z} \to C^{\bullet}(Z)$ are functorial in $Z$, and therefore glue together to give a morphism of complexes $q^{*}\Omega^{\bullet}_{\fX/\fY} \to C^{\bullet}$. The functoriality of the first map follows from the universal property of K\"ahler differentials. The functoriality of the second map  follows from the functoriality of \eqref{eq:almosWmbmor}, which in turn follows from the functoriality of the edge map \eqref{eq:edgemapforainf}.

It remains to show that the map of complexes $q^{*}\Omega^{\bullet}_{\fX/\fY} \to C^{\bullet}$ is an almost isomorphism (and an isomorphism in the case $\fY$ is locally of finite type over $\cO$). By Theorem \ref{thm:almosishodsa} there is a morphism
\[
C^{1} \to H^{1}(\wt{\Omega}_{\fX/\fY})\{1\}
\]
which is an almost isomorphism. We claim that the induced composition 
\begin{equation} \label{eq:tehconsdmor}
q^{*}\Omega^{1}_{\fX/\fY} \to C^{1} \to H^{1}(\wt{\Omega}_{\fX/\fY})\{ 1\}
\end{equation}
identifies with \eqref{eq:selfindsad}\{1\}. We reduce to the case $\fX = \Spf (R)$ and $\fY = \Spf (S)$ with $R = S \{t_1^{\pm 1}, \ldots, t_d^{\pm 1} \}$ (cf. Lemma \ref{lem:baschaexpect}). We consider the morphisms upon taking $Z$-sections with $Z$ of the form \eqref{eq:goodpair}. In this case, as shown in the proof of Theorem \ref{thm:HTspcssd}, \eqref{eq:selfindsad}\{1\} looks like (using the notation appearing there):
\begin{align*}
    (q^{*}\Omega^{1}_{\fX/\fY})(Z) &\to H^{1}_{\cont}(\Delta, R \wh{\otimes}_S S_{\infty})\{1 \} \xrightarrow{H^{1}(\text{\eqref{eq:fsmsfdssd}})\{1\}} H^{1}(\wt{\Omega}_{\fX / \fY})\{1\}(Z) \\
    \tfrac{dt_i}{t_i} &\mapsto \tfrac{d\log_i \otimes 1}{\zeta_p -1}.
\end{align*}
On the other hand, \eqref{eq:tehconsdmor} fits into a diagram
\[
 \xymatrix{
 \wh{\cO}^{+}_{D}(Z) \ar@{->}[r]^{\sim} \ar@{->}[d] & C^{0}_{Z} \ar@{->}[d] &&&    \\
 (q^{*}\Omega^{1}_{\fX/\fY})(Z) \ar@{->}[r] & C^{1}_{Z} \ar@{->}[r]  & H^{1}_{\cont}(\Delta, R \wh{\otimes}_S S_{\infty})\{1 \} \ar@{->}[r] \ar@/^2pc/[rr]^{H^{1}(\text{\eqref{eq:fsmsfdssd}})\{1\}}  & C^{1}(Z) \ar@{->}[r] & H^{1}(\wt{\Omega}_{\fX / \fY})\{1\}(Z),
 }
\]
where $C^{1}_{Z} \to H^{1}_{\cont}(\Delta, R \wh{\otimes}_S S_{\infty})\{1 \} $ is an isomorphism induced by 
\eqref{eq:truisnoalis}. So we need to understand the morphism $(q^{*}\Omega^{1}_{\fX/\fY})(Z) \to C^{1}_Z$ (constructed by universality of K\"ahler differentials). For this we can assume $d = 1$.  By \cite[Lemma 7.9]{BhMorSch}, the complex $L\eta_{\mu}R\Gamma_{\cont}(\Delta, \bA_{\Inf}(R_{\infty}))$ is represented by explicit Koszul complexes:
\[
L\eta_{\mu}R\Gamma_{\cont}(\Delta, \bA_{\Inf}(R_{\infty})) \cong \wh{\bigoplus}_{a \in \bZ} (\bA_{\Inf}(S_{\infty}) \xrightarrow{\tfrac{[\epsilon^{pa}]-1}{[\epsilon]-1}} \bA_{\Inf}(S_{\infty})) =: \wh{\bigoplus}_{a \in \bZ} K_a
\]
where the completion is $(p,\mu)$-adic (cf. \cite[Lemma 4.6]{Bhatspecvar}). On grading of degree $a$, the Bockstein homomorphism (after trivializing the twist $\{1 \}$ by choosing $\tilde{\xi}$ as a generator of $(\tilde{\xi})$) looks like
\begin{align*}
H^{0}(K_a /\tilde{\xi}) &\to H^{1}(K_a /\tilde{\xi})\\
1 &\mapsto \left( \tfrac{[\epsilon^{pa}]-1}{[\epsilon]-1} \right)/\tilde{\xi} = \left( \tfrac{[\epsilon^{pa}]-1}{[\epsilon]-1} \right)/\left( \tfrac{[\epsilon^{p}]-1}{[\epsilon]-1} \right) = \tfrac{[\epsilon^{pa}]-1}{[\epsilon^p]-1} = a \mod \tilde{\xi}.
\end{align*}
Moreover $t_1^{a} \in \wh{\cO}^{+}_D(Z)$ is identified with $1 \in H^{0}(K_a /\tilde{\xi})$ and $t_1^{a} \in H^{1}_{\cont}(\Delta, R \wh{\otimes}_S S_{\infty})$ identifies with $1 \in H^{1}(K_a /\tilde{\xi})$. The \emph{disappearance} of $\tfrac{d\log_1}{\zeta_p -1}$ is explained by the fact that the image of $\wt{\xi}$ via the map $\ker (\theta \circ \varphi^{-1}) \to \tfrac{\ker (\theta \circ \varphi^{-1})}{(\ker (\theta \circ \varphi^{-1}))^{2}} \cong \cO \{1\}$ is $\tfrac{d\log_1}{\zeta_p -1}$ (cf. \cite[Remark 6.9]{Bhatspecvar} and \cite[Example 4.24]{BhMorSch}). For an explicit calculation, cf. \cite[Remark 9.4]{Bhatpsdde}. The result now follows.
\end{proof}

An important immediate consequence of the de Rham specialization is the following (almost is w.r.t. $[\fm^{\flat}]$.):

\begin{cor} \label{cor:jfsdmderhamd}
The object $A\Omega_{\fX/\fY}$ is almost derived $p$-adically complete (i.e. the canonical morphism $A\Omega_{\fX/\fY} \to R\varprojlim (A\Omega_{\fX/\fY} \otimes^{\bL}_{\bZ} \bZ/p^n)$ is an almost isomorphism). Similarly for the object $R\Gamma_{A_{\Inf}}(\fX/\fY)$.
\end{cor}

\begin{proof}
First note that by Lemma \ref{lem:Odsupad}, $\wh{\cO}^{+}_{D}$ is derived $p$-adically complete. This implies $q^{*}\Omega^{\bullet}_{\fX/\fY}$ is a bounded complex each of whose terms is derived $p$-adically complete and hence itself is derived $p$-adically complete (induction via stupid truncations).
Thus by Theorem \ref{thm:regdeRhamse},  $A\Omega_{\fX/\fY} \otimes^{\bL}_{A_{\Inf}, \theta} \cO$ is almost derived $p$-adically complete. By induction on $n$, one gets that $A\Omega_{\fX/\fY} \otimes^{\bL}_{A_{\Inf}} A_{\Inf}/\xi^n$ is almost derived $p$-adically complete. Moreover by Corollary \ref{cor:objecxiadcom}, $A\Omega_{\fX/\fY}$ is almost derived $\xi$-adically complete. Therefore $A\Omega_{\fX/\fY}$ is almost derived $p$-adically complete by the definition of derived completeness (cf. \cite[{Tag 0999}]{stacks-project}). The last statement follows by \cite[{Tag 0A0G}]{stacks-project}.
\end{proof}

\begin{rem}
As in Remark \ref{rem: almostnodropallo}, in general the ``almost" in Corollary \ref{cor:jfsdmderhamd} cannot be dropped even in the case $\fY$ is locally of finite type over $\cO$.
\end{rem}

\subsection{Base change for $Y_{ \proet}\times_{\fY_{\et}}\fX_{\et}$}

In this section, in addition to the assumptions made in the beginning of \S\ref{sec:TheHodhTaspe}, we assume that $f \colon \fX \to \fY$ is proper. For later use, in particular for the comparison of $A_{\Inf}$-cohomology with $q$-crystalline cohomology (cf. Theorem \ref{thm:comparofRGqcruandRGaAin}) we record a base change result, cf. Proposition \ref{prop:bcrsmopn}. For a similar statement involving the Faltings' topos, cf. \cite[Théorème 6.5.31]{AbbGors}. Recall by \eqref{eq:morringtop}, the projections $p$ and $q$ extend to morphisms of ringed topoi $(p, p^{\sharp})$ and $(q, q^{\sharp})$. By taking modulo $p^n$ (of the structure sheaves), we denote the induced morphisms by $(p_n, p_n^{\sharp})$ and $(q_n, q_n^{\sharp})$, respectively. Similarly we obtain $f_n$ and $\nu_{Y,n}$. First we need some preparation.

\begin{lem} \label{lem:afinssitu}
Let $g \colon \fX \to \fY$ be a morphism of affine $p$-adic formal schemes of type (S)(b) over $\cO$.  For each $n \geq 1$, consider the commutative diagram of ringed topoi
\[
 \xymatrix{
 (Y_{ \proet}\times_{\fY_{\et}}\fX_{\et}, \wh{\cO}^+_{D}/p^n) \ar@{->}[rr]_{{(q_n, q_{n}^{\sharp})}} \ar@{->}[d]^{(p_n, p_{n}^{\sharp})} && (\fX_{\et}, \cO_{\fX_{\et}}/p^n) \ar@{->}[d]^{(g_n, g_{n}^{\sharp})}   \\ 
 (Y_{\proet}, \wh{\cO}^{+}_{Y}/p^n) \ar@{->}[rr]_{{(\nu_{Y,n}, \nu_{Y,n}^{\sharp})}} && (\fY_{\et}, \cO_{\fY_{\et}}/p^n).
 }
 \]
Then the cohomology sheaves of the cone of the canonical base change map 
 \begin{equation} \label{eq:affinecases}
 \nu_{Y,n}^{*}g_{n*}(\cO_{\fX_{\et}}/p^n) \to p_{n*}q_{n}^{*}(\cO_{\fX_{\et}}/p^n)
 \end{equation}
 are killed by $\fm$.
\end{lem}

\begin{proof}
Let $\fX = \Spf (R)$, $\fY = \Spf (S)$ and $Z := \varprojlim_{i \in I} \Spa (S_i, S_i^{+})$ an affinoid perfectoid of $Y_{\proet}$. 
By \cite[Lemma 4.10(i)]{SchpHrig} the LHS of \eqref{eq:affinecases} is almost equal the sheafification of the presheaf
\begin{equation*} 
Z \mapsto \varinjlim_{\fV \in \fY_{\et}, Z \to V} (\cO_{\fX_{\et}}/p^n)(\fX \times_{\fY} \fV) \otimes_{(\cO_{\fY_{\et}}/p^n)(\fV)}S_{\infty}/p^n = (R \otimes_S S_{\infty})/p^n,
\end{equation*}
where $V$ is the adic generic fiber of $\fV$ and $Z \to V$ are morphisms over $Y$. Moreover by Lemma \ref{lem:varithin}(i) and Proposition \ref{lem:morringtopnuf}, the RHS of \eqref{eq:affinecases} is almost equal to the sheafification of the presheaf
\begin{equation*} 
Z \mapsto (R \otimes_S S_{\infty})/p^n
\end{equation*}
\end{proof}

\begin{prop} \label{prop:bcrsmopn}
For each $n \geq 1$, consider the commutative diagram of ringed topoi
\[
 \xymatrix{
 (Y_{ \proet}\times_{\fY_{\et}}\fX_{\et}, \wh{\cO}^+_{D}/p^n) \ar@{->}[rr]_{{(q_n, q_{n}^{\sharp})}} \ar@{->}[d]^{(p_n, p_{n}^{\sharp})} && (\fX_{\et}, \cO_{\fX_{\et}}/p^n) \ar@{->}[d]^{(f_n, f_{n}^{\sharp})}   \\ 
 (Y_{\proet}, \wh{\cO}^{+}_{Y}/p^n) \ar@{->}[rr]_{{(\nu_{Y,n}, \nu_{Y,n}^{\sharp})}} && (\fY_{\et}, \cO_{\fY_{\et}}/p^n),
 }
 \]
 and $\sF$ a locally free $\cO_{\fX_{\et}}/p^n$-module. Then the cohomology sheaves of the cone of the canonical base change map (cf. \cite[{Tag 07A7}]{stacks-project})
 \begin{equation}\label{eq:basdmosprs}
 L\nu_{Y,n}^{*}Rf_{n*}\sF \to Rp_{n*}Lq_{n}^{*}\sF
 \end{equation}
 are killed by $\fm$.
\end{prop}

\begin{proof}
The statement is local for the étale topology on $\fY$, so we can assume $\fY$ is affine. The idea is to rewrite both sides of \eqref{eq:basdmosprs} as sheafified \v{C}ech complexes and then compare these \v{C}ech complexes via Leray's acyclicity lemma.

We deal with the LHS of \eqref{eq:basdmosprs} first.
Since $f_n$ is proper, we can choose a finite affine open covering $\cU : \fX = \bigcup \fU_i$, which trivializes $\sF$. Then we are in the set-up of \cite[{Tag 01XL}]{stacks-project}. In particular since $f_n$ is flat we have an explicit $K$-flat resolution of $Rf_{n*}\sF$: $\check{\cC}^{\bullet}(\cU, f_n, \sF)$, whose $r$th term is
\[
\check{\cC}^{r}(\cU, f_n, \sF) = \bigoplus_{i_{0}\ldots i_r} f_{n, i_{0}\ldots i_{r}*}(\cO_{\fX_{\et}}/p^n) \lvert_{\fU_{i_{0}\ldots i_{r}}},
\]
where $\fU_{i_{0}\ldots i_{r}} := \fU_{i_{0}} \cap \ldots \cap \fU_{i_{r}}$ and  $f_{n, i_{0}\ldots i_{r}} \colon \fU_{i_{0}\ldots i_{r}} \to \fY$ is the restriction of $f_n$ to $\fU_{i_{0}\ldots i_{r}}$. Therefore the LHS of \eqref{eq:basdmosprs} is represented by a complex whose $r$th term is
\begin{equation} \label{eq:pthtepul}
    \bigoplus_{i_{0}\ldots i_r} \nu_{Y,n}^{*}f_{n, i_{0}\ldots i_{r}*}(\cO_{\fX_{\et}}/p^n) \lvert_{\fU_{i_{0}\ldots i_{r}}}.
\end{equation}
By Lemma \ref{lem:afinssitu} the expression \eqref{eq:pthtepul} is almost equal to
\begin{equation} \label{eq:lshepa}
    \bigoplus_{i_{0}\ldots i_r} p_{n, i_{0}\ldots i_{r}*}(\wh{\cO}_{D}^{+}/p^n) \lvert_{Y_{ \proet}\times_{\fY_{\et}}\fU_{i_{0}\ldots i_{r}, \et}},
\end{equation}
where $p_{n, i_{0}\ldots i_{r}} \colon Y_{ \proet}\times_{\fY_{\et}}\fU_{i_{0}\ldots i_{r}, \et} \to Y_{\proet}$ is the obvious map (considered as a morphism of ringed topoi with the obvious structure sheaves modulo $p^n$). 

We now deal with the RHS of \eqref{eq:basdmosprs}, which is
$Rp_{n*}(q_{n}^{*}\sF)$. Let $Z_i = (Y \to \fY \leftarrow \fU_i)$. The covering $\cU$ gives rise to a covering $\wt{\cU}: \{Z_i \to (Y \to \fY \leftarrow \fX) \}_i$. We claim that $q_{n}^{*}\sF$ has an (almost) resolution by a complex $\check{\cC}^{\bullet}(\wt{\cU}, q_{n}^{*}\sF)$, whose $r$th term is
\[
\check{\cC}^{r}(\wt{\cU}, q_{n}^{*}\sF) = \bigoplus_{i_{0}\ldots i_r} g_{n, i_{0}\ldots i_{r}*}\wh{\cO}_{D}^{+}/p^n \lvert_{Y_{ \proet}\times_{\fY_{\et}}\fU_{i_{0}\ldots i_{r}, \et}}
\]
where $g_{n,i_{0}\ldots i_{r}} \colon Y_{ \proet}\times_{\fY_{\et}}\fU_{i_{0}\ldots i_{r}, \et} \to Y_{ \proet}\times_{\fY_{\et}}\fX_{\et}$ (considered as a morphism of ringed topoi with the obvious structure sheaves modulo $p^n$). The differentials are the usual ones attached to a \v{C}ech complex and $\varepsilon \colon q_{n}^{*}\sF \to \check{\cC}^{0}(\wt{\cU}, q_{n}^{*}\sF)$ is defined by taking the direct product of the natural maps $q_{n}^{*}\sF \to g_{n, i*}\wh{\cO}_{D}^{+}/p^n \lvert_{Y_{ \proet}\times_{\fY_{\et}}\fU_{i, \et}}$. The claim may appear standard but without a study of points for our fibered topos $Y_{ \proet}\times_{\fY_{\et}}\fX_{\et}$ (cf. the proof of \cite[{Tag 02FU}]{stacks-project}), we rely on the almost vanishing of higher \v{C}ech cohomology, cf. proof of Lemma \ref{lem:varithin}(i).

\begin{lem}
The sequence of sheaves
\[
0 \to q_{n}^{*}\sF^{a} \xrightarrow{\varepsilon} \check{\cC}^{0}(\wt{\cU}, q_{n}^{*}\sF)^{a} \to \check{\cC}^{1}(\wt{\cU}, q_{n}^{*}\sF)^{a} \to \ldots
\]
is exact.
\end{lem}

\begin{proof}
Let $Z$ be a good object of $Y_{ \proet}\times_{\fY_{\et}}\fX_{\et}$, which trivializes $q_{n}^{*}\sF$. Then it suffices to show that \[
0 \to (\wh{\cO}^{+a}_{D}/p^n)(Z) \xrightarrow{\varepsilon} \check{\cC}^{0}(\wt{\cU}, \wh{\cO}_{D}^{+}/p^n)^{a}(Z) \to \check{\cC}^{1}(\wt{\cU}, \wh{\cO}_{D}^{+}/p^n)^{a}(Z) \to \ldots
\]
is exact. Together with Proposition \ref{lem:morringtopnuf}, this is proved in the proof of Lemma \ref{lem:varithin}(i).
\end{proof}
Returning to the proof of the proposition, again by almost vanishing of higher \v{C}ech cohomology, we have
\[
R^{s}g_{n,i_{0}\ldots i_{r}*}\wh{\cO}_{D}^{+}/p^n \lvert_{Y_{ \proet}\times_{\fY_{\et}}\fU_{i_{0}\ldots i_{r}, \et}}\text{ and }R^{s}p_{n, i_{0}\ldots i_{r}*}\wh{\cO}_{D}^{+}/p^n \lvert_{Y_{ \proet}\times_{\fY_{\et}}\fU_{i_{0}\ldots i_{r}, \et}}
\] are almost zero for $s >0$. This implies $R^{s}p_{n*}(g_{n,i_{0}\ldots i_{r}*}\wh{\cO}_{D}^{+}/p^n \lvert_{Y_{ \proet}\times_{\fY_{\et}}\fU_{i_{0}\ldots i_{r}, \et}})$ is almost zero for $s>0$ (because $p_{n} \circ g_{n,i_{0}\ldots i_{r}} = p_{n,i_{0}\ldots i_{r}}$). Therefore by Leray's acyclicity lemma we have
\[
p_{n*}\check{\cC}^{\bullet}(\wt{\cU}, q_{n}^{*}\sF)^{a} = Rp_{n*}\check{\cC}^{\bullet}(\wt{\cU}, q_{n}^{*}\sF)^{a} = Rp_{n*}q_{n}^{*}\sF^{a}.
\]
Finally note that $p_{n*} \check{\cC}^{r}(\wt{\cU}, q_{n}^{*}\sF)$ gives the expression in \eqref{eq:lshepa}. This completes the proof.
\end{proof}

\begin{cor} \label{cor:baseschasimfs}
In the situation of Proposition \ref{prop:bcrsmopn}, assume in addition that for all $i \geq 0$, $R^{i}f_{n*}\sF$ is locally free of finite type. Then canonical base change map \eqref{eq:basdmosprs}, induces morphisms
 \[
 \nu_{Y,n}^{*}R^{i}f_{n*}\sF \to R^{i}p_{n*}q_{n}^{*}\sF
 \]
 whose kernel and cokernel are killed by $\fm$.
\end{cor}

\begin{proof}
This is obtained by taking the cohomology of \eqref{eq:basdmosprs} in Proposition \ref{prop:bcrsmopn}.
\end{proof}
We also record a $p$-completed version of the base change.

\begin{cor} \label{cor:pcomverorbascha}
Let $\sF$ be a locally free $\cO_{\fX_{\et}}$-module. Then the cohomology sheaves of the cone of the canonical base change map
 \[
 L\nu_{Y}^{*}Rf_{*}\sF \to Rp_{*}Lq^{*}\sF
 \]
 are killed by $\fm$.
\end{cor}

\begin{proof}
First note that by \cite[{Tag 0A1H}]{stacks-project}, $Rf_{*}\sF \otimes^{\bL}_{\cO_{\fY_{\et}}} \cO_{\fY_{\et}}/p^n$ is a perfect object of $D(\cO_{\fY_{\et}}/p^n)$ and its formation commutes with base change. Therefore by taking the derived limit we obtain $Rf_{*}\sF$ itself is a perfect object of $D(\cO_{\fY_{\et}})$. Since $\wh{\cO}^{+}_Y$ is  $p$-adically derived complete by \cite[Lemma C.11]{ZavBog}, we obtain $L\nu_{Y}^{*}Rf_{*}\sF$ is $p$-adically derived complete. Moreover by Lemma \ref{lem:Odsupad} and \cite[{Tag 0A0G}]{stacks-project}, $Rp_{*}Lq^{*}\sF$ is also derived $p$-adically complete. Therefore the result follows by taking the derived limit along $n \geq 1$ in Proposition \ref{prop:bcrsmopn} (note that the derived limit of almost zero objects is again almost zero).
\end{proof}

Finally we establish a base change result for our de Rham complex $q^{*}\Omega^{\bullet}_{\fX/\fY}$. We remind the reader that the complex is defined as a complex whose $i$th term is $q^{*}\Omega^{i}_{\fX/\fY}$ with induced differentials. In particular the differentials are non-linear so this complex lives in $D(D, p^{-1}\wh{\cO}^{+}_{Y})$ and not in $D(D, \wh{\cO}^{+}_{D})$ (even though each term is a $\wh{\cO}^{+}_{D}$-module). Therefore the following result serves as an appropriate substitute for Proposition \ref{prop:bcrsmopn}.

\begin{cor} \label{cor:bascaforderhamco}
For each $n \geq 1$, consider the commutative diagram of ringed topoi
\[
 \xymatrix{
 (Y_{ \proet}\times_{\fY_{\et}}\fX_{\et}, p^{-1}\wh{\cO}^+_{Y}/p^n) \ar@{->}[rr]_{{(\tilde{q}_n, \tilde{q}_{n}^{\sharp})}} \ar@{->}[d]^{(\tilde{p}_n, \tilde{p}_{n}^{\sharp})} && (\fX_{\et}, f^{-1}\cO_{\fY_{\et}}/p^n) \ar@{->}[d]^{(\tilde{f}_n, \tilde{f}_{n}^{\sharp})}   \\ 
 (Y_{\proet}, \wh{\cO}^{+}_{Y}/p^n) \ar@{->}[rr]_{{(\nu_{Y,n}, \nu_{Y,n}^{\sharp})}} && (\fY_{\et}, \cO_{\fY_{\et}}/p^n),
 }
 \]
 and $\sF$ a bounded complex each of whose terms are locally free $\cO_{\fX_{\et}}/p^n$-modules (with differentials $f^{-1}\cO_{\fY_{\et}}$-linear). Then the cohomology sheaves of the cone of the canonical base change map
 \begin{equation} \label{eq:newbdkdf}
 L\nu_{Y,n}^{*}R\tilde{f}_{n*}\sF \to R\tilde{p}_{n*}L\tilde{q}_{n}^{*}\sF
 \end{equation}
 are killed by $\fm$. By taking derived limits (as $n$ goes to $\infty$) a version of the statement also holds for a bounded complex each of whose terms are locally free $\cO_{\fX_{\et}}$-modules (cf. Corollary \ref{cor:pcomverorbascha}).
\end{cor}

\begin{proof}
By using stupid truncations, we immediately reduce to the case where $\sF$ is a locally free $\cO_{\fX_{\et}}/p^n$-module (viewed as an $f^{-1}\cO_{\fY_{\et}}/p^n$-module). In this case the LHS of \eqref{eq:newbdkdf} coincides with the LHS of \eqref{eq:basdmosprs}. We now compare $L\tilde{q}_{n}^{*}\sF$ and $Lq_{n}^{*}\sF$ as sheaves. The key point is that the tensor product defining $\wh{\cO}^{+}_{D}$ (cf. Definition \ref{def:comtensostrushea}) is derived because $f$ is flat.
\begin{align*}
    Lq_{n}^{*}\sF &= q_{n}^{-1}\sF \otimes^{\bL}_{q^{-1}\cO_{\fX_{\et}}/p^n} \wh{\cO}^{+}_{D}/p^n \\
    &= q_{n}^{-1}\sF \otimes^{\bL}_{q^{-1}\cO_{\fX_{\et}}/p^n} q^{-1}\cO_{\fX_{\et}}/p^n \otimes_{p^{-1}\nu_{Y}^{-1}\cO_{\fY_{\et}}/p^n}^{\bL} p^{-1}\wh{\cO}^{+}_{Y}/p^n \\
    &= q_{n}^{-1}\sF \otimes_{p^{-1}\nu_{Y}^{-1}\cO_{\fY_{\et}}/p^n}^{\bL} p^{-1}\wh{\cO}^{+}_{Y}/p^n \\
    &= L\tilde{q}_n^{*}\sF.
\end{align*}
Therefore the RHS of \eqref{eq:newbdkdf} coincides with the RHS of \eqref{eq:basdmosprs}. The result now follows by Proposition \ref{prop:bcrsmopn}.
\end{proof}

\section{$q$-crystalline site} \label{$q$-crystalline sitewed}
The main goal of this section is to compare the objects $A\Omega_{\fX/\fY}$ and $R\Gamma_{A_{\Inf}}(\fX/\fY)$ considered thus far with the $q$-crystalline site of Bhatt-Scholze.
Let $q$ be a formal parameter. We introduce a global (big) version of the $q$-crystalline site of \cite[\S  16]{Prisms}.  We advise the reader to familiarise themselves with the contents of \S 16 in loc.cit. before continuing. The logarithmic context is studied in \cite[\S 7]{KoshTerui}. 

As usual we denote by $A = \bZ_p [[q-1]]$ (the $(p,q-1)$-adic completion of $\bZ [q]$) with $\delta$-structure given by $\delta (q) = 0$. Let $[p]_{q} := \tfrac{q^p -1}{q-1}$ be the $q$-analogue of $p$. This gives the $q$-crystalline prism $(A, ([p]_{q}))$. The $q$-analogue of divided power structures is handled by the following: if $x$ is an element of a $[p]_{q}$-torsion free $\delta$-$A$-algebra $D$ such that $\phi(x) \in [p]_{q}D$, then we write
\[
\gamma (x) := \tfrac{\phi(x)}{[p]_{q}} - \delta(x).
\]

In order to define a big version of the $q$-crystalline site, we slightly modify the notion of a $q$-PD pair in \cite{Prisms} (cf. \cite[\S 7.1]{KoshTerui}).

\begin{defn} [$q$-PD pair] \label{defn:qPDpair}
A $q$-PD pair is given by a derived $(p, [p]_{q})$-complete $\delta$-pair $(D,I)$ over $(A, (q-1))$ such that
\begin{enumerate}
\item $(D, ([p]_{q}))$ is a bounded prism over $(A, ([p]_{q}))$ (in particular $D$ is $[p]_q$-torsion free),
    \item $\phi(I) \subset [p]_{q}D$ and $\gamma(I) \subset I$,
    \item $D/(q-1)$ is $p$-torsion free with finite $(p,[p]_{q})$-complete Tor amplitude over $D$ and 
    \item $D/I$ is $p$-adically complete.
\end{enumerate}
\end{defn}

\begin{rem}
As observed in \cite{KoshTerui}, the addition of condition (4) to \cite[Definition 16.2]{Prisms} will be convenient to define the big $q$-crystalline site (without the use of some derived geometry). 
\end{rem}
The following lemma is a special case of \cite[Lemma 7.3]{KoshTerui}.

\begin{lem} \label{lem:liftinqdpet}
Let $(D,I)$ be a $q$-PD pair. If $D/I \to \ol{E}$ is $p$-completely étale, there exists a unique $q$-PD pair $(E,J)$ over $(D,I)$ lifting $D/I \to \ol{E}$.
\end{lem}

\begin{proof}
First note that there is a unique $(p,[p]_q)$-completely étale lift $D \to E$ of $D/I \to \ol{E}$ (because for any $i\in I$, $i^p \in (p,[p]_q)$). The kernel $J$ of the map $E \to \ol{E}$ is given by the $(p,[p]_q)$-adic completion of $IE$. We equip $E$ with the unique $\delta$-structure compatible with the one on $D$ (cf. \cite[Lemma 2.18]{Prisms}). It remains to check conditions (1)-(4) of Definition \ref{defn:qPDpair} for the $\delta$-pair $(E,J)$. Since $E$ is $(p,[p]_q)$-completely flat over $D$, it follows that conditions (1) and (3) are satisfied. Moreover by construction $E/J = \ol{E}$ is $p$-adically complete and so condition (4) is satisfied. It remains to check the first half of condition (2) (the second half is shown in the proof of \cite[Lemma 7.3]{KoshTerui}). Since $D \to E$ is a map of $\delta$-rings over $A$, it follows that $\phi(IE) \subset [p]_qE$.  This implies that $\phi(J) \subset [p]_qE$ because $[p]_qE$ is closed for the $(p,[p]_q)$-adic topology on $E$.
\end{proof}

Fix a $q$-PD pair $(D,I)$ and let $\fZ$ be a $p$-adic formal scheme over $D/I$. We can now define the big $q$-crystalline site $q-\textnormal{CRYS}(\fZ/D)$ (cf. \cite[Definition 7.5]{KoshTerui}) and its associated topos.

\begin{defn}[The $q$-crystalline site and topos]\label{defn:bigqcryver4}
We define the (big) $q$-crystalline site $q-\textnormal{CRYS}(\fZ/D)$ as the opposite of the following category: an object is a $q$-PD pair $(E,J)$ over $(D,I)$ and a $\Spf(D/I)$-morphism of $p$-adic formal schemes $\Spf (E/J) \to \fZ$. Morphisms are the obvious ones. We endow it with the étale topology (by Lemma \ref{lem:liftinqdpet} there exists fiber products of $q$-PD pairs when one of the morphisms is étale in the obvious sense) with covers being singletons. The corresponding topos will be denoted by $(\fZ/D)_{q-\textnormal{CRYS}}$. When $\fZ = \Spf(T)$ is affine we will denote the site and topos by $q-\textnormal{CRYS}(T/D)$ and $(T/D)_{q-\textnormal{CRYS}}$, respectively.
\end{defn}
The first version of Definition \ref{defn:bigqcryver4} (with the Zariski topology in place of the étale topology) appears in \cite[Remark 16.15(2)]{Prisms}. 

\begin{defn}[structure sheaf]
Let $\cO_{q-\textnormal{CRYS}}$ and $\ol{\cO}_{q-\textnormal{CRYS}}$ denote the structure presheaves $(E,J) \mapsto E$ and $(E,J) \mapsto E/J$. These are sheaves by definition of coverings.
\end{defn}

For $\fZ = \Spf (T)$ affine, there is the big $q$-crystalline topos  $(T/D)_{q-\textnormal{crys}}$ with underlying site $q-\textnormal{crys}(T/D)$ defined in \cite[Definition 16.12]{Prisms} with structure sheaf $\cO_{q-\textnormal{crys}}$. As a reality check, we show that in this case the cohomology does not depend on whether one uses the Zariski or étale topology.

\begin{lem} \label{lem:samecoha}
Suppose $\fZ = \Spf(T)$. Then there exists a canonical isomorphism
\begin{equation} \label{eq:morpcohoskd}
R\Gamma((T/D)_{q-\textnormal{crys}}, \cO_{q-\textnormal{crys}}) \xrightarrow{\sim} R\Gamma((T/D)_{q-\textnormal{CRYS}}, \cO_{q-\textnormal{CRYS}})
\end{equation}
\end{lem}

\begin{proof}
The inclusion of sites $q-\textnormal{crys}(T/D) \to q-\textnormal{CRYS}(T/D)$ is continuous and induces a morphism of topoi 
\begin{equation} \label{eq:mortopos}
g \colon (T/D)_{q-\textnormal{CRYS}} \to (T/D)_{q-\textnormal{crys}}
\end{equation}
with $g^{-1}\cO_{q-\textnormal{crys}} = \cO_{q-\textnormal{CRYS}}$. The morphism \eqref{eq:mortopos} induces the morphism \eqref{eq:morpcohoskd}. Choosing a weakly final object as in \cite[Construction 16.13]{Prisms} in $q-\textnormal{crys}(T/D)$, one obtains a \v{C}ech-Alexander complex. This is an explicit representative for the RHS of \eqref{eq:morpcohoskd}. It is easy to check (by freeness) that this weakly final object is also weakly final in the big site $q-\textnormal{CRYS}(T/D)$. Therefore by \v{C}ech theory the same \v{C}ech-Alexander complex is an explicit representative for the LHS of \eqref{eq:morpcohoskd}. Finally the formation of the \v{C}ech-Alexander complex is compatible with the morphism \eqref{eq:morpcohoskd}. 
\end{proof}

Recall the small étale site of $\fZ$ is denoted by $\fZ_{\et}$. We denote the big étale site of $\fZ$ by $\fZ_{\Et}$ and $\fZ_{\Et}^{\textnormal{Aff}}$ the site consisting of affine $p$-adic formal schemes over $\fZ$ whose covers are singleton étale maps. If there is no confusion we will use the same notation for their topoi $\fZ_{\et}$, $\fZ_{\Et}$ and $\fZ_{\Et}^{\textnormal{Aff}}$, respectively.

As usual we have a cocontinuous and continuous functor (by Lemma \ref{lem:liftinqdpet}):
\begin{align*}
q-\textnormal{CRYS}(\fZ/D) &\to \fZ_{\Et}^{\textnormal{Aff}} \\
(E,J, \Spf (E/J) \to \fZ) &\mapsto (\Spf (E/J) \to \fZ),
\end{align*}
which sits in a diagram of functors of sites
\begin{equation} \label{eq:mordfssife}
q-\textnormal{CRYS}(\fZ/D) \to \fZ_{\Et}^{\textnormal{Aff}} \to \fZ_{\Et} \leftarrow \fZ_{\et}.
\end{equation}
The last two functors induce morphisms of sites and applying \cite[{Tag 00XO}]{stacks-project} to the first functor we obtain a diagram of topoi
\begin{equation} \label{eq:compssitsds}
(\fZ/D)_{q-\textnormal{CRYS}} \to \fZ_{\Et}^{\textnormal{Aff}} \xleftarrow{\sim} \fZ_{\Et} \to \fZ_{\et},
\end{equation}
which gives a map of topoi $u^{q}_{\fZ} \colon (\fZ/D)_{q-\textnormal{CRYS}} \to \fZ_{\et}$. In the next section we will construct an explicit underlying site for the oriented product of topoi
\begin{equation} \label{eq:orprotopfr}
(\fY/D)_{q-\textnormal{CRYS}} \ola{\times}_{\fY_{\et}} \fX_{\et},
\end{equation}
which will serve as the $q$-crystalline analogue for the fiber product $Y_{\proet} \times_{\fY_{\et}} \fX_{\et}$ appearing in the pro-étale setting.

\subsection{Product of topoi revisited} \label{subsec:laxproductop}

We now assume that we have a $\Spf(D/I)$-morphism $f \colon \fX \to \fY$ of $p$-adic formal schemes. The oriented product \eqref{eq:orprotopfr} certainly exists by the universal property \cite[Exposé XI, Théorème 1.4]{TravdeGabb} and the fact that any morphism of topoi essentially comes from a morphism of sites (cf. \cite[{Tag 03A2}]{stacks-project}). The difficulty in the current situation is that the morphism $(\fY/D)_{q-\textnormal{CRYS}} \to \fY_{\et}$ is induced by some composition \eqref{eq:compssitsds} (in particular it does not come directly from a morphism of the sites) and so the same construction for an underlying site (for the oriented product) given in \cite[Exposé XI]{TravdeGabb} does not work. However, it turns out that \eqref{eq:compssitsds} has enough properties to recover an explicit site for \eqref{eq:orprotopfr}. 

\begin{rem}
Regarding the construction of an underlying site, the fact that the left edge morphism $(\fY/D)_{q-\textnormal{CRYS}} \to \fY_{\et}$ does not come from a morphism of the sites forces us to work with an oriented product of topoi rather than a fiber product. Moreover the right edge morphism $\fX_{\et} \to \fY_{\et}$ does come from a morphism of sites and so the choice of orientation \eqref{eq:orprotopfr} is a natural candidate over $(\fY/D)_{q-\textnormal{CRYS}} \ora{\times}_{\fY_{\et}} \fX_{\et}$. This is in contrast to our work on the pro-étale side, where we had the freedom to work with the fiber product (in the case $\fX$ and $\fY$ are of type (S)(b)) $Y_{\proet} \times_{\fY_{\et}} \fX_{\et}$ or either of the two oriented products.
\end{rem}

For the time being we work in general with the set up being as follows:
$C_1$, $C_2$, $E$, $E'$ and $D$ are five sites such that
\begin{enumerate}
    \item $E$ is subcanonical, and $C_2$ and $D$ admit finite projective limits,
    \item there is a diagram of functors
    \begin{equation} \label{eq:nediawithact}
    C_1 \xrightarrow{g_1} E' \xrightarrow{a} E \xleftarrow{g} D 
    \end{equation}
    where $g_1$ is both continuous and cocontinuous, and the last two functors induce morphisms of sites $E \to E'$ and $E \to D$,
    \item $E \to E'$ induces an equivalence of topoi, and
    \item there is a functor $f_2 \colon D \to C_2$ commuting with finite projective limits, which induces a morphism of sites $C_2 \to D$.
\end{enumerate}
Denote $f_1 \colon C_1 \to E$ the composition $a \circ g_1$. Let $C$ be the category of 5-tuples 
\[
(U,V,W, f_1(U) \to g(W), V \to f_2(W))
\]
where $U, V$ and $W$ are objects of $C_1$, $C_2$ and $D$, respectively, and $f_1(U) \to g(W)$ and $V \to f_2(W)$ are morphisms in $E$ and $C_2$, respectively. We will denote such an object by $(U \to W \leftarrow V)$. Let $(U \to W \leftarrow V)$ and $(U' \to W' \leftarrow V')$ be two objects of $D$. A morphism from $(U \to W \leftarrow V)$ into $(U' \to W' \leftarrow V')$ is the datum of three morphisms $U \to U'$, $V \to V'$ and $W \to W'$ in $C_1$, $C_2$ and $D$, respectively such that the diagrams
\[
 \xymatrix{
 f_1(U) \ar@{->}[r] \ar@{->}[d] & g(W) \ar@{->}[d]   \\ 
 f_1(U') \ar@{->}[r] & g(W'),
 }
 \hspace{10mm}
 \xymatrix{
 f_2(W) \ar@{->}[d] & V \ar@{->}[d] \ar@{->}[l]   \\ 
 f_2(W')  & V' \ar@{->}[l]
 }
\]
are commutative. We equip $C$ with the topology generated by coverings
\[
\left\{ (U_i \to W_i \leftarrow V_i) \to (U \to W \leftarrow V) \right\}_{i \in I}
\]
of the following three types:
\begin{enumerate} [(a)]
    \item $V_i=V, W_i = W \text{ }$ for all $i$ and $(U_i \to U)_{i \in I}$ is a covering.
    \item $U_i=U, W_i = W \text{ }$ for all $i$ and $(V_i \to V)_{i \in I}$ is a covering.
    \item $I$ is a singleton: $(U \to W \leftarrow V) \to (U' \to W' \leftarrow V')$, with $U = U'$ and $V \to V'$ is induced by base change from $W \to W'$.
\end{enumerate}
Denote the resulting topos by $\wt{C}$. 

Note that condition (4) gives a morphism of topoi $\wt{C}_2 \to \wt{D}$ and treating $g_1$ as a cocontinuous functor, conditions (2)-(3) naturally give a morphism of topoi $\wt{C}_1 \to \wt{D}$ (the morphism of topoi $E' \to E$ is given by the adjoint functors $a^{s}$ and $a_{s}$). One can therefore consider the oriented product of topoi $\wt{C}_1 \ola{\times}_{\wt{D}} \wt{C}_2$. We now come to the main result of this section, which gives an explicit description of an underlying site for the said oriented product:

\begin{prop} \label{prop:explsciorep}
There is a canonical equivalence of topoi $\wt{C}_1 \ola{\times}_{\wt{D}} \wt{C}_2 \simeq \wt{C}$.
\end{prop}

\begin{proof}
We first rewrite the objects of $C$ which will make the comparison with the construction given in \cite[Exposé XI, 1.1]{TravdeGabb} easier. Since $E$ is subcanconical, the datum of a morphism $f_1(U) \to g(W)$ in $E$ is equivalent to giving a morphism in the topos $\wt{E}$ between the associated representable sheaves 
\begin{equation} \label{eq:repspsrsf}
h_{f_1(U)} \to h_{g(W)}.
\end{equation}
By \cite[{Tag 04D3}]{stacks-project}, \eqref{eq:repspsrsf} takes the shape $a_{s}g_{1,s}h_{U}^{a} \to g_{s}h_{W}^{a}$ (where $h_{U}^{a}$ and $h_{W}^{a}$ are the associated sheaves of $h_{U}$ and $h_{W}$, respectively). Now $a_s$ is a left adjoint to $a^{s}$ and $g_{1,s}$ is left adjoint to $g_{1}^{s}$ (this is because both $a$ and $g_1$ are continuous functors of sites). Thus \eqref{eq:repspsrsf} is equivalent to giving a morphism
\begin{equation} \label{eq:pullbackilsu}
h_{U}^{a} \to g_{1}^{s}a^{s}g_{s}h_{W}^{a}.
\end{equation}
Finally note that $g_{1}^{s}a^{s}g_{s}$ is precisely the pullback of the morphism of topoi $\wt{C}_{1} \to \wt{D}$. 

Having rewritten the objects of $C$, it remains to compare our construction of $C$ with the construction in \cite[Exposé XI, 1.1]{TravdeGabb}. By \cite[{Tag 03A2}]{stacks-project} we can \emph{enlarge} the site $C_1$ to a site $C'$ such that the morphism of topoi $\wt{C}_1 \to \wt{D}$ is equivalent to a morphism of topoi $\wt{C'} \to \wt{D}$ induced from a morphism of sites $C' \to D$. In fact we choose $C' \subset \wt{C}_{1}$ as in the proof of loc.cit. together with the inclusion of the final sheaf in the first inductive step (in particular $C'$ admits finite projective limits). Moreover $C'$ contains the objects $h_{U}^{a}$. Let $u \colon D \to C'$ be the corresponding functor (it is given by $W \mapsto g_{1}^{s}a^{s}g_{s}h_{W}^{a}$ by the proof of \cite[{Tag 03A2}]{stacks-project}). Note that $u$ commutes with finite projective limits. Denote by $\ul{C}$ the underlying site of $\wt{C}_1 \ola{\times}_{\wt{D}} \wt{C}_2$ associated to the morphism of sites $C' \to D$ and $C_2 \to D$, as given by \cite[Exposé XI, 1.1]{TravdeGabb}. There is a functor of sites
\begin{align*}
    C &\to \ul{C} \\
    (U,V,W, f_1(U) \to g(W), V \to f_2(W)) &\mapsto (h^{a}_{U},V,W, h^{a}_{U} \to u(W), V \to f_2(W)),
\end{align*}
which is special cocontinuous (because $C_1 \to C'$ is special cocontinuous). The associated morphism of topoi induced by the cocontinuity of $C \to \ul{C}$, $\wt{C} \to \wt{\ul{C}}$ is then an equivalence by \cite[{Tag 03A0}]{stacks-project}.
\end{proof}

We will denote the site $C$ by $C_1 \ola{\times}_{D} C_2$. As usual there are natural projections
\[
p_1 \colon \wt{C}_1 \ola{\times}_{\wt{D}} \wt{C}_2 \rightarrow \wt{C}_1 \text{ and } p_2 \colon \wt{C}_1 \ola{\times}_{\wt{D}} \wt{C}_2 \rightarrow \wt{C}_2.
\]
Since $C_1$ may lack a final object, the description of $p_2$ given in \cite[Exposé XI, \S 1.3]{TravdeGabb} does not work. Nevertheless one can describe it as follows:

\begin{lem} \label{lem:desfuncsoif}
Let $e_{C_2}$ and $e_{D}$ be the final objects of $C_2$ and $D$. Then $p_1$ and $p_2$ are induced by the functors of sites
\begin{align*}
     C_1 &\to C_1 \ola{\times}_{D} C_2 &  C_1 \ola{\times}_{D} C_2 &\to C_2 \\
    U &\mapsto (U \to e_{D} \leftarrow e_{C_2}) \text{ \hspace{14mm} and}& (U \to W \leftarrow V) &\mapsto V,
\end{align*}
respectively. The first functor induces a morphism of sites and the second functor is cocontinuous.
\end{lem}

\begin{proof}
We use the notation in the proof of Proposition \ref{prop:explsciorep}. Let $e_{C'}$ be the final object of $C'$. Then the projections $\wt{\ul{C}} \to \wt{C}'$ and $\wt{\ul{C}} \to \wt{C}_2$ are induced by the functors of sites
\begin{align*}
     C' &\to \ul{C}  &  C_2  &\to \ul{C} \\
    U' &\mapsto (U' \to e_{D} \leftarrow e_{C_2}) \text{ \hspace{14mm} and}& V &\mapsto (e_{C'} \to e_D \leftarrow V),
\end{align*}
respectively (cf.  \cite[Exposé XI, \S 1.3]{TravdeGabb}). The second functor (and also the first) has a cocontinuous left adjoint given by 
\begin{align*}
\ul{C} &\to C_2 \\
(U' \to W \leftarrow V) &\mapsto V
\end{align*}
which by \cite[{Tag 00XY}]{stacks-project} induces the same morphism of topoi as its continuous right adjoint. The claims for $p_1$ and $p_2$ now follow from the commutativity of the diagram of sites 
\[
 \xymatrix{
 \ul{C}  & C' \ar@{->}[l]   \\ 
 C_1 \ola{\times}_{D} C_2 \ar@{->}[u] & C_1 \ar@{->}[u] \ar@{->}[l]
 }
 \hspace{10mm}
  \xymatrix{
 \ul{C} \ar@{->}[r]  & C_2 \ar@{=}[d]   \\ 
 C_1 \ola{\times}_{D} C_2 \ar@{->}[u] \ar@{->}[r] & C_2,
 }
\]
respectively.
\end{proof}

Let us now return to the situation of interest (compare \eqref{eq:mordfssife} with \eqref{eq:nediawithact}): the five sites under consideration are $q-\textnormal{CRYS}(\fY/D)$, $\fX_{\et}$, $\fY_{\Et}$, $\fY_{\Et}^{\textnormal{Aff}}$ and $\fY_{\et}$, respectively. We have verified in the previous section that these five sites satisfy the conditions (1)-(4). Thus by Proposition \ref{prop:explsciorep} we have constructed an explicit underlying site for the oriented product \eqref{eq:orprotopfr}. By Lemma \ref{lem:desfuncsoif} we have a simple description for the natural projections
\[
p_1 \colon (\fY/D)_{q-\textnormal{CRYS}} \ola{\times}_{\fY_{\et}} \fX_{\et} \rightarrow (\fY/D)_{q-\textnormal{CRYS}} \text{ and } p_2 \colon (\fY/D)_{q-\textnormal{CRYS}} \ola{\times}_{\fY_{\et}} \fX_{\et} \rightarrow \fX_{\et},
\]
and there is a 2-map $fp_2 \to u^{q}_{\fY}p_1$.

\begin{defn} \label{defn:strucshrsfoprof}
We define the structure sheaf on $q-\textnormal{CRYS}(\fY/D) \ola{\times}_{\fY_{\et}} \fX_{\et}$ by the $p$-adically completed tensor product 
\[
\ol{\cO} := p_{1}^{-1}\ol{\cO}_{q-\textnormal{CRYS}} \wh{\otimes}_{p_{1}^{-1}(u^{q}_{\fY})^{-1}\cO_{\fY, \et}} p_{2}^{-1}\cO_{\fX, \et}.
\]
\end{defn}

\begin{rem}
A word of explanation for the definition of $\ol{\cO}$. The morphism $p_{1}^{-1}(u^{q}_{\fY})^{-1}\cO_{\fY, \et} \to p_{1}^{-1}\ol{\cO}_{q-\textnormal{CRYS}}$ is induced by the map $(u^{q}_{\fY})^{-1}\cO_{\fY, \et} \to \ol{\cO}_{q-\textnormal{CRYS}}$. The morphism $p_{1}^{-1}(u^{q}_{\fY})^{-1}\cO_{\fY, \et} \to p_{2}^{-1}\cO_{\fX, \et}$ is induced by the composition
\[
p_{1}^{-1}(u^{q}_{\fY})^{-1}\cO_{\fY, \et} \to p_{2}^{-1}f^{-1}\cO_{\fY, \et} \to p_2^{-1}\cO_{\fX, \et}.
\]
\end{rem}
By construction the projections $p_1$ and $p_2$ extend to morphisms of ringed topoi:
\[
(p_1, p_1^{\sharp}) \colon ((\fY/D)_{q-\textnormal{CRYS}} \ola{\times}_{\fY_{\et}} \fX_{\et}, \ol{\cO}) \rightarrow ((\fY/D)_{q-\textnormal{CRYS}}, \ol{\cO}_{q-\textnormal{CRYS}})
\]
and 
\[(p_2, p_2^{\sharp}) \colon ((\fY/D)_{q-\textnormal{CRYS}} \ola{\times}_{\fY_{\et}} \fX_{\et}, \ol{\cO}) \rightarrow (\fX_{\et}, \cO_{\fX, \et}).
\]

\begin{lem} \label{lem:basaqcrsdsi}
The family of objects of the form 
\begin{equation} \label{eq:crysgoose}
((E,J) \to \Spf(S) \leftarrow \Spf(R))
\end{equation}
is generating for $q-\textnormal{CRYS}(\fY/D) \ola{\times}_{\fY_{\et}} \fX_{\et}$. Moreover if $f \colon \fX \to \fY$ is smooth, then in addition one can take $\Spf (R) \to \Spf(S)$ satisfying condition (2) of Definition \ref{def:goodtriples}.
\end{lem}

\begin{proof}
Choose any object $T = ((E,J) \to W \leftarrow V)$ in $q-\textnormal{CRYS}(\fY/D) \ola{\times}_{\fY_{\et}} \fX_{\et}$. First we reduce to the case $W$ is affine. Take an affine Zariski covering of $W$ which by pullback (and after possibly refining) gives an affine Zariski covering of $\Spf(E/J)$. Since $\Spf(E/J)$ is quasi-compact, we can take a finite subfamily of the latter, which still covers $\Spf(E/J)$. By taking a disjoint union of the elements of this finite subfamily we obtain an étale covering of $\Spf(E/J)$, which is a singleton. Applying Lemma \ref{lem:liftinqdpet} to this singleton cover gives a covering of $(E,J)$, which factors through an affine $\Spf(S)$ étale over $W$ by construction. Finally by using coverings of type (c), we can assume $W$ is affine. Coverings of type (b) give the desired generating family. In the case $f$ is smooth one uses (a $\Spf(D/I)$-version of) Lemma \ref{lem:facetalmor} and coverings of type (b) to further refine.  
\end{proof}

As a reality check, we now turn to calculating the sections of $\ol{\cO}$.

\begin{lem}
For $Z$ of the form \eqref{eq:crysgoose}, $\ol{\cO}(Z) = E/J \wh{\otimes}_{S} R$, where the completion is $p$-adic.
\end{lem}

\begin{proof}
For each $n\geq 1$, consider first the uncompleted version modulo $p^n$: 
\[
\sF_n := (p_{1}^{-1}\ol{\cO}_{q-\textnormal{CRYS}} \otimes_{p_{1}^{-1}(u^{q}_{\fY})^{-1}\cO_{\fY, \et}} p_{2}^{-1}\cO_{\fX, \et})/p^n.\]
Then $\sF_n$ is the sheafification of the presheaf $Z \mapsto (E/J \otimes_{S} R)/p^n$
which is a sheaf by \cite[Lemme 1.2]{TravdeGabb}. The result now follows by taking the limit.
\end{proof}

\subsection{The object $q\Omega_{\fX/\fY}$}

In this section we construct the relative version of the $q$-crystalline complex $q\Omega_{\fX/\fY}$ living in our topos \eqref{eq:orprotopfr} and record a few of its properties.

\begin{lem}\label{lem:morsitqycrsi}
There is a functor of topoi
\begin{equation} \label{eq:relver4.4Pr}
\nu_{f}^{q} \colon (\fX/D)_{q-\textnormal{CRYS}} \to (\fY/D)_{q-\textnormal{CRYS}} \ola{\times}_{\fY_{\et}} \fX_{\et}
\end{equation}
such that for any sheaf $\sF$ in $(\fX/D)_{q-\textnormal{CRYS}}$ and any $Z$ of the form \eqref{eq:crysgoose}
\[
(\nu_{f*}^{q}\sF)(Z) = \Gamma((E/J \wh{\otimes}_{S} R)/E)_{q-\textnormal{CRYS}}, \sF)
\]
\end{lem}

\begin{proof}
In addition to the topos $\wt{P}_{\et}:= (\fY/D)_{q-\textnormal{CRYS}} \ola{\times}_{\fY_{\et}} \fX_{\et}$, the topoi 
\[
\wt{P}_{\Et}:= (\fY/D)_{q-\textnormal{CRYS}} \ola{\times}_{\fY_{\Et}} \fX_{\Et} \text{ and } \wt{P}_{\Et}^{\textnormal{Aff}} := (\fY/D)_{q-\textnormal{CRYS}} \ola{\times}_{\fY_{\Et}^{\textnormal{Aff}}} \fX_{\Et}^{\textnormal{Aff}}
\]
will also be useful. All three  have explicit sites described by the formalism of \S\ref{subsec:laxproductop}, $P_{\et}$, $P_{\Et}$ and $P_{\Et}^{\textnormal{Aff}}$, respectively. Consider the diagram of sites
\begin{equation} \label{eq:diaforrel4.4}
q-\textnormal{CRYS}(\fX/D) \xrightarrow{\nu} P_{\Et}^{\textnormal{Aff}} \xrightarrow{i} P_{\Et} \xleftarrow{j} P_{\et},
\end{equation}
where the first functor is given by $(E,J) \mapsto ((E,J) \to \Spf(E/J) \leftarrow \Spf(E/J))$ and the remaining two functors are induced by the natural inclusions
\[
\fZ_{\Et}^{\textnormal{Aff}} \to \fZ_{\Et} \leftarrow \fZ_{\et} 
\]
for $\fZ \in \{ \fX, \fY \}$. The functor $\nu$ is both continuous and cocontinuous, and $i$ and $j$ both induce morphisms of sites (moreover $i$ induces an equivalence of topoi). Treating $\nu$ as a cocontinuous functor, \eqref{eq:diaforrel4.4} induces the morphism of topoi \eqref{eq:relver4.4Pr} (as usual, one switches the adjoints for the equivalence induced by $i$). It remains to describe the pushforward of this morphism given by $j^{s}\circ i_{s}\circ{}_{s}\nu$. We compute
\begin{align*}
    (j^{s}\circ i_{s}\circ{}_{s}\nu(\sF))(Z) &\overset{(i)}{=} (i_{s}\circ{}_{s}\nu(\sF))(j(Z)) \\
    &\overset{(ii)}{=} ({}_{s}\nu(\sF))(j(Z)) \\
    &\overset{(iii)}{=} \varprojlim_{\nu((E',J')) \to j(Z)} \sF ((E',J'))
\end{align*}
where (i) is by definition of $j^s$, (ii) because $i$ induces an equivalence of topoi ($j(Z)$ is viewed as an object in $P_{\Et}^{\textnormal{Aff}}$ hereafter), and (iii) is by definition of ${}_{s}\nu$. To give a morphism $\nu((E',J')) \to j(Z)$ is equivalent to giving a morphism of $q$-PD pairs $(E',J') \to (E,J)$ in $q-\textnormal{CRYS}(\fY/D)$ together with a compatible (over $\Spf(S)$) morphism $\Spf(E'/J') \to \Spf (R)$ in $\fX_{\Et}^{\textnormal{Aff}}$. Thus the limit calculates precisely the global sections functor of the topos $((E/J \wh{\otimes}_{S} R)/E)_{q-\textnormal{CRYS}}$ for $\sF$. This finishes the proof.
\end{proof}
Lemma \ref{lem:morsitqycrsi} allows us to make the following definition:

\begin{defn} \label{defn:qcryscomple}
We define the relative $q$-crystalline complex
$q\Omega_{\fX/\fY} := R\nu_{f*}^{q}\cO_{q-\textnormal{CRYS}}$. This is a $D$-algebra in $(\fY/D)_{q-\textnormal{CRYS}} \ola{\times}_{\fY_{\et}} \fX_{\et}$ equipped with a $\phi_{D}$-semilinear endomorphism. In addition we set $R\Gamma_{q-\textnormal{CRYS}}(\fX/\fY) := Rp_{1*}q\Omega_{\fX/\fY}$.
\end{defn}

We now compare $q\Omega_{\fX/\fY}$ with \emph{relative} prismatic cohomology (cf. \cite[Theorem 16.17]{Prisms} and \cite[Theorem 7.13]{KoshTerui}).

\begin{thm} \label{thm:qOmegavsprism}
Let $Z$ be of the form \eqref{eq:crysgoose} such that $\Spf(R) \to \Spf(S)$ is a $p$-completely smooth $\Spf(D/I)$-morphism.
  Letting $(R \wh{\otimes}_{S} E/J)^{(1)}$ be the $p$-completed base change along the map $E/J \to E/([p]_q)$ induced by $\phi_{E}$. Writing $\Prism_{(R \wh{\otimes}_S E/J)^{(1)}/E}$ for the prismatic cohomology of $(R \wh{\otimes}_S E/J)^{(1)}$ relative to the bounded prism $(E, [p]_{q})$, there is a canonical isomorphism $\Prism_{(R \wh{\otimes}_S E/J)^{(1)}/E} \cong R\Gamma(Z, q\Omega_{\fX/\fY})$.
\end{thm}

\begin{proof}
The hard work is already done in \cite[\S 16]{Prisms}. Indeed by Lemma \ref{lem:morsitqycrsi}
\begin{equation} \label{eq:forfoderqOMe}
R\Gamma(Z, q\Omega_{\fX/\fY}) = R\Gamma((R \wh{\otimes}_{S} E/J)/E)_{q-\textnormal{CRYS}}, \cO_{q-\textnormal{CRYS}}).
\end{equation}
The result now follows by Theorem 16.17 in loc.cit. and the comparison made in Lemma \ref{lem:samecoha}.
\end{proof}

Finally we close this section by recording perfectness for $R\Gamma_{q-\textnormal{CRYS}}(\fX/\fY)$. In the $q$-crystalline setting one is able to obtain perfectness on the nose (without almost ambiguity). This is in contrast for the case of $R\Gamma_{A_{\Inf}}(\fX/\fY)$ (cf. Proposition \ref{prop:perfecAinfcohm}) and is an ubiquitous defect of the pro-étale topos associated to the generic fiber.

\begin{prop} \label{prop:perfectnessqcrco}
Assume that we have a proper smooth $\Spf(D/I)$-morphism $f \colon \fX \to \fY$ of $p$-adic formal schemes. 
Then for each $(E,J)$ in $(\fY/D)_{q-\textnormal{CRYS}}$
\begin{enumerate}[(i)]
\item $R\Gamma((E,J), R\Gamma_{q-\textnormal{CRYS}}(\fX/\fY))$ is perfect in $D(E)$.
\item $R\Gamma_{q-\textnormal{CRYS}}(\fX/\fY)\lvert_{(E,J)}$ is strictly perfect in $D((\fY/D)_{q-\textnormal{CRYS}}/(E,J), \cO_{q-\textnormal{CRYS}}\lvert_{(E,J)})$.
\item $R\Gamma_{q-\textnormal{CRYS}}(\fX/\fY)$ is perfect in $D((\fY/D)_{q-\textnormal{CRYS}}, \cO_{q-\textnormal{CRYS}})$.
\end{enumerate}
\end{prop}

\begin{proof}
We first show (i). In this proof we will abuse notation and denote $\cO_{q-\textnormal{CRYS}}$ for the structure sheaves on both $(\fY/D)_{q-\textnormal{CRYS}}$ and $(\fX/D)_{q-\textnormal{CRYS}}$. First note that $f$ induces a cocontinuous functor of sites
\begin{align*}
f_{q-\textnormal{CRYS}} \colon q-\textnormal{CRYS}(\fX/D) &\to q-\textnormal{CRYS}(\fY/D) \\
(E',J') &\mapsto (E',J')
\end{align*}
and although $R\Gamma_{q-\textnormal{CRYS}}(\fX/\fY)$ was originally defined via the intermediate oriented product, we can also write it as $Rf_{q-\textnormal{CRYS}*}\cO_{q-\textnormal{CRYS}}$. Unwinding the definitions we get
\begin{equation} \label{eq:pushforisjusqcr}
R\Gamma((E,J), R\Gamma_{q-\textnormal{CRYS}}(\fX/\fY)) = R\Gamma((\Spf(E/J) \times_{\fY} \fX)/E)_{q-\textnormal{CRYS}}, \cO_{q-\textnormal{CRYS}}).
\end{equation}
By a globalization\footnote{For the corresponding globalization in the log-smooth setting we refer to \cite[Theorem 7.13]{KoshTerui}.} of \cite[Theorem 16.17]{Prisms} (obtained by sheafifying the result in loc.cit.) we get
\begin{equation} \label{eq:globalizationdsfs}
R\Gamma((\Spf(E/J) \times_{\fY} \fX)/E)_{q-\textnormal{CRYS}}, \cO_{q-\textnormal{CRYS}}) = R\Gamma((\Spf(E/[p]_q) \times_{\fY} \fX)_{\et}, \Prism_{(\Spf(E/[p]_q) \wh{\otimes}_{\fY} \fX)/E})
\end{equation}
where $\Spf(E/[p]_q) \times_{\fY} \fX$ is obtained by the $p$-completed base change of $\Spf(E/J) \times_{\fY} \fX$ along $E/J \to E/[p]_q$ induced by $\phi_E$. Since $f \colon \fX \to \fY$ was assumed to be proper, the RHS of \eqref{eq:globalizationdsfs} is a perfect complex in $D(E)$ by the Hodge-Tate comparison for prismatic cohomology (cf. Theorem 1.8(2) in loc.cit.). Thus \eqref{eq:pushforisjusqcr} implies $R\Gamma((E,J), R\Gamma_{q-\textnormal{CRYS}}(\fX/\fY))$ is a perfect complex in $D(E)$ for each $(E,J)$ in $(\fY/D)_{q-\textnormal{CRYS}}$. 

We now prove part (ii). This is a consequence of part (i) together with the \emph{quasi-coherence} behaviour of prismatic/$q$-crystalline cohomology (cf. Theorem 1.8(5) in loc.cit.). Indeed these two points imply that the canonical morphism
\[
\ul{R\Gamma((E,J), R\Gamma_{q-\textnormal{CRYS}}(\fX/\fY))} \otimes_{E}^{\bL} \cO_{q-\textnormal{CRYS}}\lvert_{(E,J)} \to  R\Gamma_{q-\textnormal{CRYS}}(\fX/\fY))\lvert_{(E,J)}
\]
is a quasi-isomorphism. Finally part (iii) is an immediate consequence of part (ii).
\end{proof}

\subsection{The Hodge-Tate comparison}
In this short section we record the Hodge-Tate comparison morphism in the setting of the $q$-crystalline site (cf. \cite[Theorem 4.11]{Prisms} and \cite[Theorem 5.3]{KoshTerui}) in the affine case. We then glue these morphisms to obtain a global statement. This will be a formal consequence of the Hodge-Tate comparison morphism constructed in \cite{Prisms} together with Theorem \ref{thm:qOmegavsprism}. We assume $f\colon \fX \to \fY$ is a smooth $\Spf(D/I)$-morphism of $p$-adic formal schemes.

\begin{cons}[The Hodge-Tate comparison map: affine case]
\emph{Let $Z$ be of the form \eqref{eq:crysgoose}. Denote by $\Omega_{R/S}^{\bullet} \wh{\otimes}_{S} E/[p]_{q}$ the complex whose $i$-th term is $\Omega_{R/S}^{i} \wh{\otimes}_{S} E/[p]_{q}$ (the completion is $p$-adic) with induced differentials coming from the de Rham complex. Now by the comparison with prismatic cohomology and $q$-crystalline cohomology (cf. Theorem \ref{thm:qOmegavsprism})
\begin{equation} \label{eq:1stisomfth}
\Prism_{(R \wh{\otimes}_S E/J)^{(1)}/E} \otimes_{E}^{\bL} E/[p]_q \cong R\Gamma(Z, q\Omega_{\fX/\fY}) \otimes_{E}^{\bL} E/[p]_{q}.
\end{equation}
Moreover by \cite[Theorem 4.11]{Prisms}, we have an isomorphism of differential graded $E/[p]_{q}$-algebras
\begin{equation} \label{eq:2ndisorth}
\Omega^{\bullet}_{(R \wh{\otimes}_{S} E/J)^{(1)}/(E/[p]_{q})} \xrightarrow{\sim} H^{\bullet}(\Prism_{(R \wh{\otimes}_S E/J)^{(1)}/E} \otimes_{E}^{\bL} E/[p]_q)\{\bullet\},
\end{equation}
where the differentials on the RHS are induced by the Bockstein construction. Moreover by \cite[Chapter I, Proposition 5.1.5]{FujKat}
\begin{equation} \label{eq:3rdlastiso}
\Omega^{\bullet}_{(R \wh{\otimes}_{S} E/J)^{(1)}/(E/[p]_{q})} \cong \Omega_{R/S}^{\bullet} \wh{\otimes}_{S} E/[p]_{q}.
\end{equation}
Finally \eqref{eq:1stisomfth}-\eqref{eq:3rdlastiso} together with functoriality of the Bockstein construction gives the following:}
\end{cons}

\begin{cor}[Hodge-Tate comparison: affine case] \label{cor:HTforqcrys}
For $Z$ of the form \eqref{eq:crysgoose} the map
\begin{equation} \label{eq:afisHotda}
\eta_{\fX/\fY}^{q}(Z) \colon  \Omega_{R/S}^{\bullet} \wh{\otimes}_{S} E/[p]_{q} \xrightarrow{\sim} H^{\bullet}(R\Gamma(Z, q\Omega_{\fX/\fY}) \otimes_{E}^{\bL} E/[p]_{q})\{\bullet \}
\end{equation}
constructed above is an isomorphism of differential graded $E/[p]_q$-algebras.
\end{cor}
We end this section by recording the globalization of Corollary \ref{cor:HTforqcrys}. In order to formulate the statement we will need appropriately twisted structure sheaves to accommodate the twist appearing in Theorem \ref{thm:qOmegavsprism}.  In particular we consider the variant $\ol{\cO}_{q-\textnormal{CRYS}}^{(1)}$ defined by $(E,J) \mapsto E/[p]_{q}$. There is a morphism of sheaves of rings $\ol{\cO}_{q-\textnormal{CRYS}} \to \ol{\cO}_{q-\textnormal{CRYS}}^{(1)}$ induced by Frobenius lifts. Similarly one defines $\ol{\cO}^{(1)}$ as in Definition \ref{defn:strucshrsfoprof} with $\ol{\cO}_{q-\textnormal{CRYS}}^{(1)}$ in place of $\ol{\cO}_{q-\textnormal{CRYS}}$ and we get an induced morphism of sheaves of rings $\ol{\cO} \to \ol{\cO}^{(1)}$. The morphism of sheaves of rings $p_2^{\sharp}$ giving a morphism of ringed topoi
\begin{equation} \label{eq:tsifmsorinds}
(p_2, p_2^{\sharp(1)}) \colon ((\fY/D)_{q-\textnormal{CRYS}} \ola{\times}_{\fY_{\et}} \fX_{\et}, \ol{\cO}^{(1)}) \to (\fX_{\et}, \cO_{\fX_{\et}}).
\end{equation}
Let $p_2^{*}\Omega_{\fX/\fY}^{\bullet}$ denote the complex whose $i$th term is $p_2^{*}\Omega_{\fX/\fY}^{i}$ and induced differentials (here $p_2$ is considered as the twisted morphism of ringed topoi \eqref{eq:tsifmsorinds}). For the next result note that $p_1^{-1}\ol{\cO}^{(1)}_{q-\textnormal{CRYS}}$ is the sheaf\footnote{A priori $Z \mapsto E/[p]_q$ is a presheaf, but one can easily verify the sheaf condition for each of the coverings of type (a)-(c). For coverings of type (a) this follows because $\ol{\cO}^{(1)}_{q-\textnormal{CRYS}}$ is sheaf. Coverings of type (b)-(c) leave sections unchanged.} $Z \mapsto E/[p]_q$.

\begin{cor}[Hodge-Tate comparison: global case] \label{cor:HTforqcrysde}
The maps $\eta_{\fX/\fY}^{q}(Z)$ in Corollary \ref{cor:HTforqcrys} glue to give an isomorphism
\[
\eta^{q}_{\fX/\fY} \colon p_2^{*}\Omega^{\bullet}_{\fX/\fY} \xrightarrow{\sim} H^{\bullet}(q\Omega_{\fX/\fY} \otimes^{\bL}_{D} D/[p]_q)\{\bullet\}
\]
of differential graded $p_1^{-1}\ol{\cO}^{(1)}_{q-\textnormal{CRYS}}$-algebras.
\end{cor}

\begin{proof}
Unravelling the definitions $p_2^{*}\Omega^{i}_{\fX/\fY}$ is the sheafification of the presheaf
\[
Z \mapsto  \Omega_{R/S}^{i} \wh{\otimes}_{S} E/[p]_{q}.
\]
In addition, by taking an injective resolution of $q\Omega_{\fX/\fY}$, one shows that $H^{i}(q\Omega_{\fX/\fY} \otimes^{\bL}_{D} D/[p]_q)\{i\}$ is the sheafification of the presheaf
\[
Z \mapsto H^{i}(R\Gamma(Z, q\Omega_{\fX/\fY}) \otimes_{E}^{\bL} E/[p]_{q})\{i\}.
\]
The result now follows from Corollary \ref{cor:HTforqcrys} (for the glueing, one uses the functoriality of \eqref{eq:afisHotda}).
\end{proof}

\subsection{Base change for $(\fY/D)_{q-\textnormal{CRYS}} \protect\ola{\times}_{\fY_{\et}}\fX_{\et}$}

In this section we record the relevant base change for later use, namely when comparing relative $A_{\Inf}$-cohomology with $q$-crystalline cohomology. We continue to use the twisted sheaves developed in the previous section. For what follows we consider the commutative diagram of (twisted) ringed topoi
\begin{equation} \label{eq:commtdasirifs}
 \xymatrix{
 ((\fY/D)_{q-\textnormal{CRYS}} \ola{\times}_{\fY_{\et}} \fX_{\et}, \ol{\cO}^{(1)}) \ar@{->}[rrrr]_{{(p_2, p_{2}^{\sharp(1)})}} \ar@{->}[d]^{(p_1, p_{1}^{\sharp(1)})} &&&& (\fX_{\et}, \cO_{\fX_{\et}}) \ar@{->}[d]^{(f, f^{\sharp})}   \\ 
 ((\fY/D)_{q-\textnormal{CRYS}}, \ol{\cO}_{q-\textnormal{CRYS}}^{(1)}) \ar@{->}[rrrr]_{{(u_{\fY}^{q}, u_{\fY}^{q\sharp(1)})}} &&&& (\fY_{\et}, \cO_{\fY_{\et}}).
 }
 \end{equation}
As usual we handle the affine case first.

\begin{lem} \label{lem:afinbaschcascru}
Suppose $\fX$ and $\fY$ are in addition affine.  
Then for each $n \geq 1$ the canonical base change map induced by \eqref{eq:commtdasirifs}
 \begin{equation} \label{eq:affcasds}
 u_{\fY}^{q*}f_{*}(\cO_{\fX_{\et}}/p^n) \to p_{1*}p_{2}^{*}(\cO_{\fX_{\et}}/p^n)
 \end{equation}
 is an isomorphism.
\end{lem}

\begin{proof}
Recall $u^{q}_{\fY}$ is induced by the diagram of sites 
\[
q-\textnormal{CRYS}(\fY/D) \xrightarrow{g_1} \fY_{\Et}^{\textnormal{Aff}} \xrightarrow{a} \fY_{\Et} \xleftarrow{g} \fY_{\et}
\]
and so $(u^{q}_{\fY})^{-1} = g_{1}^{s}a^{s}g_{s}$. Unravelling the definitions, the LHS of \eqref{eq:affcasds} is the sheafification of the presheaf
\begin{equation} \label{eq:sokdtenfsp}
(E,J) \mapsto \varinjlim_{\Spf(E/J) \to g(W)} ((\cO_{\fX_{\et}}/p^n)(W \times_{\fY} \fX) \otimes_{ \cO_{\fY_{\et}}(W)} E/[p]_{q}).
\end{equation}
Since $\fX$ and $\fY$ are affine, the expression inside the limit of \eqref{eq:sokdtenfsp} is in fact independent of $W$ (for $W$ affine) and so \eqref{eq:sokdtenfsp} simplifies to the sheafification of
\begin{equation} \label{eq:lhspresvald}
(E,J) \mapsto (\cO_{\fX_{\et}}/p^n)(\fX) \otimes_{ \cO_{\fY_{\et}}(\fY)} E/[p]_{q}.
\end{equation}
We now deal with the RHS of \eqref{eq:affcasds}. First note that $p_{2}^{*}(\cO_{\fX_{\et}}/p^n)$ is the sheaf
\[
Z \mapsto \ol{\cO}^{(1)}(Z)/p^n = (E/[p]_{q} \otimes_{S} R)/p^n
\]
where $Z$ is of the form \eqref{eq:crysgoose}. Therefore $p_{1*}p_{2}^{*}(\cO_{\fX_{\et}}/p^n)$ is the sheaf
\begin{equation} \label{eq:rhsdsovald}
(E,J) \mapsto (\cO_{\fX_{\et}}(\fX) \otimes_{\cO_{\fY_{\et}}(\fY)} E/[p]_{q})/p^n.
\end{equation}
Clearly \eqref{eq:lhspresvald} coincides with \eqref{eq:rhsdsovald}.
\end{proof}

We will also need vanishing of cohomology.

\begin{lem} \label{lem:vaihighqcrc}
Let $Z$ be of the form \eqref{eq:crysgoose}. The cohomology groups $H^{i}(Z, \ol{\cO}^{(1)}/p^n)$ vanish for $i >0$ and $n \geq 1$.
\end{lem}

\begin{proof}
It suffices to observe that any covering of $Z$ induces an étale covering of $\Spf(E/[p]_{q} \wh{\otimes}_{S} R)$. The result now follows by vanishing of high \v{C}ech cohomology for affine schemes.
\end{proof}

\begin{prop} \label{prop:bcqcrst}
Suppose $f$ is proper and smooth,
 and $\sF$ a locally free $\cO_{\fX_{\et}}/p^n$-module of finite type for some $n \geq 1$. Then the canonical base change map induced by \eqref{eq:commtdasirifs}
 \begin{equation} \label{eq:basechapriver}
 Lu_{\fY}^{q*}Rf_{*}\sF \to Rp_{1*}Lp_{2}^{*}\sF
 \end{equation}
 is a quasi-isomorphism.
\end{prop}

\begin{proof}
The proof is very similar to the proof of the counterpart in the pro-étale setting (cf. Proposition \ref{prop:bcrsmopn}) and so we omit the details. In the current situation however, we get a result on the nose because:
\begin{enumerate}
    \item We know the sections of $\ol{\cO}^{(1)}/p^n$ (in the pro-étale setting the sections of $\wh{\cO}^{+}_{D}/p^n$ were only known up to almost ambiguity). This led us to a result on the nose of the affine case (cf. Lemma \ref{lem:afinbaschcascru}).
    \item By the proof of Lemma \ref{lem:vaihighqcrc} we know the vanishing of higher \v{C}ech cohomology of $\ol{\cO}^{(1)}/p^n$ for coverings of $Z$ of the form \eqref{eq:crysgoose}. The analogous result in the pro-étale setting (cf. the proof of Lemma \ref{lem:varithin}(i)) was only known up to almost ambiguity.  
\end{enumerate}
\end{proof}

\subsection{$q$-crystalline vs pro-étale fibered products}

In this section we specialize to the case $(D,I) = (A_{\Inf}, (\xi))$ and assume that we have a $\Spf(\cO)$-morphism $f \colon \fX \to \fY$ of $p$-adic formal schemes of type (S)(b). In \S\ref{subsec:laxproductop}, we constructed the site 
\begin{equation} 
q-\textnormal{CRYS}(\fY/A_{\Inf}) \ola{\times}_{\fY_{\et}} \fX_{\et}
\end{equation}
On the pro-étale side, in \S\ref{sec:prodoftopo}, we construced a site $D$. Let $D'$ be the subsite of $D$ consisting of objects appearing in Proposition \ref{prop:basisDnosmor}. In this case, $D$ and $D'$ have the same associated topoi: $Y_{\proet} \times_{\fY_{\et}} \fX_{\et}$. We want to compare the $q$-crystalline complex $q\Omega_{\fX/\fY}$ (cf. Definition \ref{defn:qcryscomple}) and $A\Omega_{\fX/\fY}$. These two objects live in different topoi $(\fY/D)_{q-\textnormal{CRYS}} \ola{\times}_{\fY_{\et}} \fX_{\et}$ and $Y_{\proet} \times_{\fY_{\et}} \fX_{\et}$, respectively. So we need to first construct a morphism between the topoi. 

\begin{lem} \label{lem:allcoconfns}
The functor
\begin{align} \label{eq:somkdmr}
    u \colon D' &\to q-\textnormal{CRYS}(\fY/A_{\Inf}) \ola{\times}_{\fY_{\et}} \fX_{\et} \nonumber \\
    (\varprojlim_{i \in I} \Spa(S_{i}, S_{i}^{+}) \to \Spf(S) \leftarrow \Spf(R)) &\mapsto ((A_{\Inf}(S_{\infty}), (\xi)) \to \Spf(S) \leftarrow \Spf(R))
\end{align}
is cocontinuous.
\end{lem}

\begin{rem} \label{rem:neksinmors}
We explain how the morphism $\Spf(S_{\infty}) \to \Spf(S)$ is defined so that the RHS of \eqref{eq:somkdmr} makes sense. Let $S^{+}$ be the integral closure of $S$ in $S[\tfrac{1}{p}]$. Then the LHS of \eqref{eq:somkdmr} gives a morphism $\varprojlim_{i \in I} \Spa(S_{i}, S_{i}^{+}) \to \Spa(S[\frac{1}{p}], S^{+})$ in $Y_{\proet}$. This gives a continuous morphism $S^{+} \to \varinjlim S_i^{+}$. The composition $S \to S^{+} \to \varinjlim S_i^{+} \to S_{\infty}$ gives the desired morphism.
\end{rem}

\begin{proof}
We show the cocontinuity for coverings of type (a) first. Let $(E_1,J_1) \to (A_{\Inf}(S_{\infty}), (\xi))$ be a covering in $q-\textnormal{CRYS}(\fY/(A_{\Inf})$. We verify the cocontinuity condition for this cover. By definition we have a $p$-completely étale covering map $\Spf(E_1/J_1) \to \Spf (S_{\infty})$. So $E_1/J_1$ is the $p$-adic derived completion of an étale $S_{\infty}$-algebra $\tfrac{S_{\infty}[x_1, \dots, x_n]}{(f_1, \ldots, f_n)}$. Then
    \[
    E_1/(p,J_1) = (S_{i_1}^{+}/p)[x_1, \ldots, x_n]/(\ol{f}_1, \ldots, \ol{f}_n) \otimes_{S_{i_1}^{+}/p} \varinjlim_{i \in I} S_i^{+}/p
    \]
    for some $i_1 \in I$. Moreover by increasing $i_1$ if necessary, we can assume $\tfrac{(S_{i_1}^{+}/p)[x_1, \dots, x_n]}{(\ol{f}_1, \ldots, \ol{f}_n)}$ is faithfully flat over $S_{i_1}^{+}/p$ (cf. \cite[{Tag 07RR}]{stacks-project}). By \cite[{Tag 04D1}]{stacks-project} there exists some lift of $\tfrac{(S_{i_1}^{+}/p)[x_1, \dots, x_n]}{(\ol{f}_1, \ldots, \ol{f}_n)}$ to an étale $S_{i_1}^{+}$-algebra $T$. Then $E_1/J_1$ is the $p$-adic derived completion of $T \otimes_{S_{i_1}^{+}} \varinjlim_{i \in I} S_i^{+}$. We claim that
    \[
    \Spa (\wh{T}[\tfrac{1}{p}], T') \to \Spa(S_{i_1}, S_{i_1}^{+}), 
    \]
    where $\wh{T}$ is the $p$-adic completion of $T$, and $T'$ is the integral closure of $\wh{T}$ in $\wh{T}[\tfrac{1}{p}]$, is an étale covering map. Indeed it is étale by \cite[Corollary 1.7.3(iii)]{HubEtadic}. It is also a covering because $T/p$ is faithfully flat over $S_{i_1}^{+}/p$. It follows that the fiber product
    \begin{equation} \label{eq:fiborpprot}
    \Spa (\wh{T}[\tfrac{1}{p}], T') \times_{\Spa(S_{i_1}, S_{i_1}^{+})} \varprojlim_{i \in I} \Spa(S_i, S_i^{+})
    \end{equation}
    exists in $Y_{\proet}$ and is a covering of $\varprojlim_{i \in I} \Spa(S_i, S_i^{+})$. We can rewrite \eqref{eq:fiborpprot} as
    \[
    \varprojlim_{i \in I} \Spa(\wh{T}[\tfrac{1}{p}] \wh{\otimes}_{S_{i_1}} S_i, {S'_{i}}^{+})
    \]
    where ${S'_{i}}^{+}$ is the integral closure of $T' \otimes_{S_{i_1}^{+}} S_{i}^{+}$ in $\wh{T}[\tfrac{1}{p}] \wh{\otimes}_{S_{i_1}} S_i$. Therefore we obtain a composition of maps
    \[
    T \otimes_{S_{i_1}^{+}} \varinjlim_{i \in I} S_i^{+} \to \varinjlim_{i \in I}(T' \otimes_{S_{i_1}^{+}} S_{i}^{+}) \to \varinjlim {S'_{i}}^{+}. 
    \]
    This verifies the cocontinuty condition for coverings of type (a).
    
    Coverings of type (b) are trivial to handle. We now deal with coverings of type (c): let $((A_{\Inf}(S_{\infty}), (\xi)) \to \fY' \leftarrow \fX')$ be such a cover of the RHS of \eqref{eq:somkdmr}. By using the argument in the proof of Lemma \ref{lem:basaqcrsdsi}, we can refine this to a cover of the form $((A_{\Inf}(S'_{\infty}), (\xi)) \to \Spf(S') \leftarrow \Spf(R'))$ where $\Spf(S'_{\infty}) \to \Spf(S_{\infty})$ is an étale covering. This is a covering of type (a) followed by a covering of type (c). Thus by the first part of this proof, we have reduced to checking the cocontinuity condition for the case where we have a cover of the form $((A_{\Inf}(S_{\infty}), (\xi)) \to \Spf(S') \leftarrow \Spf(R'))$ of type (c) of the RHS of \eqref{eq:somkdmr}. This means that $\Spf(S_{\infty}) \to \Spf(S)$ factors through $\Spf(S') \to \Spf(S)$. It suffices to show that $\varprojlim_{i \in I} \Spa(S_{i}, S_{i}^{+}) \to \Spa(S[1/p], S^+)$ in $Y_{\proet}$ factors through $\Spa(S'[\tfrac{1}{p}], S'^{+}) \to \Spa(S[\tfrac{1}{p}], S^{+})$, where $S'^{+}$ is the integral closure of $S'$ in $S'[\tfrac{1}{p}]$. Observe that a composition $\Spf(S_{\infty}) \to \Spf(S') \to \Spf(S)$ is equivalent to a continuous $S$-morphism $S' \to S_{\infty}$. Since $S/p \to S'/p$ is of finite presentation, there exists $i_1 \in I$ such that $S'/p \to \varinjlim_{i \in I} S^{+}_{i}/p$ factors through $S^{+}_{i_1}/p$. Moreover by increasing $i_1$ if necessary, we can assume $S \to \varinjlim_{i \in I} S^{+}_i$ factors through $S_{i_1}^{+}$. Then for each $n \geq 1$ we have a commutative square
    \[
      \xymatrix{
 S'/p^n \ar@{->}[r] \ar@{.>}[rd]  & S^{+}_{i_1}/p    \\ 
 S/p^n \ar@{->}[u] \ar@{->}[r] & S^{+}_{i_1}/p^n, \ar@{->}[u]
 }
 \]
 where the left vertical arrow is étale and so there is a unique map $S'/p^n \to S_{i_1}^{+}/p^n$ rendering the diagram commutative. Thus upon taking limits we obtain a continuous map $S' \to S^{+}_{i_1}$ compatible with the maps from $S$. The result now follows.
\end{proof}

Lemma \ref{lem:allcoconfns} induces a morphism of topoi
\begin{equation} \label{eq:morfroprotoqcr}
\mu \colon Y_{\proet} \times_{\fY_{\et}} \fX_{\et} \to (\fY/A_{\Inf})_{q-\textnormal{CRYS}} \ola{\times}_{\fY_{\et}} \fX_{\et}.
\end{equation}

\subsection{Comparison with $q$-de Rham cohomology}

In this short section we compare $q\Omega_{\fX/\fY}$ with a \emph{relative} $q$-de Rham complex (cf. \cite[Theorem 16.22]{Prisms} and \cite[Theorem 7.17]{KoshTerui}). We work locally and suppose $\Spf(R) \to \Spf(S)$ is a smooth $\Spf(\cO)$-morphism. In addition we assume $R$ is \emph{very small} over $S$, i.e. $R/p$ is generated over $S/p$ by units. This means that there is a canonical functorial surjection $P \to R \wh{\otimes}_{S} S_{\infty}$ with kernel $J$, where $P := \bA_{\Inf}(S_{\infty})[\{t_s\}_{s \in R^{\times}}]^{\wedge}$ and the completions are $(p, \tilde{\xi})$-adic. In this context $q = [\epsilon] \in A_{\Inf}$ and in order to define $q$-de Rham cohomology in the case that we need, we verify the following result.

\begin{lem}
$\bA_{\Inf}(S_{\infty})$ is a flat $A$-algebra.
\end{lem}

\begin{proof}
By \cite[Remark 4.31]{BhMorSch}, it suffices to show that $S_{\infty}^{\flat}$ is $(\epsilon -1)$-torsion free. This follows from Lemma \ref{baslem:torfree} and \cite[(2.1.2.2)]{CesSch}.
\end{proof}

We may regard $P$ as a $\delta$-$\bA_{\Inf}(S_{\infty})$-algebra in a unique way by setting $\delta(x_s)=0$ for each $s \in R^{\times}$. Let $D_{J,q}(P)$ be the $q$-PD envelope of $P \to R \wh{\otimes}_{S} S_{\infty}$. The main result is the following.

\begin{thm} \label{thm:qOmegavsqdeRham}
 Let $Z$ be of the form the RHS of \eqref{eq:somkdmr} such that in addition $R$ is very small over $S$. Then there is a canonical functorial quasi-isomorphism $R\Gamma(Z, q\Omega_{\fX/\fY}) \cong q\Omega^{*, \square}_{D_{J,q}(P)/\bA_{\Inf}(S_{\infty})}$.
\end{thm}

\begin{proof}
As with the comparison with prismatic cohomology, the hard work is already done in \cite[\S 16]{Prisms}. Indeed by Lemma \ref{lem:morsitqycrsi}
\[
R\Gamma(Z, q\Omega_{\fX/\fY}) = R\Gamma((R \wh{\otimes}_{S} S_{\infty})/\bA_{\Inf}(S_{\infty}))_{q-\textnormal{CRYS}}, \cO_{q-\textnormal{CRYS}}).
\]
The result now follows by Theorem 16.21 in loc.cit. and the comparison made in Lemma \ref{lem:samecoha}.
\end{proof}

\subsection{Comparison with $A\Omega_{\fX/\fY}$}

We specialize to the case $(D,I) = (A_{\Inf}, (\xi))$. Having established the following ingredients:
\begin{enumerate}
    \item a way to pass from the pro-étale setting to the $q$-crystalline setting (cf. \eqref{eq:morfroprotoqcr}),
    \item the Hodge-Tate comparison for $q\Omega_{\fX/\fY}$ (cf. Corollary \ref{cor:HTforqcrys}), and
    \item the Hodge-Tate comparison for $A\Omega_{\fX/\fY}$ (cf. Theorem \ref{thm:HTspcssd}),
\end{enumerate}
in this section we are now ready to compare $q\Omega_{\fX/\fY}$ and $A\Omega_{\fX/\fY}$ (cf. \cite[Theorem 17.2]{Prisms} and \cite[Theorem 8.1]{KoshTerui}). 

\begin{thm} \label{thm:AOmegvsqCrys}
Assume that we have a smooth $\Spf(\cO)$-morphism $f \colon \fX \to \fY$ of $p$-adic formal schemes of type (S)(b). Suppose $\fX$ and $\fY$ are flat over $\cO$. There is a canonical functorial morphism 
\begin{equation} \label{eq:compmorAOMeqOMe}
\mu^{-1}q\Omega_{\fX/\fY} \to A\Omega_{\fX/\fY}
\end{equation}
in $D(D, p^{-1}W(\wh{\cO}^{+, \flat}_Y))$ compatible with the Frobenius such that the cohomology sheaves of the cone of \eqref{eq:compmorAOMeqOMe} 
 are killed by $W(\fm^{\flat}) \subset A_{\Inf}$. If in addition $\fY$ is locally of finite type over $\cO$, then \eqref{eq:compmorAOMeqOMe} is an isomorphism.
\end{thm}

\begin{rem}
We explain how to view the LHS and RHS of \eqref{eq:compmorAOMeqOMe} as objects in $D(D, p^{-1}W(\wh{\cO}^{+, \flat}_Y))$. For the RHS this is already explained in Definition \ref{defn:WversofAInf} together with Lemma \ref{lem:quasiisormAOmegW}. For the LHS, first note that $p^{-1}W(\wh{\cO}^{+, \flat}_Y) = \mu^{-1}p_1^{-1}\cO_{q-\textnormal{CRYS}}$. Moreover, by temporarily regarding $\nu_{f}^{q}$ as the following morphism of ringed topoi (we abuse notation and use $\cO_{q-\textnormal{CRYS}}$ to denote the structure sheaves on both $(\fX/D)_{q-\textnormal{CRYS}}$ and $(\fY/D)_{q-\textnormal{CRYS}}$)
\[
(\nu_{f}^{q},\nu_{f}^{q\sharp}) \colon ((\fX/D)_{q-\textnormal{CRYS}}, \cO_{q-\textnormal{CRYS}}) \to ((\fY/D)_{q-\textnormal{CRYS}} \ola{\times}_{\fY_{\et}} \fX_{\et},p_1^{-1}\cO_{q-\textnormal{CRYS}})
\]
naturally equips $q\Omega_{\fX/\fY}$ as an object in $D(q-\textnormal{CRYS}(\fY/D) \ola{\times}_{\fY_{\et}} \fX_{\et}, p_1^{-1}\cO_{q-\textnormal{CRYS}})$.
\end{rem}

\begin{proof}
We follow the main ideas of the proof of \cite[Theorem 17.2]{Prisms}. Let $Z$ be of the form the LHS of \eqref{eq:somkdmr}. One difference in the (current) relative case is that the derived $Z$-sections of the presheaf $A\Omega_{\fX/\fY}^{\text{psh}}$ identify with the derived $Z$-sections of $A\Omega_{\fX/\fY}$ only up to almost ambiguity (cf. \eqref{eq:adjuisoap}). For this reason (together with the fact that the derived $Z$-sections of the presheaf version are explicit), we will work directly with the presheaf.

We will use the notation developed in \S\ref{subsec:presver}. In particular, addition to the site $D$, the site $D^{\text{psh}}$ will be useful. In fact we will work with the subsites $D_{\text{vs}} \subset D$ and $D^{\text{psh}}_{\text{vs}} \subset D^{\text{psh}}$ consisting only of good objects $Z$ such that $R$ is very small over $S$. Note that $D_{\text{vs}}$ and $D$ induce the same topoi. We have $\mu^{-1}q\Omega_{\fX/\fY} = \phi^{-1}q\Omega_{\fX/\fY}^{\text{psh}}$ where $q\Omega_{\fX/\fY}^{\text{psh}}$ is the complex of presheaves given by 
\[
Z \mapsto R\Gamma(u(Z), q\Omega_{\fX/\fY}),
\]
where $u$ is the functor in Lemma \ref{lem:allcoconfns}. 
By \eqref{eq:pullbackofpregivshe}, it suffices to construct a canonical functorial morphism
\begin{equation} \label{eq:prshlfands}
q\Omega_{\fX/\fY}^{\text{psh}} \to A\Omega_{\fX/\fY}^{\text{psh}}
\end{equation}
in $D(D^{\text{psh}}_{\text{vs}}, p^{-1}W(\wh{\cO}^{+, \flat}_Y))$ compatible with the Frobenius, whose cone satisfies the relevant conditions (we abuse notation and denote $p^{-1}W(\wh{\cO}^{+, \flat}_Y)$ for the presheaf $Z \mapsto \bA_{\Inf}(S_{\infty})$). We do this in several steps. We assume henceforth that $R$ is very small over $S$.

\emph{Step 1: Constructing explicit representatives for $R\Gamma(Z, q\Omega_{\fX/\fY}^{\text{psh}})$ and $R\Gamma(Z, A\Omega_{\fX/\fY}^{\text{psh}})$.} 
We rewrite each term $R\Gamma(Z, q\Omega_{\fX/\fY}^{\text{psh}})$ and $R\Gamma(Z, A\Omega_{\fX/\fY}^{\text{psh}})$ by explicit representatives of Koszul complexes. Strictly speaking this will be only up to almost ambiguity for the second term in general, and is one difference compared to the cases of \cite{Prisms} and \cite{KoshTerui}. For the term $R\Gamma(Z, q\Omega_{\fX/\fY}^{\text{psh}})$, this is made possible by comparison with the $q$-de Rham complex (cf. Theorem \ref{thm:qOmegavsqdeRham}). Indeed $R\Gamma(u(Z), q\Omega_{\fX/\fY}) \cong q\Omega^{*, \square}_{D_{J,q}(P)/\bA_{\Inf}(S_{\infty})}$ and by definition $q\Omega^{*, \square}_{D_{J,q}(P)/\bA_{\Inf}(S_{\infty})}$ is the complex
\begin{equation} \label{eq:KoszulcomforqOM}
\Kos_{c}(D_{J,q}(P);\{\nabla_{q,s}\}_{s \in R^{\times}})
\end{equation}
defined as the term-wise $(p,[p]_{q})$-completed filtered colimit over all finite subsets $\Sigma \subset R^{\times}$ of the corresponding cohomological Koszul complexes for $\{\nabla_{q,s}\}_{s \in \Sigma}$. The treatment of the term $R\Gamma(Z, A\Omega_{\fX/\fY}^{\text{psh}})$ is slightly more delicate. We use the notation from \S\ref{sec:localanal}. For each finite subset $\Sigma \subset R^{\times}$ containing $t_1, \ldots, t_d$, there is a corresponding profinite group $\Delta_{\Sigma} = \bZ_p^{\Sigma}$ with generators (over $\bZ_p$) $\{\sigma_{s}\}_{s \in \Sigma}$ and a pro-(finite étale) affinoid perfectoid $\Delta_{\Sigma}$-cover\footnote{One takes $R^{\square}_{\Sigma} := S \{\{t_s^{\pm 1}\}_{s \in \Sigma}\}$ and defines $R^{\square}_{\Sigma,m}$, and $R_{\Sigma, \infty}$ analogously as in \eqref{eq:defnofRmsqau} and \eqref{eq:definRinfty}, respectively. For instance $R^{\square}_{\Sigma,m}$ contains a $p^m$-th root $t_s^{1/p^m}$.}
\[
\Spa(R_{\Sigma, \infty}[\tfrac{1}{p}], R_{\Sigma, \infty}) \to X \times_{Y} Y_{\infty} \hspace{2mm} \text{ that refines the $\Delta = \Delta_{\{t_1, \ldots, t_d\}}$-cover } \hspace{2mm} X_{\infty} \to X \times_{Y} Y_{\infty}
\]
compatible with the natural surjection $\Delta_{\Sigma} \twoheadrightarrow \Delta$. Consider now the complex
\[
(\varinjlim_{\Sigma} \eta_{q-1}\Kos_{c}(\bA_{\Inf}(R_{\Sigma, \infty});\{\sigma_{s} -1 \}_{s \in \Sigma}))^{\wedge}
\]
defined as the term-wise $(p,[p]_{q})$-completed filtered colimit over all finite subsets $\Sigma \subset R^{\times}$ (containing $t_1, \ldots, t_d$) of the corresponding cohomological Koszul complexes. By Theorem \ref{edgemapisoAinf}, Remark \ref{rem:extenthem4.12} and \eqref{eq:evapresatsec},  there is a morphism
\begin{equation} \label{eq:KoszulforAom}
(\varinjlim_{\Sigma} \eta_{q-1}\Kos_{c}(\bA_{\Inf}(R_{\Sigma, \infty});\{\sigma_{s} -1 \}_{s \in \Sigma}))^{\wedge} \to R\Gamma(Z, A\Omega_{\fX/\fY}^{\text{psh}})
\end{equation}
whose cone has its cohomology groups killed by $W(\fm^{\flat})$ (if $\fY$ is locally of finite type over $\cO$, then it is a quasi-isomorphism).

\emph{Step 2: Constructing a map $R\Gamma(Z, q\Omega_{\fX/\fY}^{\text{psh}}) \to R\Gamma(Z, A\Omega_{\fX/\fY}^{\text{psh}})$.} By \eqref{eq:KoszulcomforqOM}-\eqref{eq:KoszulforAom}, it is enough to construct a functorial comparison map
\begin{equation} \label{eq:somcaskad}
\Kos_{c}(D_{J,q}(P);\{\nabla_{q,s}\}_{s \in R^{\times}}) \to (\varinjlim_{\Sigma} \eta_{q-1}\Kos_{c}(\bA_{\Inf}(R_{\Sigma, \infty});\{\sigma_{s} -1 \}_{s \in \Sigma}))^{\wedge}.
\end{equation}
By \cite[Lemma 7.9]{BhMorSch}, there is a canonical isomorphism
\[
\Kos_{c}(D_{J,q}(P);\{\nabla_{q,s}\}_{s \in R^{\times}}) \cong \eta_{q-1}\Kos_{c}(D_{J,q}(P);\{\gamma_{s} -1\}_{s \in R^{\times}})
\]
and so to construct \eqref{eq:somcaskad},
it suffices to construct a canonical map
\begin{equation} \label{eq:maprestocsma}
D_{J,q}(P) \to \bA_{\Inf}((\varinjlim_{\Sigma} R_{\Sigma, \infty})^{\wedge})
\end{equation}
of $\bA_{\Inf}(S_{\infty})$-algebras that intertwines $\sigma_{s}$ and $\gamma_{s}$ (the completion appearing on the RHS is $p$-adic). We have a commutative diagram
\[
 \xymatrix{
 P \ar@{->}[r] \ar@{->>}[d] & \bA_{\Inf}((\varinjlim_{\Sigma} R_{\Sigma, \infty})^{\wedge}) \ar@{->>}[d]   \\ 
 R \wh{\otimes}_S S_{\infty} \ar@{->}[r] & (\varinjlim_{\Sigma} R_{\Sigma, \infty})^{\wedge},
 }
 \]
 where the top map is induced by mapping $t_s$ to $[(\ldots, \tilde{t}_s^{1/p},\tilde{t}_s)]$ (where $\tilde{t}_s^{1/p^m}$ denotes the image of $t_s^{1/p^m} \in R^{\square}_{\Sigma,m}$ in $(\varinjlim_{\Sigma} R_{\Sigma, \infty})^{\wedge}$). The action of $\sigma_{s}$ on $\bA_{\Inf}((\varinjlim_{\Sigma} R_{\Sigma, \infty})^{\wedge})$ can be computed as in the last paragraph of \S\ref{subsec:cohcongrp} (where it is done in the case $s = t_i$ on $\bA_{\Inf}(R_{\infty})$). This action clearly intertwines with the action of $\gamma_{s}$ on $P$. Moreover $P \to \bA_{\Inf}((\varinjlim_{\Sigma} R_{\Sigma, \infty})^{\wedge})$ is a $\delta$-map and so extends uniquely to a $\delta$-map $D_{J,q}(P) \to \bA_{\Inf}((\varinjlim_{\Sigma} R_{\Sigma, \infty})^{\wedge})$. The uniqueness also implies that it intertwines $\sigma_s$ and $\gamma_s$ (to see this one considers the compositions $P \xrightarrow{\gamma_s} P \to \bA_{\Inf}(R_{\infty})$ and $P \to \bA_{\Inf}(R_{\infty}) \xrightarrow{\sigma_s} \bA_{\Inf}(R_{\infty})$). In summary we have constructed the desired map \eqref{eq:maprestocsma}.

\emph{Step 3: Showing the cone of map $R\Gamma(Z, q\Omega_{\fX/\fY}^{\text{psh}}) \to R\Gamma(Z, A\Omega_{\fX/\fY}^{\text{psh}})$ from Step 2 satisfies the relevant conditions.} It suffices to check that the map \eqref{eq:somcaskad} has cone whose cohomology groups are killed by $W(\fm^{\flat})$ (and a quasi-isomorphism if $\fY$ is locally of finite type over $\cO$). We use (a variant of) the criterion given in \cite[Lemma 17.4]{Prisms}. Let $K_1$ and $K_2$ be the source and target of \eqref{eq:somcaskad}. By the Hodge-Tate comparison for $q\Omega_{\fX/\fY}$ (cf. Corollary \ref{cor:HTforqcrys}) there is an isomorphism 
\[
 \Omega_{R/S}^{\bullet} \wh{\otimes}_{S} \bA_{\Inf}(S_{\infty})/[p]_{q} \xrightarrow{\sim} H^{\bullet}(K_1/[p]_{q})\{\bullet \}
\]
of differential graded $\bA_{\Inf}(S_{\infty})/[p]_{q}$-algebras. Let 
\begin{equation}
\label{eq:deRhamcompagfsa}
\eta_{K_1} \colon R \wh{\otimes}_{S} \bA_{\Inf}(S_{\infty})/[p]_{q} \to H^{0}(K_1/[p]_{q})
\end{equation}
be the morphism in degree 0. In the proof of the de Rham specialization (cf. Theorem \ref{thm:regdeRhamse}) we constructed an isomorphism 
\begin{equation} \label{eq:themmrofrom519}
\Omega_{R/S}^{\bullet} \wh{\otimes}_{S} \bA_{\Inf}(S_{\infty})/[p]_{q} = (q^{*}\Omega^{\bullet}_{\fX/\fY})(Z) \to C^{\bullet}_{Z}.
\end{equation}
Note that the twist appearing in $\Omega_{R/S}^{\bullet} \wh{\otimes}_{S} \bA_{\Inf}(S_{\infty})/[p]_{q}$ was implicit and is coming from \eqref{eq:truisnoalis}. 
Moreover by Remark \ref{rem:extenthem4.12}, there is a natural morphism
\begin{equation} \label{eq:secondmorrem413}
    C^{\bullet}_Z \to H^{\bullet}(K_2/[p]_{q})\{\bullet \}
\end{equation}
whose cone has its cohomology groups killed by $W(\fm^{\flat})$ (if $\fY$ is locally of finite type over $\cO$ then it is an isomorphism). Composing \eqref{eq:themmrofrom519}-\eqref{eq:secondmorrem413} gives
\begin{equation} \label{eq:deRhamparAom}
 \Omega_{R/S}^{\bullet} \wh{\otimes}_{S} \bA_{\Inf}(S_{\infty})/[p]_{q} \to H^{\bullet}(K_2/[p]_{q})\{\bullet \}
\end{equation}
whose cone has cohomology groups killed by $W(\fm^{\flat})$ in general, and if $\fY$ is locally of finite type over $\cO$, then it is an isomorphism.
 Let 
\begin{equation}
\label{eq:deRhamcompagfsasd}
\eta_{K_2} \colon R \wh{\otimes}_{S} \bA_{\Inf}(S_{\infty})/[p]_{q} \to H^{0}(K_2/[p]_{q})
\end{equation}
be the morphism in degree 0 of \eqref{eq:deRhamparAom}. We remark that   $\eta_{K_2}$ is $S_{\infty}$-linear. To see this, note that \eqref{eq:secondmorrem413} is clearly $S_{\infty}$-linear in each degree and $\wh{\cO}^{+}_D(Z) \to C^{0}_Z$ is $S_{\infty}$-linear as mentioned in the proof of Theorem \ref{thm:regdeRhamse}. One then checks that the composition induced by \eqref{eq:somcaskad} and \eqref{eq:deRhamcompagfsa}
\[
R \wh{\otimes}_{S} \bA_{\Inf}(S_{\infty})/[p]_{q} \xrightarrow{\eta_{K_1}} H^{0}(K_1/[p]_{q}) \to H^{0}(K_2/[p]_{q})
\]
identifies with $\eta_{K_2}$. This can be done by reducing to the case $R = S \{ t_1^{\pm 1} \}$ where one uses that the morphisms are $S_{\infty}$-linear and keeps track of where the coordinate $t_1$ is mapped to. Finally by a similar argument to the proof of \cite[Lemma 17.4]{Prisms} (where in place of derived Nakayama, one uses the \emph{almost} version, cf. Lemma \ref{lem:almostderNaka}), one obtains that \eqref{eq:somcaskad} has cone whose cohomology groups are killed by $W(\fm^{\flat})$ (and is a quasi-isomorphism if $\fY$ is locally of finite type over $\cO$ by loc.cit.).

In summary \emph{Steps 1-3} show that the morphism \eqref{eq:somcaskad} constructed is functorial in $Z$ and satisfies the relevant conditions. Similarly by fixing an injective resolution $\bA_{\Inf,X} \to I^{\bullet}$ in $X_{\proet}$, the edge map \eqref{eq:edgemapforainf} comes from a functorial morphism of actual complexes. This implies that \eqref{eq:KoszulforAom} globalizes. This produces the map \eqref{eq:prshlfands} as promised.
\end{proof}

Almost in the next statement is w.r.t. $[\fm^{\flat}] \subset A_{\Inf}$.

\begin{lem}[almost derived Nakayama] \label{lem:almostderNaka}
Let $(\sC, A_{\Inf})$ be a ringed site and suppose that $\sA \in D(\sC, A_{\Inf})$ is almost derived $[p]_q$-adically complete. If $\sA \otimes^{\bL}_{A_{\Inf}} A_{\Inf}/[p]_q$ is almost zero, then so is $\sA$.
\end{lem}

\begin{proof}
By induction on $n \geq 1$, one obtains that $\sA \otimes^{\bL}_{A_{\Inf}} A_{\Inf}/[p]_{q}^n$ is almost zero. Hence by taking the derived limit, we get that $R\varprojlim_{n} (\sA \otimes^{\bL}_{A_{\Inf}} A_{\Inf}/[p]_{q}^n)$ is almost zero. But by assumption the canonical morphism $\sA \to R\varprojlim_{n} (\sA \otimes^{\bL}_{A_{\Inf}} A_{\Inf}/[p]_{q}^n)$ is an almost quasi-isomorphism and so the result follows.
\end{proof}

\subsection{Comparison with $R\Gamma_{A_{\Inf}}(\fX/\fY)$} \label{sec:compwithrelAinf}

In this section we again specialize to the case $(D,I) = (A_{\Inf}, (\xi))$ and unless otherwise stated, assume that we have a $\Spf(\cO)$-morphism $f \colon \fX \to \fY$ of $p$-adic formal schemes of type (S)(b). In addition we suppose $\fX$ and $\fY$ are flat over $\cO$. We now move to the comparison between (relative) $q$-crystalline cohomology ($R\Gamma_{q-\textnormal{CRYS}}(\fX/\fY)$, cf. Definition \ref{defn:qcryscomple}) and $A_{\Inf}$-cohomology ($R\Gamma_{A_{\Inf}}(\fX/\fY)$). Let $Y'_{\proet} \subset Y_{\proet}$ be the subsite consisting of affinoid perfectoid objects. By the proof of Lemma \ref{lem:allcoconfns}, the functor of sites
\begin{align*}
    u_{\proet} \colon Y'_{\proet} &\to q-\textnormal{CRYS}(\fY/A_{\Inf}) \\
    \varprojlim_{i \in I} \Spa(S_{i}, S_{i}^{+})  &\mapsto (\bA_{\Inf}(S_{\infty}), (\xi)) 
\end{align*}
is cocontinuous and so induces a morphism of topoi
\[
\mu_{\proet} \colon Y_{\proet} \to (\fY/A_{\Inf})_{q-\textnormal{CRYS}}.
\]
We now have a priori, two different morphisms of topoi $u^{q}_{\fY} \circ \mu_{\proet}$ and $\nu_{Y}$ from $Y_{\proet}$ to $\fY_{\et}$. As a reality check, we verify that they coincide.

\begin{lem} \label{lem:equivoftopoi}
The two morphisms of topoi $u^{q}_{\fY} \circ \mu_{\proet} \colon Y_{\proet} \to \fY_{\et}$ and $\nu_Y \colon Y_{\proet} \to \fY_{\et}$ are the same.
\end{lem}

\begin{proof}
Indeed for a sheaf $\sF$ on $Y_{\proet}$ and $\fV$ an affine object in $\fY_{\et}$
\begin{equation} \label{eq:twolimsimpl}
(u^{q}_{\fY*}\mu_{\proet,*}\sF)(\fV) = \varprojlim_{ \Spf(E/J) \to \fV}\varprojlim_{ (A_{\Inf}(S_{\infty}), (\xi)) \to (E,J)} \sF(\varprojlim_{i \in I} \Spa(S_{i}, S_{i}^{+}))
\end{equation}
where the left-most limit is taking place along $\fY_{\Et}^{\textnormal{Aff}}$ and the middle limit along $q-\textnormal{CRYS}(\fY/A_{\Inf})$. Combining these two limits implies the RHS of \eqref{eq:twolimsimpl} can be rewritten as
\[
\varprojlim_{\Spf(S_{\infty}) \to \fV} \sF(\varprojlim_{i \in I} \Spa(S_{i}, S_{i}^{+}))
\]
with the left-most limit taking place along $\fY_{\Et}^{\textnormal{Aff}}$. The morphism
$\Spf(S_{\infty}) \to \fV$ is equivalent to giving a morphism on the generic fibers $\Spa (S_{\infty}[\tfrac{1}{p}],S_{\infty}) \to V$ (in the category of adic spacs over $Y$) because $S_{\infty}$ is integrally closed in $S_{\infty}[\tfrac{1}{p}]$. We claim that giving a morphism $\Spa (S_{\infty}[\tfrac{1}{p}],S_{\infty}) \to V$ (in the category of adic spaces over $Y$) is equivalent to giving a morphism $ \varprojlim_{i \in I} \Spa(S_{i}, S_{i}^{+}) \to V$ in $Y_{\proet}$. This is similar to the proof of the last part in Lemma \ref{lem:allcoconfns} and we omit the details.
\end{proof}

\begin{lem} \label{lem:pqderiomqcr}
Assume that we have a proper smooth $\Spf(\cO)$-morphism $f \colon \fX \to \fY$ of $p$-adic formal schemes of type (S)(b) with $\fX$ and $\fY$ flat over $\cO$. Then the cohomology sheaves of the cone of the canonical morphism
\[\mu_{\proet}^{-1}R\Gamma_{q-\textnormal{CRYS}}(\fX/\fY) \to (\mu_{\proet}^{-1}R\Gamma_{q-\textnormal{CRYS}}(\fX/\fY))^{\wedge}
\]
are killed by $[\fm^{\flat}] \subset A_{\Inf}$ (the completion on the RHS is derived $[p]_q$-adic).
\end{lem}

\begin{proof}
 Take a covering $\{U_j \to Y \}_{j}$ by affinoid perfectoid objects of $Y$ and write $U_j = \varprojlim_{i \in I_j} \Spa(S_{ij}, S_{ij}^{+})$. Then
\[
\mu_{\proet}^{-1}R\Gamma_{q-\textnormal{CRYS}}(\fX/\fY) \lvert_{U_j} = \mu_{\proet}^{-1}(R\Gamma_{q-\textnormal{CRYS}}(\fX/\fY)\lvert_{(\bA_{\inf}(S_{j,\infty}),(\xi))}).
\]
By Proposition \ref{prop:perfectnessqcrco}(ii), $(R\Gamma_{q-\textnormal{CRYS}}(\fX/\fY)\lvert_{(\bA_{\inf}(S_{j,\infty}),(\xi))}$ is represented by a strictly perfect object in $D((\fY/A_{\Inf})_{q-\textnormal{CRYS}}/(\bA_{\inf}(S_{j,\infty}),(\xi)), \cO_{q-\textnormal{CRYS}}\lvert_{(\bA_{\inf}(S_{j,\infty}),(\xi))})$. Now an easy calculation gives $\mu_{\proet}^{-1}\cO_{q-\textnormal{CRYS}} = W(\wh{\cO}^{+, \flat}_{Y})$. Moreover the cone of the canonical morphism
\[
W(\wh{\cO}^{+, \flat}_{Y}) \to \bA_{\Inf,Y}
\]
has its cohomology sheaves killed by $[\fm^{\flat}]$ (cf. \cite[Lemma 5.6]{BhMorSch}). Since $\bA_{\Inf,Y}$ is derived $[p]_q$-complete by Corollary \ref{cor:xiadcomiofAinf}, the result follows.
\end{proof}

\begin{thm} \label{thm:comparofRGqcruandRGaAin}
Assume that we have a proper smooth $\Spf(\cO)$-morphism $f \colon \fX \to \fY$ of $p$-adic formal schemes of type (S)(b) with $\fX$ and $\fY$ flat over $\cO$. There is a canonical functorial morphism 
\begin{equation} \label{eq:compmorAOMeqOMe1} 
\mu_{\proet}^{-1}R\Gamma_{q-\textnormal{CRYS}}(\fX/\fY) \to R\Gamma_{A_{\Inf}}(\fX/\fY)
\end{equation}
in $D(Y_{\proet}, W(\wh{\cO}^{+, \flat}_Y))$ compatible with the Frobenius such that the cohomology sheaves of the cone of \eqref{eq:compmorAOMeqOMe1} 
 are killed by $[\fm^{\flat}] \subset A_{\Inf}$.
\end{thm}

\begin{rem}
We explain how to view the LHS and RHS of \eqref{eq:compmorAOMeqOMe1} as objects in $D(Y_{\proet},  W(\wh{\cO}^{+, \flat}_Y))$. For the RHS this is already explained in Definition \ref{defn:WversofAInf} together with Lemma \ref{lem:quasiisormAOmegW}. For the LHS, first note $R\Gamma_{q-\textnormal{CRYS}}(\fX/\fY)$ is naturally an object of $D(q-\textnormal{CRYS}(\fY/A_{\Inf}), \cO_{q-\textnormal{CRYS}})$. Since $\mu_{\proet}^{-1}\cO_{q-\textnormal{CRYS}} = W(\wh{\cO}^{+, \flat}_{Y})$, we get that the LHS can be viewed as an object in $D(Y_{\proet},  W(\wh{\cO}^{+, \flat}_Y))$. 
\end{rem}

\begin{proof}
The strategy is to work modulo $[p]_q$ and to gradually massage the LHS towards the RHS of \eqref{eq:compmorAOMeqOMe1}. The following schematic diagram summarises this ``massage process":
\[
 \xymatrix{
 \mu_{\proet}^{-1}R\Gamma_{q-\textnormal{CRYS}}(\fX/\fY)/[p]_q \ar@{-->}[r] ^{(\textnormal{I})} & \Omega_{\fX/\fY}^{\bullet}\lvert_{(\fY/A_{\Inf})_{q-\textnormal{CRYS}} \ola{\times}_{\fY_{\et}} \fX_{\et}} \ar@{-->}[d]^{(\textnormal{II})} \\
 R\Gamma_{A_{\Inf}}(\fX/\fY)/[p]_q
 & \Omega_{\fX/\fY}^{\bullet}\lvert_{Y_{ \proet}\times_{\fY_{\et}}\fX_{\et}} \ar@{-->}[l]_{(\textnormal{III})}, 
 }
 \]
 where steps (I) and (III) are combinations of the relevant Hodge-Tate comparisons, and step (II) is a combination of the base changes from the pro-étale and $q$-crystalline side. Finally after working modulo $[p]_q$, one can lift the results as both sides of \eqref{eq:compmorAOMeqOMe1} are (almost) derived $[p]_q$-adically complete. We now begin the proof.

We define \eqref{eq:compmorAOMeqOMe1} as the composition
\[
\mu_{\proet}^{-1}R\Gamma_{q-\textnormal{CRYS}}(\fX/\fY) \to Rp_{*}\mu^{-1}q\Omega_{\fX/\fY} \xrightarrow{Rp_{*}(\eqref{eq:compmorAOMeqOMe})} R\Gamma_{A_{\Inf}}(\fX/\fY)
\]
where the first map is induced by the canonical base change map for the commutative\footnote{Commutativity can be checked at the level of sites: by definition $\mu$ and $\mu_{\proet}$ are induced by cocontinuous functors and as mentioned in the proof of Lemma \ref{lem:desfuncsoif}, so are $p$ and $p_1$.} diagram of ringed topoi
\begin{equation} \label{eq:baschadeisd1}
 \xymatrix{
 (Y_{\proet} \times_{\fY_{\et}} \fX_{\et}, p^{-1}W(\wh{\cO}^{+, \flat}_{Y})) \ar@{->}[rr]_{\mu} \ar@{->}[d]^{p} && ((\fY/D)_{q-\textnormal{CRYS}} \ola{\times}_{\fY_{\et}} \fX_{\et}, p_{1}^{-1}\cO_{q-\textnormal{CRYS}}) \ar@{->}[d]^{p_1}   \\ 
 (Y_{\proet},W(\wh{\cO}^{+, \flat}_{Y})) \ar@{->}[rr]_{\mu_{\proet}} && ((\fY/D)_{q-\textnormal{CRYS}},\cO_{q-\textnormal{CRYS}}).
 }
\end{equation}
Note that since the morphism \eqref{eq:compmorAOMeqOMe} is constructed in $D(D, p^{-1}W(\wh{\cO}^{+, \flat}_Y))$, the morphism \eqref{eq:compmorAOMeqOMe1} is in $D(Y_{\proet}, W(\wh{\cO}^{+, \flat}_Y))$.

By Corollary \ref{cor:objecxiadcom}, $A\Omega_{\fX/\fY}$ is almost derived $[p]_q$-adically complete and so the same is true for $R\Gamma_{A_{\Inf}}(\fX/\fY)$. Similarly by Lemma \ref{lem:pqderiomqcr}, $\mu_{\proet}^{-1}R\Gamma_{q-\textnormal{CRYS}}(\fX/\fY)$ is almost derived $[p]_q$-adically complete. Therefore by \emph{almost derived Nakayama} (cf. Lemma \ref{lem:almostderNaka}), it suffices to show that the induced morphism 
\begin{equation} \label{eq:afderNak}
\mu_{\proet}^{-1}R\Gamma_{q-\textnormal{CRYS}}(\fX/\fY) \otimes^{\bL}_{A_{\Inf}} A_{\Inf}/[p]_q \to R\Gamma_{A_{\Inf}}(\fX/\fY) \otimes^{\bL}_{A_{\Inf}} A_{\Inf}/[p]_q.
\end{equation}
is an almost quasi-isomorphism. The LHS of \eqref{eq:afderNak} can be rewritten as
\[
\mu_{\proet}^{-1}Rp_{1*}(q\Omega_{\fX/\fY} \otimes^{\bL}_{A_{\Inf}} A_{\Inf}/[p]_q).
\]
We have a spectral sequence 
\[
R^{i}p_{1*}H^{j}(q\Omega_{\fX/\fY} \otimes^{\bL}_{A_{\Inf}} A_{\Inf}/[p]_q) \implies R^{i+j}p_{1*}(q\Omega_{\fX/\fY} \otimes^{\bL}_{A_{\Inf}} A_{\Inf}/[p]_q)
\]
whose terms on the second page satisfy
\begin{align*}
    R^{i}p_{1*}H^{j}(q\Omega_{\fX/\fY} \otimes^{\bL}_{A_{\Inf}} A_{\Inf}/[p]_q) &\overset{(a)}{\cong}  R^{i}p_{1*}p_2^{*}\Omega^{j}_{\fX/\fY}\{-j\} \\
    &\overset{(b)}{\cong} 
    H^{i}(Lu^{q*}_{\fY}Rf_{*}\Omega^{j}_{\fX/\fY}\{-j\}) \\
    &\overset{(c)}{\cong} 
    H^{i}((u^{q}_{\fY})^{-1}Rf_{*}\Omega^{j}_{\fX/\fY}\{-j\} \otimes^{\bL}_{(u^{q}_{\fY})^{-1}\cO_{\fY_{\et}}} \ol{\cO}_{q-\textnormal{CRYS}}^{(1)}),
\end{align*}
where (a) follows from the (global) Hodge-Tate comparison (cf. Corollary \ref{cor:HTforqcrysde}), (b) follows from base change (cf. Proposition \ref{prop:bcqcrst}), noting that $\Omega^{j}_{\fX/\fY}$ is a locally free $\cO_{\fX_{\et}}$-module, and (c) is by definition of $u^{q*}_{\fY}$. We compute further
\begin{align*}
\mu^{-1}_{\proet}R^{i}p_{1*}H^{j}(q\Omega_{\fX/\fY} \otimes^{\bL}_{A_{\Inf}} A_{\Inf}/[p]_q) &\overset{(1)}{\cong} H^{i}(\nu_{Y}^{-1}Rf_{*}\Omega^{j}_{\fX/\fY}\{-j\} \otimes^{\bL}_{\nu_{Y}^{-1}\cO_{\fY_{\et}}} \mu^{-1}_{\proet}\ol{\cO}_{q-\textnormal{CRYS}}^{(1)}) \\
&\overset{(2)}{\cong} H^{i}(\nu_{Y}^{-1}Rf_{*}\Omega^{j}_{\fX/\fY}\{-j\} \otimes^{\bL}_{\nu_{Y}^{-1}\cO_{\fY_{\et}}} \wh{\cO}^{+}_{Y}) \otimes_{A_{\Inf}/\xi, \varphi} A_{\Inf}/[p]_q \\
&\overset{(3)}{\cong} H^{i}(L\nu_{Y}^{*}Rf_{*}\Omega^{j}_{\fX/\fY}\{-j\}) \otimes_{A_{\Inf}/\xi, \varphi} A_{\Inf}/[p]_q
\end{align*}
where (1) follows from Lemma \ref{lem:equivoftopoi}, (2) follows from the fact that $\mu^{-1}_{\proet}\ol{\cO}_{q-\textnormal{CRYS}}^{(1)}$ is the sheaf $\varprojlim_{i \in I} \Spa(S_{i}, S_{i}^{+}) \mapsto \bA_{\Inf}(S_{\infty})/[p]_q$, and (3) follows by definition upon viewing $\nu_{Y}$ as a morphism of ringed topoi $(Y_{\proet}, \wh{\cO}^{+}_{Y}) \to (\fY_{\et}, \cO_{\fY_{\et}})$.

By Corollary \ref{cor:pcomverorbascha}, the base change morphism 
 \[
 L\nu_{Y}^{*}Rf_{*}\Omega^{j}_{\fX/\fY}\{-j\} \to Rp_{*}Lq^{*}\Omega^{j}_{\fX/\fY}\{-j\}
 \]
is an almost isomorphism (via the canonical map $A_{\Inf} \to A_{\Inf}/\xi$, $[\fm^{\flat}]$ is sent to $\fm$). Therefore the expression on the RHS of (3) is almost isomorphic to
\[
R^{i}p_{*}q^{*}\Omega^{j}_{\fX/\fY}\{-j\} \otimes_{A_{\Inf}/\xi, \varphi} A_{\Inf}/[p]_q.
\]
Moreover by the Hodge-Tate comparison (cf. Theorem \ref{thm:HTspcssd}), there is an almost isomorphism 
\[
R^{i}p_{*}q^{*}\Omega^{j}_{\fX/\fY}\{-j\} \otimes_{A_{\Inf}/\xi, \varphi} A_{\Inf}/[p]_q \to R^{i}p_{*}H^{j}(\wt{\Omega}_{\fX / \fY}) \otimes_{A_{\Inf}/\xi, \varphi} A_{\Inf}/[p]_q.
\]
There is a spectral sequence (with second page terms)
\[
R^{i}p_{*}H^{j}(\wt{\Omega}_{\fX / \fY}) \otimes_{A_{\Inf}/\xi, \varphi} A_{\Inf}/[p]_q \implies Rp_{*}(\wt{\Omega}_{\fX / \fY}\otimes_{A_{\Inf}/\xi, \varphi} A_{\Inf}/[p]_q)
\]
By construction the morphism \eqref{eq:compmorAOMeqOMe} was shown to be compatible with the Hodge-Tate comparisons on the $q$-crystalline side and the pro-étale side, and so we conclude that there is an almost quasi-isomorphism 
\begin{equation} \label{eq:frisofsfs}
\mu_{\proet}^{-1}Rp_{1*}(q\Omega_{\fX/\fY} \otimes^{\bL}_{A_{\Inf}} A_{\Inf}/[p]_q) \to Rp_{*}(\wt{\Omega}_{\fX / \fY} \otimes_{A_{\Inf}/\xi, \varphi} A_{\Inf}/[p]_q).
\end{equation}
compatible with \eqref{eq:afderNak}.  Finally by \eqref{eq:hodgtatespecmap} the map 
\begin{equation} \label{eq:sefrisofsfs}
R\Gamma_{A_{\Inf}}(\fX/\fY) \otimes^{\bL}_{A_{\Inf}} A_{\Inf}/[p]_q \to Rp_{*}(\wt{\Omega}_{\fX / \fY} \otimes_{A_{\Inf}/\xi, \varphi} A_{\Inf}/[p]_q)
\end{equation}
is an almost quasi-isomorphism (with the twist by $\varphi$ being implicit there). Combining \eqref{eq:frisofsfs}-\eqref{eq:sefrisofsfs}, gives that \eqref{eq:afderNak} is an almost quasi-isomorphism.
\end{proof}

\subsection{Perfectness of $R\Gamma_{A_{\Inf}}(\fX/\fY)$}

In this short section we record (almost) perfectness of $R\Gamma_{A_{\Inf}}(\fX/\fY)$ together with base change (cf. Proposition \ref{prop:perfecAinfcohm} for the precise result). We will rely on the comparison with $q$-crystalline cohomology (cf. Theorem \ref{thm:comparofRGqcruandRGaAin}), however similar results can be obtained (albeit via a longer argument) using only the de Rham specialization of $A\Omega_{\fX/\fY}$ (cf. Theorem \ref{thm:regdeRhamse}). Almost in the next proposition is with respect to $[\fm^{\flat}] \subset A_{\Inf}$.
 
\begin{prop} \label{prop:perfecAinfcohm}
Assume that we have a proper smooth $\Spf(\cO)$-morphism $f \colon \fX \to \fY$ of $p$-adic formal schemes with $\fX$ and $\fY$ flat over $\cO$. Let $U \in Y_{\proet}$ be an affinoid perfectoid object determining the perfectoid space $\Spa(R,R^{+})$. Then 
\begin{enumerate} [(i)]
\item $R\Gamma(U, R\Gamma_{A_{\Inf}}(\fX/\fY))$ is almost quasi-isomorphic to a perfect object in $D(W(R^{\flat+}))$.
\item $R\Gamma_{A_{\Inf}}(\fX/\fY)\lvert_{U}$ is almost quasi-isomorphic to a strictly perfect object in $D(Y_{\proet}/U, W(\wh{\cO}^{+, \flat}_{Y}\lvert_{U}))$.
\item $R\Gamma_{A_{\Inf}}(\fX/\fY)$ is almost quasi-isomorphic to a perfect object in $D(Y_{\proet}, W(\wh{\cO}^{+, \flat}_{Y}))$.
    \item For a morphism of affinoid perfectoids $U' \to U$ (in $Y_{\proet}$) with $U'$ determining the perfectoid space $\Spa(R', R'^{+})$, the natural morphism 
    \[
    R\Gamma(U, R\Gamma_{A_{\Inf}}(\fX/\fY)) \otimes^{\bL}_{W(R^{\flat +})} W(R'^{\flat +}) \to R\Gamma(U', R\Gamma_{A_{\Inf}}(\fX/\fY))
\]
has the cohomology groups of its cone killed by $[\fm^{\flat}]$.
\end{enumerate}
\end{prop}

\begin{proof}
Parts (i)-(iii) are an immediate consequence of
\begin{enumerate}
    \item the analogous perfectness results of $R\Gamma_{q-\textnormal{CRYS}}(\fX/\fY)$ (cf. Proposition \ref{prop:perfectnessqcrco}(i)-(iii)) and
    \item the comparison of $R\Gamma_{A_{\Inf}}(\fX/\fY)$ with $R\Gamma_{q-\textnormal{CRYS}}(\fX/\fY)$ (cf. Theorem \ref{thm:comparofRGqcruandRGaAin}).
\end{enumerate}
Finally part (iv) follows from part (ii).
\end{proof}

\subsection{Comparison with relative $p$-adic étale cohomology}

Assume that we have a proper smooth $\Spf(\cO)$-morphism $f \colon \fX \to \fY$ of $p$-adic formal schemes locally of finite type and flat over $\cO$. In this section we compare $R\Gamma_{q-\textnormal{CRYS}}(\fX/\fY)$ with $p$-adic étale cohomology (i.e. $Rf_{\eta, \proet*}\wh{\bZ}_p$). This will involve considering the derived $p$-adic completions as $\mu \notin [\fm^{\flat}]$.

\begin{thm} \label{thm:qrysvsetale}
We have an identification
 \[
(\mu_{\proet}^{-1}R\Gamma_{q-\textnormal{CRYS}}(\fX/\fY))^{\wedge} \otimes_{A_{\Inf}}^{\bL} A_{\Inf}[\tfrac{1}{\mu}] \cong Rf_{\eta, \proet*}\wh{\bZ}_p \otimes_{\wh{\bZ}_p}^{\bL} \bA_{\Inf,Y}[\tfrac{1}{\mu}],
\]
where the completion on the LHS is derived $p$-adic.
\end{thm}

\begin{proof}
Denote by $R\Gamma_{A_{\Inf}}(\fX/\fY)^{\wedge}$ the derived $p$-adic completion of $R\Gamma_{A_{\Inf}}(\fX/\fY)$. Then Theorem \ref{thm:comparofRGqcruandRGaAin} and \cite[Lemma 3.17]{SemistabAinfcoh} induces a quasi-isomorphism
\[
(\mu_{\proet}^{-1}R\Gamma_{q-\textnormal{CRYS}}(\fX/\fY))^{\wedge} \otimes_{A_{\Inf}}^{\bL} A_{\Inf}[\tfrac{1}{\mu}] \cong R\Gamma_{A_{\Inf}}(\fX/\fY)^{\wedge} \otimes_{A_{\Inf}}^{\bL} A_{\Inf}[\tfrac{1}{\mu}]
\]
as $\mu \in W(\fm^{\flat})$. The result now follows from Theorem \ref{thm:padcohse}.
\end{proof}

\section{Applications}\label{sec:sec7applica}

\subsection{Relative Hodge-Tate spectral sequence}\label{sec:relHTspseq}

In this section we record a relative Hodge-Tate spectral sequence in Theorem \ref{thm:RelHTSSoincorf} and is independent of the $q$-crystalline section. This only uses the results developed in \S\ref{sec:TheHodhTaspe}. In the context of a discretely valued complete nonarchimedean extension of $\bQ_p$ with perfect residue field, such a sequence already exists by \cite[Corollary 2.2.4]{CaScGenSV} (cf. \cite[Remarque 1.4.7]{AbbGors}). In the context of schemes, results have been obtained in \cite{AbbGors}. It is an interesting question to understand the compatibility of these spectral sequences. Our approach to establishing the relative Hodge-Tate spectral sequence mirrors that of loc.cit. (we both pass through a fiber product of topoi) of which the following table summarises the similarities of the two approaches.
\begin{center}
    \begin{tabular}{ | p{7cm} | p{7cm} |}
    \hline
    \textbf{Our approach} & \textbf{Approach of Abbes-Gros} \\ \hline
    Use of pro-étale site & Use of Faltings' site \\ \hline
    Hodge-Tate specialization of $A\Omega_{\fX/\fY}$ (cf. Theorems \ref{thm:almosishodsa} and \ref{thm:HTspcssd}) & Hodge-Tate comparison of the \emph{structure sheaf} of the relative Faltings' topos (cf. \cite[Théorème 6.6.4]{AbbGors}) \\ \hline
    Base change for fiber product of topoi (cf. Proposition \ref{prop:bcrsmopn}) & Base change for the relative Faltings' site (cf. \cite[Théorème 6.5.31]{AbbGors})  \\
    \hline
    \end{tabular}
\end{center}
Let us mention one difference between the two approaches is that we get rid of the ``junk torsion" (via the décalage functor) while Abbes-Gros do not.

We begin by comparing the sheaf of K\"ahler differentials on the special and generic fibers.

\begin{lem} \label{lem:cohthmforadr}
Suppose $f \colon \fX \to \fY$ is a proper smooth morphism of $p$-adic formal schemes of type (S)(b) over $\cO$. Suppose $\fX$ and $\fY$ are flat over $\Spf(\cO)$ and $X$ and $Y$ are locally of finite type over $\Spa(C, \cO)$. Consider the commutative diagram of ringed topoi
\[
\xymatrix{
 (X, \cO_{X}) \ar@{->}[rr]^{\text{sp}_{\fX}} \ar@{->}[d]^{f_{\eta}} && (\fX, \cO_{\fX}) \ar@{->}[d]^{f_{\Zar}}   \\ 
(Y, \cO_{Y}) \ar@{->}[rr]^{\text{sp}_{\fY}} && (\fY, \cO_{\fY}),
 }
\]
and $\sF$ a locally free $\cO_{\fX}$-module. Then the canonical base change morphism
\[
 L\text{sp}_{\fY}^{*}Rf_{\Zar*}\sF \to Rf_{\eta*}L\text{sp}_{\fX}^{*}\sF
\]
is a quasi-isomorphism.
\end{lem}

\begin{proof}

First note that by \cite[{Tag 0A1H}]{stacks-project}, $Rf_{\Zar*}\sF \otimes^{\bL}_{\cO_{\fY}} \cO_{\fY}/p^n$ is a perfect object of $D(\cO_{\fY}/p^n)$ and its formation commutes with base change. Therefore by taking the derived limit we obtain $Rf_{\Zar*}\sF$ itself is a perfect object of $D(\cO_{\fY})$. Therefore $L\text{sp}_{\fY}^{*}Rf_{\Zar*}\sF$ is a perfect object of $D(\cO_{Y})$.

The claim is local for the Zariski topology on $\fY$, so we can assume it is affine, say $\fY = \Spf(S)$.
Let $U \subseteq Y$ be any open affinoid of $Y$ such that $L\text{sp}_{\fY}^{*}Rf_{\Zar*}\sF \lvert_{U}$ is strictly perfect. Then for any rational subset $U' \subseteq U$, by Tate acyclicity
\begin{equation} \label{eq:frisfsifsds}
R\Gamma(U, L\text{sp}_{\fY}^{*}Rf_{\Zar*}\sF) \otimes_{\cO_{Y}(U)}^{\bL} \cO_{Y}(U') \cong R\Gamma(U', L\text{sp}_{\fY}^{*}Rf_{\Zar*}\sF).
\end{equation}
Moreover by Kiehl's finiteness theorem \cite[Theorem 3.3]{Kiehl} 
\begin{equation} \label{eq:rhscohrel}
 (R^{i}f_{\eta*}\text{sp}_{\fX}^{*}\sF)(U) \otimes_{O_{Y}(U)} \cO_{Y}(U') = (R^{i}f_{\eta*}\text{sp}_{\fX}^{*}\sF)(U').
\end{equation}
Thus by Tate acyclicity and flatness of $\cO_{Y}(U')$ over $\cO_{Y}(U)$ (cf. \cite[pg. 530]{Hugfs}), \eqref{eq:rhscohrel} implies 
\begin{equation} \label{eq:seconequsds}
    R\Gamma(U, Rf_{\eta*}L\text{sp}_{\fX}^{*}\sF) \otimes_{\cO_{Y}(U)}^{\bL} \cO_{Y}(U') \cong R\Gamma(U', Rf_{\eta*}L\text{sp}_{\fX}^{*}\sF).
\end{equation}
Therefore \eqref{eq:frisfsifsds} and \eqref{eq:seconequsds} implies that it suffices to check that the induced morphism 
\[
 R\Gamma(U, L\text{sp}_{\fY}^{*}Rf_{\Zar*}\sF) \to R\Gamma(U, Rf_{\eta*}L\text{sp}_{\fX}^{*}\sF)
\]
is a quasi-isomorphism where in addition $U$ comes from some affine open $\fU \subseteq \fY$.

We can assume without loss of generality that $\fU = \fY$ and $Rf_{\Zar*}\sF$ is represented by a strictly perfect complex $\sE^{\bullet}$. In that case $R\Gamma(U, L\text{sp}_{\fY}^{*}Rf_{\Zar*}\sF)$ is represented by the complex $\sE^{\bullet}(\fY) \otimes_S S[\tfrac{1}{p}]$, which in turn is just $R\Gamma(\fY, Rf_{\Zar*}\sF) \otimes_{S}^{\bL} S[\tfrac{1}{p}]$. Choose some finite open affine cover of $\fX$. Then $R\Gamma(\fY, Rf_{\Zar*}\sF)$ is computed via the associated \v{C}ech complex $C^{\bullet}$ and $R\Gamma(U, Rf_{\eta*}L\text{sp}_{\fX}^{*}\sF)$ is computed by $C^{\bullet} \otimes_{S} S[\tfrac{1}{p}]$. Since $S[\tfrac{1}{p}]$ is flat over $S$, the result now follows.
\end{proof}

\begin{thm} \label{thm:RelHTSSoincorf}
Suppose $f \colon \fX \to \fY$ is a proper smooth morphism of $p$-adic formal schemes of type (S)(b) over $\cO$. Assume that $\fX$ and $\fY$ are flat over $\cO$. There is an $E_2$-spectral sequence 
\begin{equation} \label{eq:firsSSthm}
E_2^{i,j} = H^{i}(L\nu_{Y}^{*}Rf_{*}\Omega^{j}_{\fX/\fY}\{-j\}) \implies H^{i+j}(R\Gamma_{A_{\Inf}}(\fX/\fY) \otimes^{\bL}_{A_{\Inf}, \theta \circ \varphi^{-1}} \cO)
\end{equation}
in $D(Y_{\proet}, \cO^a)$. Moreover, if both $\fX$ and $\fY$ are locally of finite type over $\cO$ then upon inverting $p$, \eqref{eq:firsSSthm} takes the shape
\begin{equation} \label{eq:secondSSthsmd}
E_2^{i,j} = H^{i}(Rf_{\eta*}\Omega^{j}_{X/Y} \otimes^{\bL}_{\cO_{Y}} \wh{\cO}_{Y}(-j)) \implies R^{i+j}f_{\eta, \proet *}\wh{\bZ}_p \otimes_{\wh{\bZ}_p} \wh{\cO}_{Y}
\end{equation}
in $D(Y_{\proet}, C)$ (here the $f_{\eta *}$ is pushforward between analytic topologies). In addition, if for any $i,j \geq 0$, $R^{i}f_{\eta*}\Omega^{j}_{X/Y}$ is locally free of finite type over $\cO_{Y}$, then \eqref{eq:secondSSthsmd} takes the shape
\begin{equation} \label{eq:thirdsStfsdm}
E_2^{i,j} = R^{i}f_{\eta*}\Omega^{j}_{X/Y} \otimes_{\cO_{Y}} \wh{\cO}_{Y}(-j) \implies R^{i+j}f_{\eta, \proet *}\wh{\bZ}_p \otimes_{\wh{\bZ}_p} \wh{\cO}_{Y}
\end{equation}
and it degenerates.
\end{thm}

\begin{rem}
Let $K$ be a discretely valued complete nonarchimedean extension of $\bQ_p$ with perfect residue field and ring of integers $\cO_K$. Then if $f_{\eta}$ comes from base change of a proper smooth morphism of smooth adic spaces over $\Spa(K, \cO_K)$, then by \cite[Theorem 8.8(ii)]{SchpHrig} and base change, for any $i,j \geq 0$, $R^{i}f_{\eta*}\Omega^{j}_{X/Y}$ is locally free of finite type over $\cO_{Y}$.
\end{rem}

\begin{proof}
By Theorem \ref{thm:almosishodsa} and the projection formula (cf. \cite[{Tag 0944}]{stacks-project}) we get a morphism in $D(Y_{\proet}, \cO)$
\[
R\Gamma_{A_{\Inf}}(\fX/\fY) \otimes^{\bL}_{A_{\Inf}, \theta \circ \varphi^{-1}} \cO \to Rp_{*}\wt{\Omega}_{\fX / \fY}
\]
whose image in $D(Y_{\proet}, \cO^a)$ is an almost isomorphism. Therefore we obtain an $E_2$-spectral sequence (the Cartan-Leray spectral sequence)
\[
R^{i}p_{*}H^{j}(\wt{\Omega}_{\fX / \fY}) \implies H^{i+j}(R\Gamma_{A_{\Inf}}(\fX/\fY) \otimes^{\bL}_{A_{\Inf}, \theta \circ \varphi^{-1}} \cO)
\]
in $D(Y_{\proet}, \cO^{a})$. By Theorem \ref{thm:HTspcssd}, we get a morphism $R^{i}p_{*}q^{*}\Omega^{j}_{\fX/\fY}\{-j\} \to R^{i}p_{*}H^{j}(\wt{\Omega}_{\fX / \fY})$ in $D(Y_{\proet}, \cO)$, whose image in $D(Y_{\proet}, \cO^a)$ is an almost isomorphism. By Corollary \ref{cor:pcomverorbascha}, the morphism 
 \[
 H^{i}(L\nu_{Y}^{*}Rf_{*}\Omega^{j}_{\fX/\fY}\{-j\}) \to R^{i}p_{*}q^{*}\Omega^{j}_{\fX/\fY}\{-j\}
 \]
is an almost isomorphism. This proves \eqref{eq:firsSSthm}.

Suppose now both $X$ and $Y$ are locally of finite type over $\Spa(C, \cO)$.  We now invert $p$ and perform some yoga to obtain \eqref{eq:secondSSthsmd}. Indeed 
\begin{align*}
    H^{i}(L\nu_{Y}^{*}Rf_{*}\Omega^{j}_{\fX/\fY}\{-j\}) \otimes_{\cO} C &\overset{(i)}{\cong} H^{i}(L\nu_{Y}^{*}Rf_{*}\Omega^{j}_{\fX/\fY}\{-j\} \otimes_{\cO}^{\bL} C) \\
    &\overset{(ii)}{\cong} H^{i}(Rf_{*}\Omega^{j}_{\fX/\fY} \otimes_{\cO_{\fY_{\et}}}^{\bL} \wh{\cO}_{Y}(-j)) \\
    &\overset{(iii)}{\cong}
    H^{i}(Rf_{\Zar,*}\Omega^{j}_{\fX/\fY} \otimes^{\bL}_{\cO_{\fY}} \wh{\cO}_{Y}(-j)) \\
    &\overset{(iv)}{\cong}
    H^{i}(Rf_{\eta*}\Omega^{j}_{X/Y} \otimes^{\bL}_{\cO_Y} \wh{\cO}_{Y}(-j))
\end{align*}
where (i) follows from localization being flat, (ii) is by definition of $L\nu_{Y}^{*}$ (we have dropped the $\nu_Y^{-1}$) and the fact that the cokernel of $\cO(1) \hookrightarrow \cO \{1\}$ is killed by $p^{1/(p-1)}$, (iii) from the fact that coherent cohomology can be computed in the étale or Zariski site, and (iv) is a consequence of Lemma \ref{lem:cohthmforadr}. The identification 
\[
H^{i+j}(R\Gamma_{A_{\Inf}}(\fX/\fY) \otimes^{\bL}_{A_{\Inf}, \theta \circ \varphi^{-1}} \cO)[\tfrac{1}{p}] \cong R^{i+j}f_{\eta, \proet *}\wh{\bZ}_p \otimes_{\wh{\bZ}_p} \wh{\cO}_{Y}
\]
follows from Theorem \ref{thm:padcohse} (note that $\mu$ is invertible in $A_{\Inf}/\tilde{\xi}$). This proves \eqref{eq:secondSSthsmd} of which \eqref{eq:thirdsStfsdm} is a direct consequence.

We now prove degeneration of \eqref{eq:thirdsStfsdm}. By Lemma \ref{lem:procomsisasli} the complex $Rf_{\eta, \proet*}\wh{\bZ}_p$ is adic lisse and in particular it is locally a perfect complex of $\wh{\bZ}_p$-modules. Thus 
\[
R^{i+j}f_{\eta, \proet *}\wh{\bZ}_p \otimes_{\wh{\bZ}_p} \wh{\cO}_{Y} = R^{i+j}f_{\eta, \proet *}\wh{\bZ}_p \otimes_{\wh{\bZ}_p} \wh{\bZ}_p[\tfrac{1}{p}] \otimes_{\wh{\bZ}_p[\tfrac{1}{p}]} \wh{\cO}_{Y}
\]
is a locally free $\wh{\cO}_Y$-module. Recall that locally free sheaves of finite rank on $Y_{\proet}$ correspond to locally free sheaves of finite rank on the analytic topology on $Y$ by \cite[Lemma 7.3]{SchpHrig}). So the rank (locally) of such objects is determined by the rank at geometric points of rank 1. Fix an integer $m \geq 0$. By \cite[Proposition 2.6.1]{HubEtadic} applied\footnote{It is enough to look at $R^{i+j}f_{\eta, \et*}(\bZ/p^n)$ for each $n$ as $Rf_{\eta, \proet*}\wh{\bZ}_p$ is adic lisse.\label{seconfs}} to $R^{i+j}f_{\eta, \et*}(\bZ/p^n)$ for each $n$, commutation with arbitrary base change for the terms $R^{i}f_{\eta*}\Omega^{j}_{X/Y}$ (by locally freeness assumption) and the degeneration of the Hodge-Tate spectral sequence (cf. \cite[Theorem 13.3(ii)]{BhMorSch}), we get a covering $\{U_{k} \to Y \}_{k}$ in $Y_{\proet}$ such that restricted to $U_k$ each term in the quadrant $i+j \leq m$ of the second page of \eqref{eq:thirdsStfsdm} is free (as well as the abutment $R^{i+j}f_{\eta, \proet *}\wh{\bZ}_p \otimes_{\wh{\bZ}_p} \wh{\cO}_{Y}$) and 
\begin{equation} \label{eq:sumofrankequalfs}
\sum_{i+j = l} \textrm{rank}(R^{i}f_{\eta*}\Omega^{j}_{X/Y} \otimes_{\cO_{Y}} \wh{\cO}_{Y}(-j)\lvert_{U_{k}}) = \textrm{rank}(R^{l}f_{\eta, \proet *}\wh{\bZ}_p \otimes_{\wh{\bZ}_p} \wh{\cO}_{Y}\lvert_{U_{k}})
\end{equation}
for all $k$ and any $ 0 \leq l \leq m$. Since affinoid perfectoids form a basis for $Y_{\proet}$, we can assume further that each $U_k$ is affinoid perfectoid and determines a perfectoid space $\Spa(S_{k}, S_{k}^{+})$. Then by vanishing of higher cohomology (cf. \cite[Lemma 4.10(v)]{SchpHrig}), taking the $U_k$-sections of \eqref{eq:thirdsStfsdm} gives a spectral sequence consisting of finite free $S_k$-modules. It suffices to show that the resulting spectral sequence degenerates. Consider a differential $d_{2}^{i,j} \colon E_{2}^{i,j}(U_k) = S_{k}^{a} \to E_{2}^{i+2, j-1}(U_k) = S_{k}^{b}$ for any $i+j \leq m-1$. Then for each minimal prime ideal $\fp$ of $S_k$, the localization $S_{k, \fp}$ is a field as $S_k$ is reduced (cf. \cite[{Tag 00EU}]{stacks-project}). Thus by \eqref{eq:sumofrankequalfs}, $d_{2}^{i,j}$ localized at $\fp$ is the zero morphism. Thus the induced morphism
\begin{equation} \label{eq:upgrsfmod}
\prod_{\fp \text{ minimal}} (S_{k, \fp})^{a} \to  \prod_{\fp \text{ minimal}} (S_{k, \fp})^{b},
\end{equation}
where the products run over all minimal prime ideals of $S_k$, is the zero morphism. Moreover $S_k \to \prod_{\fp \text{ minimal}} S_{k, \fp}$ is an embedding (cf. \cite[{Tag 00EW}]{stacks-project}). Therefore $d_{2}^{i,j}$ is the zero morphism. This proves degeneration.
\end{proof}

\subsection{Torsion in étale pushforward}

The result in this section makes use of the $q$-crystalline theory from \S\ref{$q$-crystalline sitewed}. Although one could avoid this by using the de Rham specialization for $A\Omega_{\fX/\fY}$ (cf. Theorem \ref{thm:regdeRhamse}), one would still need to prove a perfectness result such as Proposition \ref{prop:perfecAinfcohm}.

\begin{thm} \label{thm:torsionbigerindic}
Suppose $f \colon \fX \to \fY$ is a proper smooth morphism of $p$-adic formal schemes locally of finite type and flat over $\cO$. Fix some $j \geq 0$. Suppose that $R^{i}f_{\dR*}\cO_{\fX}$ is a finite locally free $\cO_{\fY}$-module for all $i \geq j$. Then for all $i \geq j$
\begin{enumerate}
    \item $R^{i}f_{\eta, \proet *}\wh{\bZ}_p$ is $p$-torsion free, 
    \item if $\varprojlim_{i \in I} \Spa(S_{i}, S_{i}^{+})$ is an affinoid perfectoid of $Y_{\proet}$ determining the perfectoid space $\Spa(S_{\infty}[\tfrac{1}{p}], S_{\infty})$, then $H^{i}(R\Gamma((\bA_{\Inf}(S_{\infty}),(\xi)), R\Gamma_{q-\textnormal{CRYS}}(\fX/\fY)))$ is a finite locally free $\bA_{\Inf}(S_{\infty})$-module and
    \item if $\fY$ is in addition smooth and $(E,J)$ is any object in $(\fY/A_{\Inf})_{q-\textnormal{CRYS}}$, then 
    
    $H^{i}(R\Gamma((E,J), R\Gamma_{q-\textnormal{CRYS}}(\fX/\fY)))$ is a finite locally free $E$-module.
\end{enumerate}

\begin{rem} \label{rem:indivaTsjma}
A slight variant of part (2) is also proved in \cite[Lemma 1.38]{MorrTsuj} by another method. The key point in the argument of loc.cit. is to take a judicious base change to a rank 1 perfectoid valuation ring, at which point the result of \cite[Theorem 1.8]{BhMorSch} can be applied. We take a different approach and instead rely on the comparison with relative $p$-adic étale cohomology (cf. Theorem \ref{thm:qrysvsetale}). T. Tsuji also pointed out that (1) immediately follows from the analagous result in \cite[Theorem 14.5(ii)]{BhMorSch} (cf. Remark 14.4 in loc.cit.) together with \cite[Proposition 2.6.1]{HubEtadic} and Lemma \ref{lem:procomsisasli}.
\end{rem}

\end{thm}

\begin{proof}
We first show part (2) and we use the notation from \S\ref{$q$-crystalline sitewed}. Let $(E,J)$ be an object in $(\fY/A_{\Inf})_{q-\textnormal{CRYS}}$. Then by the de Rham comparison with $q$-crystalline/prismatic cohomology (cf. \cite[Corollary 15.4]{Prisms}) and \eqref{eq:pushforisjusqcr}
\begin{equation}
\label{eq:psuhfsomf}
R\Gamma_{\dR}((\Spf(E/[p]_{q}) \times_{\fY} \fX)/\Spf(E/[p]_{q}))
 \cong R\Gamma((E,J), R\Gamma_{q-\textnormal{CRYS}}(\fX/\fY)) \otimes^{\bL}_{A_{\Inf}, \theta} \cO.
\end{equation}
Note that by Proposition \ref{prop:perfectnessqcrco}(i), the RHS of \eqref{eq:psuhfsomf} is already ($p$-adically derived) complete. Since $R^{i}f_{\dR*}\cO_{\fX}$ is a finite locally free $\cO_{\fY}$-module for all $i \geq j$ and $Rf_{\dR *}\cO_{\fX}$ is a perfect complex and commutes with arbitrary base change, $R^{i}f_{\dR*}\cO_{\fX}$ commutes with arbitrary base change and so $R^{i}f'_{\dR*}\cO_{\Spf(E/[p]_q) \times_{\fY} \fX}$ is a  finite locally free $\cO_{\Spf(E/[p]_q)}$-module (here $f' \colon \Spf(E/[p]_q) \times_{\fY} \fX \to \Spf(E/[p]_q)$ is obtained from base change) for all $i \geq j$. By affine vanishing of higher cohomology, it follows that $H^{i}$ of the LHS of \eqref{eq:psuhfsomf} is a finite locally free $E/[p]_q$-module for all $i \geq j$. Therefore 
\begin{equation} \label{eq:adikiandas}
H^{i}(R\Gamma((E,J), R\Gamma_{q-\textnormal{CRYS}}(\fX/\fY)) \otimes^{\bL}_{A_{\Inf}, \theta} \cO)
\end{equation}
is a finite locally free $E/[p]_q$-module for all $i \geq j$.

 Let $\varprojlim_{i \in I} \Spa(S_{i}, S_{i}^{+})$ be an affinoid perfectoid of $Y_{\proet}$ determining the perfectoid space $\Spa(S_{\infty}[\tfrac{1}{p}], S_{\infty})$ and suppose $(E,J) = (\bA_{\Inf}(S_{\infty}), (\xi))$. We now proceed by descending induction on $i \geq j$ to prove that 
\[
H^{i}:= H^{i}(R\Gamma((E,J), R\Gamma_{q-\textnormal{CRYS}}(\fX/\fY)))
\]
is a finite locally free $E$-module. Since $R\Gamma((E,J), R\Gamma_{q-\textnormal{CRYS}}(\fX/\fY))$ is a perfect complex in $D(E)$ (cf. Proposition \ref{prop:perfectnessqcrco}(i)), it is quasi-isomorphic to a complex $E^{\bullet}$ of finite projective $E$-modules with $E^{k} = 0$ for $k \geq m+1$ for some $m \geq i$. Then $H^{m}$ is a finite $E$-module and $H^{m} \otimes_{A_{\Inf}, \theta} \cO$ coincides with \eqref{eq:adikiandas} for $i = m$. Assume without loss of generality $H^{m} \otimes_{A_{\Inf}, \theta} \cO$ is a finite free $E/[p]_q$-module and choose a basis $\ol{B}$ (this does not affect the form of $(E,J)$, as it is enough to take a Zariski cover of $\fY$). By Nakayama's lemma, the basis $\ol{B}$ lifts to a generating set $B$ of $H^m$. Moreover by Proposition \ref{prop:perfectnessqcrco}(ii), \cite[Theorem 6.5(ii)]{SchpHrig} (in particular the sentence immediately following \cite[Lemma 5.6]{BhMorSch}) and the equality $\mu_{\proet}^{-1}\cO_{q-\textnormal{CRYS}} = W(\wh{\cO}^{+, \flat}_{Y})$ we get 
\begin{equation} \label{eq:equaltenspfro}
H^{m}[\tfrac{1}{\mu}] = H^{m}(R\Gamma(\varprojlim_{i \in I} \Spa(S_{i}, S_{i}^{+}), (\mu^{-1}_{\proet}R\Gamma_{q-\textnormal{CRYS}}(\fX/\fY))^{\wedge}))[\tfrac{1}{\mu}].
\end{equation}

By Theorem \ref{thm:qrysvsetale}, \eqref{eq:equaltenspfro} implies $H^{m}[\tfrac{1}{\mu}, \tfrac{1}{p}]$ is the $E[\tfrac{1}{\mu}, \tfrac{1}{p}]$-module 
\begin{equation} \label{eq:somkinexp}
(R^{m}f_{\eta, \proet*}\wh{\bZ}_p)(\varprojlim_{i \in I}\Spa(S_{i},S_{i}^{+})) \otimes_{(\wh{\bZ}_p)(\varprojlim_{i \in I}\Spa(S_{i},S_{i}^{+}))} E[\tfrac{1}{\mu}, \tfrac{1}{p}].
\end{equation}
Since $Rf_{\eta, \proet*}\wh{\bZ}_p$ is adic lisse we can trivialize it (locally on $Y_{\proet}$) and after applying base change as in footnote \ref{seconfs}, \eqref{eq:somkinexp} base changed to $\bA_{\Inf}(S_{\infty}')[\tfrac{1}{\mu}, \tfrac{1}{p}]$ (where $\varprojlim_{i' \in I'} \Spa(S'_{i'}, {S'_{i'}}^{+}) \to \varprojlim_{i \in I} \Spa(S_{i}, S_{i}^{+})$ is a pro-étale map of affinoid perfectoids and $\varprojlim_{i' \in I'} \Spa(S'_{i'}, {S'_{i'}}^{+})$ determines the perfectoid space $\Spa(S_{\infty}'[\tfrac{1}{p}], S_{\infty}')$) is a free $\bA_{\Inf}(S_{\infty}')[\tfrac{1}{\mu}, \tfrac{1}{p}]$-module of rank
\begin{equation} \label{eq:somekdsa}
 \textrm{rank}_{\bZ_p}(( H^{m}_{\et}(X \times_{Y} \Spa(K, K^{\circ}), \bZ_p)))  
\end{equation} 
for any rank $1$ geometric point $\Spa(K, K^{\circ})$ of $Y$ which lies above $\varprojlim_{i' \in I'}\Spa(S'_{i'},{S'_{i'}}^{+})$. By the comparison of $B^{+}_{\dR}$-cohomology with étale cohomology (cf. \cite[Theorem 13.1]{BhMorSch}) and the fact that $B^{+}_{\dR}$-cohomology specializes to de Rham cohomology (cf. the expression immediately after \cite[Definition 13.18]{BhMorSch}), \eqref{eq:somekdsa} is equal\footnote{One also uses Lemma \ref{lem:cohthmforadr} to rewrite $R^{i}f'_{\dR*}\cO_{\Spf(E/[p]_q) \times_{\fY} \fX}$ as the de Rham pushforward along the generic fiber.} to $|\ol{B}|$. Now
\begin{equation} \label{eq:valusdasheac}
\varprojlim_{i \in J}\Spa(T_{i},T_{i}^{+}) \mapsto (R^{m}f_{\eta, \proet*}\wh{\bZ}_p)(\varprojlim_{i \in J}\Spa(T_{i},T_{i}^{+})) \otimes_{(\wh{\bZ}_p)(\varprojlim_{i \in J}\Spa(T_{i},T_{i}^{+}))} \bA_{\Inf}(T_{\infty})[\tfrac{1}{\mu}, \tfrac{1}{p}]
\end{equation}
is a sheaf on affinoid perfectoids in $Y_{\proet}$. Since \eqref{eq:somekdsa} equals $|\ol{B}| = |B|$, the generating set $B$ is a basis (pro-étale locally on $\varprojlim_{i \in I}\Spa(S_{i},S_{i}^{+})$) for the sections of \eqref{eq:valusdasheac}. Therefore, it must be a basis for  \eqref{eq:somkinexp}. Hence, as $\bA_{\Inf}(S_{\infty})$ is $\mu$-torsion (and $p$-torsion) free, it is a basis for $H^m$. This proves the base case of the induction. For the inductive step, one repeats the argument noting that the locally finite freeness of $H^{i}$ for $i > i' \geq j$ implies $H^{i'} \otimes_{A_{\Inf}, \theta} \cO$ coincides with \eqref{eq:adikiandas} for $i = i'$. Moreover since $R\Gamma((E,J), R\Gamma_{q-\textnormal{CRYS}}(\fX/\fY))$ is a perfect complex in $D(E)$, the inductive step also shows $H^{i'}$ is a finite $E$-module. This proves (2). 

Part (1) then follows from Theorem \ref{thm:qrysvsetale}: indeed 
\[
(R^{i}f_{\eta, \proet*}\wh{\bZ}_p \otimes_{\wh{\bZ}_p}^{\bL} \bA_{\Inf,Y}[\tfrac{1}{\mu}])(\varprojlim_{i \in I}\Spa(S_{i},S_{i}^{+}))
\]
is a finite locally free $\bA_{\Inf}(S_{\infty})[\tfrac{1}{\mu}]$-module for all $i \geq j$ and so $( R^{i}f_{\eta, \proet*}\wh{\bZ}_p)(\varprojlim_{i \in I}\Spa(S_{i},S_{i}^{+}))$ is $p$-torsion free.

We now show part (3) by descent to the case of part (2). First we show the map $S \to S_{\infty}$ (cf. Remark \ref{rem:neksinmors}) is quasisyntomic (for small enough Zariski neighbourhood $\Spf(S)$ of $\fY$ and a choice of $S_{\infty}$). Since $\fY$ is assumed to be smooth, we can take $S$ to be a small $p$-completely smooth $\cO$-algebra (recall this means that there is an étale map $\Spf(S) \to \Spf(\cO \{t_1^{\pm 1}, \ldots, t_n^{\pm 1} \})$). Above $\Spa(C\{t_1^{\pm 1}, \ldots, t_n^{\pm 1} \}, \cO \{t_1^{\pm 1}, \ldots, t_n^{\pm 1} \})$, we have the affinoid perfectoid covering 
\[
\varprojlim_{m \geq 0} \Spa(C\{t_1^{\pm 1/p^m}, \ldots, t_n^{\pm 1/p^m} \}, \cO \{t_1^{\pm 1/p^m}, \ldots, t_n^{\pm 1/p^m} \}).
\] 
The induced map $\cO \{t_1^{\pm 1}, \ldots, t_n^{\pm 1} \} \to \cO \{t_1^{\pm 1/p^{\infty}}, \ldots, t_n^{\pm 1/p^{\infty}} \}$ is quasisyntomic. The map $S \to S_{\infty}$, induced by ($p$-completed) base change, is again quasisyntomic by \cite[Lemma 4.16(2)]{bhatmorscho2}. Now suppose $\Spf(E/J) \to \fY$ factors through $\Spf(S)$. Then again by loc.cit. the induced map $E/\tilde{\xi} \to S_{\infty} \wh{\otimes}_{S^{(1)}} E/\tilde{\xi}$ is quasisyntomic (the completion is $p$-adic and $S^{(1)} := S \otimes_{A_{\Inf}/\xi, \varphi} A_{\Inf}/\tilde{\xi}$). By \cite[Proposition 7.11]{Prisms}, there exists an object
\[
(E' \to E'/\tilde{\xi} \leftarrow S_{\infty} \wh{\otimes}_{S^{(1)}} E/\tilde{\xi}) \in ((S_{\infty} \wh{\otimes}_{S^{(1)}} E/\tilde{\xi})/E)_{\Prism},
\]
and moreover the map $S_{\infty} \wh{\otimes}_{S^{(1)}} E/\tilde{\xi}) \to E'/\varphi (\xi)$ is $p$-completely faithfully flat. Therefore the composition $E/\tilde{\xi} \to E'/\tilde{\xi}$ is also $p$-completely faithfully flat.
By universality (cf. Lemma 4.7 in loc.cit.), the induced map $S_{\infty} \to E'/\tilde{\xi}$ lifts uniquely to a map $(\bA_{\Inf}(S_{\infty}), \tilde{\xi}) \to (E', \tilde{\xi})$ of prisms. Therefore by part (2)
\begin{equation} \label{eq:onksishir}
H^{i}(R\Gamma((\bA_{\Inf}(S_{\infty}),(\xi)), R\Gamma_{q-\textnormal{CRYS}}(\fX/\fY))) \otimes_{\bA_{\Inf}(S_{\infty})} E'
\end{equation}
is a finite locally free $E'$-module. The expression \eqref{eq:onksishir} coincides with
\[
H^{i}(R\Gamma((E,J), R\Gamma_{q-\textnormal{CRYS}}(\fX/\fY))) \wh{\otimes}_{E} E',
\]
where the completion is $(p, \tilde{\xi})$-adic (to see this, identify $q$-crystalline cohomology with prismatic cohomology \eqref{eq:pushforisjusqcr}-\eqref{eq:globalizationdsfs} and apply \cite[Theorem 1.8(5)]{Prisms}).
By a similar argument as the proof of \cite[Lemma A.12]{PrismaticDieudonn} (in particular using Lemma A.8 of loc.cit.), one concludes that 
\[
H^{i}(R\Gamma((E,J), R\Gamma_{q-\textnormal{CRYS}}(\fX/\fY)))
\]
is a finite locally free $E$-module. In general (in the case when the morphism $\Spf(E/J) \to \fY$ does not necessarily factor through $\Spf(S)$), there is a covering $(E', J') \to (E,J)$ (i.e. $\Spf(E'/J') \to \Spf(E/J)$ is a $p$-completely étale covering), such that 
$H^{i}(R\Gamma((E',J'), R\Gamma_{q-\textnormal{CRYS}}(\fX/\fY)))$ is a finite locally free $E'$-module. One then repeats the argument as the proof of \cite[Lemma A.12]{PrismaticDieudonn}, as before.
\end{proof}

\begin{rem}
In the setting of Theorem \ref{thm:torsionbigerindic}, there is also a Frobenius structure on 
\begin{equation} \label{eq:relqcryrelbkf}
H^{i}(R\Gamma((\bA_{\Inf}(S_{\infty}),(\xi)), R\Gamma_{q-\textnormal{CRYS}}(\fX/\fY)))
\end{equation}
and by the isogeny theorem  (cf. \cite[Theorem 1.8(6)]{Prisms}), the linearization morphism (induced by Frobenius) is an isomorphism after inverting $\tilde{\xi}$. 
Together with part (2) of Theorem \ref{thm:torsionbigerindic}, the expression \eqref{eq:relqcryrelbkf} gives rise to a \emph{relative Breuil-Kisin-Fargues module}. This is made precise in \cite{MorrTsuj}, where various categories (e.g. $\mu$-small generalised representations over $\bA_{\Inf}(S_{\infty})$ for a particular $S_{\infty}$ and relative Breuil-Kisin-Fargues modules with Frobenius structure) are shown to be equivalent. In the language of loc.cit. \eqref{eq:relqcryrelbkf} can be upgraded to a generalised representation over $\bA_{\Inf}(S_{\infty})$ (with the $\bZ_p(1)^{d}$-action\footnote{In this context, one takes a small Zariski neighbourhood $\Spf(S)$ of $Y$ which can be described using $d$-tori-coordinates. The $S$-algebra $S_{\infty}$ is then obtained from $S$ by adjoining all $p$-th roots of the tori-coordinates and this gives the usual $\bZ_p(1)^{d}$-action on  $\bA_{\inf}(S_{\infty})$, which in turn induces a $\bZ_p(1)^{d}$-action on \eqref{eq:relqcryrelbkf}.}  satisfying a certain triviality condition modulo $\mu$) with Frobenius structure.
\end{rem}

\subsection{The relative absolute crystalline comparison}

Let $f \colon \fX \to \fY$ be a smooth $\Spf(\cO)$-morphism of $p$-adic formal schemes of type (S)(b). Assume that $\fX$ and $\fY$ are flat over $\cO$. In this section we prove a relative version of the absolute crystalline comparison isomorphism (cf. \cite[Theorem 12.1]{BhMorSch}). 

As usual let $X := \fX \times_{\Spf(\cO)} \Spec (\cO/p)$ and $Y := \fY \times_{\Spf(\cO)} \Spec(\cO/p)$ be the fibers modulo $p$ of $\fX$ and $\fY$, respectively. In addition let $\textnormal{CRYS}(X/A_{\textnormal{crys}})$ and $\textnormal{CRYS}(Y/A_{\textnormal{crys}})$ be the absolute (big) crystalline sites of $X$ and $Y$, respectively. We denote the corresponding topoi by $(X/A_{\textnormal{crys}})_{\textnormal{CRYS}}$ (resp. $ (Y/A_{\textnormal{crys}})_{\textnormal{CRYS}}$). Let $\cO^{\textnormal{CRYS}}_{X/\bZ_p}$ be the structure sheaf on $ (X/A_{\textnormal{crys}})_{\textnormal{CRYS}}$. The morphism $f$ induces a functor of topoi
\[
f^{\textnormal{CRYS}} \colon (X/A_{\textnormal{crys}})_{\textnormal{CRYS}} \to (Y/A_{\textnormal{crys}})_{\textnormal{CRYS}}.
\]
Let $Y'_{\proet} \subset Y_{\proet}$ be the subsite consisting of affinoid perfectoid objects. By killing $\mu = q - 1$ in the target of $u_{\proet}$ in \S \ref{sec:compwithrelAinf}, we obtain a cocontinuous functor of sites
\begin{align*}
    u_{\proet}^{\textnormal{CRYS}} \colon Y'_{\proet} &\to \textnormal{CRYS}(Y/A_{\textnormal{crys}}) \\
    \varprojlim_{i \in I} \Spa(S_{i}, S_{i}^{+})  &\mapsto (\bA_{\Inf}(S_{\infty})/\mu, (\xi/\mu)) 
\end{align*}
which induces a morphism of topoi
\[
\mu_{\proet}^{\textnormal{CRYS}} \colon Y_{\proet} \to (Y/A_{\textnormal{crys}})_{\textnormal{CRYS}}.
\]
We can now state the absolute crystalline comparison in the current context:

\begin{thm}
Assume that we have a smooth $\Spf(\cO)$-morphism $f \colon \fX \to \fY$ of $p$-adic formal schemes of type (S)(b). Suppose in addition that $\fX$ and $\fY$ are flat over $\cO$. There is a canonical functorial morphism (the completion on the target is derived $p$-adic)
\begin{equation} \label{eq:crysverssht}
(\mu^{\textnormal{CRYS}}_{\proet})^{-1}Rf_{*}^{\textnormal{CRYS}}\cO^{\textnormal{CRYS}}_{X/\bZ_p} \to R\Gamma_{A_{\Inf}}(\fX/\fY) \wh{\otimes}_{A_{\Inf}} A_{\textnormal{crys}}
\end{equation}
in $D(D, W(\wh{\cO}^{+, \flat}_Y))$ compatible with the Frobenius such that the cohomology sheaves of the cone of \eqref{eq:crysverssht} 
 are killed by $[\fm^{\flat}] \subset A_{\Inf}$. 
\end{thm}

\begin{proof}
Let $r \colon (\fY/A_{\Inf})_{q-\textnormal{CRYS}} \to (Y/A_{\textnormal{crys}})_{\textnormal{CRYS}}$ be the natural morphism of topoi such that $\mu_{\proet}^{\textnormal{CRYS}} = r \circ \mu_{\proet}$. Then by \cite[Theorem 16.14]{Prisms} $r^{-1}Rf_{*}^{\textnormal{CRYS}}\cO^{\textnormal{CRYS}}_{X/\bZ_p} \cong R\Gamma_{q-\textnormal{CRYS}}(\fX/\fY) \wh{\otimes}_{A_{\Inf}} A_{\textnormal{crys}}$. The result now follows from Theorem \ref{thm:comparofRGqcruandRGaAin}.
\end{proof}

\bibliography{reference.bib}
\bibliographystyle{alpha2}

\end{document}